\documentclass[10pt]{article}
\usepackage[english]{babel}
\usepackage{amsfonts,bm,enumerate,enumitem}
\usepackage[utf8]{inputenc}
\usepackage{amsthm,yhmath}
\usepackage{ mathrsfs }
\usepackage{amsmath,tikz-cd}
\usepackage{mathrsfs}
\usepackage{enumitem}
\usepackage{mathtools}
\usepackage{bbm}
\usepackage{makeidx}
\usepackage{graphicx}
\usepackage[noadjust]{cite}
\usepackage{frontespizio}
\usepackage{color}
\usepackage[makeroom]{cancel}
\usepackage[normalem]{ulem}
\usepackage{ amssymb }
\usepackage{a4wide}
\usepackage{hyperref}
\usepackage{mathtools}
\usepackage{enumitem}
\usepackage[titletoc]{appendix}
\usepackage{esint,soul}

\numberwithin{equation}{section}

\newtheorem{theorem}{Theorem}[section]
\newtheorem{corollary}[theorem]{Corollary}
\newtheorem{lemma}[theorem]{Lemma}
\newtheorem{proposition}[theorem]{Proposition}
\newtheorem{definition}[theorem]{Definition}

\newtheorem{example}[theorem]{Example}
\newtheorem{remark}[theorem]{Remark}

\newcommand{\N}{\mathbb{N}}
\newcommand{\Q}{\mathbb{Q}}
\newcommand{\R}{\mathbb{R}}
\newcommand{\Z}{\mathbb{Z}}

\newcommand{\restr}[1]{\lower3pt\hbox{$|_{#1}$}}

\newcommand{\TV}{{\sf TV}}

\newcommand{\X}{{\rm X}}
\newcommand{\Y}{{\rm Y}}
\renewcommand{\Z}{{\rm Z}}

\newcommand{\sfd}{{\sf d}}

\newcommand{\mm}{{\mathfrak m}}
\newcommand{\z}{{\sf z}}									

\newcommand{\Tr}{{\rm tr}}								

\renewcommand{\L}{{\rm Lin}}

\newcommand{\CD}{{\sf CD}}
\newcommand{\TCD}{{\sf TCD}}

\newcommand{\RCD}{{\sf RCD}}
\newcommand{\MCP}{{\sf Mcp}}
\newcommand{\sMCP}{{\sf sMcp}}

\newcommand{\W}{{\rm W}}

\renewcommand{\d}{{\rm d}}
		
\newcommand{\lims}{\varlimsup}
\newcommand{\limi}{\varliminf}

\newcommand{\nchi}{{\raise.3ex\hbox{$\chi$}}}

\newcommand{\eps}{\varepsilon}
\newcommand{\Id}{{\rm Id}}

\newcommand{\fr}{\penalty-20\null\hfill$\blacksquare$}         

\DeclareMathOperator*{\essinf}{\rm ess-inf}

\newcommand\la{\langle}
\newcommand\ra{\rangle}

\newcommand{\Hom}{{\rm Hom}}
\newcommand{\Lin}{{\rm Lin}}

\newcommand{\ls}{{\sf l}}
\newcommand{\us}{{\sf u}}
\newcommand{\LSC}{{\sf LSC}}

\newcommand{\tip}{{\sf tip}\,}
\newcommand{\p}{{\rm p}}
\newcommand{\lattice}{cone with joins}
\newcommand{\lattices}{cones with joins}
\newcommand{\Lattice}{Cone with joins}
\newcommand{\mcp}{{\sf Mcp}}
\newcommand{\LL}{{\sf L}}
\renewcommand{\O}{{\mathcal O}}
\newcommand{\C}{{\mathbb C}}

\DeclareMathOperator*{\Lsup}{{{\rm Lin}\,-\sup}}
\DeclareMathOperator*{\Linf}{{{\rm Lin}\,-\inf}}
\DeclareMathOperator*{\psup}{{\p\,-\sup}}
\DeclareMathOperator*{\pinf}{{\p\,-\inf}}

\newcommand{\hn}{{\sf hn}}
\renewcommand{\W}{{\rm W}}
\renewcommand{\C}{{\rm C}}
\newcommand{\J}{{\rm J}}
\renewcommand{\t}{{\sf t}}
\newcommand{\s}{{\sf s}}

\title{Hyperbolic Banach spaces  }
\author{Nicola Gigli\ \thanks{SISSA, ngigli@sissa.it}  }

\begin{document}

\maketitle

\begin{abstract}
The standard theory of Banach spaces is built upon the notions of vector space, triangle inequality and Cauchy completeness. Here we propose a `hyperbolic' variant of this `elliptic' framework where general linear combinations are replaced by linear combinations with non-negative coefficients, triangle inequality is replaced by reverse triangle inequality and Cauchy completeness is replaced by the order-theoretic notion of directed completeness.


\medskip

The motivation for our investigation is in  non-smooth Lorentzian geometry: we believe that to unlock the full potential of the field, and ultimately extract more informations about the smooth world, some version of `Lorentzian functional analysis' is needed, especially in relation to timelike lower Ricci curvature bounds.
%

An example of structure we investigate is obtained by starting with a Banach space, multiplying it by $\R$ and considering the `future cone' in there. Because of this, some of the results  in this manuscript  might be read through the lenses of  standard Banach spaces theory. From this perspective, the classical Hahn-Banach and Baire category theorems can be seen as consequences  of statements obtained here.

A different kind of  example is that of  $L^p$ spaces for $p\in[-\infty,1]$. Their structure and  natural duality relations for `H\"older conjugate' exponents $p,q\leq 1$ with $\tfrac1p+\tfrac1q=1$ fit particularly well in our framework, to the extent that they have been  an important  source of inspiration for the axiomatization chosen in this paper.

We also investigate the notion of directed completeness regardless of any algebraic structure, as we believe it is central even in the finite-dimensional non-smooth Lorentzian framework, for instance to achieve a compactness theorem \`a la Gromov. This study unveils  connections between  Geroch--Kronheimer--Penrose's concept of ideal point in a spacetime, Beppo Levi's monotone convergence theorem and certain aspects of domain theory.

\end{abstract}
\newpage

\tableofcontents

\newpage

\begin{section}{Introduction}

\begin{subsection}{Motivation and overview}
Functional analysis is a cornerstone of modern mathematics. Among other things,  it offers powerful existence results valid in a wide variety of settings.

With this paper we want to  lay the ground for a functional-analytic theory valid in Lorentzian signature, aiming at giving a solid framework to non-smooth Lorentzian calculus. Our main message is that, perhaps, such theory is actually possible.

Let us explain what we mean and our motivation by giving some context. In the seminal paper \cite{KS18} the authors introduced a non-smooth analog of the concept of spacetime, ideally playing the same role that the concept of metric space has in Riemannian geometry. The precise axiomatization of these objects is part of an ongoing intense conversation within the community (see e.g.\ \cite{MinSuh24},  \cite{CavMonTCD}, \cite{Octet}) but there is consensus over the fact that, as postulated in \cite{KS18},  the geometry of these objects is encoded by a `time separation' (or Lorentzian distance) function satisfying the reverse triangle inequality. Inspired by the successes of the classical metric theory, one of the main motivations behind this definitions was that of having a purely synthetic theory of curvature bounds in Minkoswkian signature, ultimately leading to a better understanding of smooth Lorentzian manifolds and thus of the mathematics of General Relativity.

This latter scope has been one of the reasons behind the introduction in \cite{CavMonTCD} of the notion of $\TCD$ spaces, i.e.\ nonsmooth spacetimes with Ricci curvature bounded from below in the timelike directions (in analogy with \cite{Lott-Villani09}, \cite{Sturm06I}, \cite{Sturm06II}). Even though already such paper contains non-trivial geometric statements, it is clear that to make full use of the curvature assumptions a sophisticated differential calculus is needed. In this direction some further steps have been made in \cite{Octet}, where a non-smooth causal differential calculus has been developed. Among other things, such calculus allows to define `infinitesimally Minkowskian spacetimes', i.e.\ those structures that from an infinitesimal perspective look like Minkowskian spacetimes, as opposed to more general Finsler-Minkowskian ones (much like infinitesimal Hilbertianity \cite{Gigli12} selects those metric measure spaces that resemble Riemannian manifolds, as opposed to Finsler manifolds, on small scales).

Further development of such non-smooth calculus requires new theoretical framework: building such framework is the main scope of this manuscript. To emphasize the need for this, we notice that virtually all the existence results pertaining differential operators  in the setting of $\CD$/$\RCD$ spaces are ultimately related to existence theorems in functional analysis, such as Banach-Alaoglu's theorem, Riesz' theorem, Lax-Milgram's theorem, and gradient flow theory, just to mention a few. None of these is anymore available in the current setting, as the presence of the reverse triangle inequality prevents any Banach-like notion classically intended to emerge. To give a concrete example, consider the question:
\[
\text{What is the differential of a causal function on a non-smooth spacetime?}
\]
that surely sooner or later will need an answer. The analog one in elliptic signature had a number of technically different answers (see e.g.\ \cite{Cheeger00}, \cite{Weaver01}, \cite{Gigli14}, \cite{EBS24}), depending on the precise set of underlying assumptions, but all of these are grounded on Banach spaces theory. This simply cannot be the case in our setting, that instead requires a reverse triangle inequality to be in place. 

Giving up the standard triangle inequality in favour of the reverse one opens the key question: which sort of completeness are we going to assume for our Banach-like structures, now that the concept of Cauchy sequence makes, de facto, no sense  anymore? Inspired by the study of the causal structure of smooth spacetime and of the basic examples of $L^p$ spaces for $p\leq 1$ we propose to endow our objects with an order structure (naturally induced by their algebraic properties, see  \eqref{eq:deforder} below) and to require such partial order to be directed complete, i.e.\ any directed subset must have a supremum. Reasons behind this proposal are:
\begin{itemize}
\item[i)] The fact that it mirrors Cauchy-completeness of metric spaces (as opposed to global hyperbolicity, that mirrors properness, i.e.\ compactness of closed balls), making it viable in infinite dimensional settings such as those we study here. 

In other words, we have in mind the rough analogy
\begin{equation}
\label{eq:schemino}
\begin{array}{cc}
\text{Riemannian signature } \qquad& \qquad\text{ Lorentzian signature}\\
\hline
\text{proper metric space}  \qquad& \qquad\text{globally hyperbolic spacetime}\\
\text{complete  metric space} \qquad& \qquad\text{(locally) causally directed complete  spacetime}
\end{array}
\end{equation}
see Proposition \ref{prop:2completi} for more on this. In particular, much like Cauchy completeness, directed completeness has nothing to do with the underlying linear structure and thus can be imposed on more general geometries.
\item[ii)] The existence of a general well-behaved completion procedure. This gives robustness to the theory, as one might always start from an incomplete object and then suitably complete it.  See Theorem \ref{thm:excompl}. 

Notably, the two-sided completion  leads to structures undesirable from the perspective of a `working analyst' (see Section \ref{se:twosided}), a fact that is in line with known difficulties  one encounters in adding at the same time  ‘future’ and
‘past’ infinity to a spacetime (see e.g.\ \cite{FHM11} and \cite{MR03}).
\item[iii)] Its natural link, via the concept of reflector in category theory, with maps between partial orders that respect directed suprema. This kind of maps, called Scott continuous in domain theory, emerge in a number of relevant examples ($L^p$ spaces for $p\leq 1$ and in the duality between measures and functions) via Beppo Levi's monotone convergence theorem. For this reason, and also because we are not endowing our spaces with the Scott topology, we call this property  Monotone Convergence Property, or $\mcp $ in short. 

\item[iv)] The fact that the directed completion of Minkowski spacetime consists in adding future time and future null infinity to it, see Example \ref{ex:mink} in Section \ref{se:exbelli}. We regard this as an important sanity check of our proposal, as it shows that such notion is linked to key concepts in General Relativity, from which this whole paper is somehow coming from. We conjecture that equality between directed completion and that coming from Geroch--Kronheimer--Penrose's notion of ideal points is valid in more general physically relevant spacetimes.

In relation to the above,  we recall that categorical aspects of Geroch-Kronheimer-Penrose's construction have already been investigated in \cite{Har98}.

\item[v)] All this is relevant not only in the infinite-dimensional framework studied here, but also in finite-dimensional metric geometry. For instance, in looking for a compactness theorem for spacetimes \`a la Gromov, it is reasonable to expect that one will find a limit spacetime via a completion procedure starting from a dense (possibly countable) structure. 
\end{itemize}

\bigskip

This paper owns a lot to \cite{Octet}, as the initial motivation for initiating this investigation was the desire to better understand $L^p-L^q$ duality for H\"older conjugate exponents $p,q\in[-\infty, 1]$ that was surfacing there. Conversely, some of the findings in  here helped shaping some aspects in  \cite{Octet}. In particular, the above considerations about directed completeness have been a source of inspiration for the concept of `forward complete' spacetime introduced in  \cite{Octet}, there showed to be sufficient to get very general existence results, such as a limit curve theorem valid in absence of global hyperbolicity.

The procedure of taking - some version of - forward completion has been an important ingredient in the development of a Gromov pre-compactness theorem in Lorentzian setting  \cite{MS25}. More precisely, it has been useful to obtain uniqueness of limits and recast a continuous spacetime out of a countable one, as discussed above.

\end{subsection}
\begin{subsection}{Main results}

Our {\bf first main result} is a 
\[
\text{general construction of directed completion of partial orders,}
\]
see Theorem \ref{thm:excompl}. Notice that neither the existence result nor the actual construction are new: in \cite{Mark76} and \cite{DCPO10} the same construction has been given and these two papers offer different proofs, also different from ours, that it produces the desired completion. In a nutshell, the construction consists in taking the Scott-closure of the image of the given partial order into the collection of its Scott-closed subsets.

The added value of our study is in a cleaner understanding of such completion procedure via, in particular, the notion of \emph{set with tip}, see Definition \ref{def:tip}, that can be seen as a sort of generalization of the well known notion of directed set. It is mainly thanks to this concept that we are able to conduct our general study of directed completion and, more importantly, to link it to the algebraic structure of cones, see in particular the definition \ref{def:tdp} of directed decomposition property and the related discussion we give below.

The kind of algebraic structures we want to pair with this concept of completeness  is that of \emph{wedge}, namely a set coming with an operation of `sum' and `product by positive scalars'. Such structures come with a canonical reflexive and transitive relation $\leq$ defined as
\begin{equation}
\label{eq:deforder}
v\leq w\qquad\text{ whenever there is $z$ such that }\qquad v+z=w.
\end{equation}
A key condition that we impose is that this is a partial order, i.e.\ that $v\leq w$ and $w\leq v$ occurs only when $v=w$, see Definition \ref{def:wedge} for the precise formulation.

Notably, we are \emph{not} insisting, and for good reasons, see the discussion  below and Remark \ref{re:perchenocanc}, on the sum being cancellative, i.e.\ we are not imposing that $a+v=b+v$ must imply $a=b$. Still, in this generality we can prove -- see Lemma \ref{le:canc1} -- that
\begin{equation}
\label{eq:canc1intro}
a+v\leq b+v\qquad\Rightarrow\qquad a+\tfrac1n v\leq b+\tfrac1nv\quad\forall n\in\N,\ n>0,
\end{equation}
that in turn can be used to show that the operation of sum has the $\mcp$ -- see Proposition \ref{prop:operazioniMCP}.

We shall call \emph{cone} a wedge that is directed complete.  Cones are the main objects studied in the paper and relevant maps between them are linear ones with the $\mcp$. Surely proposing a functional analytic theory grounded on the concept of cone (:=directed complete wedge), as opposed to that of Banach space (:=Cauchy complete normed vector space) is a relevant change, so we should offer some good reasons for doing so. These are:
\begin{itemize}
\item[i)] Morphisms easily take the value $+\infty$. When this happens, dealing with vector space structures can be rather problematic, as if $L$ is a linear map with $L(v)=+\infty$ and $L(w)=-\infty$, then what is it the value of $L(v+w)$? Dealing only with positive maps  solves this problem.
\item[ii)] The completion procedure works well, from the perspective of the theory we are trying to build, only if mono-directional, which suggests we should either work with directed complete or filtered complete partial orders. See Section \ref{se:twosided} for more on this.
\item[iii)] The choice of directed completeness as opposed to filtered one has to do with the `fields of scalars'. If one accepts to deal with cone structures, then she is naturally led to deal, in the  1-dimensional space, with either $[-\infty,0]$ or $[0,+\infty]$. Of these, only the latter one is closed w.r.t.\ multiplication, somehow breaking the symmetry and hinting at a theory where `increasing to $+\infty$' is more natural than `decreasing to $-\infty$'. A further clue in this direction is given by Beppo Levi's theorem, that somehow tells that  positive functionals on positive objects respect suprema of increasing sequences, not infima of decreasing ones.
\item[iv)] Cones are extremely natural from the perspective of General Relativity, from which this  work is coming from, as  the admissible directions of motions for an object in a spacetime  are those  in its future cone.
\end{itemize}
See also Remark \ref{re:perchenocanc} for more on these. We also point out that such natural emergence of the concept of cone leads to the consideration that, perhaps, instead of  studying tangent `spaces' to non-smooth Lorentzian structures, we should instead consider  tangent `cones' and, more generally, that one might investigate pointed convergence of spacetimes with the chosen point being the order-theoretical minimum.

We mentioned a general completion procedure and the relevance of complete objects paired with some basic algebraic structure. It is then natural to ask whether if we start with a wedge and directly-complete it we obtain a cone. Unfortunately this is not entirely clear due to the possible conflict of two natural partial orders in the completion. See Remark \ref{re:abstrcompl} for comments in this direction and for the  (unsatisfactory) description of an abstract wedge-completion in the same spirit of the classical Stone-\v{C}ech compactification / Freyd’s General Adjoint Functor Theorem /  Yoneda's lemma. See also the next Section \ref{se:venue} for a possible solution to this problem.

Coming back to key definitions, we notice that any cone $\C$ is a directed subspace of itself, because for any $v,w\in \C$ we have $v,w\leq v+w$ by definition. Being the cone complete, it follows that there exists a maximal element in $\C$, that we shall denote $\infty$. By nature, $\infty$ is an absorbing element of the sum, i.e.\ it satisfies
\[
v+\infty=\infty\qquad\forall v\in \C.
\]
Because of this, the Groethendick group associated to a cone is trivial and thus  we should be a bit   careful if we want to associate a vector space to a cone. In many practical situations this is nevertheless possible by first removing elements with non-zero \emph{part at infinity}, see below, and the procedure is  in some sense the inverse one of that associating  to a vector space a positive cone and (directed) completing it, see Section \ref{se:finitedim} for more on this and also the recent \cite{KO25}.

The important concept of part at infinity $\eps v$ of an element $v$ of a cone is defined as the supremum of the set of $w$'s  such that $w\leq \lambda v$ for every $\lambda \in(0,+\infty)$: being this collection closed by addition, it is directed, so the supremum exists. The name and notation are chosen to reflect the fact that  in some sense $\eps v$  is what remains of $v$ after multiplying it by small positive numbers. This concept plays a crucial role in the work, especially in the more restrictive context of \emph{cones with joins}, i.e.\ cones that are also  lattices in the order theoretic sense and in which the `inf' operation is consistent with the sum, see Definition \ref{def:conejoins} for the precise notion.

In a cone with  joins the cancellation property \eqref{eq:canc1intro} above takes the particularly useful form
\begin{equation}
\label{eq:canc2intro}
a+v\leq b+v\qquad\Rightarrow\qquad a+\eps v\leq b+\eps v.
\end{equation}
Among other things, this will be crucial to prove our {\bf second main result}, namely a 
\[
\text{general extension theorem \`a la Hahn-Banach for maps from a wedge to a cone with joins.}
\]
Shortly said, such extension theorem allows to find a linear map $M$ satisfying 
\begin{equation}
\label{eq:sandwich}
\varphi\leq M\leq\psi
\end{equation}
 for given $\varphi$ superadditive and $\psi$ subadditive, assuming $M$ to be initially defined only on a subwedge and satisfying \eqref{eq:sandwich} in there, see Theorem \ref{thm:extnuovo} for the precise formulation.  The lack of the cancellation property as classically intended  makes the proof of our extension statement quite more complicated  than  analogous preexisting ones.  In this direction, it is worth to point out that  applying our  result to the `future cone' of a vector space we naturally recover the classical Hahn-Banach theorem, see Remark \ref{re:implicaHB}. 
 
 Interestingly, and unfortunately, despite this latter remark our extension result is not fully satisfactory in our setting. The problem is that, unlike in the classical `elliptic' world, a bound like \eqref{eq:sandwich} does \emph{not} imply the desired continuity of $M$, i.e.\ the $\mcp$, regardless of what we know about $\varphi$ and $\psi$. In fact, more generally, it seems that the $\mcp$ \emph{cannot be quantified}, posing relevant conceptual challenges in the field, as classical `$\eps-\delta$ arguments' seem to have no role here.
 
 It is surely true that in a variety of settings, for a linear operator `positivity' implies some sort of `continuity', a fact that possibly might have a role also in our setting, but it is not hard in our framework, using precisely our extension theorem, to find linear maps without the $\mcp$.
 
 The question is thus:
 \[
\begin{split}
&\text{Under which conditions on $\varphi,M,\psi$ as in \eqref{eq:sandwich} we can replace $M$ with}\\
&\text{a new linear map $\tilde M$ that still satisfies \eqref{eq:sandwich} but also has the $\mcp$?}
\end{split}
\]
 Out {\bf third main result}, loosely speaking, is
 \[
\begin{split}
\text{having identified a general set of conditions for the above to be possible.}
\end{split}
\]
The crucial assumption is on the source cone: it should have the \emph{directed decomposition property}, so named for the similarity with the Riesz decomposition property. Technicalities apart, this means that if $(v_n)$ is an increasing sequence with supremum $v$ and  $v=w+z$, then there must be increasing sequences $(w_n),(z_n)$ with sup $w,z$ respectively such that  $w_n+z_n\leq v_n$ for every $n$ (see Definition \ref{def:tdp} for the precise notion). 

The last bit of structure we shall consider is a \emph{hyperbolic norm}, i.e.\ a map $\hn$ from a given wedge $\W$ into $[0,+\infty]$ such that
\[
\hn(\lambda  v+\eta w)\geq\lambda \hn(v)+\eta \hn(w)\qquad\forall \lambda,\eta\in[0,+\infty),\ v,w\in\W.
\]
A hyperbolic Banach space is simply a cone equipped with a hyperbolic norm, see Section \ref{se:hbs}. In this paper we shall not impose any additional condition on the norm. In particular, the boring choices $\hn\equiv 0$ and $\hn(v)=+\infty$ for any $v\neq 0$ are admissible, though surely not particularly interesting. Understanding the role of the norm in dictating the geometric/analytic structure of hyperbolic Banach spaces is a topic left open to further explorations in the area. We emphasize that we are not even asking the norm to have the $\mcp$, as it  fails for some basic example we want to keep within the borders of the theory (such as  the $L^p$-norm for $p<0$, see Remark \ref{rem:lpmcp}).

In particular, and unlike the classical case, the norm has nothing to do with the concept of continuity we care about, as here the $\mcp$ property is only related to the order structure. Still, it is natural to use the norm to define the `chronological' topology as that generated by sets of the form 
\[
\begin{split}
I^+(v)&:=\{w\in \C: w=v+z\text{ for some $z$ s.t.\ }\hn(z)>0\},\\
I^-(v)&:=\{w\in \C: w+z=v\text{ for some $z$ s.t.\ }\hn(z)>0\},
\end{split}
\]
as $v$ varies in $\C$. Our {\bf fourth main result} is an
\[
\text{order-theoretic  Baire category theorem  applicable to chronological topologies of cones with joins,}
\]
see Theorem \ref{thm:Baire} and Theorem \ref{thm:BaireJ} and  notice that our version of the Baire category theorem implies the standard one for complete metric spaces, see Remark \ref{re:Baire}.

\medskip

We conclude pointing out that in this manuscript we shall make full use of the Axiom of Choice. This is relevant in particular, unsurprisingly, in the proof of our extension theorem \ref{thm:extnuovo}. Still, we believe and hope that in suitable separable settings, only Countable Dependent Choice will be needed, in line with traditional functional analysis. See Definition \ref{def:sep} for some notions of separability that emerge naturally in our framework.

\end{subsection}
\begin{subsection}{Some venues of further research}\label{se:venue}

The material collected here leaves open several key problems. We collect the main ones.
\begin{itemize}
\item[i)] The value of a new definition is measured on how much it allows to better understand   pre-existing concepts. In this sense, even though this paper leaves open  several natural `internal' problems, we developed this theory with the goal of building a more sophisticated calculus on metric measure spacetimes. Much like the notions given in \cite{Gigli14} had a role in studying $\RCD$ spaces that in turn helped answering questions formulated in the smooth category, so we hope that the notions presented here will be useful for the study of metric spacetimes as introduced in the seminal paper \cite{KS18} (see also the surveys \cite{CavMonLorSurvey}, \cite{McCannSurvey} and \cite{BraunPers} for presentations that are both more recent, more linked to the lower Ricci bounds that this paper is hinting to, and for relevant bibliographical references) and then ultimately helping addressing questions about the shape of smooth Lorentzian manifolds. 

\item[ii)] An important goal we aim to achieve with this paper is to convince the reader that a functional-analytic theory based on the concepts of wedge and linear maps with the $\mcp$ is both possible and interesting. This does not mean, though, that we believe that ours is the only axiomatization one could provide for these notions, nor that it is the best one. Actually, we expect that some fine tuning of the  definitions will occur. 

In particular, perhaps the most questionable choice we made is that of insisting that  the partial order is dictated by the sum operation, as in \eqref{eq:deforder}, rather than it coming as extra structure having some natural compatibility  with the algebraic one. We made this choice both because it is natural and to favour simplicity. Yet, this simplicity might backfire in some case. Most notably, when dealing with wedges of functions, one might prefer to deal with the pointwise order rather than with the sum-induced one. An example where this happens is the `non-linear' wedge of lower semicontinuous and non-negative functions on a given Polish space, that surely we want to be covered by the theory: equipped with the pointwise order this is a very natural complete lattice, but the sum-induced partial order is not even directed complete (see Remark \ref{re:lscnocone} -- still, this does not prevent us from having a hyperbolic interpretation of the classical Riesz' theorem, see Theorem \ref{thm:riesz}).

We point out that there is some wiggle room in our fundamental results, that might therefore accomodate different choices for the partial orders involved, see for instance Remark \ref{re:flexext}.

As stimulus for potential future investigations in the area, we mention a couple of advantages one could gain by giving up the sum-induced partial order in favour of a more general choice:
\begin{itemize}
\item[1)] The operation of directed completion and the algebraic structure interact well, thus producing cones out of wedges can become simpler and more natural. Compare with Remark \ref{re:abstrcompl}.
\item[2)] The space of morphisms (i.e.\ linear maps with the $\mcp$) between two cones is trivially directed-complete w.r.t.\ the pointwise order, thus a cone itself if we accept this more general interpretation of the definition. In particular, the dual of a cone equipped with the pointwise order would always be a cone and every cone would map in its bidual  (not necessarily injectively, though). See Proposition \ref{prop:Wstarcone} and its proof.
\end{itemize}
In line with what said above, we believe that the best axiomatization will be that producing the most relevant results pertaining structures that live outside the theory.

\item[iii)] We carefully avoided proposing a `best' topology and understanding if there is any is crucial for the development of the subject. The difficulties we face in picking one are reminiscent of a known dichotomy one deals with even when studying smooth spacetimes: on one hand, the manifold topology is clearly the correct one, on the other any event sees/is seen only events in its past/future, and these do not form a neighbourhood of the manifold topology. 

We briefly comment on the two most obvious candidates for a topology of a hyperbolic Banach space.   One is the Scott topology (see \cite{CompendiumLattices}) that is naturally linked to the notion of continuity we care about, i.e.\ the $\mcp$. However, the fact that this is typically just $T_0$  (and in particular not Hausdorff) makes it an hard call from an analyst' perspective. The other is the chronological topology, that emerges  also in the  smooth category and for which we have the existence result Theorem \ref{thm:BaireJ}. Still,  this in general has no connection with the $\mcp$ and in simple situations can be the discrete topology (see the discussion at the end of Section \ref{se:Baire}).

We also notice that, even though the definitions are very natural,  the weak$^*$-like topologies that we discuss in Section \ref{se:wstar} appear to be a bit artificial and as of now we do not have relevant existence results  related to them.

This problem is related to the difficulties in associating an intrinsically defined `nice' topology to sets equipped with just a time separation function (but see \cite[Section 1]{McCann24} for some positive results), that in our setting is related  to the unclear role of the hyperbolic norm. In this direction, notice that our choice of insisting on the causal structure  is reminiscent of Minguzzi's idea of regarding a spacetimes through the eyes of Nachbin's notion of  topological spaces equipped with closed orders (see \cite{Ming10}): further studies are needed to understand whether this is just a formal analogy or there is a deeper connection.

Related to all this, it is worth to repeat that the lack of a classically intended triangle inequality makes it complicated to give a meaning to `quantitative estimates'. For instance, there seems to be no way of making quantitative the notion of $\mcp$, not even for linear functionals. This of course is a big difference with  the classical Banach setting, where  the operator norm quantifies the continuity of a linear map.
 
\end{itemize}

\end{subsection}

\noindent{\bf Acknowledgments} 
I am  grateful to D. Trevisan for the example discussed in Section \ref{se:dario}, to A. Lerario for the one in Section \ref{se:antonio} and to R. Oleinik for conversations on the basic structure of wedges and in particular for Example \ref{ex:Roman} in Section \ref{se:finitedim}. I thank them also for the permission of  sharing these.

This study was funded by the European Union - NextGenerationEU, in the framework of the PRIN Project Contemporary perspectives on geometry and gravity (code 2022JJ8KER – CUP G53D23001810006). The views and opinions expressed are solely those of the author and do not necessarily reflect those of the European Union, nor can the European Union be held responsible for them.

\end{section}

\begin{section}{Directed completion of partial orders}

\begin{subsection}{Definition and construction}
\label{se:constr}
A {\bf partially ordered} set $(\X,\leq)$ is a set $\X$ equipped with a relation $\leq$ that is reflexive, antisymmetric and transitive, i.e.
\[
\begin{split}
x&\leq x,\\
x&\leq y\quad \text{and}\quad y\leq x\qquad\Rightarrow\qquad x=y,\\
x&\leq y\quad\text{and}\quad y\leq z\qquad\Rightarrow\qquad x\leq z,
\end{split}
\]
for every $x,y,z\in\X$.  A {\bf directed} subset $D\subset\X$ is a subset so that for any finite subset $F\subset D$ there is $z\in D$ with $x\leq z$ for any $x\in F$ (notice that this applies also to $F$ being the empty set, forcing directed subsets to be non-empty).  A  {\bf lower} set $L\subset \X$ is  subset so that  $x\in L$ and $y\in \X$ with $y\leq x$ imply $y\in L$.

Given $A\subset\X$ the {\bf supremum} (or join) $\sup A\in\X$ is an element   $s\in \X$  so that:
\begin{itemize}
\item[i)] $x\leq s$ for any  $x\in A$  
\item[ii)] if $s'\in\X$ is such that $x\leq s'$ for any $x\in A$, then $s\leq s'$.
\end{itemize}
A supremum does not necessarily exists, but it is trivially unique: if $s_1,s_2$ are both suprema, then   $s_1\geq s_2$ and $s_2\geq s_1$, from which it  follows $s_1=s_2$.

A {\bf directed-complete} partial order (dcpo) is a partial order $(\X,\leq)$ for which every directed subset has a supremum.

A subset $A$ of a partial order $\X$ is called  {\bf directed-sup-closed} if for any directed subset of $A$ that admits supremum in $\X$ we have that such supremum belongs to $A$. 

We define two closure operators: given a partial order $(\X,\leq)$ and an arbitrary $A\subset\X$ we put
\begin{equation}
\label{eq:chiusure}
\begin{split}
\bar A&:=\text{smallest $B\subset\X$ containing $A$ and directed-sup-closed},\\
\widehat A&:=\text{smallest $B\subset\X$ containing $A$, lower set and directed-sup-closed},
\end{split}
\end{equation}
where in both cases `smallest' is intended in the sense of inclusion. The definitions are well posed because, trivially, the intersection of any arbitrary family of directed-sup-closed (resp.\ lower) subsets is still a directed-sup-closed set (resp.\ lower set). Notice, though, that   the intersection of directed subsets is not necessarily directed. It is simple to check, and well known, that these are closure operators in the sense of Kuratowski. We quickly give the details. Since the empty set is both directed-sup-closed and a lower set, we see that $\bar\emptyset=\widehat\emptyset=\emptyset$. The inclusion $A\subset\bar A$ and  identity $\bar{\bar A}=\bar A$ are obvious from the definition; same for $\widehat A$. Thus it remains to prove that $\overline{A\cup B}\subset \bar A\cup\bar B$ and that $\widehat{A\cup B}\subset \widehat A\cup\widehat B$. For this it suffices to prove that the union of two directed-sup-closed  sets is still directed-sup-closed (applying this to $\bar A$ and $\bar B$ we conclude that $\overline{A\cup B}\subset \bar A\cup\bar B$  and applying to $\widehat A$ and $\widehat B$ and noticing that union of lower sets is still a lower set we get $\widehat{A\cup B}\subset \widehat A\cup\widehat B$). Thus say that $A,B\subset\X$ are directed-sup-closed and let $D\subset A\cup B$ be directed and having sup $s$ in $\X$. We want to prove that $s\in A\cup B$. If there is $d\in D$ such that every $e\in D$ with $d\leq e$ belongs to $A$, then clearly the collection of such $e$'s is a directed set with $\sup E=\sup D=s$ and since $A$ is directed-sup-closed we conclude $s\in A$. Otherwise, for any $d\in D$ there is $e\in D\cap B$ with $d\leq e$: this suffices to conclude that $D\cap B$ is directed with $\sup (D\cap B)=\sup D=s$, and since $D\cap B\subset B$ we also get $s\in B$.

The topology induced by $A\mapsto\widehat A$ is known as Scott topology (see e.g.\ \cite{CompendiumLattices}) and that induced by $A\mapsto \bar A$ known as $D$-topology (see \cite{DCPO10}).

In studying the operator $A\mapsto \bar A$ it is natural to define the somewhat analog operator $A\mapsto A^\uparrow$ as
\[
A^\uparrow:=\{\text{suprema of directed sets $D\subset A$ admitting supremum in $\X$}\}.
\]
It is clear that $A^\uparrow$ can be used to detect if a set is directed-sup-closed:
\begin{equation}
\label{eq:easybar}
A=\bar A\quad\text{(i.e.\ $A$ is directed-sup-closed)}\qquad\Leftrightarrow\qquad A=A^\uparrow.
\end{equation}
It is also clear that $A\subset A^\uparrow\subset \bar A$. It is simple to see that in general the second inclusion may be strict - see for instance Example \ref{ex:doppiafreccia} in Section  \ref{se:techex} (that it is essentially the same example as in  \cite{JohnstoneScott}), so one might iterate the operation and produce the sequence $A\subset A^\uparrow\subset (A^\uparrow)^\uparrow\subset\cdots$. More generally,  we can proceed by (transfinite) induction and  define $A^{\uparrow^\alpha}$ for every ordinal $\alpha$ by putting:
\begin{equation}
\label{eq:arrowalpha}
\begin{array}{lll}
A^{\uparrow^0}&\!\!\!:=A,\\
A^{\uparrow^{\alpha+1}}&\!\!\!:=(A^{\uparrow^{\alpha}})^\uparrow,&\qquad\text{ for every ordinal $\alpha$},\\
A^{\uparrow^\alpha}&\!\!\!:=\bigcup_{\beta<\alpha}A^{\uparrow^\beta},&\qquad\text{ for every limit ordinal $\alpha>0$.}
\end{array}
\end{equation}
A trivial transfinite induction argument shows that $A^{\uparrow^\alpha}\subset \bar A$ for every ordinal $\alpha$.  The  map $\alpha\mapsto A^{\uparrow^\alpha}$ cannot be injective (as is defined from a proper class to the set of subsets of $\X$), hence for some $\alpha\lneq \beta $ we must have $A^{\uparrow^\alpha}=A^{\uparrow^\beta}$. Since  $\alpha\mapsto A^{\uparrow^\alpha}$ is also non-decreasing (w.r.t.\ inclusion), we see that $A^{\uparrow^\alpha}=A^{\uparrow^\gamma}$ for every $\alpha\leq\gamma\leq\beta$, and in particular for $\gamma:=\alpha+1$. Thus from \eqref{eq:easybar} we see that $\bar A=A^{\uparrow^\alpha}$. In other words, we established that
\begin{equation}
\label{eq:barordinali}
\bar A=A^{\uparrow^\alpha}\qquad\text{for every ordinal $\alpha$ sufficiently big}
\end{equation}
(of course, this is just an instance of the general phenomenon telling that a monotone map from the ordinals to a set must be eventually constant).
Notice that Example \ref{ex:alphafreccia} in Section \ref{se:techex} shows that for any ordinal $\alpha$ we can find a partial order $(\X,\leq)$ and $A\subset\X$ such that $A^{\uparrow^\alpha}\neq \bar A$. 

A similar construction shows that 
\begin{equation}
\label{eq:hatordinali}
\text{iterating the operator $\qquad A\ \mapsto\ \downarrow\!(A^\uparrow)\qquad$ we eventually end up in $\widehat A$,}
\end{equation}
where $\downarrow \!A:=\{x\in\X:x\leq a\text{ for some }a\in A\}$ (if $A=\{a\}$ is a singleton we simply write $\downarrow\!a$).

 A different viewpoint on \eqref{eq:barordinali} is obtained via the concept   of {\bf tip} of a set: 
\begin{definition}[Tip of a set]\label{def:tip}
We say that $A\subset\X$ has a tip provided $\bar A$ has a maximum and in this case such maximum is the tip of $A$. The tip of $A$ will be denoted $\tip A$. 
\end{definition}
For instance, singletons have tips (the element itself) and more generally directed sets having the sup also have the tip (it being the same as the sup).  Example \ref{ex:doppiafreccia}  in Section \ref{se:techex}  shows that there are sets admitting tips that are not directed. 

We point out a key conceptual difference between the property of being directed and that of admitting tip: the former is intrinsic, i.e.\ only concerns the order relation between elements of the subset considered, while the latter is extrinsic, i.e.\ it depends on whether an element with certain given properties exists possibly  outside the subset.

 It is obvious that if $A$ has a tip it also has supremum, it being the tip itself. The viceversa might not hold: if $x,y\in\X$ are two unrelated points admitting supremum, the set $\{x,y\}=\overline{\{x,y\}}$ does not have a maximum, and thus $\{x,y\}$ does not have a tip.

We claim that
\begin{equation}
\label{eq:bartip}
\bar A=\{\text{tips of sets $B\subset A$ admitting a tip}\}.
\end{equation}
Indeed, since $\tip B\in\bar B\subset \bar A$ for every  $B$ as above, the inclusion  $\supset$ holds. For  $\subset$ we notice that since singletons have tips, the right hand side contains $A$, hence to conclude it suffices to prove that it is directed-sup-closed. Let thus $D$ be directed, contained in the right hand side and admitting sup in $\X$, call it $s$. Then for every $d\in D$ there is $B_d\subset A$ with $\tip B_d=d$. We put   $B:=\cup_{d\in D}B_d\subset A$ and claim that  $\tip B=s$. If we show this we are done. Since clearly $s$ is an upper bound for $B$, to conclude we need to prove that $s\in\bar B$. To see this notice that by the choice of $B_d$ we know that for every $d\in D$ we have $d\in \bar B_d\subset\bar B$. In other words we have $D\subset \bar B$ and thus $s=\sup D\in \bar B$.

In the following we shall often use the following trivial fact:
\begin{equation}
\label{eq:tiparrow}
\text{$B\subset\X$ has tip}\qquad\Rightarrow\qquad \widehat B=\,\downarrow\!\tip B.
\end{equation}
Indeed, since $\tip B$ is an upper bound of $B$ we have $B\subset \,\downarrow\!\tip B$ and thus $ \widehat B\subset\,\downarrow\!\tip B$. Conversely, from $\tip B\in \bar B\subset\widehat B$ we deduce that $\downarrow\!\tip B\subset \widehat B$. In particular, if $B$ has tip then $\widehat B$ has a maximum. The viceversa does not hold, see Example \ref{ex:notip} in Section \ref{se:techex}.

\begin{definition}[Monotone Convergence Property] Let $(\X_1,\leq_1)$ and $(\X_2,\leq_2)$   be two partial orders and $T:\X_1\to\X_2$  a map. We say that $T$ has the Monotone Convergence Property ($\mcp$ in short) if it respects directed suprema, i.e.\ whenever for any $D\subset\X_1$ directed having supremum the set $T(D)\subset\X_2$ also has supremum and it holds
\begin{equation}
\label{eq:defmcp}
T(\sup D)=\sup T(D).
\end{equation}
\end{definition}
The choice $D:=\{x,y\}$ with $x\leq_1 y$ in \eqref{eq:defmcp} shows that any such $T$ must be monotone. Some equivalent characterizations of maps with the $\mcp$ are collected in the following lemma. Below for a set $A$ and element $u$ we shall write $A\leq u$ meaning that $u$ is an upper bound for $A$, i.e.\ that $a\leq u$ for any $a\in A$.
\begin{lemma}\label{le:MCP}
Let $(\X_1,\leq_1)$ and $(\X_2,\leq_2)$ be two partial orders and $T:\X_1\to\X_2$. Then the following are equivalent:
\begin{itemize}
\item[a)] $T$ has the Monotone Convergence Property,
\item[b)] $T$ is monotone and for every $A\subset\X_1$ and $u\in\X_2$ with $T(A)\leq_2 u$ we have $T(\widehat A)\leq_2 u$,
\item[c)] $T$ is monotone and for every $A\subset\X_1$ we have $T(\bar A)\subset\overline{T(A)}$,
\item[d)] For every $A\subset\X_1$ we have $T(\widehat A)\subset\widehat{T(A)}$,
\item[e)] For every $A\subset\X_1$ having tip the set $T(A)\subset\X_2$ has tip and $T(\tip A)=\tip T(A)$.
\end{itemize}
\end{lemma}
\begin{proof}\ \\
\noindent{$(a)\Rightarrow(d)$}  Fix $A\subset\X$, put $A^0:=A$ and recursively define $A^{\alpha+1}:=\downarrow\! ((A^\alpha)^\uparrow)$ and, for limit ordinals $\alpha$, put $A^\alpha:=\cup_{\beta<\alpha}A^\beta$. Then \eqref{eq:hatordinali} tells that $A^\alpha=\widehat A$ for any ordinal $\alpha$ sufficiently big, so to conclude it suffices to prove that $T(A^\alpha)\subset \widehat{T(A)}$ for any $\alpha$. This will be done by transfinite recursion. It is obvious for $\alpha=0$ and that $T(A^\beta)\subset \widehat{T(A)}$ for any $\beta<\alpha$ implies $T(A^\alpha)\subset \widehat{T(A)}$ if $\alpha$ is a limit ordinal. Now assume that $T(B)\subset \widehat{T(A)}$ holds for some set $B$ and notice that the $\mcp$ and the fact that $ \widehat{T(A)}$ is directed-sup-closed ensure that $T(B^\uparrow)\subset  \widehat{T(A)}$. Also, we have already noticed that maps with the $\mcp$ are monotone and since $ \widehat{T(A)}$ is a lower set we also deduce that $T(\downarrow\!(B^\uparrow))\subset  \widehat{T(A)}$. This shows that if $T(A^\alpha)\subset \widehat{T(A)}$ then we have $T(A^{\alpha+1})\subset \widehat{T(A)}$, concluding the proof by induction.\\
\noindent{$(d)\Rightarrow(b)$}  Picking $A:=\{a\}$ in $(d)$ and noticing that, trivially, we have $\hat{a}=\downarrow\!a$, we see that $T$ is monotone. The same observation ensures that  $T(A)\leq_2 u$ (i.e.\ $T(A)\subset\,\downarrow\!u$) implies $\widehat{T(A)}\leq_2 u$, as  $\widehat{T(A)}\subset\widehat{\downarrow\!u}=\downarrow\!u$. \\
\noindent{$(b)\Rightarrow(a)$} Let $D\subset\X$ be directed admitting supremum, call it $s$. The assumed monotonicity implies $T(D)\leq_2 T(s)$, thus to conclude we need to prove that if $T(D)\leq_2 u$ , then $T(s)\leq_2 u$. But this is obvious, because our assumption tells that  $T(\widehat D)\leq_2 u$, so the claim follows noticing that $s\in \bar D\subset\widehat D$.\\
\noindent{$(a)\Rightarrow(c)$} Monotonicity of maps with the $\mcp$ has been observed before. Then observe that if $B\subset\X_2$ is directed-sup-closed and contains $T(A)$, then the $\mcp$ implies that $T(A^\uparrow)\subset B$ as well. The conclusion follows by transfinite recursion as in the implication $(a)\Rightarrow(d)$ (recall \eqref{eq:barordinali}).\\
\noindent{$(c)\Rightarrow(e)$} By monotonicity $T(\tip A)$ is an upper bound of $T(A)$ and thus of $\overline{T(A)}$. On the other hand, since $\tip A\in \bar A$, from the assumption we also get $T(\tip A)\in\overline{T(A)}$. We thus proved that $T(\tip A)=\max\overline{T(A)}=\tip T(A)$, as desired.\\
\noindent{$(e)\Rightarrow(a)$} Let $A\subset\X_1$ be directed admitting supremum. Then such supremum is also the tip and the assumption that $T(A)$ has tip implies in particular that $T(A)$ has supremum. Then $T(\tip A)=\tip T(A)$ reduces to \eqref{eq:defmcp}.
\end{proof}

The characterization in item $(d)$ shows that $T$ has the $\MCP$ if and only if it is continuous w.r.t.\ the Scott topologies on both the source and the target spaces.  For this reason, maps of this kind are called Scott continuous in the literature, especially that related to domain theory (see e.g.\ \cite{CompendiumLattices} and references therein).

We are interested in the completion process associated to the concept of directed complete partial order and maps with the Monotone Convergence Property. The following definition is very natural, being that of reflector of the inclusion of the full subcategory of dcpo into that of partial orders (see e.g.\ \cite{AbstrConcrCat} for more on this concept), where in both cases morphisms are maps with the $\MCP$ - notice that this makes sense because the identity map on any partial order has the $\MCP$ and that composition of maps with the $\MCP$ has the $\MCP$:
\begin{definition}[Directed completion]\label{def:completamento}
Let $(\X,\leq)$ be a partial order. A directed completion of $(\X,\leq)$ is given by a directed complete partial order $(\bar\X,\bar\leq)$ and a map $\iota:\X\to\bar\X$ with the Monotone Convergence Property that is universal in the following sense: for any directed complete partial order $(\Z,\leq_\Z)$ and map $T:\X\to\Z$ with  the $\MCP$ there is a unique map $\bar T:\bar\X\to\Z$ with the $\MCP$ so that $\bar T\circ\iota=T$, i.e.\ making the following diagram commute
\begin{equation}
\label{eq:diagrammino}
\begin{tikzcd}[node distance=1.5cm, auto]
\X\arrow{dr}[swap]{T} \arrow{r}{\iota}&\bar\X\arrow{d}{\bar T}\\
& \Z
\end{tikzcd}
\end{equation}
\end{definition}
A standard and easy to check property of these sort of definitions is that the resulting object is unique up to unique isomorphism. More precisely, if $(\tilde\X,\tilde\leq)$ and $j:\X\to\tilde\X$ is another directed completion, then there is a unique monotone bijection $\varphi:\bar \X\to\tilde\X$ whose inverse is also monotone (in particular both $\varphi$ and its inverse have the $\MCP$) such that $\varphi\circ \iota=j$. Indeed, the universal properties of the two completions produce maps $\varphi:\bar\X\to\tilde\X$ and $\psi:\tilde\X\to\bar\X$ with the $\MCP$ such that $\varphi\circ\iota=j$ and $\psi\circ j=\iota$. Hence $\psi\circ\varphi:\bar\X\to\bar\X$ has the $\MCP$ and satisfies $(\psi\circ\varphi)\circ\iota=\iota$, thus by the uniqueness of $\bar T$ (with $T=\iota$) we deduce that $\psi\circ\varphi$ is the identity on $\bar\X$. Similarly, $\varphi\circ\psi$ is the identity on $\tilde\X$, proving that $\varphi,\psi$ are one the inverse of the other and the claim.

The question is thus existence. As usual, an explicit construction is found within the power set, that is complete w.r.t.\ inclusion. In our case, given the partial order $(\X,\leq)$ we shall actually look at
\begin{equation}
\label{eq:defY}
\Y:=\Big\{A\subset \X\ :\ \widehat A=A\Big\}.
\end{equation}
Notice that $\Y$, equipped with the order given by the inclusion, is a directed complete partial order: even more, any family $(A_i)\subset\Y$ admits a supremum, it being given by $\widehat{\cup_iA_i}$ (the proof that this is the supremum is trivial --  this claim shows that $\Y$ is a complete lattice).

There is a natural embedding of $\X$ into $\Y$, defined as
\begin{equation}
\label{eq:iotay}
\iota(x):=\,\downarrow\! x.
\end{equation}
The couple given by $\Y$ and $\iota$ is not the completion of $\X$, the problem being that $\Y$ is too big. For instance, to an $\X$ that only contains  two unrelated points, and is therefore already complete, is associated a $\Y$ that also contains a `bottom' and a `top' point.

Still, completion can be found within $\Y$, as illustrated in the following theorem, see also \cite{Mark76} and \cite{DCPO10} for earlier, different, proofs of existence of directed completion:
\begin{theorem}[Existence of directed completion]\label{thm:excompl}
Let $(\X,\leq)$ be a partial order. 

Then a directed completion $(\bar\X,\bar\leq)$ and $\iota:\X\to\bar\X$ exists. Moreover:
\begin{itemize}
\item[i)] For any $x\in\bar\X$ the set $\iota(\X)\cap\downarrow\! x\subset\bar\X$ has tip, it being precisely $x$,
\item[ii)] For any $(\Z,\leq_\Z)$ directed complete partial order and $T:\X\to\Z$ with the Monotone Convergence Property, the map $\bar T:\bar\X\to\Z$ as in \eqref{eq:diagrammino} given by the very definition of completion is given by
\begin{equation}
\label{eq:chieetbar}
\bar T(\bar x)=\sup_{x\in\X\atop \iota(x)\leq\bar x}T(x)\qquad\forall \bar x\in\bar\X,
\end{equation}
where it is part of the claim the fact that the supremum exists. 
\item[iii)] An explicit definition of $\bar\X$ and $\iota$ is given by the choice
\begin{equation}
\label{eq:defbax}
\bar\X:=\overline{\iota(\X)}\ \subset\ \Y
\end{equation}
where here $\iota:\X\to\Y$ is the map defined in \eqref{eq:iotay} and the directed-sup-closure of $\iota(\X)$ is taken in $\Y$.

 With this choice, an element of $\Y$, i.e.\ a subset $A$ of $\X$ with $A=\widehat A$, belongs to $\bar \X$ if and only if the set $\{\iota(a):a\in A\}\subset\Y$ has a tip. In this case the tip is exactly $A$ (compare with \eqref{eq:bartip}).
\end{itemize}
\end{theorem}
\begin{proof}
We shall prove that $\bar\X$ as defined in \eqref{eq:defbax} and the map $\X\ni x\mapsto \iota(x)=\,\downarrow\! x\in\bar\X\subset\Y$ has the required universal property. We shall prove so via transfinite induction, proving at the same time property \eqref{eq:chieetbar}.  For every $\alpha$ ordinal let us put for brevity $\X_\alpha:=\iota(\X)^{\uparrow^\alpha}\subset\Y$, where $\iota(\X)^{\uparrow^\alpha}$ is defined as in \eqref{eq:arrowalpha}. Then by \eqref{eq:barordinali} we know that $\X_\alpha=\bar\X$ for every $\alpha$ sufficiently big.

Let us  now fix a directed complete partial order $(\Z,\leq_\Z)$ and a map $T:\X\to\Z$ with the Monotone Convergence Property. We shall prove existence, uniqueness and property \eqref{eq:chieetbar} of $\bar T$  by induction. Start defining $T_0:\X_0\to\Z$ as $T_0(\iota(x)):=T(x)$. Trivially, this is the only definition for $T_0$ for which we have $T_0\circ\iota=T$. Also, since $T$ has the $\MCP$, and in particular is monotone, we see that the property 
\begin{equation}
\label{eq:Talpha}
A\in \X_\alpha\qquad\Rightarrow\qquad T_\alpha(A)= \sup_{a\in A}T(a)
\end{equation}
holds for $\alpha=0$. We shall prove by induction that for every $\alpha$ there is a unique map $T_\alpha:\X_\alpha\to\Z$ with the $\MCP$ that extends those previously defined and that such extension satisfies \eqref{eq:Talpha}. 

Suppose this has been proved for $T_\alpha:\X_\alpha\to\Z$ and let us prove the claims for $T_{\alpha+1}:\X_{\alpha+1}\to\Z$. Let $(A_i)_{i\in I}\subset\X_\alpha$ be a directed family and $A\in\X_{\alpha+1}$ its supremum, i.e.\ $A=\widehat{\cup_iA_i}\subset\X$.  Since $T_\alpha$ is monotone, the family $(T(A_i))_{i\in I}$ is directed in $\Z$ and thus admits supremum: clearly, any extension $T_{\alpha+1}$ of $T_\alpha$ with the $\MCP$ must satisfy $T_{\alpha+1}(A)=\sup_iT(A_i)$. It is a priori not clear, though, whether $\sup_iT(A_i)$ only depends on $A$, and not on the particular choice of the directed family $(A_i)$ having $A$ as supremum. To prove that this is the case we notice that
\begin{equation}
\label{eq:quasiapiu1}
\sup_iT_\alpha(A_i)\stackrel{\eqref{eq:Talpha}}=\sup_{a\in\cup_iA_i}T(a)\stackrel*=\sup_{a\in A}T(a),
\end{equation}
where in $*$ we use the $\MCP$ of $T:\X\to\Z$. This suffices to prove that the definition $T_{\alpha+1}(A):=\sup_iT(A_i)$ is well posed, showing at once existence and uniqueness of the desired extension as well as property \eqref{eq:Talpha} for $\alpha+1$.

Suppose now that $\alpha>0$ is a limit ordinal and that $T_\beta$ has been defined and satisfies \eqref{eq:Talpha} for each $\beta<\alpha$. In this case, being $\X_\alpha$ equal to $\cup_{\beta<\alpha}\X_\beta$ there is a unique possible extension $T_\alpha$ of the $T_\beta$'s to $\X_\alpha$ and clearly $T_\alpha$ satisfies \eqref{eq:Talpha}. Let us check that $T_\alpha:\X_\alpha\to\Z$ has the $\MCP$. Thus let $(A_i)_{i\in I}\subset \X_\alpha$ be a directed family having supremum $A$ in $\X_\alpha$. As before, we have $A=\widehat{\cup_iA_i}$. We have
\[
T_\alpha(A)\stackrel{\eqref{eq:Talpha}}=\sup_{a\in A} T(a)\stackrel{*}=\sup_{a\in\cup_iA_i}T(a)\stackrel{\eqref{eq:Talpha}}=\sup_iT_\alpha(A_i),
\]
where in $*$ we use the $\MCP$ of $T:\X\to\Z$. We thus established that the only extension of the $T_\beta$'s has the $\MCP$ and satisfies \eqref{eq:Talpha}.

This concludes the proof of the universal property of $\bar\X$ and $\iota$. Also, with this choice of $\bar\X$ and $\iota$ the claimed property \eqref{eq:chieetbar} is a restatement of \eqref{eq:Talpha}. The fact that  \eqref{eq:chieetbar} holds regardless of the particular representation of $\bar\X,\iota$ trivially follows by the universal property (equivalently: from  uniqueness of the completion).

Claim $(iii)$ is now a direct consequence of the definition \eqref{eq:defbax} and the property \eqref{eq:bartip} applied with $A:=\iota(\X)$. For this explicit version of the completion, the claim $(i)$ is also clear for the same reasons just recalled, and it is obvious that the claim is independent on the particular representation of $\bar\X,\iota$. 
\end{proof}
\end{subsection}

\begin{subsection}{Some comments}\label{se:commcompl}

Let us start noticing that recognising whether a map has the $\mcp$ is a task that  in some cases gets simplified by the structure of the source space:
\begin{definition}\label{def:sep}
We say that a partial order $(\X,\leq)$ is \emph{first countable} provided for any $D\subset\X$ directed admitting supremum there is a non-decreasing sequence $(x_n)\subset\, \downarrow\! D$ having the same supremum.

We say that is is \emph{second countable} if there is an at most countable set $N$ such that for any $D\subset\X$ directed admitting supremum there is a non-decreasing sequence $(x_n)\subset N\cap \!\downarrow\! D$ having the same supremum. 
\end{definition}
It is easy to see that for first countability we might equivalently ask $(x_n)\subset D$.

The relevance of first countability is due to the following statement. Second countability, instead, will not play a role in this manuscript.
\begin{proposition}\label{prop:sepmcp}
Let $\X,\Y$ be partial orders, with $\X$ first countable. Then $T:\X\to\Y$ has the $\mcp$ if and only if  
\begin{equation}
\label{eq:mcpsep}
T(\sup_nx_n)=\sup_nT(x_n)\qquad\forall (x_n)\subset \X\text{ non decreasing and admitting supremum.}
\end{equation}
\end{proposition}
\begin{proof}
The `only if' is clear. For the `if' we notice that picking $(x_n)$ with finite range, from assumption \eqref{eq:mcpsep} we   get that $T$ is monotone. Then we pick $D\subset\X$  directed with sup $s$ and $(x_n)\subset\,\downarrow\! D$ non-decreasing with $\sup_nx_n=s$. Then clearly $T(s)$ is an upper bound for $T(D)$ and if $y\in\Y$ is another upper bound then it is also an upper bound for $T(x_n)$ for every $n\in\N$ and we conclude noticing that the assumption \eqref{eq:mcpsep} gives $T(s)=\sup_nT(x_n)\leq y$.
\end{proof}
\begin{remark}{\rm Notably, many examples of hyperbolic Banach spaces  we are going to consider are first countable  in the above sense, see for instance Lemma \ref{le:esssup}, Proposition \ref{prop:2completi} and Proposition \ref{prop:lsc}.  Easy examples show that the directed completion of a first countable  is not necessarily first countable, see e.g.\ Example \ref{eq:complsubset} in Section \ref{se:exbelli} or the closely related Remark \ref{re:noalcomplet}. 
}\fr\end{remark}
We now investigate further the concept of directed completion. A subset $A$ of a partial order $\X$  is called {\bf dense} provided $\bar A=\X$. It is not hard to find examples of dense subsets such that $\X$ is not the directed completion of $A$ (being intended that $A$ has the induced partial order and the map $\iota$ is the inclusion). Indeed, if one calls  $B,j$ the directed completion of $A$ and $\bar\iota:B\to\X$ the map associated to the inclusion $\iota:A\to \X$, then  $\bar\iota$ is, in general, neither injective nor surjective (see e.g.\ Example \ref{eq:densecompl} in Section \ref{se:techex}).

Because of this, it is important to have a criterion to recognize directed completion: this is the scope of the following proposition.
\begin{proposition}[Recognising the completion]\label{prop:criterio} Let $(\X,\leq)$ and $(\Y,\leq)$ be partial orders, with $\Y$ directed complete and $\j:\X\to\Y$  a map.

Then the following are equivalent:
\begin{itemize}
\item[i)] $(\Y,\j)$ is the directed completion of $\X$, 
\item[ii)]  $\j(\X)$ is dense in $\Y$ and for any $B\subset\X$  such that $\j(B)\subset\Y$ has tip we have $\j^{-1}(\widehat{\j(B)})=\widehat B$, i.e.\ for any $a\in\X$ we have
\begin{equation}
\label{eq:criterio}
\j(a)\in\widehat{\j(B)}\subset\Y\qquad\Leftrightarrow\qquad a\in\widehat B\subset\X.
\end{equation}
\end{itemize}
\end{proposition}
\begin{proof}\ \\
\noindent{$(i)\Rightarrow(ii)$} Even though a direct proof based on the explicit construction provided by Theorem \ref{thm:excompl} is possible, we shall argue via universal property. We start proving  \eqref{eq:criterio}. The implication $\Leftarrow$ comes from the $\MCP$ of $\j$. The converse implication might be proved directly from  the explicit construction of $\bar\X$ in Theorem \ref{thm:excompl}; we give also an argument based on the universal property.  Suppose $\Rightarrow$ in  \eqref{eq:criterio} does not hold, i.e.\ that there are $B\subset\X$ so that $\j(B)$ has tip ${\sf t}$ and $a\in\X$ with $\j(a)\in\widehat{\j(B)}=\,\downarrow\!\!{\sf t}$ and $a\notin \widehat B$. Let $\Z:=\{0,1\}$ with the order $0\leq 1$ (this is obviously complete) and $T:\X\to\Z$ be defined as $0$ on $\widehat B$ and 1 on $\X\setminus\widehat B$. Since $\widehat B$ is a lower set and directed-sup-closed, the map $T$ has the $\MCP$ and thus by the universal property there is $\bar T:\Y\to \Z$ with the $\MCP$ so that $\bar T\circ\j=T$. Since $\j(a)\leq{\sf t}$ we must have $1=T(a)=\bar T(\j (a))\leq \bar T({\sf t})$. On the other hand, the $\MCP$ of $\bar T$ yields $\bar T({\sf t})=\sup_{y\in\j (B)}\bar T(y)=\sup_{x\in B}\bar T(\j(x))=\sup_{x\in B}T(x)=0$, giving the desired contradiction.

We now prove that $\j(\X)$ is dense. Let $\Y_1,\Y_2$ be two copies of $(\Y,\leq)$ and $\j_1,\j_2$ the corresponding maps from $\X$ to $\Y_1,\Y_2$, respectively. Let  $\hat\Y$ be obtained by  identifying the directed sup-closures of the images of $\X$ via $\j_1,\j_2$ in  the disjoint union  $\Y_1\sqcup\Y_2$, i.e.\ we put the equivalence relation $\sim$ on $\Y_1\sqcup\Y_2$ declaring two different points equivalent iff they are of the form $\tip \j_1(A)$ and $\tip \j_2(A)$ for some $A\subset\X$. We put a partial order on $\hat\Y$ by declaring $y_1\leq y_2$ if there are representatives $y_1',y_2'$ of their equivalence classes that both belong  to either  $\Y_1$ or $\Y_2$ and so that $y_1'\leq y_2'$ in such partial order. It is readily verified that this is well posed, and is a directed complete partial order. Now consider the two maps $I_1,I_2:\Y\to\hat\Y$ obtaining by composing the natural identification of $\Y$ with $\Y_1,\Y_2$ respectively with the projection of $\Y_1\sqcup\Y_2$ in $\hat\Y$. It is clear that $I_1,I_2$ both have the $\MCP$ and satisfy $I_1\circ\j=I_2\circ\j=:\hat\j$. By the (uniqueness part of the) universal property of $(\Y,\j)$ applied with $\Z:=\hat\Y$ and $T:=\hat \j$ it follows that $I_1=I_2$, proving that for every $y_1\in\Y_1$ there is $y_2\in\Y_2$ so that $y_1\sim y_2$, and viceversa. This is (equivalent to) the density claim.\\
\noindent{$(ii)\Rightarrow(i)$}  Start picking $a\leq b$ in $\X$ and choosing $B:=\{b\}$ in \eqref{eq:criterio}: the implication $\Leftarrow$ tells that $\j(a)\leq\j(b)$ proving that $\j$ is monotone. Now let $B\subset\X$ be directed and admitting supremum $s$. The monotonicity of $\j$ shows that $\j(s)$ is an upper bound of $\j(B)$, thus also of $\widehat{\j(B)}$. Also,  the implication $\Leftarrow$ tells that $\j(s)\in \widehat{\j(B)}$ proving that $\j(s)$ is the maximum of $\widehat{\j(B)}$. This proves also that $\j(s)$ is the supremum of $\j(B)$ (as it is an upper bound and any other upper bound is also an upper bound for $\widehat{\j(B)}$ and thus $\geq\j(s)$); in other words we proved that if $B$ is directed and admitting supremum, then $\j(B)$ has supremum and $\j(\sup B)=\sup\j(B)$, i.e.\ we established the $\MCP$ of $\j$.
 
To conclude, we need to prove that $(\Y,\j)$ has the universal property. Thus let $(\Z,\leq_\Z)$ be a directed complete partial order and $T:\X\to\Z$ be with the $\MCP$. Define $\X_0:=\j(\X)\subset\Y$ and then $\X_\alpha:=\X_0^{\uparrow^\alpha}\subset\Y$. By \eqref{eq:barordinali} and the assumed density of $\j(\X)$ we see that for $\alpha$ sufficiently big we have $\X_\alpha=\Y$. We define $T_0:\X_0\to\Z$ as $T_0:=T\circ\j^{-1}$. Since \eqref{eq:criterio} tells that $a\leq b$ if and only if $\j(a)\leq \j(b)$,  we see that $\j$ is injective so that $T_0$ is well defined and has   the $\MCP$. We claim that for every ordinal $\alpha$ there is a unique $T_\alpha:\X_\alpha\to\Z$ with the $\MCP$ extending $T_0$ and that such map is given by the fomula
\begin{equation}
\label{eq:formulatalpha}
T_\alpha(y)=\sup_{x\in\X:\j(x)\leq y}T(x)\qquad\forall y\in \X_\alpha,
\end{equation}
where it is part of the claim the fact that the supremum on the right hand side exists. We prove this by induction. For $\alpha=0$ the claims are all true. Suppose they hold for some $\alpha$ and let us prove them for $\alpha+1$. Let $D\subset\X_\alpha$ directed and $y:=\sup D\in\X_{\alpha+1}$. In order for the $\MCP$ of $T_{\alpha+1}$ to hold we must impose $T_{\alpha+1}(y):=\sup T_\alpha (D)$ (such supremum exists because $T_\alpha(D)\subset\Z$ is directed and $\Z$ is complete), so we must check that this is a good definition, i.e.\ that it depends only on $y$ and not on the particular choice of $D$. To prove this let $B:=\{x\in\X:\j(x)\leq_\Y d\text{ for some }d\in D\}=\j^{-1}(\downarrow\!D)$ and notice that applying \eqref{eq:formulatalpha} to $d\in D$ we deduce that $s:=\sup T(B)\in\Z$ exists and is equal to $\sup T_\alpha(D)$. On the other hand, we have $s=\sup T(\widehat B)$ (the inclusion $T( B)\subset T(\widehat B)$ tells that any upper bound of $T(\widehat B)$ is also an upper bound of $T( B)$ while the $\MCP$ of $T$ grants that any upper bound of $T( B)$ is an upper bound of $T(\widehat B)$  - recall item $(b)$ of Lemma \ref{le:MCP}). By construction $y$ is the tip of $\j(B)$ (recall also \eqref{eq:bartip}), hence  
\begin{equation}
\label{eq:duevolte}
\sup T_\alpha (D)=\sup T(B)=\sup T(\widehat B)\stackrel{\eqref{eq:criterio}}=\sup T(\j^{-1}(\widehat{\j(B)}))=\sup T(\j^{-1}(\downarrow\!y))=\sup_{x\in\X:\j(x)\leq y}T(x),
\end{equation}
proving at once that $\sup T_\alpha (D)$ only depends on $y$, that formula \eqref{eq:formulatalpha} for $\alpha+1$ holds and that $T_{\alpha+1}$ has the $\MCP$.

For the case of limit ordinals, assume all the stated properties hold for every $\beta<\alpha$. Then the definition of $T_\alpha$ on $\X_\alpha=\cup_{\beta<\alpha}\X_\beta$ is forced and the identity \eqref{eq:formulatalpha} holds trivially. The only thing left to prove is the $\MCP$ for $T_\alpha$. Thus let $D\subset \X_\alpha$ be directed with supremum $y\in\X_\alpha$. Let, as before $B:=\{x\in\X:\j(x)\leq d\text{ for some }d\in D\}$  and notice that the equalities in \eqref{eq:duevolte} are still justified and that the last term on the right is equal, by \eqref{eq:formulatalpha} to $T_\alpha(y)$. This  proves that $\sup T_\alpha (D)=T_\alpha(y)$, as desired, and   concludes the proof.
\end{proof}

We have already mentioned that Example \ref{ex:alphafreccia}  in Section \ref{se:techex} shows that for any ordinal $\alpha$ there are a partial order $\X$ and a subset $A\subset \X$ such that $A^{\uparrow^\alpha}=\bar A$ and $A^{\uparrow^\beta}\subsetneq\bar A$  for all $\beta<\alpha$. We can ask a similar question for the directed completion of a partial order: if $\bar\X,\iota$ is the completion of $\X$, is it the case that $\iota(\X)^\uparrow=\bar\X$ or might it be necessary to take more steps? Example \ref{eq:doppiafrecciacompl} in Section \ref{se:techex}  shows that in general we might indeed have $\iota(\X)^\uparrow\subsetneq\bar\X$.

It is therefore interesting to know when the construction simplifies and one step suffices. Inspecting the above mentioned example  we see that there are directed subsets $D$ of $\X$ such that $D^\uparrow$ is not directed. If this does not happen, the construction simplifies. Before turning to the statement notice that if $D$ is a directed subset then so is $\downarrow\! D$. Thus if in a given partial order $D^\uparrow$ is directed whenever $D$ is, by \eqref{eq:hatordinali} we see that we also have
\begin{equation}
\label{eq:hatdirected}
D\subset\X\text{ directed}\qquad\Rightarrow\qquad\widehat D\subset\X\text{ directed.}
\end{equation}
With this said, we have:
\begin{proposition}\label{prop:complunostep}
Let $(\X,\leq)$ be a partial order satisfying \eqref{eq:hatdirected} and let $(\bar\X ,\iota)$ be its completion. 

Then $\iota(\X)^\uparrow$ is the whole $\bar\X$. Also, an explicit example of $(\bar\X,\iota)$ is given by 
\begin{equation}
\label{eq:completionsemplice}
\bar\X:=\big\{A\subset\X\ :\ A=\widehat A\text{ and $A$ is directed}\big\}
\end{equation}
ordered by inclusion, with $\iota:\X\to\bar\X$ given by $\iota(x):=\,\downarrow\!x$, as in \eqref{eq:iotay}. 

Moreover, we can replace the criterion $(ii)$ in Propositon \ref{prop:criterio} with:
\begin{itemize}
\item[ii')] $\j(\X)$ is dense in $\Y$ and \eqref{eq:criterio} holds for any $B\subset\X$ directed.
\end{itemize}
\end{proposition}
\begin{proof} Let $\Y,\iota$ and $\X_\alpha$ be as in the proof of Theorem \ref{thm:excompl}. Then $\X_1$ is, by the very definition, equal to $\{\widehat D\subset\X:\text{$D$ is directed}\}$, so that by the assumption \eqref{eq:hatdirected} we see that $\X_1$ is the right hand side of \eqref{eq:completionsemplice}. Thus if we prove that $\X_2\subset\X_1$ we are done. The construction tells that $\X_2=\{\widehat{\cup_iA_i}:A_i\in\X_1\text{ form a directed family}\}$, thus let $(A_i)\subset\X_1$ be a directed family, put $A:=\cup_iA_i\subset\X$ and pick $a,b\in A$. Given that the $A_i$'s are a directed family, for some index $k$ we have $a,b\in A_k$ and given that $A_k$ is directed there is $c\in A_k\subset A$ with $a,b\leq c$. This shows that $A$ is a directed subset of $\X$, hence by \eqref{eq:hatdirected} $\widehat A=\widehat{\cup_iA_i}$ is also directed, showing that  $\X_2\subset\X_1$, as desired.

We pass to the last claim. Since $(ii')$ is weaker than $(ii)$ of Proposition  \ref{prop:criterio}, we clearly have $(i)\Rightarrow(ii')$. For the converse implication  the same arguments given in the proof of $(ii)\Rightarrow(i)$ in  Proposition  \ref{prop:criterio}  prove that $\j$ has the $\MCP$. Now let $(\bar\X,\iota)$ be the completion of $\X$. Since $\j:\X\to \Y$ has the $\MCP$ and $\Y$ is complete, there is (a unique) $\bar\j:\bar\X\to\Y$ with the $\MCP$ such that $\bar j\circ \iota=\j$. We claim that
\[
\text{$\bar \j:\bar\X\to\Y$ is a bijection and that for any $b_1,b_2\in\bar \X$ we have $b_1\leq b_2\qquad\Leftrightarrow\qquad \bar\j(b_1)\leq  \bar\j(b_2)$.}
\]
The implication `$b_1\leq b_2\ \Rightarrow\ \bar\j(b_1)\leq\bar\j(b_2)$' is obvious. For the opposite one it is convenient to take $\bar\X$ as in \eqref{eq:completionsemplice}. Thus we have $D_1,D_2\subset\X$ directed with $\widehat D_i=D_i$ and so that $\bar\j(D_1)\leq \bar\j(D_2)$ and our goal is to prove that $D_1\subset D_2$. Since $\bar\j$ has the $\MCP$, we have $\bar\j(D_i)=\sup \j(D_i)$ hence  for any $a\in D_1$ we have $j(a)\leq \sup \j(D_2)$ and since $\j(D_2)$ is directed we have  $\sup \j(D_2)\in \widehat{\j(D_2)}$. We thus have that $a\in D_1$ implies $\j(a)\in  \widehat{\j(D_2)}$, hence by \eqref{eq:criterio} we deduce $a\in D_2$, thus proving that $D_1\subset D_2$, as desired.

We thus proved that $\bar\j$ is an isomorphism with the $\MCP$ from $\bar\X$ to its image in $\Y$. Since $\bar\j(\bar\X)\supset \j(\X)$ we also know that $\bar\j(\X)$ is dense in $\Y$, so to conclude it suffices to prove that $\bar\j(\bar\X)^\uparrow=\bar\j(\bar\X)$. Thus let $(y_i)\subset \j(\bar\X)$ be directed and put $\bar x_i:=\j^{-1}(y_i)$. Then $(\bar x_i)$ is directed in $\bar\X$, hence admits supremum $\bar x$. Since  $\bar\j$ has the $\MCP$, we know that $\bar\j(\bar x)$ is the supremum of $(y_i)$, proving that such supremum belongs to $\bar\j(\bar\X)$, as desired.
\end{proof}
The supremum of a subset of a partial order is sometimes called {\bf join} of the subset, and the infimum {\bf meet}. The join and meet of two elements $a,b$ are typically denoted by $a\vee b$ and $a\wedge b$ respectively. One says that the partial order $(\X,\leq)$ has finite joins (resp.\ meets) if for any $a,b\in\X$ the join (resp.\ meet) of $a,b$ exists. Notice that this is the same as requiring that any {non-empty} finite subset of $\X$ has a join (resp.\ meet).

If $(\X,\leq)$ is so that any arbitrary subset has both join and meet, i.e.\ supremum and infimum, then $(\X,\leq)$ is called {\bf complete lattice}. The following is well known and easy to prove:
\begin{proposition}\label{prop:joinlattice}
Let $(\X,\leq)$ be a complete partial order admitting finite joins and having minimum. Then it is  a complete lattice.
\end{proposition}
\begin{proof} We start proving that any subset $A$ of $\X$ has a supremum. If $A$ is empty, the supremum is the minimum of $\X$, otherwise we consider the set $B:=\{a_1\vee\ldots\vee a_n:n\in\N,\ (a_i)\subset A\}$ and notice that it is  directed. Hence the supremum of $B$  exists in $\X$ and since an element is an upper bound for $B$ if and only if it is an upper bound for $A$, such supremum is also the supremum of $A$, which therefore also exists.
 
 For the infimum, given $A\subset\X$ we let $C\subset\X$ be the set of lower bounds: the supremum of $C$ is the infimum of $A$.
\end{proof}
In situations where the existence of joins is relevant, it is natural to consider maps between partial orders preserving existing joins, and not just directed joins (or directed suprema, as we are calling them), and due to Proposition \ref{prop:joinlattice}, the corresponding natural notion of complete order is that of complete lattice. In this setting one  is therefore led to consider the category of partial orders with morphisms the maps that preserve existing joins, denoted ${\bf JPos}$ in the literature, and the full subcategory of complete lattices. Also ixn this framework, surely the notion of `completion as reflector' makes sense and it can be shown that any partial order $(\X,\leq)$, not necessarily admitting joins, has a completion in this sense. An explicit construction is given by the lattice $(\bar\X^{\sf \j},\leq)$ defined as
\begin{equation}
\label{eq:compljoin}
\bar\X^{\sf \j}:=\big\{A\subset\X\ :\ \text{$A$ is a lower set closed by arbitrary existing suprema}\big\}
\end{equation}
ordered by inclusion, together with the map $\iota:\X\to\bar\X^{\sf \j}$ defined as in \eqref{eq:iotay}. For this, see for instance \cite[Chapter I.4, Example C10]{AbstrConcrCat}.

It is natural to compare the above $\bar\X^{\sf \j}$ with the directed completion $\bar\X$. In doing so one quickly notices that $\bar\X^{\sf \j}$ always has a minimum element, while $\bar\X$ has minimum if and only if $\X$ does. It is therefore better to consider the {\bf enhanced directed completion}, say, $\bar\X_o$ of $\X$ defined as $\bar\X$ if $\X$ has minimum and as $\{\perp\}\cup \bar\X$ otherwise, where $\perp$ is an extra term declared to be smaller than all the others. Notice that this is precisely the directed completion we would obtain if only the empty set was declared to be directed, as in this case the extra term $\perp$ in the completion would be its supremum (in the literature, a set that is either directed or empty is called semidirected). 

Notice also that  if $\X$ has no minimum, then the least element of $\bar\X^{\sf \j}$ is the empty set, while if $\X$ has a minimum $m$, then the least element of  $\bar\X^{\sf \j}$ is $\{m\}$ (the empty set does not belong to  $\bar\X^{\sf \j}$ in this case, because its supremum exists, being it $m$, so that $\emptyset $ is not closed by existing suprema).

With all this said, we have the following result:
\begin{proposition}[Comparison of completions - 1]\label{prop:compljoin}
Let $(\X,\leq)$ be a partial order admitting finite joins. Then 
\[
D\text{ directed}\qquad\Rightarrow\qquad (\downarrow\! D)^\uparrow\text{ directed.}
\]
In particular, \eqref{eq:hatdirected} holds. Moreover, the enhanced directed completion of $\X$ is a complete lattice and coincides with the completion in ${\bf JPos}$  recalled above.
\end{proposition}
\begin{proof}
Let $D\subset\X$ be directed and $x,y\in  (\downarrow\! D)^\uparrow$. Then there are directed families $(x_i)_{i\in I},(y_j)_{j\in J}\subset \,\downarrow\! D$ with $x=\sup_ix_i$ and $y=\sup_jy_j$. Since, trivially, $\downarrow\!D$ is directed, for every $i,j$ there is $a_{ij}\in \downarrow\!D$ with $x_i,y_j\leq a_{ij}$ and thus with $x_i\vee y_j\leq a_{ij}$. It follows that 
$x_i\vee y_j\in\,\downarrow\!D$ as well. It is easy to see that the set $(x_i\vee y_j)_{i,j}\subset \downarrow\! D$ is directed. Indeed, for every $i_1,i_2\in I$ and $j_1,j_2\in J$ there are $i\in I$ 
and $j\in J$ so that $x_{i_1}$ and  $x_{i_2}$ are both $\leq x_i$ and also  $y_{j_1}$ and $y_{j_2}$ that are both $\leq y_j$. Hence $x_{i_1}\vee y_{j_1}$ and $x_{i_2}\vee y_{j_2}$ are both $\leq x_{i}\vee y_{j}$, proving the claim. We now claim that $x\vee y$ is the supremum of $(x_i\vee y_j)_{i,j}$ (in particular: such supremum exists even if $\X$ is not complete). The claim is obvious because $x\vee y$ is clearly an upper bound, and any upper bound $u$ must also be an upper bound for the $x_i$'s, and thus $\geq x$, and an upper bound for the $y_j$'s, and thus $\geq y$. Hence $x\vee y\leq u$, proving the claim.

We thus established that for $x,y\in  (\downarrow\! D)^\uparrow$ we have  $x\vee y\in  (\downarrow\! D)^\uparrow$ as well, proving that $ (\downarrow\! D)^\uparrow$ is directed, as claimed. Now the fact that  \eqref{eq:hatdirected} holds is a direct consequence of \eqref{eq:hatordinali}.

To prove that the enhanced directed completion $\bar\X_o$ is a complete lattice it suffices,  by Proposition \ref{prop:joinlattice},  to prove that it has joins. In turn, by (the representation given in) Proposition \ref{prop:complunostep} to this aim it suffices to prove that given two directed subset $A,B\subset\X$ with $\widehat A=A$ and $\widehat B=B$ there exists their join in $\bar\X_o$. We claim that $C:=\widehat{A\vee B}$ is the required join, where $A\vee B:={\{a\vee b:a\in A,\ b\in B\}}$. Indeed, it is obvious that any set as in the right hand side of \eqref{eq:completionsemplice} that contains $A$ and $B$ must also contain $C$. On the other hand the discussion made above shows that $A\vee B$ is directed, hence so is $C$ by  \eqref{eq:hatdirected}. It follows that $C$ is as in  the right hand side of \eqref{eq:completionsemplice} and thus that it is the join of $A$ and $B$, proving that such join  exists.

For the last claim it now suffices to prove that the set $\bar\X^{\sf \j}$ defined in \eqref{eq:compljoin} coincides with the one $\bar\X$ defined in \eqref{eq:completionsemplice}, possibly with the addition of a lowest element in case it didn't have it already. The inclusion $\bar\X^{\sf \j}\subset\bar\X$ is obvious. For the opposite one let $A\in\bar \X$ and $a,b\in A$. Then since $A$ is directed there is $c\in A$ bigger than both $a$ and $b$, and thus of their join. Since $A$ is a lower set, we have $a\vee b\in A$. It follows by induction that $A$ is closed by finite joins and since it is closed by directed suprema, the same arguments in Proposition \ref{prop:joinlattice}  show that $A$ is closed by arbitrary joins, proving that $A\in \bar\X^{\sf \j}$, as desired.
\end{proof}
Another situation in which recognizing the directed completion is particularly simple is that of spaces having the truncation property, defined as:
\begin{definition}[Truncation property]\label{def:trunc} Let $(\Y,\leq)$ be a partial order and $a\in\Y$. We say that $a$ truncates $\Y$ provided
\begin{equation}
\label{eq:truncpropr}
\text{$B\subset\Y$ with tip and  }a\in\widehat B\qquad\text{(i.e.\ $a\leq\tip B$)}\qquad\Rightarrow\qquad a\in\widehat{(\downarrow\!B)\cap(\downarrow\!a)}.
\end{equation}
If \eqref{eq:truncpropr} only holds whenever $B$ is directed, we shall say that $a$ directly truncates $\Y$.

If \eqref{eq:truncpropr} holds for every $a$ in a given subset $\X$ of $\Y$ we shall say that $\X$ truncates $\Y$ (or directly truncates $\Y$ in case only directed sets are involved).

 If $\X=\Y$ we say that $\Y$ has the truncation (resp.\ directed truncation) property.
\end{definition}
The name comes from the following rough intuition. Say for simplicity that the space has meets, i.e.\ minima, of any couple of points and that for some non-decreasing sequence $(b_n)$ admitting sup $b$ and some  element $a$ we have $a\leq b$. Then one might expect that the `truncated' sequence $n\mapsto a\wedge b_n$, which is also clearly non-decreasing, has $a$ as supremum. This, though, is not necessarily the case, as shown by simple examples (see e.g.\ Example \ref{ex:notrunc} in Section \ref{se:techex}) and it is therefore interesting to consider cases where this sort of examples do not occur. The truncation property goes in this direction, as also shown in the next statement:
\begin{proposition}\label{prop:checktrunc}
Let $(\Y,\leq)$ be a complete partial order having meets and $\X\subset\Y$ be a lower subset (i.e.\ $\X=\,\downarrow\!\X$). Then the following are equivalent:
\begin{itemize}
\item[i)] For every  $a\in\X$ and directed set $D\subset\Y$  with $a\leq \sup D$ we have $a=\sup\{a\wedge d:d\in D\}$,
\item[ii)] For every $b\in\X$ the map $x\mapsto x\wedge b$ from $\Y$ into itself has the $\MCP$.
\end{itemize} 
Moreover, if these hold then $\X$ truncates $\Y$.
\end{proposition}
\begin{proof} The implication $(ii)\Rightarrow(i)$ is obvious. For the converse one, fix $b\in\X$,  $D\subset\Y$  directed and let $a:=b\wedge(\sup D)\leq\sup D$. Then $a\in\X$ and the assumption $(i)$ tells that $b\wedge \sup(D)=a=\sup\{a\wedge d:d\in D\}=\sup\{b\wedge d:d\in D\}$ (the last equality follows from $b\wedge d=a\wedge d$ for any $d\in D$, that in turn is an easy consequence of the definition of $a$). By the arbitrariness of $D$, this is precisely the $\MCP$ of $x\mapsto x\wedge b$.

Let us now prove that if these hold then $\X$ truncates $\Y$. Let $B\subset\Y$ be with tip and $a\in \X$ with  $a\leq \tip B$. The fact that $x\mapsto x\wedge a$ has the $\MCP$ yields $a=(\tip B)\wedge a$ is equal to the tip of $\{b\wedge a:b\in B\}$, hence belongs to $\overline{\{b\wedge a:b\in B\}}$. Since clearly $\{b\wedge a:b\in B\}\subset (\downarrow\!B)\cap(\downarrow\!a)$, the conclusion $a\in\widehat{(\downarrow\!B)\cap(\downarrow\!a)}.$ follows.
\end{proof}
The  truncation property is related to directed completion via  the next result:
\begin{proposition}\label{prop:trunccompl}
Let $(\Y,\leq)$ be a complete partial order  and $\X\subset\Y$ a dense lower subset that truncates $\Y$.

Then $\Y$, together with the inclusion $\j$, is the directed completion of $\X$.

If $\X$ also has joins, we can weaken the assumption that $\X$ truncates $\Y$ into $\X$ directly truncates $\Y$ to still deduce  that $(\Y,\j)$ is the directed completion of $\X$.
\end{proposition}
\begin{proof}

We apply Proposition \ref{prop:criterio}. $\X=\j(\X)$ is dense in $\Y$ by assumption. Since $\X$ is a lower set in $\Y$, the inclusion $\j:\X\to\Y$ has the $\MCP$, because if $D\subset\X$ is directed with supremum $a$ as subset of $\X$ and supremum $b$ as subset of $\Y$ (the latter exists as $\Y$ is complete) then obviously $b\leq a$, as in $\Y$ there is a larger pool of upper bounds of $D$, and thus $b\in\X$, which forces $a=b$.

It follows that the implication $\Leftarrow$ in \eqref{eq:criterio} holds. For $\Rightarrow$, let  $B\subset\X\subset\Y$ be having tip in $\Y$ and $a\in\X$ be in $\widehat B^\Y$ (meaning that the operation $B\mapsto \widehat B$ is performed in $\Y$). Since $\X$ truncates $\Y$   it follows that $a\in \widehat{(\downarrow\!B)\cap(\downarrow\!a)}^\Y$ (since $\X$ is a lower set and $B\subset\X$, the meaning of $\downarrow\!B$ is unambiguous --  same for $\downarrow\!a$). If we prove that $ \widehat{(\downarrow\!B)\cap(\downarrow\!a)}^\Y= \widehat{(\downarrow\!B)\cap(\downarrow\!a)}^\X$ the proof is complete, as we then have $a\in  \widehat{(\downarrow\!B)\cap(\downarrow\!a)}^\X\subset\widehat B^\X$, as desired.

Since $C:=(\downarrow\!B)\cap(\downarrow\!a)$ has $a$ as upper bound,  to conclude it suffices to prove that if $\X$ is a lower set of the complete partial order $\Y$ and $a\in\X$ is an upper bound of $C\subset\X$, then $\widehat C^\X=\widehat C^\Y$. This is a consequence of the completeness of $\Y$: here the truncation property has no role. To see this, recalling \eqref{eq:hatordinali} to conclude it suffices to prove that for such $C$ the operation  $C\mapsto\,\downarrow\! (C^\uparrow)$ performed in $\X$ and $\Y$ produces the same result. Since $\X$ is a lower set, this holds if and only if the operation  $C\mapsto  C^\uparrow$ performed in $\X$ and $\Y$ produces the same result and in turn this reduces to proving that any directed set $D\subset \downarrow\!a$ admits  supremum in $\X$ and such supremum  agrees with the supremum in $\Y$ (that exists as $\Y$ is complete). Let thus $d$ be the supremum of $D$ in $\Y$ and notice that since $a$ is an upper bound of $D$, we must have $d\leq a$ and thus $d\in\X$. If $u\in\X\subset\Y$ is another upper bound of $D$, then we must have $d\leq u$, or else $d$ is not the supremum of $D$ in $\Y$. Thus $d$ is also the supremum of $D$ in $\X$, as desired.

The last  statement now follows easily. The fact that the inclusion has the $\MCP$ can be proved in the same way, and this settles the implication $\Leftarrow$ in \eqref{eq:criterio}. Then according to Propositions \ref{prop:compljoin} and \ref{prop:complunostep} to conclude it suffices to prove $\Rightarrow$ for $B\subset\X\subset\Y$ directed. Thus, as before, let $a\in\X$ be in $\widehat B^\Y$ and notice that since $\X$ directly truncates $\Y$ we have $a\in \widehat{(\downarrow\!B)\cap(\downarrow\!a)}^\Y$. The conclusion now follows as above, as the truncation property is not anymore required.
\end{proof}
It might be worth to point out that in this last proposition we cannot conclude that $\Y$ has the truncation property (i.e.\ that it truncates itself), see Remark \ref{re:noalcomplet} for a counterexample.

We conclude the section comparing the directed completion with the completion in the sense of   Dedeking-MacNeille, that much like the completion in ${\bf JPos}$ produces a complete lattice out of any partial order. In general the directed and the  Dedeking-MacNeille completions do not coincide, see Example \ref{ex:compldiverse} in Section \ref{se:techex} and Proposition \ref{prop:compcompl} below.

Unlike the directed completion, the Dedeking-MacNeille does not arise as reflector (but rather as injective hull in the category of partial orders and monotone maps). Here we shall be satisfied with the following practical presentation of it:
\begin{definition}[Dedeking--MacNeille  completion]
Let $(\X,\leq)$ be a partial order. For any $A\subset\X$ we define $A^\us,A^\ls$ as  the set of upper bounds and  lower bounds of $A$, i.e.:
\begin{equation}
\label{eq:usls}
\begin{split}
A^\us&:=\{x\in\X\ :\ a\leq x\quad\forall a\in A\},\\
A^\ls&:=\{x\in\X\ :\ a\geq x\quad\forall a\in A\}.
\end{split}
\end{equation}
A \emph{cut} of $\X$ is a couple $(A,B)$ with $A^\us=B$ and $B^\ls=A$.

The Dedeking-MacNeille  completion $\bar\X^{DM}$ is the collection of all cuts, where $(A_1,B_1)\leq (A_2,B_2)$ whenever $A_1\subset A_2$. The original order $\X$ is embedded in $\bar\X^{DM}$ via the map $x\mapsto \iota_{DM}(x):=(\downarrow \!x,\,(\downarrow 
\!x)^\us)$ (the fact that this is a cut is easy to check).
\end{definition}
We collect some comments.  Notice that $A_1\subset A_2$ implies  $A_1^\us\supset A_2^\us$ and that $A\subset A^{\us\ls}$. Hence for any $A\subset\X$ we have $A^{\us\ls\us}=(A^\us)^{\ls\us}\supset A^\us$ and $A^{\us\ls\us}=(A^{\us\ls})^\us\subset A^\us$, whence $A^{\us\ls\us}=A^\us$. It follows that 
\begin{equation}
\label{eq:ululA}
\text{for any $A\subset \X$ we have $A^{\us\ls}=A^{\us\ls\us\ls}$}
\end{equation}
and then that the Dedekind-MacNeille completion can be equivalently described as the collection of subsets $A\subset\X$ with $A=A^{\us\ls}$, with the order given by inclusion: to each such $A$ we can associate the cut $(A,A^\us)$ and viceversa to  each cut $(A,B)$ we associate the set $A$.

It is then clear why such completion is a complete lattice: the supremum (resp.\ infimum) of the arbitrary family $(A_i)$ is given by $(\cup_iA_i)^{\us\ls}$ (resp.\ $(\cap_iA_i)^{\us\ls}$). This makes also transparent the fact that the embedding of $\X$ into  $\bar\X^{DM}$ is both join-dense and meet-dense, i.e.\ any element of  $\bar\X^{DM}$ is the supremum and the infimum of appropriate subsets of $\X$.

We want to compare $\bar \X^{DM}$ with the directed completion $\bar \X$ and since the former always has a minimum, as in the case of the completion $\bar\X^{\sf \j}$ in ${\bf JPos}$ it is better to compare  $\bar \X^{DM}$ with the enhanced directed completion $\bar\X_o$ as discussed before Proposition \ref{prop:compljoin}. Notice that the least element of  $\bar \X^{DM}$ is $\emptyset^{\us\ls}=\X^\ls$ and this is the empty set if $\X$ has no minimum, otherwise it is the singleton containing just the minimum (in analogy with $\bar\X^{\sf \j}$).

\medskip

With this said, we always have a natural map  ${\sf T}:\bar\X_o\to\bar\X^{DM}$ defined as
\[
{\sf T}(\bar x):=A^{\us\ls}\qquad\text{where } A:=\{x\in\X:\iota(x)\leq\bar x\}\subset\X.
\]
A natural map ${\sf S}:\bar\X^{DM}\to \bar\X_o$ exists at least if $\X$ has finite joins, as   in this case we have
\begin{equation}
\label{eq:uldir}
\text{$A=A^{\us\ls}\qquad\Rightarrow \qquad A=\widehat A$ and $A $ is either directed or empty.}
\end{equation}
Let us verify this. The set $(A^\us)^\ls$ is a clearly a lower set and directed sup closed, so that it being equal to $A$ forces $A=\widehat A$. Also, if $a,b\in A$ are $\leq u$, then also $a\vee b\leq u$, proving that if $A$ is not empty then it is directed. In particular, we can define ${\sf S}:\bar\X^{DM}\to \bar\X_o$ as
\[
{\sf S}(A):=\sup \iota(A)\in\bar\X_o.
\]
Here the existence of the sup follows from   \eqref{eq:uldir} as it shows that  each $A\in \bar\X^{DM} $ is either empty (in which case its supremum is the least element in $\bar\X_o$) or directed, in which case its supremum exists in $\bar\X_o$ by definition of directed completion.

The following is now easy:
\begin{proposition}[Comparison of completions - 2]\label{prop:compcompl}
Let $(\X,\leq)$ be a partial order with joins. 

Then with the above notation we have ${\sf T}\circ{\sf S}={\rm Id}_{\bar \X^{DM}}$. In particular, ${\sf S}:\bar\X^{DM}\to\bar\X$ is injective and 
${\sf T}:\bar\X\to\bar\X^{DM}$ is surjective.

 Moreover, ${\sf T}$ is injective if and only if ${\sf S}$ is surjective and if and only if for every $A\subset\X$ directed with $\widehat A=A$ and $a\notin A$ there is an upper bound $u$ of $A$ such that $a\nleq u$.
\end{proposition}
\begin{proof} The fact that  ${\sf T}({\sf S}(A))=A$ for every $A\in  {\bar \X^{DM}}$ is a direct consequence of \eqref{eq:uldir} and \eqref{eq:ululA}. The statements in the second part of the claim are all equivalent to the converse implication in \eqref{eq:uldir} and it is trivial to check  that this holds if and only if the stated property is true.
\end{proof}
 see Example \ref{ex:compldiverse} in Section \ref{se:techex} for a case where the map ${\sf T}$ is not injective.

\end{subsection}

\begin{subsection}{Examples and further comments}
\begin{subsubsection}{Technical examples}\label{se:techex}

\begin{enumerate}

\item A complete partial order can be isomorphically embedded in another complete partial order with dense image without the image being the whole target set.  Consider for instance the embedding $\mathcal I$ of $\X:=[0,1]$ in $\Y:=[0,1]\cup\{2\}$ that is the identity on $[0,1)$ and sends 1 to 2. Notice that this map has not the $\MCP$ and compare with the last part of the proof of Proposition \ref{prop:complunostep}.

This also shows that the directed-sup-closure of a subset strongly depends on the whole space. Indeed, for   $A:=[0,1)\subset\X$ we have $\bar A^\X=\X$ and $\overline{ \mathcal I(A)}^\Y=[0,1]\neq \mathcal I(\X)$. In particular, $\mathcal I(A)$ is not dense in $\mathcal I(\X)$.

\item\label{ex:notip} In general a subset of the form $\downarrow\!\bar A$ is not directed-sup-closed, hence it might differ from $\widehat A$. Let for instance $\X:=[0,1]\times\{0,1\}$ be with the complete order
\[
(t,i)\leq (s,j)\qquad\Leftrightarrow\qquad \text{either \ ($i=0$ and $t\leq s$) \ or \  ($(s,j)=(1,0)$) \ or \  ($(t,i)=(s,j)$)}.
\]
In particular there is no relation between the points in $A:=[0,1)\times\{1\}$ which therefore is trivially directed-sup-closed. Still, $\downarrow A=[0,1)\times\{0,1\}$ contains the directed subset $[0,1)\times\{0\}$ whose supremum is  $(1,0)\notin A$.

Notice that we also have $\widehat A=\,\downarrow\!(1,0)$, i.e.\ $\widehat A$ has maximum, but $A$ has no tip.
\item\label{ex:doppiafreccia} We give an example of $A^\uparrow\neq \bar A$. Let $\X:=\N^2\cup\N\cup\{\top\}$ be ordered as:
\[
\begin{split}
(n,m)\leq (n',m')\qquad&\Leftrightarrow\qquad n=n'\quad\text{ and }\quad m\leq m',\\
(n,m)\leq n'\qquad&\Leftrightarrow\qquad n\leq n'\\
n\leq n'\qquad&  \text{ in the usual sense  },\\
\top\qquad&\text{ is the maximal element.}
\end{split}
\]
It is clear that this is a partial order. Let $A:=\N^2\subset\X$. Then $A^\uparrow= \N^2\cup\N$ and  $A^{\uparrow^2}=\X\supsetneq A^\uparrow$. To see that   $\top\neq A^\uparrow$  notice that a directed subset of $A$ must be entirely contained in a `branch' of the form  $\N\times\{n\}$ for some  $n\in\N$.

Observe that $(\X,\leq)$ has joins.
\item\label{ex:alphafreccia} More generally, we claim that for every ordinal $\alpha$ there is a complete partial order $\X_\alpha$ and a subset $A_\alpha\subset\X_\alpha$ so that $A_\alpha^{\uparrow^\alpha}=\X_\alpha$ but for no $\beta<\alpha$ we have $A_\alpha^{\uparrow^\beta}=\X_\alpha$. We prove this by transfinite induction.

The case $\alpha=0$ is obvious: just take $A_0=\X_0=\{\top_0\}$. 

Now suppose that the construction has been done for $\alpha$ and that $\X_\alpha$ has a maximum $\top_\alpha$ and let us build the sets for $\alpha+1$. Put $\X_{\alpha+1}:=(\N\times \X_\alpha) \cup\{\top_{\alpha+1}\}$ and $A_{\alpha+1}:=\N\times A_\alpha$, the partial order $\leq $ on $\X_{\alpha+1}$ is so that $\top_{\alpha+1}$ is the maximal element and then $(n,x)\leq (n',x')$ whenever
\[
\begin{split}
\text{either}\qquad \big(\ n=n'\ \text{ and }\  x\leq x'\text{ on $\X_\alpha$}\ \big)\qquad \text{or}\qquad \big(\ 
 n\leq n'\ \text{ and }\   x'=\top_\alpha \ \big).
\end{split}
\]
It is clear that the restriction of this order to $\N\times(\X_\alpha\setminus\{\top_\alpha\})$ consists of countable many unrelated copies of $\X_\alpha\setminus\{\top_\alpha\}$ and thus that for $\beta\leq\alpha$ we have  $A_{\alpha+1}^{\uparrow^\beta}=\N\times A_\alpha^{\uparrow^\beta}$, being intended that $A_\alpha^{\uparrow^\beta}$ is computed in $\X_\alpha$. In particular $A_{\alpha+1}^{\uparrow^\beta}\neq \X_{\alpha+1}$ for any such $\beta$ and any directed subset of  $A_{\alpha+1}^{\uparrow^\beta}$ must be included in $\{n\}\times \X_\alpha$ for some $n$. It follows that $A_{\alpha+1}^{\uparrow^\alpha}=\N\times\X_\alpha\subsetneq\X_{\alpha+1}$. On the other hand, $\{\top_\beta:\beta\leq\alpha\}\subset \N\times\X_\alpha$ is linearly ordered, hence directed, subset of $\X_{\alpha+1}$ with supremum $\top_{\alpha+1}$. Therefore $\X_{\alpha+1}=(\N\times\X_\alpha)^\uparrow=A_{\alpha+1}^{\uparrow^{\alpha+1}}$.

Suppose now that $\alpha$ is a limit ordinal and that $\X_\beta,A_\beta$ have been constructed for any $\beta<\alpha$. In this case we simply put
\[
\X_\alpha:= \bigsqcup_{\beta<\alpha}\X_\beta\qquad\text{and}\qquad A_\alpha:= \bigsqcup_{\beta<\alpha}A_\beta,
\]
with the order $\leq$ on $\X_\alpha$ being so that the order on $\X_\beta$ coincides with the existing one and elements of different $\X_\beta$'s are never related. It is then obvious that $A_\alpha^{\uparrow^\beta}$ is never $\X_\alpha$ (because it does not contain $\X_{\beta+1}$), but $A_\alpha^{\uparrow^\alpha}$ is the whole $\X_\alpha$. Notice that $\X_\alpha$ has no maximum.

If $\alpha$ is an infinite limit ordinal, we define $\X_{\alpha+1}$ as $\X_\alpha\cup\{\top_{\alpha+1}\}$ with $\top_{\alpha+1}$ being the maximum element. We also redefine the order of $\X_\alpha$ by adding the following relation:  for any $\beta<\alpha$ for which $\X_\beta$ has maximum $\top_\beta$ and any $x\in\X_{\beta'}$ for some $\beta'<\beta$, we put $x\leq\top_\beta$. It is clear that this new relation is still a partial order and that the family $\{\top_\beta:\beta<\alpha\text{ and $\top_\beta$ exists}\}$ is linearly ordered with supremum $\top_{\alpha+1}$. Putting $A_{\alpha+1}:=A_\alpha$, the construction also ensures that $A_{\alpha+1}^{\uparrow^\alpha}=\X_\alpha\subsetneq\X_{\alpha+1}$ and $A_{\alpha+1}^{\uparrow^{\alpha+1}}=\X_{\alpha+1}$, as desired.

\item\label{eq:densecompl} On the completion of dense subsets. Let $\X$ be two copies of $[0,1]$ with the 1's identified and no relation between elements in the two copies of $[0,1)$. Let $A:=\X\setminus\{1\}$, i.e.\ two copies of $[0,1)$. Then $A$ is dense in $\X$, its completion $(\Y,\iota)$ is made by two copies of $[0,1]$ and the map $T:\Y\to\X$ with the $\MCP$ so that $T\circ\iota$  is  the inclusion of $\Y$ in $\X$ is not injective (it sends both the 1's in $\Y$ to $1\in\X$). 

On the other hand, for $A,\X$ as in the Example \ref{ex:doppiafreccia} above it is easy to see that the completion $(\Y,\iota)$  of $A$ consists of countably many distinct copies of $\N\cup\{\infty\}$ and the map $T:\Y\to \X$ with the $\MCP$ so that $T\circ\iota$  is  the inclusion of $A$ in $\X$ is not surjective (the point $\top\in\X$ is not in the image).

These two examples can be combined into a single one having $T$ being neither injective nor surjective: just create a new partial order as disjoint union of these examples with no relation between the points in the different sets.

\item\label{eq:doppiafrecciacompl} We give an example of partial order $(\X,\leq)$ whose completion is not obtained in a single step in the induction procedure discussed in the proof of Theorem \ref{thm:excompl}. In other words, for $\Y$ as in \eqref{eq:defY} and $\iota:\X\to\Y$ as in \eqref{eq:iotay} we have $\bar \X\neq (\iota(\X))^\uparrow$. Equivalently, in the construction of $\bar\X$ as in Theorem \ref{thm:excompl} we have that it strictly contains $\{\widehat D:D\subset\X\text{ is directed}\}$. Notice that  choosing as $\X$ the set $A$ in Example \ref{ex:doppiafreccia} does not work, as the abstract completion of such order is simply given by countably many independent copies of $\N\cup\{\infty\}$. Still, a suitable modification of such example obtained by `linking' the different branches of $A$ does the job.

Let $\X:=\N\times\N\times(\N\cup\{\infty\})$  be ordered so that $(n,m,l)\leq (n',m',l')$ if and only if one of these hold:
\[
\begin{array}{lll}
n=n'&\qquad m=m'&\qquad l\leq l'\\
n=n'&\qquad  m\leq m'&\qquad l'=\infty \\
n< n'&\qquad m,l\leq m'&\qquad l'=\infty
\end{array}
\]
It is clear that this is a partial order. If a directed subset $D$ does not contain elements of the form  $(n,m,\infty)$, then it must be contained in $\{(n,m,l): l\in \N\cup\{\infty\} \}$ for some fixed $n,m$. Thus either it has a maximum, or the supremum is  $(n,m,\infty)$. If instead $D$ contains an element of the form $(n,m,\infty)$ but no maximum, then all the other elements of the directed set must have the same first component `$n$'. It follows that  its supremum in the completion will be a new element, that we call $s_n$ (as it is easy to check that it depends solely on $n$). It is also easy to check that in the completion we have $s_{n'}\leq s_n$ if and only if $n'\leq n$, so that the $s_n$'s form a totally ordered, hence directed, set in the completion: their supremum does not arise as supremum of a directed subset of $\X$.

We parse the same argument following the notation of the proof of Theorem \ref{thm:excompl}. The elements of $\X_1\setminus\X_0$ are subsets of  $\X$ of the form $\widehat D$ with $D\subset\X$ directed. The arguments above show that these are precisely sets of the form  
\[
A_n:=\{(n,m,\infty):m\in\N\}\cup\{(n',m,l):n'\leq n,\ m\in\N,\ l\in\N\cup\{\infty\}\}
\] 
for some fixed $n\in\N$. These arise as supremum of the directed set  $\{\iota(n,m,\infty):m\in\N\}$. It is clear that $A_{n'}\subset A_n$ for $n'\leq n$. The element of $\X_2\setminus \X_1$ is the whole $\X$, that arise as supremum of the $A_n$'s. 

In particular, $\X$ is a lower set and directed-sup-closed (trivially: any order is so within itself) and is not of the form $\widehat D$ for any $D\subset\X$ directed.

\item Adding a single point can destroy the notion of directed-sup-closure and thus that of completeness. Let for instance $\X:=[0,1]$ be with its standard order and $\Y:=\X\cup\{1'\}$ with $1'$ being an upper bound of $[0,1)$ and with no relation to $1$. Then the directed set $[0,1)$ has supremum $1$ in $\X$ (and thus is dense in $\X$) but no supremum in $\Y$, so its directed-sup-closure in $\Y$ coincides with itself. Compare with the proof of  Proposition \ref{prop:trunccompl}.
\item\label{ex:notrunc} Let $\X$ be made by two copies of $[0,1]$, each with the standard order, having the two 0's identified and so the two 1's. This is a complete lattice, in particular has meets, but does not have the truncation property. Indeed, if $B\subset\X$ is one of the two copies of $(0,1)$ we have $\widehat B=\X$ but if $a$ is an element of `the other' $(0,1)$ we have $(\downarrow\! B)\cap(\downarrow\! a)=\{0\}$.
\item\label{ex:compldiverse} The directed completion and Dedekind-MacNeille completion might differ. Let for instance $\X:= (0,1)\times\{1,2,3,4\}$ be equipped with the order
\begin{equation}
\label{eq:esempietto}
(t,n)\leq (s,m)\qquad\Leftrightarrow\qquad t\leq s\text{ and either $m=4$ or $m=n$ or $n=1$.}
\end{equation}
Notice that $\X$ has joins and meets of any couple of elements (but no maximum or minimum). 

The directed completion is $(0,1]\times \{1,2,3,4\}$  ordered as  in \eqref{eq:esempietto}. On the other hand the Dedekind-MacNeille completion is $\X\cup\{\perp,\top\}$ with $\perp,\top$ being respectively the smallest and largest element. Notice that the map ${\sf T}$ as in Proposition \ref{prop:compcompl} is not injective because for $A_i:=(0,1)\times\{i\}$ we have $A_i^\us=\emptyset$ for any $i$ (and $A_i=\{x\in\X:\iota(x)\leq (1,i)\}$).
\item\label{ex:extell} Let $\X$ be a set and $\ell:\X^2\to\{-\infty\}\cup[0,+\infty)$ be a \emph{time separation} function on it, i.e.\ a map satisfying 
\begin{equation}
\label{eq:timesep}
\begin{split}
\ell(x,x)&=0,\\
\ell(x,y)+\ell(y,z)&\leq\ell(x,z),\\
\ell(x,y),\ell(y,x)&\geq0\qquad\Rightarrow\qquad x=y,
\end{split}
\end{equation}
for any $x,y,z\in\X$. Define the relation $\leq $ by declaring $x\leq y$ if and only if $\ell(x,y)\geq 0$ and notice that the above ensures that this is a partial order. We can therefore complete it and then wonder whether $\ell$ can be extended to the completion. Notice that for any $x\in\X$ the map $\X\ni y\mapsto \ell(x,y)\in\{-\infty\}\cup[0,+\infty)$ is monotone. If it has the $\MCP$ then we can uniquely extend $\ell$ to a map, still denoted $\ell(x,\cdot)$, from $\bar\X$ to $\{-\infty\}\cup[0,+\infty]$. Similarly, for every $y\in\X$ the map $\X\ni x\mapsto -\ell(x,y)\in (-\infty,0]\cup\{+\infty\}$ is monotone and if it has the $\MCP$ it can be extended to a map,  still denoted $-\ell(\cdot,y)$, from $\bar\X$ to  $\{-\infty\}\cup[0,+\infty]$.

Thus under the stated assumptions, all in all quite natural in settings where the time separation is continuous, we can extend $\ell$ to a relevant portion of the completion. Still, such function is not yet defined   on $(\bar\X\setminus\X)^2\subset\bar\X^2$ and in general there seems to be no obvious way to do so. Consider for instance the case $\X:=\N\times\{0,1\}$ with 
\[
\begin{split}
\begin{array}{rl}
\ell\big((n,i),(m,i)\big)&:=\ \left\{\begin{array}{rl}
0,&\qquad\text{if }n\leq m,\\
-\infty,&\qquad\text{if }n>m,
\end{array}
\right.\quad i=0,1,\\
\ \\
\ell\big((n,1),(m,0)\big)&:=\ -\infty\qquad\forall n,m\in\N,\\
\ \\
\ell\big((n,0),(m,1)\big)&:=\ \left\{\begin{array}{rl}
1,&\qquad\text{if }n\leq m,\\
0,&\qquad\text{if }n>m.
\end{array}
\right.
\end{array}
\end{split}
\]
It takes only a moment to realize that $\ell$ satisfies \eqref{eq:timesep} and that the partial order  induced by $\ell$ is so that the two subsets $\{(n,0)\}_{n\in\N}$ and $\{(m,1)\}_{m\in\N}$  have the natural order  and with $(n,0)\leq (m,1)$ for every $n,m$. Then the completion $\bar\X$ of $\X$ simply adds the suprema of these two subsets, call them $(\infty,0)$ and $(\infty,1)$ respectively and the extended $\ell$ isatisfies $\ell((\infty,0),(m,1))=0$ and $\ell((n,0),(\infty,1))=1$ for every $n,m\in\N$.  We then see that letting $n\to\infty$ in the first identity or $m\to\infty$ in the second one gives the two different values $0$ and $1$ for $\ell((\infty,0),(\infty,1))$. Notice that both the choices, as well as that of any number in between, satisfy \eqref{eq:timesep}.
\end{enumerate}
\end{subsubsection}

\begin{subsubsection}{Examples more related to the geometry and analysis we are interested in}\label{se:exbelli}

We turn to some examples of completions that are relevant for our discussion. 

We point out that the first three of these, the directed completions coincide with that of Dedekind-MacNeille, as can be proved by  a direct application of Proposition \ref{prop:compcompl}. Still, to illustrate how the concepts we discussed work, we describe them relying on notions presented in this note.

\begin{enumerate}\setcounter{enumi}{10}
\item\label{eq:complsubset} \underline{Completion of space of finite subsets of a given set}. Let $\Y$ be any set and let $\X$ be the collection of finite subsets of $\Y$, ordered by inclusion. Then the completion is the whole power set $\mathcal P(\Y)$ of $\Y$ still ordered by inclusion,  with $\iota$  being the inclusion of $\X$ into  $\mathcal P(\Y)$. To see why we shall apply Proposition \ref{prop:trunccompl}. It is clear that $\X$ is a lower subset of $\mathcal P(\Y)$ and since the supremum of any  family in  $\mathcal P(\Y)$  is the union of its members, it is also clear that $\X$ is dense in $\mathcal P(\Y)$. To conclude it therefore suffices to prove that $\mathcal P(\Y)$ has the truncation property (being $\X$ stable by unions it would suffice the directed truncation property) and to this aim we shall apply Proposition \ref{prop:checktrunc}.  Let $\mathcal B\subset \mathcal P(\Y)$ be any family and $A\in  \mathcal P(\Y)$: since clearly $A\wedge B=A\cap B$ for any $B\in \mathcal P(\Y)$, the $\MCP$ property of $B\mapsto B\wedge A$ follows by the distributive law $A\cap (\cup_{B\in\mathcal B}B)=\cup_{B\in\mathcal B}(A\cap B)$.

Notice that $\X$ is first countable (second countable iff $\Y$ is countable) in the sense of Definition \ref{def:sep} (because any directed subset of $\X$ is finite), however if $\Y$ is uncountable, then $\mathcal P(\Y)$ is not so (because there is a chain of subsets of $\Y$ without countable cofinality).

 \item\label{eq:ell1} \underline{Completion of  spaces of sequences}. Let $\X:=\ell_1^+$, namely the set of non-negative and summable sequences of real numbers, ordered componentwise, i.e.\ for $a=(a_n), b=(b_n)\in\X$ we have $a\leq b$ iff $a_n\leq b_n$ for every $n\in\N$. We claim that the completion is the space $\Y$ of $[0,+\infty]$-valued sequences, still ordered componentwise, with the inclusion $\j$. It is trivial to check that $\Y$ is a complete lattice, the supremum of any family being obtained componentwise. Indeed, given   $B\subset\Y$, if we denote by $\bar b_n$ the supremum of  $\{b_n:(b_i)\in B\}\subset[0,+\infty]$ for every $n\in\N$, then $\bar b:=(\bar b_n)$ is the supremum of $B$ in $\Y$ (it is an upper bound and for any other upper bound $u$ and $n\in\N$ we must have $u_n\geq \bar b_n$). 
 
 Since clearly  $\X$ is a dense lower set in $\Y$, according to  Proposition \ref{prop:trunccompl} to conclude it suffices to prove that $\Y$ has the truncation property. As above, this easily follows from Proposition \ref{prop:checktrunc}, as the characterization of suprema we just discussed  shows that for any fixed $a\in\Y$  and $B\subset\Y$ the supremum of $\{b\wedge a:b\in B\}$ is the sequence $\bar b\wedge a=(\bar b_n\wedge a_n)$, proving that $x\mapsto x\wedge a$ has the truncation property.
 
In this example working with summable sequences has little to no role (but will play a role in Section \ref{se:twosided}): the spaces of non-negative sequences in $\ell^p$ have the same completion for every $p\in[1,\infty]$, with the same proof. If instead we complete the whole $\ell^p$, thus  including  sequences with negative terms, than the completion is given by the sequences in $\R\cup\{+\infty\}$ bounded from below by some element of $\ell^p$.
\item\label{ex:c} \underline{Completion of continuous functions}. Let $(M,\tau)$ be a metrizable topological space (the metrizability assumption can be weakened, but doing so would deviate from the main point of this example). We shall denote by $\leq_\p$ the pointwise order of functions from $M$ to $[0,+\infty]$, i.e.\ for  $f,g:M\to[0,+\infty]$ we write $f\leq_\p g$ to mean that $f(x)\leq g(x)$ for every $x\in M$.  We also denote by $\pinf$ and $\psup$ the pointwise inf and sup of a family of functions.

Let $\X:=C(M,[0,+\infty])$ be equipped with the pointwise order $\leq_\p$. 

We use the metrizability of $M$ via the following well-known and easy to prove consequence:
\begin{equation}
\label{eq:usclsc}
\begin{split}
\text{$f:M\to[0,+\infty]$ is upper semicontinuous }\quad&\Leftrightarrow\quad f=\pinf\{g\in \X\ :\ f\leq_\p g \},\\
\text{$f:M\to[0,+\infty]$ is lower semicontinuous }\quad&\Leftrightarrow\quad f=\psup\{g\in \X \ :\ f\geq_\p g \}.
\end{split}
\end{equation}
Let $\Y$ be the set of lower semicontinuous functions from $M$ to $[0,+\infty]$  equipped with the pointwise order. Since the pointwise supremum of any family of lower semicontinuous functions is still lower semicontinuous, the pointwise supremum is also the supremum in $\Y$, which therefore according to Proposition \ref{prop:joinlattice} is a complete lattice (of course, the infimum of an arbitrary family is not the pointwise infimum, but rather its lower semicontinuous envelope - see also the discussion below). For any $f\in\Y$ the collection of $g\in\X$ with $g\leq_\p f$ is stable by joins, hence a directed family. Then \eqref{eq:usclsc} shows that $f$ is the supremum in $\Y$ of such family, proving that $\X$ is dense in $\Y$.

However, in general $\Y$ (with the inclusion) is NOT the completion of $\X$, the problem being that the inclusion of $\X$ in $\Y$ in general does not have the $\MCP$. Consider for instance the case $M:=[0,1]$ and let $f_n(x):=1\wedge (nx)$, so that $f_n\in \X$ with pointwise supremum (and thus supremum in $\Y$) the function $f\in\Y\setminus\X$ equal to 0 in 0 and 1 everywhere else. This sequence also has supremum in $\X$, it being the function identically 1: this   follows noticing that any continuous function $h$ that is $\geq_\p f_n$ for every $n$ must be $\geq 1$ on $(0,1]$ and thus everywhere. This shows that the suprema in $\X$ and $\Y$ of the same increasing sequence differ, so that the inclusion of $\X$ into $\Y$ does not have the $\MCP$, as claimed.

To identify the completion it is better to introduce the upper/lower semicontinuous envelopes: given $f:M\to[0,+\infty]$ we put
\[
\begin{split}
f^\us:=(\pinf)\{g\in \X\ :\ f\leq_\p g\}\qquad\text{i.e.}\qquad f^\us(x):=\lims_{y\to x}f(y)\quad\forall x\in M,\\
f^\ls:=(\psup)\{g\in \X\ :\ f\geq_\p g\}\qquad\text{i.e.}\qquad f^\ls(x):=\limi_{y\to x}f(y)\quad\forall x\in M.
\end{split}
\]
The stated equivalences are easy to check (e.g.\ from  \eqref{eq:usclsc}). Also, it is easy to check that $f^\us$ (resp.\ $f^\ls$) is the pointwise smallest upper semicontinuous function that is $\geq_\p f$ (resp.\ pointwise biggest lower semicontinuous function that is $\leq_\p f$). Hence if $f$ is upper semicontinuous we have $f\geq_\p f^{\ls\us}$ and if it is lower semicontinuous we have $f\leq_\p f^{\us\ls}$. From the monotonicity of $f\mapsto f^\us,f^\ls$ it then follows that
\begin{equation}
\label{eq:ulul}
f^{\us\ls\us\ls}=f^{\us\ls}\qquad\text{ for any $f:M\to[0,+\infty]$.}
\end{equation}
It is now easy to see that the set
\[
\Z:=\{f:M\to[0,+\infty]\ :\ f=f^{\us\ls}\}
\]
ordered with $\leq_\p $ is a complete lattice: 
\begin{equation}
\label{eq:eqnum}
\text{the supremum of any family $(f_i)\subset\Z$ is given by $(\psup_if_i)^{\us\ls}$,}
\end{equation}
 as it follows from the definitions and \eqref{eq:ulul}. Notice that $\X\subset\Z\subset\Y$. We claim that $\Z$, with the inclusion, is the completion of $\X$ and to prove this we shall apply Propositions \ref{prop:complunostep} and \ref{prop:compljoin} (this is duable as $\X$ admits finite joins).  

We have already observed that $\X$ is dense in $\Y$ and this easily implies that it is also dense in $\Z$. Let us now prove that the embedding of $\X$ into $\Z$ has the $\MCP$. This amounts at proving that for any directed (but in fact the argument works for  arbitrary) family $(f_i)\subset\X$ admitting supremum $f$ in $\X$ we have that $f$ coincides with the supremum of $(f_i)$ in $\Z$, that is  $(\psup_if_i)^{\us\ls}$.  Indeed, we have $f\geq_\p f_i$ for every $i$, hence $f\geq_\p \psup_if_i$ and thus $f=f^{\us\ls}\geq_\p (\psup_if_i)^{\us\ls}$. On the other hand, $(\psup_if_i)^{\us}$ is the pointwise infimum of all the continuous functions $h$ that are $\geq_\p\psup_if_i$: any such $h$ must be $\geq_\p f$ (as $f$ is the supremum in $\X$ of the $f_i$'s)  proving that $(\psup_if_i)^{\us}\geq_\p f$ and thus that $ (\psup_if_i)^{\us\ls}\geq_\p f^\ls=f$.

Keeping in mind Propositions \ref{prop:criterio}, the fact that the embedding $\j$ of $\X$ into $\Z$ has the $\MCP$ proves implication $\Leftarrow$ in \eqref{eq:criterio}. According to Propositions  \ref{prop:complunostep} and  \ref{prop:compljoin} to conclude it remains to prove that if $B=(f_i)\subset\X$ is directed with $f:=  (\psup_if_i)^{\us\ls}$ being its supremum in $\Z$ and $g\in \X$ is $\leq_\p f$, then $g\in \widehat B^\X$. Since $\{f'\in\X:f'\leq_\p f_i\text{ for some }f_i\in B\}$ contains the family $(f_i\wedge g)$, the claim will be proved if we show that the supremum in $\X$ of this latter family is $g$ and in turn this will follow if we show that the supremum in $\Z$ of $(f_i\wedge g)$ is $g$.

Thus to conclude it will suffice to prove the following slightly more general claim (notice the analogy with the truncation property and Proposition \ref{prop:checktrunc}): for $(f_i)\subset\X$ and $h\in\X$ arbitrary we have
\begin{equation}
\label{eq:claimulg}
(\psup_if_i)^{\us\ls}\wedge h=(\psup_if_i\wedge h)^{\us\ls}.
\end{equation}
This is a trivial consequence of the continuity of $h$, indeed for any $f:M\to[0,+\infty]$ and $x\in M$ we have $(f\wedge h)^\us(x)=\lims_{y\to x}f\wedge h(y)=(\lims_{y\to x}f(y))\wedge h(x)=f^\us\wedge h(x)$, having used the continuity of $h$   in the second equality. Similarly we have  $(f\wedge h)^\ls(x)=f^\ls\wedge h(x)$ and the claim \eqref{eq:claimulg} easily follows.

We end this example pointing out that, rather trivially by symmetry, another completion of $\X$ is given by the space 
\[
\Z':=\{f:M\to[0,+\infty]\ :\ f=f^{\ls\us}\}
\]
and the inclusion. To see this simply notice that the maps $f\mapsto f^\us$ and $f\mapsto f^\ls$ from $\Z$ to $\Z'$ and from $\Z'$ to $\Z$ respectively are order isomorphism, one the inverse of the other and respect the inclusion of $\X$ into both $\Z$ and $\Z'$.

\item\label{ex:mink} \underline{Completion of Minkowski spacetime}. Let $(V,\|\cdot\|)$ be a  normed vector  space. Put $M:=\R\times B$  and equip it  with the partial order $\preceq$ defined as
\[
(t,v)\preceq (s,w)\qquad\Leftrightarrow\qquad \|v-w\|\leq s-t.
\]
If $V=\R^3$ then $M$ is the standard Minkowski space and $\preceq$ the causal order. In relation with the rough scheme \eqref{eq:schemino} in the introduction, we notice that $M$ with the product topology is globally hyperbolic if and only if $V$ is finite dimensional, while it is (locally, see Definition \ref{def:localcompl} below) directed complete if and only if $V$ is Banach (see   Proposition \ref{prop:2completi} for the simple proof).

We are interested in studying the directed completion of $M$ when $V$ is a Hilbert space, that we shall henceforth denote $H$. We claim that  in this case the directed completion of $M$ is  $M\cup ``\text{future null infinity}"\cup``\text{future time infinity}"$, where $``\text{future time infinity}"$ consists of one point that is the maximal element of the completion, and $``\text{future null infinity}"=\R\times S$, where $S\subset H$ is the unit sphere and the order $\preceq$ on $M\cup``\text{future null infinity}" $ is defined as
\begin{equation}
\label{eq:mbarord}
\begin{array}{rlll}
(c,w)&\!\!\!\!\!\preceq (c',w')\qquad&\Leftrightarrow\qquad w=w'\quad\text{and}\quad c\leq c',&\quad\forall (c,w),(c',w')\in \R\times S \\
(t,v)&\!\!\!\!\!\preceq (c,w)\qquad&\Leftrightarrow \qquad \la v,w\ra\geq t-c,&\quad\forall (c,w)\in \R\times S,\ (t,v)\in M.
\end{array}
\end{equation}
Here the corresponding map $\iota$ is the inclusion. It is immediate to verify that this is a directed complete partial order and that   this completion fully reproduces the causal completion of the standard Minkowski space if $H=\R^3$. The proof  that this is really the the directed completion of $M$ follows by a simple  study  of the geometry of the spacetime. The key point is that the collection $\bar M:=\{\widehat D:D\subset M\text{ is directed}\}$ ordered by inclusion is isomorphic to $M\cup ``\text{future null infinity}"\cup``\text{future time infinity}"$  and in particular is complete. Since $\bar M$ as just defined   corresponds at what we obtain in the first step in the proof of Theorem \ref{thm:excompl}, the completeness claim proves that $\bar M$, and thus $M\cup ``\text{future null infinity}"\cup``\text{future time infinity}"$, is the directed completion of $M$.

To study $\bar M$ amounts at studying sets of the form $\widehat D$ for $D\subset M$ directed. Let us thus fix such $D$ and notice that since $(t',v')\preceq (t,v)$ implies $t-t'\geq |v-v'|\geq ||v|-|v'||\geq 0$, the functions $D\ni (t,v)\mapsto t,t-|v|$ are both monotone, hence admit limits $T,c\in(-\infty,+\infty]$ respectively (the limits are intended as limit of nets). We distinguish the following cases:
\begin{itemize}
\item[a)] `Future time infinity'. Suppose $c=+\infty$. Then from $t-\bar t-|v-\bar v|\geq t-\bar t-|v|-|\bar v|$ we deduce that for any $(\bar t,\bar v)\in M$ we have  $\lim_{(t,v)\in D}t-\bar t-|v-\bar v|=+\infty$, proving that $(\bar t,\bar v)\in\downarrow\! D$ and thus that $\widehat D=M$.
\item[b)] `$M$'. Suppose $c,T<+\infty$. Let $(t_1.v_1),(t_2.v_2)\in D$ and then $(t,v)\in D$ be with $(t_i,v_i)\preceq (t,v)$, $i=1,2$. We have $t-t_i\geq |v-v_i|$ and thus $2T-t_1-t_2\geq 2t-t_1-t_2\geq |v_1-v_2|$, hence $\lims_{(t_1,v_1),(t_2,v_2)\in D}|v_1-v_2|\leq \lims_{(t_1,v_1),(t_2,v_2)}2T-t_1-t_2=0$. This proves that $D\ni (t,v)\mapsto v\in H$ is a Cauchy net, hence converging to some $\bar v\in H$. We claim that $\sup D=(T,\bar v)$ (and thus that $\widehat D=\,\downarrow\!\!(T,\bar v)$). Fix $(t',v')\in D$ and pass to the limit in $(t,v)\in D$ with $(t,v)\succeq(t',v')$ in the bound $t-t'\geq |v-v'|$ to deduce that $T-t'\geq |\bar v-v'|$, i.e.\ $(T,\bar v)\succeq (t',v')$ proving that $(T,\bar v)$ is an upper bound for $D$. Now suppose that $(\tilde t,\tilde v)\in M$ is another upper bound. Then passing to the limit in $(t,v)\in D$ in the bound $\tilde t-t\geq |\tilde v-v|$ we deduce that $(T,\bar v)\leq (\tilde t,\tilde v)$, as claimed.
\item[c)] `Future null infinity'. Suppose $c<+\infty$ and $T=+\infty$. Then the limit of $|v|$ in $D$ is $+\infty$ and  from $t-|v|=t(1-\tfrac{|v|}t)$ we deduce that the limit of $\tfrac{t}{|v|}$ in $D$ exists and is 1. Now let $(t,v),(t',v')\in D$ be with $(t,v)\preceq (t',v')$: starting from $c+|v'|-t\geq t'-t\geq |v'-v|\geq ||v'|-|v||$, squaring, dividing by $|v||v'|$ and with little manipulation we get 
\[
-\tfrac{c^2-|v|^2+t^2}{|v||v'|}-\tfrac{2c}{|v|}-\tfrac{2ct}{|v||v'|}+\tfrac {2t}{|v|}\leq 2\la\tfrac{v}{|v|},\tfrac{v'}{|v'|}\ra\leq2.
\]
Passing to the limit we obtain that $\lim_{(t.v)\in D}\lim_{(t',v')\in D}\la\tfrac{v}{|v|},\tfrac{v'}{|v'|}\ra=1$, from which it easily follows that $D\ni (t,v)\mapsto \tfrac{v}{|v|}\in H$ is a Cauchy net, thus converging to some $w\in S$. 

We claim that
\begin{equation}
\label{eq:claimdhat}
\widehat D=\{(t,v)\in M: \la v,w\ra\geq t-c\}
\end{equation}
and observe that if this holds,  the correspondence $\widehat D\to (c,w)$ is an isomorphism of partial orders (the $\widehat D$'s are ordered by inclusion and the $(c,w)$'s as in \eqref{eq:mbarord}). To prove the claim we notice that for any $(\bar t,\bar v)\in M$ we have
\[
|v-\bar v|=\sqrt{|v-\bar v|^2}=|v|(\sqrt{1-2\tfrac{\la v,\bar v\ra}{|v|^2}+\tfrac{|\bar v|^2}{|v|^2}})=|v|-\la\tfrac v{|v|},\bar v\ra+O(|v|^{-1})
\]
and therefore $\lim_{(t,v)\in D}t-\bar t-|v-\bar v|=c-\bar t+\la w,v\ra$. The claim \eqref{eq:claimdhat} easily follows.
\end{itemize}
\item\label{it:prod}\underline{Completion of products}.  If $(\X_1,\leq_1),(\X_2,\leq_2)$ are two partial orders, then the \emph{product order} $\leq_\times$ on $\X_1\times\X_2$ is defined as $(a_1,a_2)\leq_\times(b_1,b_2)$ whenever $a_1\leq_1b_1$ and $a_2\leq_2b_2$. The following is a simple, yet interesting, general fact:
\begin{proposition}[Componentwise continuity implies continuity]\label{prop:mcpprod}
Let $(\X_1,\leq_1)$, $(\X_2,\leq_2)$ and  $(\Z,\leq_\z)$ be partial orders and $T:\X_1\times\X_2\to \Z$. Then $T$ has the $\MCP$ if and only if $T(\cdot,x_2):\X_1\to\Z$ and $T(x_1,\cdot):\X_2\to \Z$ all have the $\MCP$ for any $x_1\in\X_1$ and $x_2\in\X_2$.
\end{proposition}
\begin{proof}
The `only if' is clear. Now assume that when one variable is frozen the map has the $\MCP$ and let $D\subset\X_1\times\X_2$ be directed with $(x_1,x_2)=\sup D$. Let $D_1,D_2$ be the projections of $D$ on the first and second factor respectively and notice that, rather trivially from the definition of product order, $D_1,D_2$ are directed with suprema $x_1,x_2$ respectively. Then we have
\[
T(x_1,x_2)=\sup_{a_2\in D_2} T(x_1,a_2)=\sup_{a_2\in D_2}\sup_{\tilde a_1\in D_1} T(\tilde a_1,a_2)=\sup_{(a_1,a_2)\in D\atop (\tilde a_1,\tilde a_2)\in D} T(\tilde a_1,a_2)=\sup_{(b_1,b_2)\in D} T(b_1,b_2),
\]
where in the last equality we used that $D$ is directed and componentwise  monotonicity of $T$.
\end{proof}
The following is now expected. Notice that other completions, such as  the Dedekind-MacNeille one, do  not respect products (because the product of complete lattices is in general not a complete lattice).
\begin{proposition}[`Completion of product is  product of completions']\label{prop:complprod}
Let $(\X_1,\leq_1)$ and\linebreak  $(\X_2,\leq_2)$ be two partial orders and $(\overline\X_1,\iota_1),(\overline\X_2,\iota_2)$ be their completions. 

Then the completion of $(\X_1\times\X_2,\leq_\times)$ is $\overline \X_1\times\overline\X_2$ with the map $(x_1,x_2)\mapsto (\iota_1(x_1),\iota_2(x_2))$.
\end{proposition}
\begin{proof} Let $(\overline{\X_1\times\X_2},\j)$ be the completion of $\X_1\times\X_2$. We claim that there exists a unique map ${\sf T}:\overline \X_1\times\overline\X_2\to \overline{\X_1\times\X_2}$ such that ${\sf T}\circ\iota=\j$. From the density of $\iota(\X_1\times\X_2)$ in $\overline \X_1\times\overline\X_2$ (that is an easy consequence of density in the factors) uniqueness of ${\sf T}$ is clear. For existence start fixing $x_2\in\X_2$ and notice that the map $\X_1\ni x_1\mapsto \j(x_1,x_2)\in\overline{\X_1\times\X_2}$ has, for trivial reasons, the $\MCP$. By the universal property of $(\X_1,\iota_1)$ we see that there is a (unique) map ${\sf T}_{x_2}:\overline\X_1\to \overline{\X_1\times\X_2}$ with the $\MCP$ such that ${\sf T}_{x_2}(\iota_1(x_1))=\j(x_1,x_2)$ for every $x_1\in\X_1$. We claim that for every  $y_1\in\overline\X_1 $ the map $\X_2\ni x_2\mapsto {\sf T}_{x_2}(y_1)\in  \overline{\X_1\times\X_2}$ has the $\MCP$. Indeed, fix $y_1\in\overline\X_1$ and let $A\subset\X_1$ with tip be so that $y_1=\sup\iota_1(A)$ (the choice $A:=\{x_1:\iota_1(x_1)\leq y_1\}$ does the job). Let $D\subset\X_2$ be directed admitting sup in $\X_2$. Then we have
\[
\begin{split}
\sup_{x_2\in D}{\sf T}_{x_2}(y_1)&=\sup_{x_2\in D}\sup_{x_1\in A}{\sf T}_{x_2}(\iota_1(x_1))=\sup_{x_2\in D}\sup_{x_1\in A}\j(x_1,x_2)=\sup_{x_1\in A}\sup_{x_2\in D}\j(x_1,x_2)\\
&=\sup_{x_1\in A}\j(x_1,\sup D)=\sup_{x_1\in A}{\sf T}_{\sup D}(x_1)={\sf T}_{\sup D}(y_1),
\end{split}
\]
 proving the claim. By the universal property of $(\X_2,\iota_2)$ it follows that there is ${\sf T}(y_1,\cdot):\overline\X_2\to\overline{\X_1\times\X_2}$ such that ${\sf T}(y_1,\iota_2(x_2))={\sf T}_{x_2}(y_1)$ for every $x_2\in\X_2$. The map $\overline\X_1\times\overline\X_2\ni (y_1,y_2)\mapsto {\sf T}(y_1,y_2)\in\overline{\X_1\times\X_2}$ has the $\MCP$ in each component, thus by Proposition \ref{prop:mcpprod} above it has the $\MCP$. The fact that ${\sf T}\circ\iota=\j$ is clear from the construction.

Having proved existence and uniqueness of such  ${\sf T}$, the rest easily follows. Indeed, since $\iota$ has the $\MCP$ (again, by Proposition \ref{prop:mcpprod} above), by the universal property of $\overline{\X_1\times\X_2}$ we know that there is a unique ${\sf S}:\overline{\X_1\times\X_2}\to\overline\X_1\times\overline\X_2$ such that ${\sf S}\circ\j=\iota$.

The map ${\sf S}\circ{\sf T}:\overline\X_1\times\overline\X_2\to \overline\X_1\times\overline\X_2$ has the $\MCP$ and is the identity on the dense subset $\iota(\X_1\times\X_2)$ of $\overline\X_1\times\overline\X_2$. Hence ${\sf S}\circ{\sf T}$ is the identity on $\overline\X_1\times\overline\X_2$ (this claim can be proved, e.g., by transfinite recursion on $\iota(\X_1\times\X_2)^{\uparrow^\alpha}$). Similarly, the map ${\sf T}\circ{\sf S}:\overline{\X_1\times\X_2}\to\overline{\X_1\times\X_2}$ has the $\MCP$ and is the identity on the dense subset  $\j(\X_1\times\X_2)$ of $\overline{\X_1\times\X_2}$, hence, as above, is the identity on $\overline{\X_1\times\X_2}$.

We thus proves that ${\sf T}$ and ${\sf S}$ are one the inverse of each other and by construction they satisfy ${\sf T}\circ\iota=\j$  and ${\sf S}\circ \j=\iota$. Since such ${\sf T}$, ${\sf S}$ are also unique, the proof is complete by the very definition of uniqueness for the completion.
\end{proof}

\end{enumerate}

\end{subsubsection}

\end{subsection}
\begin{subsection}{Variants}

\begin{subsubsection}{Sequential version}\label{se:seqcompl}

In some applications it might be that one is not interested in suprema of all directed subsets, but only of non-decreasing sequences. When this is the case, the directed completion is possibly too big, so to say, and one should stop the construction earlier. Before coming to the actual definitions and statements, it is worth to recall the following basic fact, valid in any partial order $(\X,\leq)$:
\begin{equation}
\label{eq:countabledirected}
\text{any  countable and directed $D\subset\X$ admits a cofinal non-decreasing sequence}.
\end{equation}
This means that for each such $D$ there is $(x_n)\subset D$ with $x_n\leq x_{n+1}$ so that for every $d\in D$ we have $d\leq x_n$ for some $n\in\N$ (and thus for all $n$'s sufficiently big). Because of this, asking for existence of  suprema of countable directed sets is the same as asking for suprema of non-decreasing sequences. The proof of \eqref{eq:countabledirected} is easy: let $(d_n)$ be an enumeration  of the elements of $D$, put $x_0:=d_0$ and then recursively find $x_{n+1}$ that is $\geq x_n$ and $\geq d_n$.
\begin{remark}{\rm
In fact in \cite{Mark76} it has been proved a  more general statement: a partial order is directed complete if and only if it is chain complete (i.e.\ every totally ordered subset has a supremum).

Still, recall that, as it is well known, not every directed set has a totally ordered cofinal set: in Example \ref{eq:complsubset} in Section \ref{se:exbelli} any totally ordered family of finite subsets is at most countable and if $\Y$ is uncountable no such family can be cofinal.
\fr
}\end{remark}

\begin{definition}[Sequential Monotone Convergence Property] Let $(\X_1,\leq_1)$, $(\X_2,\leq_2)$ be two partial orders and $T:\X_1\to\X_2$. We say that $T$ has the sequential Monotone Convergence Property ($\sMCP$ in short) if for any non-decreasing sequence $(x_n)\subset\X_1$ admitting supremum $x_\infty\in\X_1$ we have that $T(x_\infty)=\sup_nT(x_n)$.
\end{definition}
Clearly, composition of maps with the $\sMCP$ has the $\sMCP$ and so does the identity map on any partial order. The following definition is also natural:
\begin{definition}[Sequential directed completeness]
A partial order $(\X,\leq)$ is called sequentially directed-complete partial orders (sdcpo, in short) if any non-decreasing sequence has a supremum. 
\end{definition}
There are several analogies between the concepts just described and those in Section \ref{se:constr}. For instance, it is natural to declare a subset $A$ of a partial order to be {\bf sequentially-directed-sup-closed} if any non-decreasing sequence in $A$ admitting  supremum in $\X$ has supremum in $A$. Then  the sequential-directed-sup-closure ${\overline A}^s$ is defined as
\[
{\overline A}^s:=\text{smallest $B\subset\X$ containing $A$ and  sequentially-directed-sup-closed}.
\] 
In analogy with \eqref{eq:barordinali}, this closure can be obtained by a recursive operation starting from the definition of $A^{\uparrow_s}$ as
\[
A^{\uparrow_s}:=\big\{\text{supremum of non-decreasing sequences in }A\big\}.
\]
Indeed, defining $A^{\uparrow^\alpha_s}$ for every ordinal $\alpha$ as $A^{\uparrow^0_s}:=A$, then $A^{\uparrow^{\alpha+1}_s}:=(A^{\uparrow^\alpha_s})^{\uparrow_s}$ and finally for limit ordinals $A^{\uparrow^\alpha_s}:=\cup_{\beta<\alpha}A^{\uparrow^\beta_s}$, it is easy to see that ${\overline A}^s=A^{\uparrow^\alpha_s}$ for every $\alpha$ sufficiently big. The proof is the same as for \eqref{eq:barordinali}. In fact, in this case we can bound a priori the length of the iteration, as we always have
\begin{equation}
\label{eq:barordinali2}
{\overline A}^s=A^{\uparrow^{\omega_1}_s},
\end{equation}
where $\omega_1$ is the first uncountable ordinal. Indeed, if  $(x_n)\subset A^{\uparrow^{\omega_1}_s}$, then $x_n\in A^{\uparrow^{\alpha_n}_s}$ for suitable countable ordinals $\alpha_n$ so that putting $\alpha:=\sup_{n\in\N}\alpha_n$ we have $\alpha<\omega_1$ and thus the supremum of $(x_n)$, if it exists at all, is found in $A^{\uparrow^{\alpha+1}_s}\subset A^{\uparrow^{\omega_1}_s}$. Notice also that Example \ref{ex:alphafreccia} in Section \ref{se:techex} with $\alpha:=\omega_1$ shows a case were $A^{\uparrow^{\alpha}_s}\subsetneq {\overline A}^s$ for every $\alpha<\omega_1$.

We then  have the following analogue of Theorem \ref{thm:excompl}:
\begin{theorem}[Directed sequential completion]\label{thm:dirseqcompl}
Let $(\X,\leq)$ be a partial order. Then there is a sequentially directed complete partial order $({\overline\X}^s,\bar\leq)$ and a map $\iota:\X\to{\overline\X}^s$ with the $\sMCP$ such that the following holds. For any sequentially complete partial order $(\Z,\leq_\Z)$ and any $T:\X\to\Z$ with the $\sMCP$ there is a unique $\bar T:{\overline\X}^s\to\Z$ with the $\sMCP$ such that $\bar T\circ\iota=T$.

The couple ${\overline\X}^s,\iota$ is unique up to unique isomorphism, i.e.\ if  $(\tilde\X,\tilde\leq)$ and $\tilde\iota$ have the same properties, then there is a unique map  $\mathcal J:\bar \X\to\tilde\X$ with the $\sMCP$ so that $\tilde\iota=\mathcal J\circ\iota$.

Finally, the map $\bar T:{\overline\X}^s\to\Z$ with the $\sMCP$ corresponding to $T:\X\to\Z$ that we mentioned above has the explicit expression:
\begin{equation}
\label{eq:chieetbar2}
T(\bar x)=\sup_{x\in\X:\ \iota(x)\leq\bar x}T(x)\qquad\qquad\forall \bar x\in\bar\X,
\end{equation}
where it is part of the claim that the supremum in the right hand side exists.
\end{theorem}
\begin{proof}\ \\
{\sc Uniqueness}  The universal properties of the two completions produce maps $\varphi:{\overline\X}^s\to\tilde\X$ and $\psi:\tilde\X\to{\overline\X}^s$ with the $\sMCP$ such that $\varphi\circ\iota=\tilde\iota$ and $\psi\circ \tilde\iota=\iota$. Hence $\psi\circ\varphi:\bar\X\to\tilde\X$ has the $\sMCP$ and satisfies $(\psi\circ\varphi)\circ\iota=\iota$, thus by the uniqueness of $\bar T$ (with $T=\iota$) we deduce that $\psi\circ\varphi$ is the identity on $\bar\X$. Similarly, $\varphi\circ\psi$ is the identity on $\tilde\X$, proving that $\varphi,\psi$ are one the inverse of the other and the claim.\\
{\sc Existence} The construction is a variant of the one in Theorem \ref{thm:excompl}.  Let $\Y$ and $\iota:\X\to\Y$ be as in the proof of Theorem \ref{thm:excompl}. Put $\X_0:=\iota(\X)\subset\Y$ and then $\X_\alpha:=\X_0^{\uparrow^\alpha_s}$. We claim that  ${\overline\X}^s:=\X_{\omega_1}$, together with $\iota$, does the job.  The fact that ${\overline\X}^s$ is sequentially directed-sup-closed follows from \eqref{eq:barordinali2}, so we need only to prove the universal property and formula \eqref{eq:chieetbar2}. 

This can be proved following closely the arguments in the proof of Theorem \ref{thm:excompl}: the only difference is that here we are considering only suprema of sequences and because of this the assumed sequential completeness of $\Z$ and the fact that $T:\X\to\Z$ has the sequential $\MCP$ suffice to perform the construction. We omit the details.
\end{proof}

\end{subsubsection}

\begin{subsubsection}{Local completion}\label{se:loccompl}
In some circumstances one might be content with a completion procedure that does not add `points at infinity' but only `fills existing gaps'. In this case, the relevant concept is that of locally directed complete partial order:
\begin{definition}[Locally directed complete partial order]\label{def:localcompl} We say that a partial order $(\X,\leq)$ is locally directed complete (or simply locally complete) if any directed subset admitting an upper bound has a supremum.

Similarly, we say that it is locally sequentially directed complete if any non-decreasing sequence admitting an upper bound has a supremum.
\end{definition}
The following statement can now be easily proved:
\begin{theorem}[Local (sequential) directed completion]\label{thm:loccompl}
Let $(\X,\leq)$ be a partial order. Then there is a locally directed complete (resp.\ locally sequentially complete) partial order $(\overline\X,\bar\leq)$ and a map $\iota:\X\to{\overline\X}^s$ with the $\MCP$ (resp.\ $\sMCP$) such that the following holds. For any locally complete (resp.\ locally sequentially complete) partial order $(\Z,\leq_\Z)$ and any $T:\X\to\Z$ with the $\MCP$ (resp.\ $\sMCP$) there is a unique $\bar T:{\overline\X}^s\to\Z$ with the $\MCP$ (resp.\ $\sMCP$)  such that $\bar T\circ\iota=T$.

The couple ${\overline\X}^s,\iota$ is unique up to unique isomorphism, i.e.\ if  $(\tilde\X,\tilde\leq)$ and $\tilde\iota$ have the same properties, then there is a unique map  $\mathcal J:\bar \X\to\tilde\X$ with the $\MCP$ (resp.\ $\sMCP$)  so that $\tilde\iota=\mathcal J\circ\iota$.

Finally, the map $\bar T:{\overline\X}^s\to\Z$ with the $\MCP$ (resp.\ $\sMCP$)  corresponding to $T:\X\to\Z$ that we mentioned above has the explicit expression:
\[
T(\bar x)=\sup_{x\in\X:\ \iota(x)\leq\bar x}T(x)\qquad\qquad\forall \bar x\in\bar\X,
\]
where it is part of the claim that the supremum in the right hand side exists.
\end{theorem}
\begin{proof} The uniqueness part follows precisely as for the directed and sequential direct completion. Existence also follows along similar lines: it suffices, in the recursion process, to be sure   that at each step  only suprema of sets with upper bounds are added. Alternatively, start  from the completion as in Theorem \ref{thm:excompl} (resp.\ Theorem \ref{thm:dirseqcompl}) and consider the subspace of those points being $\leq \iota(x)$ for some $x\in\X$: it is clear that this is locally directed (sequentially) complete and that it has the desired universal property.
\end{proof}

\end{subsubsection}

\begin{subsubsection}{(Lack of reasonable) two-sided completion}\label{se:twosided}
The whole discussion made in this paper has been built upon the concept of supremum of directed sets, but of course every definition/result has a symmetric analogue  where one deals infima of filtered sets (a subset $A$ of a partial order is filtered if for any $a,b\in A$ there is $c\in A$ with $c\leq a$ and $c\leq b$).

It is totally obvious that  a  version of  Theorem \ref{thm:excompl} (and of its variants Theorems \ref{thm:dirseqcompl}, \ref{thm:loccompl}) can be proved in the category of partial orders where morphisms respect  filtered infima, rather than directed suprema (either mimic the arguments or  apply the given versions to the `dual' order $\leq'$ defined as $x\leq'y$ iff $y\leq x$).

A more compelling question concerns whether one can do both at the same time, i.e.\ whether there is a version of Theorem \ref{thm:excompl}, or its variants, in the category of partial orders where morphisms are maps  respecting both directed suprema and filtered infima. 

This is unclear to us: an abstract existence result can be established, see Remark \ref{re:abstrcompl}, but we are not aware of any constructive proof. More importantly,  such completion  has some  undesirable properties from the perspective of (at least some) working analyst. The example we have in mind is the partial order $\X:=\ell^1_+$, whose two-sided completion, if it exists at all,  is not $[0,\infty]^\N$ with the inclusion, despite $[0,\infty]^\N$  being both `upward' and `downward' complete and the inclusion respecting both directed suprema and filtered infima.

To see why, consider the map $T:\X\to[0,+\infty]$ sending $(a_n)$ to $T((a_n)):=\sum_{n}a_n$.  Here the target space $[0,+\infty]$ is equipped with its canonical order, that  is clearly both `upward' and `downward' complete.

It is easy to see that $T:\X\to[0,+\infty]$ respects both directed suprema and filtered infima: this can be seen for instance via an application of the  dominated convergence theorem, noticing that  any directed (and similarly filtered) set in $\X$ admits a countable cofinal subset linearly ordered (see also Lemma \ref{le:esssup}). Since we already know from Example \ref{eq:ell1} that  $[0,\infty]^\N$ is the directed completion of $\X$, we know   that $T$ can be extended to a map $\bar T:[0,\infty]^\N\to[0,+\infty]$  respecting directed suprema and from formula \eqref{eq:chieetbar} we see that such extension is given by
\[
\bar T((a_n))=\sup_{(b_n)\in\X\atop b_n\leq a_n\ \forall n}T((b_n))=\sup_{(b_n)\in\X\atop b_n\leq a_n\ \forall n}\sum_n b_n=\sum_na_n\qquad\forall (a_n)\in [0,\infty]^\N.
\]
Notice that  Beppo Levi's monotone convergence theorem ensures that $\bar T$ respects directed suprema. We thus defined $\bar T:[0,+\infty]^\N\to[0,+\infty]$ as the only possible extension of $T$ that respects directed suprema. The problem is that $\bar T$ does not respect filtered infima, proving that $[0,+\infty]^\N$ is not the desired two-sided completion, as it does not allow for an extension of $T$ respecting both directed suprema and filtered infima.  To see why $\bar T$ does not respect filtered infima just  let $A_i\in[0,+\infty]^\N$ be the sequence whose first $i$ entries are 0 and the rest $+\infty$. Then clearly $\bar T(A_i)=+\infty$ for every $i\in\N$, and $A_i\geq A_{i+1}$. Also,  we have $(0,0,\ldots)=\inf_iA_i$ but $\bar T((0,0,\ldots))=0\neq+\infty=\inf_i\bar T(A_i)$. This shows that $[0,+\infty]^\N$ is not the two-sided completion of $\X$, as claimed, and also that such completion cannot be found by iterating a process of  `upward' and `downward' completions, as in this case the first step brings us to $[0,+\infty]^\N$ which is already two-sided complete.

We stress the fact that in checking that $\bar T((a_n)):=\sum_na_n$ does not respect filtered infima we are, evidently, just witnessing the fact that for positive functions Beppo Levi's theorem works only for increasing sequences and not for decreasing ones.

\medskip

The same line of thought, brought a bit to the more abstract level, shows that the two-sided completion $(\bar\X,\iota)$ of $\X$ must be quite weird. Indeed, for fixed $i\in\N$ we can consider the directed set $D_i\subset\X$ made of sequences having 0 at the first $i$ entries. The supremum of $\iota(D_i)$ must exist in $\bar\X$: call it $A_i$. Notice that from $D_i\supset D_{i+1}$ we get $A_i\geq A_{i+1}$ in $\bar \X$ and thus there exists the infimum $A_\infty:=\inf_iA_i$ in $\bar \X$. Since $0:=(0,0,\ldots)\in\X$ belongs to each of the $D_i$'s, we have $\iota(0)\leq A_i$ for every $i$ and thus $\iota(0)\leq A_\infty$. The argument above about the extension of $T:\X\to[0,+\infty]$ to the completion shows that $\bar T(A_\infty)=+\infty$ and $\bar T(\iota(0))=T(0)=0$, so that in particular $\iota(0)\neq A_\infty$. 

In fact, this completion should contain many more objects like $A_\infty$. To see why let $I\subset\N$ be arbitrary and for $i\in\N$ define $D^I_i\subset\X$ as the set of sequences that are 0 in the first $i$ entries and also in all the entries with index outside $I$. As before, $D^I_i$ is directed, thus we can define $A^I_i:=\sup\iota(D^I_i)\geq A^I_{i+1}$ and also $A^I_\infty:=\inf_iA^I_i$  (with this notation, the element $A_\infty$ previously defined would be $A^\N_\infty$). It is easy to see that if $I\setminus J\subset \N$ is finite, then $A^I_\infty\leq A^J_\infty$. Also, if $I\setminus J\subset \N$ is infinite, then  we can see $A^I_\infty\neq A^J_\infty$. Indeed,   let $S:\X\to [0,+\infty]$ be defined as $S((a_n)):=\sum_{n\in I\setminus J}a_n$, notice that this respects directed suprema and filtered infima and let $\bar S:\bar \X\to[0,+\infty]$ the extension that also respects directed suprema and filtered infima. Then the same arguments previously used show that $\bar S(A^I_\infty)=+\infty$ and $\bar S(A^J_\infty)=0$, proving that $A^I_\infty\neq A^J_\infty$.

Now recall that we can find a family $\{I_\lambda:\lambda\in\R\}$ of subsets of $\N$ with $I_\lambda\subset I_\eta$ for $\lambda\leq \eta$ and $\#(I_\eta\setminus I_\lambda)=\aleph_0$ when $\lambda< \eta$ (put $I_\lambda:=\{n\in\N:r_n<\lambda\}$, where  $(r_n)$ is an enumeration of the rationals), to deduce that the map $\R\ni\lambda\mapsto A^{I_\lambda}_\infty\in\bar\X$ is injective, monotone and with values between $\iota(0)$ and $A^\N_\infty$.

\end{subsubsection}

\end{subsection}

\begin{subsection}{Spaces of maps with the Monotone Convergence Property}\label{se:mapsmcp}

Let $\X$ be a set and $\Y$ a partial order. Then we can surely define a \emph{pointwise partial order} $\leq_\p$ on the collection $\Y^\X$ of maps from $\X $ to $\Y $ by declaring
\[
T\leq_\p S\qquad\Leftrightarrow\qquad T(x)\leq_\Y S(x)\qquad\forall x\in\X.
\]
If both $\X$ and $\Y$ are partial orders, we will denote by  $\MCP(\X,\Y)$ the collection of maps with the $\MCP$ from  $\X$ to  $\Y$. It is easy to check that
\begin{equation}
\label{eq:mcpcomplete}
\big(\mcp(\X,\Y),\leq_\p\!\big)\qquad
\text{is a directed complete partial order if $\Y$ is so},
\end{equation}
and that  for any  directed family  $(T_i)_{i\in I}\subset\mcp(\X,\Y)$  
\begin{equation}
\label{eq:mcpcomplete2}
\text{the supremum $T:=\sup_iT_i$ is given by the pointwise formula $T(x):=\sup_iT_i(x)\quad\forall x\in\X$,}
\end{equation}
see also \cite[Lemma 4.1(i)]{CompendiumLattices} for a more general presentation of this fact.  To  see the above, it suffices to check that  $T$  as in \eqref{eq:mcpcomplete2} has the $\MCP$ as well (the fact that $T$ is well defined follows from the completeness of $\Y$). Thus let $D\subset\X$ be directed admitting supremum, call it $x$, and notice that
\[
T(x)=\sup_{i\in I}T_i(x)=\sup_{i\in I}\sup_{d\in D}T_i(d)=\sup_{(d,i)\in D\times I}T_i(d)=\sup_{d\in D}\sup_{i\in I}T_i(d)=\sup_{d\in D} T(d),
\]
as desired.

More can be said on the structure of $\mcp(\X,\Y)$ if  $\Y$ is also a lattice. Start noticing that in this case  we can consider the pointwise maximum $T\vee_\p S$ of any two maps $T,S:\X\to\Y$ and the same arguments used above show that this map has the $\MCP$ if both $T$ and $S$ have it. It follows from \eqref{eq:mcpcomplete} that 
\begin{equation}
\label{eq:mcplattice}
\big(\mcp(\X,\Y),\leq_\p\!\big)\qquad\text{is a  complete lattice if $\Y$ is so.}
\end{equation}
In particular, we can define:
\begin{definition}[Projection over maps with the $\mcp$]\label{def:prmcp}
Let $\X,\Y$ be partial orders, with $\Y$ being a complete lattice and $T:\X\to \Y$  a map.

The map ${\sf Pr}(T):\X\to\Y$ is defined as the pointwise supremum of all the $S\in\MCP(\X,\Y)$ such that $S\leq_\p T$.
\end{definition}
By \eqref{eq:mcplattice} and \eqref{eq:mcpcomplete2} we see that ${\sf Pr}(T)$ is well defined and belongs to $\MCP(\X,\Y)$. It is also clearly an idempotent operator and the identity over $\MCP(\X,\Y)$, whence the terminology of `projection' and the notation. It is also monotone, meaning that
\begin{equation}
\label{eq:monpr}
S\leq_\p T\qquad\Rightarrow\qquad {\sf Pr}(S)\leq_\p{\sf Pr}(T),
\end{equation}
as it is clear from the definition.

For later use, in order to understand the properties of ${\sf Pr}(T)$ it will be convenient to have a  more explicit construction of it. The  idea is to replace $T:\X\to\Y$ with the map $P(T):\X\to\Y$ defined as
\begin{equation}
\label{eq:defprt}
P(T)(x):=\inf_{A\subset \X\text{ directed}\atop x=\sup A}\sup\, T(A).
\end{equation}
Concretely,  it is not clear whether $P(T)$ so defined has the $\MCP$, so the actual projection is found via an iteration of this procedure (an iteration that, to the best of our understanding, is necessary even if we  work with  sets with tips rather than directed ones in formula \eqref{eq:defprt}).

\begin{proposition}[Construction of the projection]\label{prop:projmcp}
Let $\X,\Y$ be partial orders, with $\Y$ being a complete lattice and $T:\X\to \Y$  a map.

Recursively define a map $P_\alpha T:\X\to\Y$ for  any ordinal number $\alpha$ as follows: $P_0T:=T$, $P_{\alpha+1}T=P(P_\alpha T)$ (where $P$ is given by formula \eqref{eq:defprt}) and, for limit ordinals $\alpha$, we put $P_\alpha T(x):=\inf_{\beta<\alpha}P_\beta T(x)$ for any $x\in \X$.

Then the sequence eventually stabilizes on ${\sf Pr}(T)$, i.e.\  for some ordinal $\bar\alpha$ we have $P_\alpha T(x)={\sf Pr}(T)(x)$ for any $\alpha\geq\bar\alpha$.
\end{proposition}
\begin{proof}
For any $x\in\X$ the singleton $\{x\}$ is a directed set with supremum $x$, hence admissible in the definition \eqref{eq:defprt}. It follows that the map $\alpha\mapsto P_\alpha T(x)$ is non-increasing, hence for cardinality reasons there must be some ordinal $\alpha_x$ such that the map is constant after that. Then $\bar\alpha:=\sup_{x\in\X}\alpha_x$ has the required properties.

In particular, we have $P(P_{\bar\alpha}T)=P_{\bar\alpha+1}T=P_{\bar\alpha}T$ and thus, by the very definition \eqref{eq:defprt} we see that $P_{\bar\alpha}T$ has the $\MCP$. Since clearly $P_{\bar\alpha}T\leq_\p T$, it remains to prove that if  $S\in\mcp(\X,\Y)$ is $\leq_\p T$, then it is $\leq_\p P_{\alpha}T$ for any ordinal $\alpha$. This, is trivial by construction and transfinite induction. Indeed, if $S\leq_\p T$ has the $\MCP$, then for every $x\in\X$  we have
\[
P(T)(x)=\inf_{A\subset \X\text{ directed}\atop x=\sup A}\sup\, T(A)\geq \inf_{A\subset \X\text{ directed}\atop x=\sup A}\sup\, S(A)=\inf_{A\subset \X\text{ directed}\atop x=\sup A} S(x)=S(x),
\]
i.e.\ $S\leq_\p P(T)$. This shows that $S\leq_\p P_\alpha T$ implies  $S\leq_\p P_{\alpha+1} T$. Moreover, if $\alpha$ is a limit ordinal and we know that $S\leq_\p P_\beta T$ for any $\beta<\alpha$, then it follows from the definition of $P_\alpha T$ that  $S\leq_\p P_\alpha T$ as well, concluding the proof by transfinite recursion.
\end{proof}

\end{subsection}

\end{section}

\begin{section}{Hyperbolic Banach Spaces}

\begin{subsection}{Wedges, cones and cones with joins}

\begin{definition}[Wedge]\label{def:wedge}
A prewedge $\W$ is a set equipped with:
\begin{itemize}
\item[i)] A commutative and associative binary operation of `sum'  with identity element, denoted $0$ (in the terminology of abstract algebra $(\W,+,0)$ is called a commutative monoid).
\item[ii)] An operation of product with non-negative reals, i.e.\ a map $[0,+\infty)\times \W\ni (\lambda, v)\mapsto \lambda\cdot v\in \W$ (whose result shall be typically written $\lambda v$) that satisfies
\[
\begin{array}{rll}
\lambda\cdot(\eta\cdot v)\!\!\!&=(\lambda\eta)\cdot v,\qquad&\forall \lambda,\eta\in[0,\infty),\ v\in \W,\\
0\cdot v\!\!\!&=0,\qquad&\forall v\in \W,\\
1\cdot v\!\!\!&=v,\qquad&\forall v\in \W,\\
(\lambda+\eta)\cdot v\!\!\!&=\lambda \cdot v+\eta\cdot v,\qquad&\forall \lambda,\eta\in[0,\infty),\ v\in \W,\\
\lambda\cdot(v+w)\!\!\!&=\lambda\cdot v+\lambda\cdot w,\qquad&\forall \lambda\in[0,\infty),\ v,w\in \W.
\end{array}
\]
\end{itemize}
We equip $\W$ with the relation $\leq $ defined as:
\begin{equation}
\label{eq:defpo}
\text{ $v\leq w $ if and only if for some $z\in \W$ we have $v+z=w$. }
\end{equation}
Notice that $+$ being associative and with identity ensures that $\leq$ is reflexive and transitive, hence a preorder.   

We say that $\W$ is a wedge if moreover:
\begin{itemize}
\item[iii)] The relation $\leq$ is a partial order, i.e.\  $v+z=w$ and $w+z'=v$ imply $v=w$. We remark that we are \emph{not} requiring that $z=z'=0$ here,
\item[iv)] The following compatibility between product and order relation is in place:
\begin{equation}
\label{eq:supmultbase}
 v=\sup_{\eta<1}\eta v\qquad\qquad\forall  v\in \W.
\end{equation}
\end{itemize}
\end{definition}
From the calculus properties encoded in the definition it directly follows that
\begin{equation}
\label{eq:basebase}
\begin{array}{rll}
v\leq w&\qquad\Rightarrow\qquad &\lambda v\leq \lambda w,\\
\lambda\leq\eta&\qquad \Rightarrow\qquad &\lambda v\leq\eta v,\\
v_1\leq w_1\qquad v_2\leq w_2&\qquad\Rightarrow\qquad&v_1+v_2\leq w_1+w_2
 \end{array}
\end{equation}
hold for any $v,v_1,v_2,w,w_1,w_2\in \W$ and $\lambda,\eta\in[0,+\infty)$.  We shall frequently use these relations without further notice. Observe that replacing $v$ with $\lambda v$ in \eqref{eq:supmultbase} we get
\begin{equation}
\label{eq:supmult}
\lambda v=\sup_{\eta<\lambda}\eta v\qquad\qquad\forall  v\in \W,\ \lambda\in(0,+\infty).
\end{equation}
Also,  \eqref{eq:supmult} yields the following analogous property:
\begin{equation}
\label{eq:infmult}
\lambda v=\inf_{\eta>\lambda}\eta v\qquad\qquad\forall  v\in \W,\ \lambda\in(0,+\infty).
\end{equation}
Indeed, \eqref{eq:basebase} shows that  $\lambda v$ is a lower bound for $\{\eta v:\eta>\lambda\}$. Now suppose that $w$ is another lower bound  and notice that for any $\eta>\lambda$ we have $\eta^{-1}w\leq \eta^{-1}\eta v=v$ and thus  $v\geq \sup_{\eta>\lambda}\eta^{-1}w=\sup_{\gamma<\lambda^{-1}}\gamma w=\lambda^{-1}w$, having used \eqref{eq:supmult} in the last step. It follows that $\lambda v\geq w$ and thus that \eqref{eq:infmult} holds.

Notice that since $0+v=v$ holds for every $v\in \W$, we have that $0$ is the minimum of $\W$, i.e.:
\[
0=\min \W=\sup \emptyset.
\]

In general the cancellation property does \emph{not} hold in wedges of interest for us, i.e.\ the inequality  $v+a\leq w+a$  does not imply $v\leq w$ and similarly the equality $v+a=w+a$ does not imply $v=w$  (see in particular the discussion around the definition of the infinity element $\infty$ in \eqref{eq:definfty}). Still, the following is in place:
\begin{lemma}[Cancellation property - first version]\label{le:canc1}
Let $\W$ be a prewedge and $a,b,v\in \W$. Then:
\begin{equation}
\label{eq:cancellation}
a+v\leq b+v\qquad\qquad\Leftrightarrow\qquad \qquad a+\tfrac1n v\leq b+\tfrac1n  v\qquad \forall n\in\N,\ n>0.
\end{equation}
\end{lemma}
\noindent (Note: these are also equivalent to $ a+\lambda v\leq b+\lambda  v$ \ $\forall \lambda>0$, but we shall use the form in \eqref{eq:cancellation})
\begin{proof}
 $\Leftarrow$ is obvious (pick $n=1$). For $\Rightarrow$ fix $n\in\N$, $n>0$, and notice that for any $i=1,\ldots,n$ we have
\[
\begin{split}
\tfrac in a+\tfrac1nv&=\tfrac{i-1}na+\tfrac1n(a+v)\leq \tfrac{i-1}na+\tfrac1n(b+v)=\tfrac1n b+\tfrac{i-1}na+\tfrac1n v.
\end{split}
\]
The conclusion follows chaining these inequalities starting from  $i=n$ up to $i=1$.
\end{proof}

If $\W$ is a wedge and $A\subset \W$ a subset, we shall denote by $v+A$ (resp.\ $\lambda A$) the set $\{v+w:w\in A\}$ (resp.\ $\{\lambda w:w\in A\}$).
\begin{proposition}\label{prop:operazioniMCP}
Let $\W$ be a wedge, $A\subset \W$ and $\lambda\in(0,+\infty)$. Then $A$ has sup (resp.\ inf) if and only if $\lambda A$ has sup (resp.\ inf) and in this case we have
\begin{equation}
\label{eq:basemult}
\sup(\lambda A)=\lambda\sup A\qquad\qquad (\text{resp. }\inf(\lambda A)=\lambda\inf A).
\end{equation}
In particular, the map $\W\ni v\mapsto \lambda v\in \W$ has the $\MCP$.

Also, for any $v\in \W$ we have that: if $A$ has sup  then so does $v+A$ and 
\begin{equation}
\label{eq:basesomma}
\sup(v+A)=v+\sup A.
\end{equation}
In particular, the map $\W^2\ni (v,w)\mapsto v+w\in \W$ has the $\MCP$ (on the source space we are putting the product order, i.e.\ $(v_1,v_2)\leq (w_1,w_2)$ if and only if $v_1\leq w_1$ and $v_2\leq w_2$).
\end{proposition}
\begin{proof} By the first in \eqref{eq:basebase} we see that $u$ is an upper bound of $A$ (resp.\ lower bound) if and only if $\lambda u$ is an upper bound of $\lambda A$ (resp.\ lower bound). The first claim \eqref{eq:basemult} follows.

For \eqref{eq:basesomma} notice that $v+\sup(A)\geq v+w$ for any $w\in A$, showing that $v+\sup A$ is an upper bound for $v+A$. Conversely, let $u$ be an upper bound for $v+A$. Then in particular $u\geq v$, i.e.\ $u=v+z$ for some $z\in \W$. Then for any $w\in A$ we have $v+z\geq v+w$, so that the cancellation property \eqref{eq:cancellation} yields  $\tfrac1nv+z\geq\tfrac1nv+w\geq w$ and thus ultimately that $\tfrac1nv+z\geq\sup A$ for every $n\in\N$, $n>0$. It follows that   $(1+\tfrac1n)u\geq v+\tfrac1nv+z\geq v+\sup A$ and being thus true for any $n\in\N$, taking that $\inf$ in $n$ and using  \eqref{eq:infmult} we deduce $u\geq v+\sup A$, concluding the proof of \eqref{eq:basesomma}. The claim about addition having the $\MCP$ now follows from Proposition \ref{prop:mcpprod}.
\end{proof}

\begin{definition}[Cone]\label{def:cone}
A cone $\C$ is a wedge such that the order relation $\leq$ is directed complete.
\end{definition}
For any $v,w\in \C$ we have $v\leq v+w$ and $w\leq v+w$, showing that the whole cone $\C$ is directed, hence admits a supremum which, being an element of $\C$, is a  maximum. We shall denote it $\infty$, so that
\begin{equation}
\label{eq:definfty}
\infty=\max \C=\inf \emptyset.
\end{equation}
We point out that $\infty\in \C$ is a so-called absorbing element of the addition, meaning that   
\begin{equation}
\label{eq:abs}
v+\infty=\infty\qquad\forall v\in \C,
\end{equation}
because $v+\infty$ is a well defined  element of $\C$, hence $\leq\infty$. In particular,  the standard cancellation property does not hold on cones. 

\begin{remark}[Completing a wedge]\label{re:abstrcompl}{\rm
Proposition \ref{prop:operazioniMCP} above  together with Proposition \ref{prop:mcpprod} ensures that the operations of sum and product by a scalar can be uniquely extended, maintaining the $\mcp$, to the directed completion $\bar \W$ of any given wedge $\W$. Such completion therefore carries a canonical structure of prewedge and the discussion below clarifies that it is actually a wedge. 

However, it is not clear to us whether at this level of generality the completion is a cone (but see Remark \ref{re:wedgecomplcpiu} for a possible counterexample). The problem is that on it we have two natural partial orders: one, call it $\leq_1$, is the one coming from the structure of directed completion and the other, call it $\leq_2$,  coming from the wedge structure. In order for $\bar \W$ to be a cone we should know that it is $\leq_2$-complete, but we only know that it is $\leq_1$ complete. We do not know if $\leq_1$ and $\leq_2$ coincide or not, but it is easy to see that $x\leq_2 y$ implies $x\leq_1y$ (so that in particular $\leq_2$ is antisymmetric and thus a  partial order, as claimed before). Indeed, if $x\leq_2 y$ then  there is $z\in\bar \W$ such that $x+z=y$. Now let $A_x,A_z\subset \W$ be  sets with tips in $\bar \W$ (here we think $\W\subset \bar \W$) with tips $x,z$ respectively and notice that an argument by transfinite induction and the $\mcp$ of the sum show that the set $A_y:=A_x+A_z=\{a_x+a_z:a_x\in A_x,a_z\in A_z\}\subset \W$ has tip in $\bar \W$, the tip being    $y$. We thus know that the sets $A_x,A_y\subset \W$ have tips $x,y$ respectively and that for every $a_x\in A_x$ there is $a_y\in A_y$ with $a_x\leq a_y$. Hence $A_x\subset\, \downarrow\! A_y$, thus $\widehat A_x\subset\widehat A_y$ proving that   $x\leq_1y$, as claimed.

This  suggests that the completion should be looked for in the  category of wedges where morphisms are linear maps satisfying the $\mcp$ (we will  call these `morphism' in Section \ref{se:morphisms}). In this direction, an abstract existence result   comes from the following general construction, reminiscent of embedding in the bidual and Yoneda's lemma, as an object is studied via morphisms into all (relevant) objects. Let $\W$ be a wedge and consider the collection $B$ of couples $(T,\C)$ where $\C$ is a cone and $T:\W\to \C$ is linear with the $\mcp$. Then we consider the product 
\[
\C_\W:=\prod_{(T,\C)\in B}\C
\]
equip it with componentwise operations and notice that, rather trivially, with these $\C_\W$ is  a cone (if only it were a set, see also below - notice that in this case   $\C_\W$   would be, together with the natural projections, the product of the given cones in the categorical sense). Finally we consider the `evaluation map' ${\sf ev}:\W\to \C_\W$ sending $v\in \W$ into the element of $\C_\W$ whose $(T,\C)$-component is $T(x)\in \C$. Routine arguments show that the directed-sup-closure $\overline{{\sf ev}(\W)}$ of ${\sf ev}(\W)\subset \C_\W$ in $\C_\W$ is the desired completion, together with the map ${\sf ev}$. Notice that this tells nothing on the structure of such completion: for instance it is unclear whether the evaluation map ${\sf ev}$ from $\W$ to its completion as injective (this comes from  difficulties in producing linear maps with the $\mcp$ -- but see Section \ref{se:morphisms}).

This sort of procedure is standard and is the one behind Freyd's General Adjoint Functor Theorem, see e.g.\ \cite[Section V.6]{MacLane98} (except that we avoid considering a solution \emph{set} because the above product is a proper class). For instance, it is the one producing the Stone-\v{C}ech compactification, and if applied to the category of metric spaces with  1-Lipschitz maps it reproduces the standard completion (notice that in such category the distance on the product is the `sup': this can take the value $+\infty$, but on the image of the evaluation map such distance is finite). 

The added value of this line of thought is in showing existence of the desired completion, but its level of abstraction makes it of limited use  in many applications. 
It is worth to notice that  we could have applied the above strategy also to show the existence of the two-sided completion mentioned in Section \ref{se:twosided} or as alternative route for existence of directed completion in Section \ref{se:constr}. This latter choice, however would have prevented us from knowing `how the completion is made' and how we can recognize it, as studied in Section \ref{se:commcompl}. 

Finally, we remark that, as usual, the construction above has to be taken with a bit of caution from the set-theoretic perspective (see e.g.\ \cite{Jech2003}). If one works in ZF, then she should limit herself to just one choice of $(T,\C)$ per isomorphism class and to cones $\C$ of `not too high cardinality compared to that of $\W$' (meaning `so that there is an epimorphism from $\W$ to $C$' --  it has to be understood who such cardinality actually is). Alternatively, working in NGB the objects $B,\C_\W$ make sense as a proper classes, then the image of $\W$ via ${\sf ev}$ is a set and so is its directed-sup-closure. 
}\fr\end{remark}

Given a cone $\C$ and   $v\in \C$ the sets  
\[
\begin{split}
&\{w:w\leq \lambda v\text{ for every }\lambda>0\}\qquad\text{and}\qquad \{w:w\leq \lambda v\text{ for some }\lambda>0\}
\end{split}
\]
are  both closed by sum, hence directed, hence their sup exist. These deserve a special name
\[
\begin{split}
\eps v:=&\sup\{w:w\leq \lambda v\text{ for every }\lambda>0\}=\inf \{  \lambda v\ :\ \lambda>0\},\\
\infty v:=&\sup\{w:w\leq \lambda v\text{ for some }\lambda>0\}=\sup \{  \lambda v\ :\ \lambda>0\}.
\end{split}
\] 
We shall sometimes refer to $\eps v$ as to the `part at infinite' of $v$, because it is `what remains of $v$ after multiplying it by a very small positive number'.  

The notation is also somehow  justified by the validity of the following calculus rules, that hold for any $v,w\in \C$ and $\lambda\in(0,+\infty)$:
\begin{equation}
\label{eq:epsinfty}
\begin{array}{rlrl}
\eps(v+w)\!\!\!\!&\geq \eps v+\eps w,&\qquad\qquad\qquad \infty(v+w)\!\!\!\!&= \infty v+\infty w,\\
\lambda v+\eps v\!\!\!\!&=\lambda v,&\qquad\qquad\qquad \lambda v+\infty v\!\!\!\!&=\infty v,\\
\lambda(\eps v)\!\!\!\!&=\eps v,&\qquad\qquad\qquad\lambda(\infty v)\!\!\!\!&=\infty v,\\
\eps(\eps v)\!\!\!\!&=\eps v,&\qquad\qquad\qquad\infty(\infty v)\!\!\!\!&=\infty v,\\
\infty(\eps v)\!\!\!\!&=\eps v,&\qquad\qquad\qquad \eps(\infty v)\!\!\!\!&=\infty v,\\
\infty v+\eps v\!\!\!\!&=\infty v.&&
\end{array}
\end{equation}
All of these are rather straightforward consequences of the definitions and \eqref{eq:basemult}, \eqref{eq:basesomma}. We comment just on the first row. For the equality we have
\[
\infty(v+w)=\sup_{\lambda>0}\lambda(v+w)=\sup_{\lambda>0}(\lambda v+\lambda w)=\sup_{\lambda,\eta>0}(\lambda v+\eta w)\stackrel{\eqref{eq:basesomma}}=\sup_{\lambda >0}\lambda v +\sup_{ \eta>0} \eta w=\infty v+\infty w.
\]
For the inequality:
\[
\forall \lambda>0\qquad\eps v+\eps w\leq\lambda v+\lambda w=\lambda(v+w)\qquad\Rightarrow\qquad \eps v+\eps w\leq\eps(v+w).
\]
The difficulty in getting the other inequality is that we do not know whether the bound $z\leq v+w$ implies that we can write $z=z_1+z_2$ for some $z_1\leq v$ and $z_2\leq w$. This difficulty also prevents us from proving that $\inf (A+v)=(\inf A)+v$ (here the inequality $\geq$ is trivial). See also the axiomatization of cones with joins given below and the definition \ref{def:tdp} of directed decomposition property.

In particular, given a cone $\C$ we can naturally define  its `cone at infinity' as
\begin{equation}
\label{eq:cinfinity}
\C_{\rm infinity}:=\{v\in \C\ :\  v=\infty v\}\ =\ \{v\in \C\ :\  v=\eps v\}\ =\ \{v\in \C\ :\  v= v+v\}.
\end{equation}
The above properties easily imply that $\C_{\rm infinity}$ is wedge. To see that it is a cone notice that  \eqref{eq:basesomma}  and the last characterization in \eqref{eq:cinfinity} show that a directed family $(v_i)$ in $\C_{\rm infinity}$  has a $\C$-supremum in $\C_{\rm infinity}$, so that to conclude it suffices to prove that partial order in $\C_{\rm infinity}$ coming from its wedge structure coincides with the restriction of that in $\C$ and since the former is clearly included in the latter, we need only to prove that: if $v,w\in \C_{\rm infinity}$ are so that $v\leq w$ in $\C$, then there is $z\in \C_{\rm infinity}$ such that $v+z=w$. To see this, let $z'\in \C$ be so that $v+z'=w$ and put $z:=\infty z'$: we have   $w=\infty w= \infty v+\infty z'=v+z$, concluding the proof.
\begin{remark}\label{re:finitepart}{\rm
It is natural to also define the `finite part'  of the cone, or wedge, $\C$ as
\[
\C_{\rm finite}:=\{v\in \C\ :\ \eps v=0\}.
\]
Notice, however, that we don't know in this generality  whether this set is closed by sum, thus if it is a wedge (but we don't have counterexamples either). If this turns out to be the case, one can canonically associate to it, and thus to $\C$,  a vector space $V$ and a natural inclusion $\C_{\rm finite}\hookrightarrow V$: the construction is that of  the Groethendick group associated to the wedge, see also \cite{KO25} and Section \ref{se:finitedim} for some  comments. Notice that this procedure inverts the  procedure that associates a Banach spacetime to a classical Banach space  as described in Section \ref{se:Banst} (akin to that which starts from $\R^3$, builds the Minkowski space and then pick the causal/directed completion of its future cone).

It would also seem reasonable to only consider cones that are the (directed) completion of their finite part, as in some sense one would think that the `true geometry' is encoded in the finite part, and that elements with non-zero part at infinity should only arise as suitable limits of finite ones. While this perspective is sound, and would allow to exclude `bad' examples such as those presented in Example \ref{ex:nofinpart} below or in Section \ref{se:finitedim}, it is unclear to us if this property passes to duals. Similarly, we do not know whether the property `the finite part is dense' passes to duals or not.
}\fr\end{remark}

\begin{example}\label{ex:nofinpart}{\rm
The singleton $\{0\}$ is a cone. The only one for which $0=\infty$.

The two elements set  $\{0,+\infty\}$, with the obvious operations, is a cone. It is not the completion of its finite part.

More interestingly, consider the wedge $\W:=[0,+\infty)^2$  (with component-wise operations). It is clear that its completion is $\C:=[0,+\infty]^2$, that this is a cone (recall also Proposition \ref{prop:mcpprod}) and that $\W$ is its finite part. Still, there are other cones having $\W$ as their finite part, such as:
\[
\begin{split}
\C_1&:=\W\cup\{(+\infty,+\infty)\},\\
\C_2&:= W\cup([0,+\infty]\times\{+\infty\}),
\end{split}
\]
with the obvious operations. See also Section \ref{se:finitedim} for some further finite dimensional examples and comments.
}\fr\end{example}

In this generality it is not clear to us whether $\inf_nw+\tfrac1nv$ exists, nor whether in case it   is equal to $w+\eps v$. Still, we have that in any cone $\C$, for $v,w\in \C$ it holds
\begin{equation}
\label{eq:quasieps}
\max\big\{z\in \C :\ w\leq z\leq w+\tfrac1nv\quad\forall n\in\N,\ n>0\big\}= w+\eps(w+ v).
\end{equation}
Indeed, $w\leq w+\eps(w+ v)$ clearly holds and for every $n,m\in\N$, $n\geq m>0$ we have 
\[
w+\eps(w+ v)\leq w+\tfrac1n(w+v)\leq (1+\tfrac1n)(w+\tfrac1mv).
\]
Taking the inf in $n$ recalling \eqref{eq:infmult} we see that $w+\eps(w+v)$ belongs to the set on the left hand side of \eqref{eq:quasieps}. Let $z$ be another element and notice that since $z\geq w$ by assumption we can write $z=w+z'$ for some $z'\in \C$. Then we have $w+z'\leq w+\tfrac1n v$ for any $n$, so from \eqref{eq:cancellation} we get $z'\leq \tfrac1n(w+v)$ for any $n\in\N$, and thus $z'\leq \eps(w+v)$, proving \eqref{eq:quasieps}.

This discussion clarifies what is needed to get a cleaner cancellation property:
\begin{proposition}[Cancellation property - second version]\label{prop:percancellazione}
Let $\C$ be a cone with the following property: if $A\subset \C$ admits infimum, then $v+A$ admits infimum for every $v\in \C$ and
\begin{equation}
\label{eq:infsumprima}
\inf(v+A)=v+\inf A.
\end{equation}
 Then
\begin{equation}
\label{eq:cancellationtrue}
a+v\leq b+v\qquad\qquad\Leftrightarrow\qquad\qquad a+\eps v\leq b+\eps v.
\end{equation}
\end{proposition}
\begin{proof}
$\Leftarrow$ is obvious while for $\Rightarrow$ we use \eqref{eq:infsumprima} to get $a+\eps v=
a+\inf\{\tfrac1nv:n\in\N\}=\inf\{a+\tfrac1nv:n\in\N\}$ and conclude with \eqref{eq:cancellation}.
\end{proof}
The assumption \eqref{eq:infsumprima} is one of the two key requirements we impose on the   following important subclass of spaces:
\begin{definition}[\Lattice]\label{def:conejoins}
A \lattice\ is a cone $\J$ in the sense of Definition \ref{def:cone} such that moreover:
\begin{itemize}
\item[i)] Any subset admits an infimum, called $\inf$,
\item[ii)] Such infimum respects sums, i.e.\  for every $v\in \J$ and $A\subset \J$ we have
\begin{equation}
\label{eq:infsum}
\inf(v+A)=v+\inf A.
\end{equation}
\end{itemize}
\end{definition}
From the order theoretic perspective a cone with join is a complete lattice (recall Proposition \ref{prop:joinlattice}); the  sup and inf of two elements will be denoted $v\vee w$ and $v\wedge w$ respectively, as usual.

Obviously, in cones with joins Proposition \ref{prop:percancellazione} applies, so that the cancellation formula \eqref{eq:cancellationtrue} holds. Another consequence of  assumption \eqref{eq:infsum} is
\begin{equation}
\label{eq:epssomma}
\eps(v+w)=\eps v+\eps w\qquad\forall v,w\in \J,
\end{equation}
that we shall frequently use without further notice. To see this notice that we have  $\inf_{\lambda>0}(\lambda v+\lambda w)=\inf_{\lambda,\eta>0}(\lambda v+\eta w)$ and then using twice \eqref{eq:infsum} we get \eqref{eq:epssomma}.

A direct consequence of \eqref{eq:basesomma} and \eqref{eq:infsum} is that for $v,w\in \J$ with $v\leq w$ the set $\{z\in \J: v+z=w\}$ is a complete lattice and thus admits a maximal and minimal element. We shall denote the maximal element by $w-v$, i.e.\ we put
\begin{equation}
\label{eq:defdiff}
w-v:=\max\{z\in \J\ :\ v+z=w\}\qquad\forall v,w\in \J,\ v\leq w
\end{equation}
and observe that from \eqref{eq:epsinfty} and \eqref{eq:cancellationtrue} it follows that
\begin{equation}
\label{eq:maxdiff}
v+z=w\qquad\Rightarrow\qquad z+\eps v=z+\eps w=w-v.
\end{equation}
Indeed, from \eqref{eq:cancellationtrue} we see that if $v+z'=w=v+z''$, then $z'+\eps v=z''+\eps v$. In particular, this is true for $z'':=w-v$ thus by maximality we deduce  $(w-v)+\eps v=w-v$ and from all this \eqref{eq:maxdiff} easily follows. These considerations imply that the operation of taking the difference is linear in the following sense: for $v_1\leq w_1$, $v_2\leq w_2$ and $\lambda_1,\lambda_2\in[0,\infty)$ we have
\begin{equation}
\label{eq:difflin}
\lambda_1(w_1-v_1)+\lambda_2(w_2-v_2)=(\lambda_1w_1+\lambda_2w_2)-(\lambda_1v_1+\lambda_2v_2).
\end{equation}
Indeed, we clearly have
\[
(\lambda_1v_1+\lambda_2v_2)+(\lambda_1(w_1-v_1)+\lambda_2(w_2-v_2))=\lambda_1w_1+\lambda_2w_2
\]
and
\[
\begin{split}
\lambda_1(w_1-v_1)+\lambda_2(w_2-v_2)&\stackrel{\eqref{eq:maxdiff}}=\lambda_1(w_1-v_1)+\lambda_1\eps w_1+\lambda_2(w_2-v_2)+\lambda_2\eps w_2\\
(\text{by \eqref{eq:epssomma}})\qquad&\stackrel{\phantom{\eqref{eq:maxdiff}}}=\lambda_1(w_1-v_1)+\lambda_2(w_2-v_2)+\eps(\lambda_1 w_1+\lambda_2 w_2)
\end{split}
\]
that by \eqref{eq:maxdiff} gives the desired \eqref{eq:difflin}. From \eqref{eq:epssomma} and \eqref{eq:maxdiff} it also follows that
\begin{equation}
\label{eq:epsdiff}
\eps(w-v)=\eps w\qquad\forall v,w\in \J,\ v\leq w.
\end{equation}
Indeed \eqref{eq:maxdiff} gives   $(w-v)+\eps w=(w-v)$, hence from \eqref{eq:epssomma} we get $\eps (w-v)=\eps (w-v)+\eps w\geq\eps w$ and since the other inequality is obvious, the claim is proved.

Other basic properties of cones with joins are collected in the following statement (compare e.g., with \cite[Section 1.1]{MNP91}):
\begin{proposition}\label{prop:baselattice}
Let $\J$ be a \lattice. Then:
\begin{enumerate}[label=\emph{\alph*)}]
\item\label{it:i} \emph{Distributivity I}. For every $x,y,z\in \J$ we have
\begin{equation}
\label{eq:distrmin}
\begin{split}
(x+y)\wedge z&\leq x\wedge z+y\wedge z.
\end{split}
\end{equation}
\item\label{it:ii} \emph{Modularity}.
 For any $x,y,z,w\in \J$ we have
\begin{equation}
\label{eq:modulare1}
x+y=z+w\qquad\Rightarrow\qquad x+y=x\vee z+y\wedge w.
\end{equation}
In particular 
\begin{equation}
\label{eq:modulare2}
x+y=x\vee y+x\wedge y.
\end{equation}
\item\label{it:v} \emph{Distributivity II}. For any set $(x_i)\subset \J$ and $y\in \J$ we have
\begin{equation}
\label{eq:quasidistr}
\sup_i(x_i\wedge y)\leq( \sup_i x_i)\wedge y\leq \sup_i(x_i\wedge y)+\eps(y\vee\sup_ix_i).
\end{equation}
\item\label{it:vi} \emph{Decomposition I}. Suppose that $v_1,v_2,w_1,w_2\in \J$ are so that $v_1+v_2=w_1+w_2$. Then there are  $z_{ij}$, $i,j=1,2$, such that
\[
\begin{split}
z_{11}+z_{12}&=v_1,\\
z_{21}+z_{22}&=v_2+\eps v_1,\\
z_{11}+z_{21}&=w_1,\\
z_{12}+z_{22}&=w_2+\eps w_1. 
\end{split}
\]
\item\label{it:viii}   \emph{Decomposition II}. Suppose that $v_1,v_2,w_1,w_2\in \J$ are so that $v_1+v_2=w_1+w_2$. Then there are  $z_{ij}$, $i,j=1,2$, such that
\begin{equation}
\label{eq:perlin}
\begin{split}
z_{11}+z_{12}&=v_1,\\
z_{21}+z_{22}&=v_2,\\
z_{11}+z_{21}+\eps w_2&\geq w_1,\\
z_{12}+z_{22}+\eps w_1&\geq w_2.
\end{split}
\end{equation}
\end{enumerate}
\end{proposition}
\begin{proof} \

\noindent{\sc Distributivity I}.  Put  $v:=(x+y)\wedge z$. Then clearly  $v\leq x+y$ and $v\leq z\leq x+z$, therefore $v\leq (x+y)\wedge (x+z)=x+y\wedge z$, having used \eqref{eq:infsum}. On the other hand, we clearly have  $v\leq z\leq z+y\wedge z$, therefore 
\[
v\leq (x+y\wedge z)\wedge(z+y\wedge z)\stackrel{\eqref{eq:infsum}}=x\wedge z+y\wedge z.
\]
{\sc Modularity}. We have $y\geq y\wedge w$, thus for some  $a\in \J$ we have $y=y\wedge w+a$. Thus
\[
y\wedge w+a+x=x+y=z+w\geq z+y\wedge w\qquad\stackrel{\eqref{eq:cancellationtrue}}\Rightarrow\qquad  \eps (y\wedge w)+a+x\geq z.
\]
On the other hand we have  $ \eps (y\wedge w)+a+x\geq x$ and thus $\eps (y\wedge w)+a+x\geq x\vee z$, hence
\[
x+y=y\wedge w+a+x=y\wedge w+\eps (y\wedge w)+a+x\geq  y\wedge w+x\vee z.
\]
For the other inequality we notice that  $x\vee z\geq x$, thus for some  $b\in \J$ we have $x\vee z=x+b$. Therefore
\[
x+b+w=x\vee z+w\geq z+w=x+y\qquad\stackrel{\eqref{eq:cancellationtrue}}\Rightarrow\qquad \eps x+b+w\geq y
\]
and thus
\[
\begin{split}
x\vee z+y\wedge w=x+\eps x+b+y\wedge w\stackrel{\eqref{eq:infsum}}=x+(\underbrace{\eps x+b+y}_{\geq y})\wedge(\underbrace{\eps x+b+ w}_{\geq y})\geq x+y.
\end{split}
\]
{\sc Distributivity II}.   The first inequality is obvious. For the second we put   $v:= \sup_i(x_i\wedge y)$. Then for every $i$ we have $v\geq x_i\wedge y$ and thus
\[
v+x_i\vee y\geq x_i\wedge y+x_i\vee y\stackrel{\eqref{eq:modulare2}}=x_i+y
\]
and taking the sup in $i$, recalling \eqref{eq:basesomma} and the trivial identity  $(\sup_ix_i)\vee y=\sup_i(x_i\vee y)$ we get
\[
v+(\sup_i x_i)\vee y\geq\sup_ix_i+y\stackrel{\eqref{eq:modulare2}}=(\sup_i x_i)\vee y+(\sup_i x_i)\wedge y.
\]
The conclusion follows from the cancellation property  \eqref{eq:cancellationtrue}.

\noindent {\sc Decomposition I}. Put $z_{11}:= v_1\wedge w_1$ and then let  $z_{12},z_{21}$ be so that  $z_{11}+z_{12}=v_1$ and $z_{11}+z_{21}=w_1$. We claim that $z_{21}\leq v_2+\eps v_1$ and to prove this we notice that
\[
\begin{split}
z_{11}+z_{21}+w_2=w_1+w_2\stackrel{\eqref{eq:modulare1}}=z_{11}+v_2\vee w_2\leq z_{11}+v_2+w_2
\end{split}
\]
and thus from \eqref{eq:cancellationtrue} we get $z_{21}\leq v_2+\eps(z_{11}+w_2)\leq  v_2+\eps(w_{1}+w_2)=  v_2+\eps v_1$.
Thus there is $z$ such that  $z_{21}+z=v_2+\eps v_1$. Put $z_{22}:=z+\eps w_1$ and notice that
\[
\begin{split}
z_{21}+z_{22}=z_{21}+z+\eps w_1=v_2+\eps(v_1+w_1)=v_2+\eps(v_1+v_2+w_1)=v_2+\eps(v_1+v_2)=v_2+\eps v_1.
\end{split}
\]
Moreover
\[
z_{12}+z_{22}+w_1=\sum_{ij}z_{ij}=v_1+v_2+\eps (v_1+w_1)=v_1+v_2+\eps(v_1+v_2)=w_1+w_2
\]
and thus from  \eqref{eq:cancellationtrue}  we get  $z_{12}+z_{22}=z_{12}+z_{22}+\eps w_1=w_2+\eps w_1$.

\noindent {\sc Decomposition II}.  Put $z_{11}:=v_1\wedge w_1$, $z_{22}:=v_2\wedge w_2$ and then $z_{12}:=v_1-z_{11}$ and $z_{21}:=v_2-z_{22}$ (recall the definition \eqref{eq:defdiff}). Then  $z_{11}+z_{12}=v_1$ and   $z_{21}+z_{22}=v_2$.

Now notice that  $w_1+z_{12}\geq z_{11}+z_{12}=v_1$, hence  $w_1+z_{12}\geq v_1\vee w_1$  and thus
\[
\begin{split}
w_1+z_{12}+z_{22}\geq v_1\vee w_1+v_2\wedge w_2\stackrel{\eqref{eq:modulare1}}=w_1+w_2\qquad\stackrel{\eqref{eq:cancellationtrue}}\Rightarrow\qquad z_{12}+z_{22}+\eps w_1\geq w_2.
\end{split}
\]
proving the last inequality. The other one is proved analogously.
\end{proof}

\end{subsection}

\begin{subsection}{Spaces of linear maps}

Here we shall start investigating properties of spaces of linear maps.

Given two wedges $\W_1,\W_2$, a map $L:\W_1\to \W_2$ is called linear provided
\[
L(\lambda_1 w_1+\lambda_2 w_2)=\lambda_1 L(w_1)+\lambda_2 L(w_2)\qquad\forall w_1,w_2\in \W_1,\ \lambda_1,\lambda_2\in[0,+\infty).
\]
We shall denote by $\Lin(\W_1,\W_2)$ the collection of linear maps from $\W_1$ to $\W_2$ and equip such set with pointwise operations of sum and product by a non-negative scalar, i.e.\ we put:
\[
\begin{split}
(L_1+L_2)(v)&:=L_1(v)+L_2(v),\\
\lambda L(v)&:=L(\lambda v),
\end{split}
\]
for every $L_1,L_2,L\in\Lin(\W_1,\W_2)$, $\lambda\in[0,\infty)$ and $v\in \W_1$. It is clear that the results of these operations are still linear maps, hence $\Lin(\W_1,\W_2)$ is a prewedge. We shall denote by $\leq_{\L}$ the induce preorder, i.e.\ we put
\[
L_1\leq_\L L_2\qquad\qquad\Leftrightarrow\qquad\qquad  L_1+M=L_2\qquad\text{for some  $M\in \Lin(\W_1,\W_2)$}.
\]
This is actually a partial order, i.e.\ $L_1\leq_\L L_2$ and $L_2\leq_\L L_1$ imply $L_1=L_2$, as can be seen using the analog property valid on $\W_2$.  On $\Lin(\W_1,\W_2)$ we  have another natural partial  order, namely the  the pointwise order  $\leq_\p$ defined as
\[
L_1\leq_\p L_2\qquad\qquad\Leftrightarrow\qquad\qquad L_1(v)\leq L_2(v)\qquad\forall v\in \W_1.
\]
It is clear that $L_1\leq_\L L_2$ implies $L_1\leq_\p L_2$, but the converse is not obvious: we can prove it only when the target is a cone with joins (thanks to \eqref{eq:difflin}). Notice also that we do \emph{not} know whether  $\Lin(\W_1,\W_2)$ is a wedge or not, as we do not know if property $(iv)$ in Definition \ref{def:wedge} is valid or not. Notice indeed that for $L\in  \Lin(\W_1,\W_2)$ we clearly have $ L\geq_\L \lambda L$ for any $\lambda\in[0,1)$ and if $M\in\Lin(\W_1,\W_2)$ is so that $M\geq_\L\lambda L$ for every $\lambda\in[0,1)$ then clearly $M\geq_p  L$. Still, concluding that $M\geq_\L L$, that is required to show \eqref{eq:supmultbase},  seems not obvious.
\begin{remark}[The infinity map]\label{rem:inftyL}{\rm
Notice that a consequence of the definition is that any linear map satisfies $L(0)=0$ (because $L(0)=L(0\cdot 0)=0 L(0)=0$). 

In particular, the maximal element in $\Lin(\W,\C)$ for $\C$ cone is the map  equal to $\infty\in \C$ at any non-zero element in $\W$ and equal to $0\in \C$ at $0\in \W$.
}\fr\end{remark}
Let us collect a couple of observations. If $\W$ is a wedge and  $(v_i),(w_i)\subset \W$ are directed (resp.\ filtered) families indexed over the same directed (resp.\ filtered) set $I$  admitting sup (resp.\ inf), then
\begin{equation}
\label{eq:supsommadir}
\sup_{i\in I}(v_i+w_i)=\sup_{i\in I}v_i+\sup_{i\in I}w_i\qquad\qquad (\text{resp. }\inf_{i\in I}(v_i+w_i)=\inf_{i\in I}v_i+\inf_{i\in I}w_i ).
\end{equation}
Indeed, $\leq$ is clear and for $\geq$ we notice that for every $i,j\in I$ there is $k\in I$ bigger than both, implying $v_i\leq v_k$ and $w_i\leq w_k$, thus $v_i+w_j\leq v_k+w_k\leq\sup_k(v_k+w_k)$, which by the arbitrariness of $i,j$ suffices to conclude (analogously for the inf).

A direct consequence of \eqref{eq:supsommadir} is that   pointwise defined sup/inf of directed/filtered families are linear, provided they exist:
\begin{equation}
\label{eq:pointsup}
\begin{split}
\text{$(L_i)_{i\in I}\subset\Lin(\W,\C)$ is $\leq_\p $-directed\ \  and \ \ } \text{$L(v):=\sup_{i\in I}L_i(v)$ exists $\forall v$}&\qquad\Rightarrow\qquad \text{$L$ is  linear,}\\
\text{$(L_i)_{i\in I}\subset\Lin(\W,\C)$ is $\leq_\p $-filtered\ \  and \ \ } \text{$L(v):=\inf_{i\in I}L_i(v)$ exists $\forall v$}&\qquad\Rightarrow\qquad \text{$L$ is  linear.}\\
\end{split}
\end{equation}
In particular, for any wedge $\W$ and cone $\C$ we have
\begin{equation}
\label{eq:linp}
\big(\,\Lin(\W,\C)\,,\,\leq_\p \!\big)\qquad\text{ is a directed complete partial order.}
\end{equation}
%
Notice that this  does not tell that $\Lin(\W,\C)$ is a cone.

If the target is a cone with joins, then we have:
\begin{proposition}\label{prop:lincone}
Let $\W$ be a wedge and $\J$ a \lattice. Then the order relations $\leq_\L$ and $\leq_\p $ on $\Lin(\W,\J)$ coincide.

In particular, $\Lin(\W,\J)$ is a cone.
\end{proposition}
\begin{proof} We already noticed that $L_1\leq_\L L_2$ implies  $L_1\leq_\p L_2$. For the converse implication, for every $v\in\W$ define $M(v):=L_2(v)-L_1(v)$ (recall \eqref{eq:defdiff}). By \eqref{eq:difflin} we know that $M$ is linear, hence the fact that $L_1(v)+M(v)=L_2(v)$ for every $v\in\W_1$ gives $L_1\leq_\L L_2$, as desired.

From the completeness in \eqref{eq:linp} we see that to conclude it suffices to prove that property $(iv)$ in Definition \ref{def:wedge} holds. This however is now obvious, as we already noticed that for any $L\in\Lin(\W,\C)$ we have $L\geq_\L \lambda L$ for any $\lambda\in[0,1)$ and that any $M\in\Lin(\W,\C)$ with the same property is $\geq_\p L$. By the first part of the proof we thus deduce that $M\geq_\L L$, establishing that $L$ is the $\leq_\L$-sup of $(\lambda L)_{\lambda<1}$.
\end{proof}
%
Even more structure is present if also the source space is a \lattice. In what follows we shall denote by $\Lsup$ and $\Linf$ the sup and inf w.r.t.\ the order $\leq_\L$. If only two elements are involved we shall write $\vee_\L$ and $\wedge_\L$, respectively.
\begin{proposition}\label{prop:linlattice}
Let $\J_1,\J_2$ be \lattices. Then $\Lin(\J_1,\J_2)$ is a \lattice\ and the Riesz-Kantorovich formulas hold, i.e.\ for $(L_i)_{i\in I}\subset \Lin(\J_1,\J_2)$ arbitrary family we have
\begin{subequations}\begin{align}
\label{eq:RKpgen}
\big(\Lsup_{i\in I}L_i\big)(v)&=\sup\Big\{\sum_{j=1}^nL_{i_j}(v_{i_j})\,:\, n\in\N,\ (i_j)\subset I, \sum_{j=1}^nv_{i_j}=v\Big\},\ \ \ \ \ \\
\label{eq:RKmgen}
\Big(\Linf_{i\in I}L_i\big)(v)&=\inf\big\{\sum_{j=1}^nL_{i_j}(v_{i_j})\,:\, n\in\N,\ (i_j)\subset I, \sum_{j=1}^nv_{i_j}=v\Big\},\ \ \ \ \ 
\end{align}
\end{subequations}
where we declare that $\big(\Linf_{i\in I}L_i\big)(0)=0$ if $I$ is empty. In particular it holds:
\begin{subequations}\begin{align}
\label{eq:RKp}
(L_1\vee_\L L_2)(v)&=\sup_{v_1+v_2=v} L_1(v_1)+L_2(v_2),\\
\label{eq:RKm}
(L_1\wedge_\L L_2)(v)&=\inf_{v_1+v_2=v} L_1(v_1)+L_2(v_2).
\end{align}
\end{subequations}
\end{proposition}
\begin{proof} We start proving that   $L_1\vee_\L L_2$ exists by establishing  the explicit formula \eqref{eq:RKp}. The validity of \eqref{eq:RKm} will follow along similar lines.

Let $L:\J_1\to \J_2$ be the map defined by the right hand side of \eqref{eq:RKp}. Let $M\in\Lin(\J_1,\J_2)$ be such that $M\geq L_1$ and $M\geq L_2$. Then clearly for any $v=v_1+v_2$ we have 
\[
M(v)=M(v_1)+M(v_2)\geq L_1(v_1)+L_2(v_2)
\]
and thus $M(v)\geq L(v)$ for any $v\in \J_1$. Since the inequalities   $L(v)\geq L_1(v)$ and $L(v)\geq L_2(v)$ are trivially valid for any $v\in \J_1$, to conclude it is  sufficient to show that $L$ is linear. The identity $L(\lambda v)=\lambda L(v)$ is a trivial consequence of the definition, so we focus on additivity.

Let $v=v'+v''$ and then pick arbitrary decompositions $v'=v'_1+v'_2$ and $v''=v''_1+v''_2$. Obviously we have $v=(v'_1+v''_1)+(v'_2+v''_2)$, thus
\[
L(v)\geq L_1(v'_1+v''_1)+ L_2(v'_2+v''_2)=\big(L_1(v'_1)+L_2(v'_2)\big)+\big(L_1(v''_1)+L_2(v''_2)\big).
\]
Taking the supremum over the choices of $v'_1,v'_2,v''_1,v''_2$ yields 
\begin{equation}
\label{eq:Lsuper}
L(v)\geq L(v')+L(v'').
\end{equation}
For the converse inequality let $v_1,v_2\in \J_1$ be so that $v_1+v_2=v=v'+v''$ and use item \ref{it:viii} in Proposition \ref{prop:baselattice} to find $z_{ij}$, $i,j=1,2$, such that
\begin{equation} 
\label{eq:splittate}
\begin{split}
v'&=z_{11}+z_{12},\\
v''&=z_{21}+z_{22},\\
v_1&\leq z_{11}+z_{21}+\eps v,\\
v_2&\leq z_{12}+z_{22}+\eps v.
\end{split}
\end{equation}
The last two  and the linearity of $L_1,L_2$ give
\begin{equation}
\label{eq:unpezzo}
\begin{split}
L_1(v_1)&\leq L_1(z_{11})+L_1(z_{21})+L_1(\eps v)\leq L(z_{11})+L(z_{21})+L_1(\eps v),\\
L_2(v_2)&\leq L_2(z_{12})+L_2(z_{22})+L_2(\eps v)\leq L(z_{12})+L(z_{22})+L_2(\eps v).
\end{split}
\end{equation}
From  $\eps v=\eps v'+\eps v''$ (recall \eqref{eq:epssomma})  we get  $L_1(\eps v)=L_1(\eps v')+L_1(\eps v'')\leq L(\eps v')+L(\eps v'')$ and similarly $L_2(\eps v)\leq L(\eps v')+L(\eps v'')$. Using these bounds in \eqref{eq:unpezzo}, adding up the two resulting inequalities and using 
\eqref{eq:Lsuper} and the first two in \eqref{eq:splittate} we get
\[
\begin{split}
L_1(v_1)+L_2(v_2)&\leq L(v')+L(v'') +2L(\eps v')+2L(\eps v'')\\
\text{(again by \eqref{eq:Lsuper})}\qquad&\leq  L(v'+2\eps v')+L(v''+2\eps v'')=L(v')+L(v'').
\end{split}
\]
Taking the supremum over all the choices of $v_1,v_2$ with $v_1+v_2=v$ gives $L(v)\leq L(v')+L(v'')$, that together with \eqref{eq:Lsuper} gives the conclusion.

We pass to the  more general formulas \eqref{eq:RKpgen} and \eqref{eq:RKmgen}. These are obvious if $I$ is empty (recall also Remark \ref{rem:inftyL} and compare with the definition of $\Linf$ for $I$ empty), so we shall assume it is not. Let $\mathcal L\subset\Lin(\J_1,\J_2)$ be arbitrary and let $\mathcal L^j$ (resp.\ $\mathcal L^m$) be the collection of finite joins (resp.\ meets) of elements in $\mathcal L$. Then  $\mathcal L^j$ (resp.\ $\mathcal L^m$) is directed (resp.\ filtered) with the same $\Lsup$ (resp.\ $\Linf$) of  $\mathcal L$. Then the conclusion is a direct consequence of the formulas already proved, of the fact that $\leq_\L$ and $\leq_\p$ agree (Proposition \ref{prop:lincone}) and of \eqref{eq:pointsup}.

To conclude we need to prove that $M+\Linf_iL_i=\Linf_i(M+L_i)$ for an arbitrary family $(L_i)\subset\Lin(\W,\J)$. We start proving that this holds for a family of two elements. We have
\[
\begin{split}
(M+L_1\wedge L_2)(v)&=M(v)+\inf_{v_1+v_2=v}\big(L_1(v_1)+ L_2(v_2)\big)\\
\text{(by \eqref{eq:infsum} in $\J_2$)}\qquad&=\inf_{v_1+v_2=v}\big(M(v)+L_1(v_1)+ L_2(v_2)\big)\\
&=\inf_{v_1+v_2=v}\big(M(v_1)+L_1(v_1)+M(v_2)+ L_2(v_2)\big)=(M+L_1)\wedge (M+L_2)(v).
\end{split}
\]
Now let $\mathcal L\subset \Lin(\W,\J)$ be an arbitrary family and ${\mathcal L}^m$ the family formed by finite meets of elements in $\mathcal L$, as before. Then ${\mathcal L}^m$ is filtered  with the same inf of $\mathcal L$ and what just proved grants that
\[
\Linf\big\{M+L\ :\ L\in\mathcal L\big\}=\Linf\big\{M+ L'\ :\  L'\in  {\mathcal L}^m\big\}.
\]
To conclude, notice that since $ {\mathcal L}^m$ and $M+ {\mathcal L}^m=\{M+ L'\ :\  L'\in  {\mathcal L}^m\}$ are filtered, by formula in \eqref{eq:pointsup} for the inf we get
\[
\begin{split}
\big(\inf\big\{M+  L' :  L'\in {\mathcal L}^m\big\}\big)(v)&=\inf\big\{M(v)+  L'(v) :  L'\in {\mathcal L}^m\big\}\\
\text{(by \eqref{eq:infsum} in $\J_2$)}\qquad&= M(v)+\inf\big\{  L'(v) :   L'\in  {\mathcal L}^m\big\}= M(v)+\big(\inf\big\{  L' :  L'\in {\mathcal L}^m\big\}\big)(v),
\end{split}
\]
thus getting the conclusion.
\end{proof}
\end{subsection}

\begin{subsection}{Extension of linear maps}
For the discussion in the section, for the sake of generality we shall work with source spaces that are just prewedges. We start defining subadditive, superadditive and linear maps:
\begin{definition}
Let $L:\W\to V$ be a map from a prewedge $\W$ to a wedge ${\rm V}$ sending $0\in\W$ into $0\in{\rm V}$. We say that:
\[
\begin{split}
\text{\rm $L$ is superadditive if: }& \qquad L(\lambda_1 v_1+\lambda_2v_2)\geq\lambda_1 L(v_1)+\lambda_2  L(v_2)\qquad\forall \lambda_1,\lambda_2\in[0,\infty),\ v_1,v_2\in \W,\\
\text{\rm $L$ is subadditive if: }& \qquad L(\lambda_1 v_1+\lambda_2v_2)\leq\lambda_1 L(v_1)+\lambda_2  L(v_2)\qquad\forall \lambda_1,\lambda_2\in[0,\infty),\ v_1,v_2\in \W.
\end{split}
\]
We say that $L$ is linear if it is both superadditive and subadditive.
\end{definition}
Notice that choosing $\lambda_2=0$ and replacing $\lambda_1\neq 0$ with $\lambda_1^{-1}$ we see that super/sub-additive maps are homogeneous, i.e.\ satisfy $L(\lambda v)=\lambda L(v)$. Also, superadditive maps, hence also linear ones, are necessarily monotone.

In this section we study the problem of extending a linear map defined on a subwedge of a given prewedge. In line with known results in the literature (see e.g.\ \cite[Chapter 2]{Peressini67} or \cite[Section 1.2]{AB85}), this will be possible when the target is a complete lattice (a cone with joins, in our context). In our framework, though, the lack of the standard cancellation property creates non-trivial additional difficulties,  some of which are addressed by the following result:
\begin{lemma}[An existence result]\label{le:keyext}
Let $\J$ be a \lattice\ and $A,\tilde A\subset \J^2$ be sets of couples such that for every $(v,w)\in A$ and $(\tilde v,\tilde w)\in \tilde A$ we have
\[
\begin{split}
v\leq w\qquad\text{ and }\qquad v+\tilde v\leq w+\tilde w.
\end{split}
\]
Then there exists $x\in \J$ such that or every $(v,w)\in A$ and $(\tilde v,\tilde w)\in \tilde A$ we have
\begin{equation}
\label{eq:trovox}
v+x\leq w\qquad\text{ and }\qquad \tilde v\leq \tilde w+x.
\end{equation}
An explicit example of such $x$ is $x:=\inf_{(v,w)\in A}w-v$ (recall \eqref{eq:defdiff}).
\end{lemma}
\begin{proof}For any $(v,w)\in A$ and $(\tilde v,\tilde w)\in\tilde A$ we have
\[
v+\tilde v\leq w+\tilde w=v+\tilde w+(w-v)\qquad\stackrel{\eqref{eq:cancellationtrue}}\Rightarrow\qquad \tilde v\leq \tilde w+(w-v)+\eps v\stackrel{\eqref{eq:maxdiff}}=\tilde w+(w-v).
\] 
Thus for  $x:=\inf_{(v,w)\in A}(w-v)$ by \eqref{eq:infsum} we have $\tilde v\leq \tilde w+x$ for any $(\tilde v,\tilde w)\in\tilde A$, proving that the second in \eqref{eq:trovox} holds. Since the first is trivial by definition of $x$, the proof is complete.
\end{proof}
To illustrate the statement of our extension result, we start with the natural definition:
\begin{definition}[Subwedge]
Let $\W$ be a prewedge and $\W'\subset \W$. We say that $\W'$ is a subwedge provided it is closed by sum and product by a scalar.
\end{definition}
Clearly, a subwedge is also a prewedge on its own. Also,  even if not important for our discussion, we notice that if $\W$ is actually a wedge, then a subwedge $\W'$ of $\W$ is a wedge on its own, i.e.\ the intrinsic relation $\leq' $ induced by the algebraic structure of $\W'$ satisfies properties $(iii)$ and $(iv)$ in Definition \ref{def:wedge}.

Before continuing, let us point out that the fact that we work with wedges can make linear maps look substantially different from the classical vector space setting. For instance, if $v\in {\rm V}$ is so that $v+v=v$, and thus $\eps v$ is well defined and equal to $v$, then the map $L:\W\to {\rm V}$ defined as 0 in 0 and  as   $v$ everywhere else is linear. More generally, if $\W'\subset \W$ is both a subwedge and a lower set (if $\tilde \W\subset \W$ is any subwedge, then $\{w\in\W:w\leq\tilde w\ \text{ for some }\tilde w\in\tilde\W\}$ is both a subwedge and a lower set), then the map defined as 0 on $\W'$ and $v$ everywhere else is also linear.

Now suppose that $\W$ is a prewedge, ${\rm V}$  a wedge and  that $\varphi,\psi:\W\to {\rm V}$ are given maps that are superadditive and subadditive respectively. Let $\W'\subset \W$ be a subwedge and $M:\W'\to \J$ linear such that
\begin{equation}
\label{eq:suW}
\varphi\leq M\leq \psi
\end{equation}
on $\W'$. We ask whether there exists an extension of $M$ to a linear map on the whole $\W$ that still satisfies \eqref{eq:suW}. Notice that if such extension exists we must have that
\begin{equation}
\label{eq:necM}
a,b,c,d\in \W,\qquad a+b\leq c+d\qquad\Rightarrow\qquad M(a)+\varphi(b)\leq M(c)+\psi(d).
\end{equation}
Indeed, if \eqref{eq:suW} holds on the whole $\W$, for $a,b,c,d$ as in \eqref{eq:necM} the monotonicity of $M$ yields
\[
\begin{split}
M(a)+\varphi(b)&\leq M(a+b)\leq M(c+d)\leq M(c)+\psi(d).
\end{split}
\]
The necessity of \eqref{eq:necM} motivates assumption \eqref{eq:assextnuovo} below. In the statement and proof of the next, crucial, result for better clarity we are going to denote by $\leq_\s$ the preorder on the \emph{s}ource space (the prewedge $\W$) and by $\leq_\t$ the order on the \emph{t}arget (the cone with joins $\J$).
\begin{theorem}[Extension of linear maps]\label{thm:extnuovo}
Let $\W$ be a prewedge, $\J$ a cone with joins and $\varphi,\psi:\W\to \J$ be superadditive and  subadditive, respectively. Also, let $\W'\subset \W$ be a subwedge and $M:\W'\to \J$ linear such that
\begin{equation}
\label{eq:assextnuovo}
\left.
\begin{array}{rl}
a+b\!\!\!&\leq_\s c+ d\\
a,c\!\!\!&\in \W'\\
b,d\!\!\!&\in \W
\end{array}
\right\}\qquad\Rightarrow\qquad M(a)+\varphi(b)\leq_\t  M(c)+\psi(d).
\end{equation}
Then there exists $\hat M:\W\to \J$ linear, extending $M$  and such that
\begin{equation}
\label{eq:extensionnuovo}
\varphi(v)\leq_\t \hat M(v)\leq_\t\psi (v)\qquad\forall v\in \W.
\end{equation}
\end{theorem}
\begin{proof}   Let $v\in \W\setminus \W'$ (if no such $v$ exists we are done) and let $\W'':=\{t v+w:t\in[0,\infty),w\in \W'\}$ be the subwedge generated by $\W'$ and $v$. We shall prove that we can extend $ M$ to $\W''$ in such a way that \eqref{eq:assextnuovo} holds with $\W''$ in place of $\W'$: by a standard application of Zorn's lemma this will be enough to conclude.

Consider the sets $Q,\tilde Q\subset\W^4$ of quadruples  in $\W$ defined as
\begin{subequations}\begin{align}
\label{eq:q}
Q&:=\{(a,b,c,d)\ : a,c\in \W',\ b,d\in \W,\ a+v+b\leq_\s c+d\},\\
\label{eq:qtilde}
\tilde Q&:=\{(\tilde a,\tilde b,\tilde c,\tilde d)\ : \tilde a,\tilde c\in \W',\ \tilde b,\tilde d\in \W,\ \tilde a+\tilde b\leq_\s \tilde c+v+\tilde d\}
\end{align}
\end{subequations}
and  then the sets $A,\tilde A\subset\J^2$ of couples in $\J$ defined as
\begin{equation}
\label{eq:coppieAA}
\begin{split}
A&:=\Big\{\Big(  M(a) +\varphi (b)\,,\,  M(c)+\psi (d)\Big)\ :\ (a,b,c,d)\in Q\Big\},\\
\tilde A&:=\Big\{\Big(  M(\tilde a) +\varphi (\tilde b)\,,\,  M(\tilde c)+\psi (\tilde d)\Big)\ :\ (\tilde a,\tilde b,\tilde c,\tilde d)\in\tilde  Q\Big\}.
\end{split}
\end{equation}
We are looking for a value of $ M(v)\in\J$ such that \eqref{eq:assextnuovo} holds with $\W''$ in place of $\W'$. In particular, \eqref{eq:assextnuovo} must hold with $a+v$ in place of $a$, i.e.\ we must have
\begin{equation}
\label{eq:pezzo1}
z+M(v)\leq_\t w\qquad\forall (z,w)\in A.
\end{equation}
Similarly,  \eqref{eq:assextnuovo} must hold with $c+v$ in place of $c$, i.e.\ we must have
\begin{equation}
\label{eq:pezzo2}
\tilde z\leq_\t \tilde w+M(v)\qquad\forall (\tilde z,\tilde w)\in\tilde A.
\end{equation}
We want to prove that  $M(v)\in\J$ satisfying both these sets of requirements exists and to this aim we shall apply Lemma \ref{le:keyext} above. We verify the assumptions of the latter. The assumption \eqref{eq:assextnuovo}  grants that $z\leq_\t w$ for any $(z,w)\in A$. Moreover, for any $(a,b,c,d)\in Q$ and any $(\tilde a,\tilde b,\tilde c,\tilde d)\in\tilde Q$ we have
\[
(a+\tilde a)+(b+\tilde b)\stackrel{\eqref{eq:qtilde}}{\leq_\s} a+b+\tilde c+v+\tilde d\stackrel{\eqref{eq:q}}{\leq_\s} (c+\tilde c)+(d+\tilde d)
\]
thus the assumption \eqref{eq:assextnuovo} together with superadditivity of $\varphi$, subadditivity of $\psi$ and linearity of $M$ on $\W'$ give
\[
 M(a)+M(\tilde a)+\varphi(b)+\varphi(\tilde b)\leq_\t   M(c)+M(\tilde c)+\psi(d)+\psi(\tilde d).
\]
By the arbitrariness of the quadruples chosen  we see that $z+\tilde z\leq_\t w+\tilde w$ holds for any $(z,w)\in A$, $(\tilde z,\tilde w)\in\tilde A$. We can thus apply Lemma \ref{le:keyext} and obtain that
\begin{equation}
\label{eq:defx}
M(v):=\inf_{(z,w)\in A}w-z
\end{equation}
satisfies   \eqref{eq:pezzo1} and \eqref{eq:pezzo2}. We now  claim that, with this choice,  more generally it holds
\begin{equation}
\label{eq:exteso}
\left.
\begin{array}{rl}
a+tv+b\!\!\!&\leq_\s c+sv+d\\
a,c\!\!\!&\in \W'\\
b,d\!\!\!&\in \W\\
t,s\!\!\!&\in[0,\infty)
\end{array}
\right\}\qquad\Rightarrow\qquad
   M(a)+t M(v)+\varphi(b) \leq_\t   M(c)+s  M(v)+\psi(d).
\end{equation}
To prove this we distinguish few cases.\\
\noindent{\sc Case 0: $t=s=0$.} In this case \eqref{eq:exteso} reduces to \eqref{eq:assextnuovo}, that holds by assumption.\\
\noindent{\sc Case 1: $s=0$, $t>0$.} Up to scaling we can assume $t=1$. In this case the quadruple $(a,b,c,d)$ in \eqref{eq:exteso} belongs to $Q$, thus  the conclusion follows directly from \eqref{eq:pezzo1}.

\noindent{\sc Case 2: $s-t> 0$.} Up to scaling we can assume $s-t=1$. Then by the cancellation property \eqref{eq:cancellation}, for every $n\in\N$, $n>0$, we have
\[
a+b\leq_\s c+(1+\tfrac1n)v+d\leq_\s (1+\tfrac1n)(c+v+d)
\]
and the arbitrariness of $n$ and \eqref{eq:infmult} imply $a+b\leq_\s c+v+d$, i.e.\ the  quadruple $(a,b,c,d)$ in \eqref{eq:exteso}   belongs to $\tilde Q$. Hence    \eqref{eq:pezzo2} gives $ M(a)+\varphi(b)\leq_\t   M(c)+\psi(d)+M(v)$ and the conclusion \eqref{eq:exteso} follows adding $t M(v)$ to both sides.

\noindent{\sc Case 3: $t\geq s>0$.} Fix an arbitrary auxiliary quintuple $(a',b',c',d')\in Q$ and let $(z',w')\in A$ be the corresponding couple as in the first in \eqref{eq:coppieAA}. Notice that $v\leq c'+d'$, fix $n\in\N$ and  observe that the assumption in \eqref{eq:exteso} and the cancellation property \eqref{eq:cancellation} give
\[
a+(t-s+\tfrac1n)v+b\leq_\s c+\tfrac1n v+d\leq_\s c+\tfrac1n(c'+d')+d,
\]
i.e.\ $(t-s+\tfrac1n)^{-1}(a,b,c+\tfrac1nc',d+\tfrac1nd')\in Q$. Thus \eqref{eq:pezzo1}, the subadditivity of $\psi$ and the linearity of $M$ on $\W'$  give
\[
M(a)+(t-s+\tfrac1n)M(v)+\varphi(b)\leq_\t   M(c)+\psi(d)+\tfrac1n w'.
\]
Taking the $\inf$ in $n\in\N$ and in the auxiliary quintuple we get
\[
\begin{split}
 M(a)&+(t-s) M(v)+\varphi(b)\leq_\t   M(c)+ \inf_{(z,w)\in A}\eps w  +\psi(d).
\end{split}
\]
Now add $s M(v) $ to both sides and use that
\[
 \inf_{(z,w)\in A}\eps w\stackrel{\eqref{eq:epsdiff}}= \inf_{(z,w)\in A}\eps (w-z)=\inf_{\lambda>0,(z,w)\in A}\lambda (w-z)\stackrel{\eqref{eq:basemult}}= \inf_{\lambda>0}\lambda\inf_{(z,w)\in A} (w-z)\stackrel{\eqref{eq:defx}}=\eps   M(v)
\]
and the fact that $s>0$ to conclude that \eqref{eq:exteso} holds.

We thus established the validity of \eqref{eq:exteso}. Picking $b=d=0$ in there and then inverting the roles of $(a,t)$ and $(c,s)$ we deduce that
\begin{equation}
\label{eq:linM}
\left.
\begin{array}{rl}
a+tv\!\!\!&= c+sv\\
a,c\!\!\!&\in \W'\\
t,s\!\!\!&\in[0,\infty)
\end{array}
\right\}\qquad\Rightarrow\qquad
   M(a)+t   M(v)=  M(c)+s   M(v).
\end{equation}
We now extend $  M$ from $\W'$ to the subwedge $\W'':=\{a+tv:a\in\W',t\in[0,\infty)\}$ by putting $  M(a+tv):=  M(a)+t  M(v)$. By \eqref{eq:linM} we know that this is a good definition, i.e.\ that the value of $  M(a+tv)$ only depends on the argument, and not on how we wrote it as combination of an element of $\W'$ and a multiple of $v$.

It is  clear from the definition that $  M$ is linear and then that \eqref{eq:exteso} is the required \eqref{eq:assextnuovo} with $\W''$ in place of $\W'$. The proof is therefore  completed.
\end{proof}
\begin{remark}[Trivial choices of $\varphi,\psi$]\label{re:trivialfipsi}{\rm The above extension theorem demands,  and then extends, the two bounds \eqref{eq:extensionnuovo}.  Notice, though, that we can easily reduce to the case where only one of these, or none, is present. Indeed, we can always apply the statement with the `zero map' $\varphi\equiv0$ ($\varphi$  is clearly linear, hence superadditive) and  the `infinity map' $\psi$ defined as $\psi(0_\W):=0_\J$ and $\psi(v)=+\infty\in\J$ for every $v\neq 0$  ($\psi$ is clearly linear, hence subadditive).

The first choice makes the  presence of $b$ in \eqref{eq:assextnuovo} irrelevant and first bound in \eqref{eq:extensionnuovo} always true. The second choice makes the   presence of $d$ in \eqref{eq:assextnuovo} irrelevant and the second bound in \eqref{eq:extensionnuovo} always true.
}\fr\end{remark}

\begin{remark}\label{re:implicaHB}{\rm
Theorem \ref{thm:extnuovo} implies the classical Hahn-Banach extension theorem. To see why, let $V$ be a real vector space and $p:V\to\R$ be such that
\begin{equation}
\label{eq:rhosubadd}
\begin{split}
p(v+w)&\leq p(v)+p (w),\\
p(\lambda v)&=\lambda p(v)
\end{split}
\end{equation}
for every $v,w\in V$ and $\lambda\geq 0$. Also,  let $V'\subset V$ be a subspace and $T:V'\to\R$ a linear map such that
\begin{equation}
\label{eq:Tbound}
T(v)\leq p(v)\qquad\forall v\in V'.
\end{equation}
We shall use Theorem \ref{thm:extnuovo}  to extend $T$ to a linear map on the whole $V$ still satisfying \eqref{eq:Tbound}. Consider on $\R\times V$ the  preorder $\preceq$ defined as
\begin{equation}
\label{eq:pov}
(t,v)\preceq (s,w)\qquad\text{ whenever }\qquad p(w-v)\leq s-t,
\end{equation}
and let $F:=\{(t,v)\in\R\times V:(0,0)\preceq (t,v)\}$ be the `future cone'. The subadditivity of $p$ encoded in \eqref{eq:rhosubadd} grants that $\preceq$ is a preorder. It is then clear that $F$ is a prewedge and that the preorder relation induced by the wedge structure coincides with $\preceq$ (as the definition \eqref{eq:pov} tells that $(t,v)\preceq (s,w)$ if and only if $(s-t,w-v)\in F$). Then consider the subwedge $F':=\{(t,v)\in F:v\in V'\}$ and  the map $M:F'\to \R$ defined as 
\begin{equation}
\label{eq:defM}
M(t,v):=t-T(v)\qquad\forall (t,v)\in F'.
\end{equation}
Such $M$ is linear and, by \eqref{eq:Tbound} and \eqref{eq:pov}, takes only non-negative values.

We can thus apply Theorem \ref{thm:extnuovo} with the prewedge $\W:=F$, the subwedge $\W':=F'$, the target $\J:=[0,+\infty]$, the map $M$ just defined, $\varphi\equiv 0$ and $\psi$ the `infinite map' set to 0 in 0 and to $+\infty$ everywhere else (see also Remark \ref{re:trivialfipsi}). It is easy to see that with these choices the assumption \ref{eq:assextnuovo} holds, indeed: if $d\neq 0$ then $\psi(d)=+\infty$ and the conclusion in \eqref{eq:assextnuovo} holds, otherwise we assume $a+b\leq_\s d$ and we want to conclude that $M(a)\leq_\t M(d)$, which is obvious by the monotonicity of $M$. Thus Theorem  \ref{thm:extnuovo} is applicable and it provides a linear extension $ \hat M:F\to[0,+\infty]$ of $M$.

Let $(t,v)\in F$, notice that $(p(-v),-v)\in F$  and thus 
\[
\hat  M(t,v)\leq \hat M\big((t,v)+(p(-v),-v)\big)=\hat  M(t+p(-v),0)\stackrel*=M(t+p(-v),0)\stackrel{\eqref{eq:defM}}=t+p(-v)<\infty,
\]
having used in the starred identity that $(t+p(-v),0)\in F'$, that in turn follows from $t+p(-v)\geq p(v)+p(-v)\geq p(0)$. This proves that $\hat  M$ never attains the value $+\infty$.  Also, using again the identity   $ \hat M(t,0)=M(t,0)=t$ valid for every $t\geq 0$  we have
\[
t+ \hat M(s,v)= \hat M(t+s,v)=s+\hat  M(t,v)\qquad\forall t,s\geq p(v),\ \forall v\in V.
\]
Thus  the value of $t- \hat M(t,v)$ is independent on $t\geq p(v)$: calling such value $ \hat T(v)$ we then see that $ \hat T$ is a real valued linear map on $V$ that, trivially from the fact that $ \hat M$ extends $M$ and from \eqref{eq:defM}, extends $T$. Also, the positivity of $ M$ on $F$ and the fact that $(p(v),v)\in F$ for every $v\in V$ give
\[
\hat  T(v)=p(v)-\hat  M(p(v),v)\leq p(v)\qquad\forall v\in V,
\]
as desired. See Section \ref{se:Banst} for more on the transition $V\mapsto \text{future cone in }\R\times V$.
}\fr\end{remark}
\begin{remark}\label{re:flexext}{\rm
There is some flexibility in Theorem \ref{thm:extnuovo} and its proof, as the partial orders $\leq_\s$ and $\leq_\t$ might in principle be different from those coming from the (pre)wedge structures: it is sufficient, somehow tautologically, that all the relations between the partial orders and the operation of sum and product by scalar that we use in the proof are still in place. In particular,  the properties of the preorder $\leq_\s$ on the source prewedge $\W$ that we used are: $0\leq_\s v$ for any $v\in\W$ plus \eqref{eq:supmultbase} and \eqref{eq:basebase} (these imply that $\leq_\s$ contains the preorder induced by the sum via \eqref{eq:defpo}). Analogously, for  $\leq_\t$ it is important that on top of the already mentioned properties we also have that it is a complete lattice and that \eqref{eq:infsum} holds (warning: these suffice for Theorem \ref{thm:extnuovo}, but other results concerning cone with joins rely on Proposition \ref{prop:baselattice}, that in turn requires the partial order to be the one induced by the sum via \eqref{eq:defpo}, for instance because from $y\geq y\wedge w$ we deduce that there is $a$ such that $y\wedge w+a=y$).

As trivial as it is, this observation might be relevant in applications where the wedge structure comes with an additional partial order, different from the one coming from \eqref{eq:defpo}. This might be the case, for instance, for wedges of functions, that canonically come with the pointwise order.
}\fr\end{remark}

We conclude the section presenting a variant of the above extension result that appears stronger, even though in fact is a consequence of what already proved.

Suppose that $\W$ is a prewedge, $\J$ a cone with joins that that  $\varphi,\psi,L:\W\to\J$ are given maps that are superadditive, subadditive and linear respectively. Let $\W'\subset \W$ be a subwedge and $M:\W'\to \J$ linear such that
\begin{equation}
\label{eq:suWvecchio}
\varphi\leq L+M\leq \psi
\end{equation}
on $\W'$. We ask whether there exists an extension of $M$ to a linear map on the whole $\W$ that still satisfies \eqref{eq:suWvecchio}. Notice that if such extension exists we must have that
\begin{equation}
\label{eq:necMvecchio}
a,b,b',c,d\in \W,\qquad a+b+b'\leq c+ d\qquad\Rightarrow\qquad (L+M)(a)+L(b)+\varphi(b')\leq (L+M)(c)+\psi(d).
\end{equation}
Indeed, if \eqref{eq:suWvecchio} holds on the whole $\W$, for $a,b,b',c,d$ as in \eqref{eq:necMvecchio} the monotonicity of $L+M$ yields
\[
\begin{split}
(L+M)(a)+L(b)+\varphi(b')&\leq(L+M)(a)+(L+M)(b)+(L+M)(b')\\
\text{(as $a+b+b'\leq c+d$)}\qquad\qquad&\leq (L+M)(c+d)\leq (L+M)(c)+\psi(d).
\end{split}
\]
As before,  \eqref{eq:necMvecchio} motivates assumption \eqref{eq:assextvecchio} below and, as before, in the statement and proof of the next result, for better clarity we are going to denote by $\leq_\s$ the preorder on the \emph{s}ource space  and by $\leq_\t$ the order on the \emph{t}arget.
\begin{corollary}\label{thm:extvecchio}
Let $\W$ be a prewedge, $\J$ a cone with joins and $\varphi,\psi,L:\W\to \J$ be superadditive, subadditive and linear, respectively. Also, let $\W'\subset \W$ be a subwedge and $M:\W'\to \J$ linear such that
\begin{equation}
\label{eq:assextvecchio}
\left.
\begin{array}{rl}
a+b+b'\!\!\!&\leq_\s c+d\\
a,c\!\!\!&\in \W'\\
b,b',d\!\!\!&\in \W
\end{array}
\right\}\qquad\Rightarrow\qquad (L+M)(a)+L(b)+\varphi(b')\leq_\t (L+M)(c)+\psi(d).
\end{equation}
Then there exists $\hat M:\W\to \J$ linear, so that $\hat M(w)=M(w)+\eps L(w)$ for every $w\in \W'$  and such that
\begin{equation}
\label{eq:extensionvecchio}
\varphi(v)\leq_\t L(v)+\hat M(v)\leq_\t\psi (v)\qquad\forall v\in \W.
\end{equation}
\end{corollary}
\begin{proof}  
Define $\tilde\varphi:\W\to\J$ as
\[
\tilde\varphi(v):=\sup_{v_1+v_2=v}\varphi(v_1)+ L(v_2),
\]
notice that, trivially, we have $\tilde\varphi(\lambda v)=\lambda \tilde\varphi (v)$. Also, for $v,w\in\W$ and arbitrary writing $v=v_1+v_2$ and $w=w_1+w_2$ we obviously have $v+w=(v_1+w_1)+(v_2+w_2)$, hence
\[
\begin{split}
\tilde\varphi(v+w)\geq \varphi(v_1+w_1)+L(v_2+w_2)\geq \varphi(v_1)+L(v_2)+\varphi(w_1)+L(w_2)
\end{split}
\]
and the arbitrariness of $v_1,v_2,w_1,w_2$ as above give $\tilde\varphi(v+w)\geq \tilde\varphi(v)+\tilde\varphi(w)$. Now, define $\tilde M:\W'\to \J$ as $\tilde M:=L+M$ and notice that the assumption \eqref{eq:assextvecchio} imply
\[
\left.
\begin{array}{rl}
a+\tilde b\!\!\!&\leq_\s c+d\\
a,c\!\!\!&\in \W'\\
\tilde b,d\!\!\!&\in \W
\end{array}
\right\}\qquad\Rightarrow\qquad \tilde M(a)+\tilde \varphi(\tilde b)\leq_\t \tilde M(c)+\psi(d).
\]
We can therefore apply the extension Theorem \ref{thm:extnuovo} and deduce that we can extend $\tilde M$ to a linear map, still denoted $\tilde M$, such that 
\begin{equation}
\label{eq:vecchiochiuso}
\tilde\varphi(v)\leq_\t\tilde M(v)\leq_\t \psi(v)\qquad\forall v\in\W.
\end{equation}
Since by construction for every $v\in\W$ we have $\tilde\varphi(v)\geq_\t L(v)$, we can define
\[
\hat M(v):=\tilde M(v)-L(v)\qquad\forall v\in \W
\]
and notice that by \eqref{eq:difflin} the map $\hat M:\W\to \J$ is linear. Also, \eqref{eq:maxdiff} gives that $\hat M=M+\eps L$ on $\W'$, so that to conclude it suffices to prove \eqref{eq:extensionvecchio}. This, however, is obvious from \eqref{eq:vecchiochiuso}, the validity of $\tilde M(v)=L(v)+\hat M(v)$ and $\tilde\varphi(v)\geq_\t \varphi(v)$  for every $v\in\W$.
\end{proof}
\end{subsection}

\begin{subsection}{Morphisms and the directed decomposition property}\label{se:morphisms}

Much like in classical Banach space theory  one is not interested in general linear maps, but rather in those that also respect Cauchy-limits, i.e.\ that are continuous, here we are not interested in general linear maps, but rather in those that respect directed suprema, i.e.\ with  the $\mcp$:
\begin{definition}[Morphism]\label{def:morphism}
Let $\W_1$ be a wedge and $\W_2$ be a cone. We define the set of morphisms from $\W_1$ to $\W_2$ as $\Hom(\W_1,\W_2):=\Lin(\W_1,\W_2)\cap \mcp(\W_1,\W_2)$.

The dual $\W^*$ of a wedge $\W$ is defined as $\Hom(\W,[0,\infty])$.
\end{definition}
Sum and multiplication by positive scalars can be defined pointwise on $\Hom(\W_1,\W_2)$. We have already noticed that the result of these operations is linear if the starting maps are linear. Using \eqref{eq:supsommadir} it is also immediate to check that the result has the $\mcp$ if the starting maps have it, hence $\Hom(\W_1,\W_2)$ is a prewedge. 

The associated  preorder  $\leq$ on $\Hom(\W_1,\W_2)$  is thus defined as: $L_1\leq L_2$ whenever there is $M\in\Hom(\W,\C)$ such that $L_1+M=L_2$. It is clear that for any $L_1,L_2\in \Hom(\W_1,\W_2)$ we have
\begin{equation}
\label{eq:3ordini}
L_1\leq L_2\qquad\qquad\Rightarrow\qquad\qquad L_1\leq_\L L_2\qquad\qquad\Rightarrow\qquad\qquad L_1\leq_\p  L_2
\end{equation}
showing in particular that $L_1\leq L_2$ and $L_2\leq L_1$ imply $L_1=L_2$, i.e.\ that $\leq$ is a partial order. We know from Proposition \ref{prop:lincone} that the second implication is an equivalence if $\W_2$ is a \lattice. We will be able to show that the first implication is also an equivalence  under the additional structural assumptions on the source space given in Definition \ref{def:tdp}, that tighten the links between linear structure and directed suprema. In particular, at this level of generality we do not know whether $\Hom(\W,\C)$ is a cone, or a wedge, not even in the particular case $\C=[0,\infty]$. 

Recalling \eqref{eq:mcpcomplete}, \eqref{eq:mcpcomplete2} and \eqref{eq:linp} it is immediate to see that for any wedge $\W$ and cone $\C$ we have that 
\begin{equation}
\label{eq:linhom}
\big(\,\Hom(\W,\C)\,,\,\leq_\p \big)\qquad\text{ is a directed complete partial order},
\end{equation}
the supremum of a $\leq_\p -$directed family being given by the pointwise formula as in \eqref{eq:pointsup}.


To further investigate the matter, we recall that  in Section \ref{se:mapsmcp} we discussed a general way to produce a map with the $\mcp$ out of a general lattice-valued map: the projection operator given by Definition \ref{def:prmcp}. Its general relations with linearity are:
\begin{proposition}[Projection and linearity]\label{prop:prsubadditive}
Let $\W$ be a wedge, $\J$ a \lattice\ and $L:\W\to \J$  linear. Then ${\sf Pr}(L):\W\to \J$ is subadditive.

Also, let $M_1,M_2:\W\to \J$ be arbitrary, not necessarily linear, maps and $\lambda_1,\lambda_2 \in[0,+\infty)$. Then
\begin{equation}
\label{eq:prsuperlin}
{\sf  Pr}(\lambda_1 M_1+\lambda_2  M_2)\geq_\p \lambda_1 {\sf  Pr}(M_1)+\lambda_2 {\sf  Pr}(M_2)
\end{equation}
with equality if either one of $M_1,M_2$ has the $\mcp$.
\end{proposition}
\begin{proof} We shall prove that if $L:\W\to \J$ is subadditive, then so is $P(L):\W\to \J$ as defined by \eqref{eq:defprt}. By transfinite recursion and Proposition \ref{prop:projmcp} this will suffice to conclude.

Let $v_1,v_2\in \W$, $\lambda_1,\lambda_2\in[0,\infty)$ and let $D_1,D_2\subset \W$ be directed sets with supremum $v_1,v_2$ respectively. It is then clear that the set $D:=\{\lambda_1w_1+\lambda_2 w_2:w_1\in D_1,w_2\in D_2\}$ is directed with supremum $\lambda_1v_1+\lambda_2 v_2$, thus
\[
\begin{split}
P(L)(\lambda_1v_1+\lambda_2 v_2)&\leq \sup_{w_1\in D_1,w_2\in D_2} L(\lambda_1w_1+\lambda_2 w_2)\\
&\leq \sup_{w_1\in D_1,w_2\in D_2}\big( \lambda_1L( w_1)+\lambda_2 L( w_2)\big)\stackrel{\eqref{eq:basemult},\eqref{eq:basesomma}}=\big(\lambda_1\sup_{D_1}L\big)+\big(\lambda_2 \sup_{D_2}L\big).
\end{split}
\]
Taking the infimum over $D_1,D_2$ and using \eqref{eq:infsum} we conclude.

We pass to the second claim and notice that the inequality \eqref{eq:prsuperlin} is obvious from the fact that the right hand side has the $\mcp$ (by \eqref{eq:supsommadir}) and is $\leq_\p\lambda_1M_1+\lambda_2 M_2$. For the equality case, possibly swapping $M_1$ and $M_2$,  by transfinite recursion it suffices to prove that if $M_2$ is monotone then $P(\lambda_1M_1+\lambda_2 M_2)\leq_\p \lambda_1P(M_1)+\lambda_2 M_2$. To see this, let $D\subset \W$ be directed with supremum $v$. Then we have
\[
\begin{split}
P(\lambda_1M_1+\lambda_2 M_2)(v)&\leq \sup_{w\in D}(\lambda_1M_1+\lambda_2 M_2)(w)=\sup_{w\in D}\big(\lambda_1M_1(w)+\lambda_2 M_2(w)\big)\\
\text{(by monotonicity of $M_2$)}\qquad&\leq\sup_{w\in D}\big(\lambda_1M_1(w)+\lambda_2 M_2(v)\big)=\big(\sup_{w\in D}\lambda_1M_1(w)\big)+\lambda_2 M_2(v)
\end{split}
\]
and taking the inf over $D$ we conclude recalling \eqref{eq:infsum}.
\end{proof}
From the above proposition we get the following result, that can be thought of as an addendum to Theorem \ref{thm:extnuovo} and Corollary \ref{thm:extvecchio}, see e.g.\ how these are combined in the proof of Proposition \ref{prop:normabiduale}.

\begin{theorem}[Existence theorem]\label{thm:exthm}
Let $\W$ be a wedge and $\J$ a \lattice. Let $\varphi,\psi:\W\to \J$ be superadditive and subadditive respectively and $L,M:\W\to \J$ linear. Assume that:
\begin{itemize}
\item[i)] We have
\begin{equation}
\label{eq:peresistenza}
\varphi\leq_\p L+M\leq_\p \psi\qquad \text{on $\W$},
\end{equation}
\item[ii)] $L$ has the $\mcp$,
\item[iii)] For some $\bar v\in \W$ such that $\eps ((L+M)(\bar v))=0$ the map $\W\ni v\mapsto \varphi(v+\bar v)$ has the $\mcp$.
\end{itemize} 
Then there is a subadditive map $\tilde M:\W\to \J$ with the $\mcp$ such that
\begin{equation}
\label{eq:exthm}
\varphi\leq_\p L+\tilde M\leq_\p \psi\qquad \text{on $\W$},
\end{equation}
an explicit choice being $\tilde M:= {\sf Pr}(M)$.
\end{theorem}
\begin{proof}
Since ${\sf Pr}(M)\leq_\p M$ the second inequality in \eqref{eq:exthm} trivially follows from the analog one in \eqref{eq:peresistenza}. For the first one we pick $\bar v$ as in $(iii)$ and $n\in\N$. Then the first in \eqref{eq:peresistenza} and the linearity of $L,M$ give
\begin{equation}
\label{eq:perpr}
\varphi(\cdot+\tfrac1n\bar v)\leq_\p L(\cdot)+M(\cdot)+\tfrac1n(L+M)(\bar v).
\end{equation}
Now observe that the trivial identity $\varphi(v+\tfrac1n\bar v)= \tfrac 1n\varphi (nv+\bar v)$ and $(iii)$ grant that $\varphi(\cdot+\tfrac1n\bar v)$ has the $\mcp$, thus applying ${\sf Pr}$ to both sides of \eqref{eq:perpr}, unsing the monotonicity of ${\sf Pr}$ (recall \eqref{eq:monpr}) and the equality case in \eqref{eq:prsuperlin} (used both with $L$ and the constant map $\tfrac1n(L+M)(\bar v)$ that obviously has the $\mcp$) we get
\[
\varphi(\cdot)\leq_\p\varphi(\cdot+\tfrac1n\bar v)\leq_\p  L(\cdot)+{\sf Pr}(M)(\cdot)+\tfrac1n(L+M)(\bar v),
\]
having also used the monotonicity of the superadditive map  $\varphi$. Taking the $\inf$ in $n\in\N$ we conclude by our assumption $\eps (L+M)(\bar v)=0$.
\end{proof}
In such statement we gained the $\mcp$ of the map we found by giving up linearity in favour of subadditivity. This is a high price to pay: we can avoid   paying it only  if the source space has the property described in the following definition, reminiscent of the Riesz decomposition property valid in Riesz spaces:
\begin{definition}[Directed decomposition property]\label{def:tdp}
We say that the cone $\C$ has the directed decomposition property provided the following holds. For any $v\in \C$, any directed set $D$ with $\sup D=v$ and any decomposition $v=v_1+v_2$ there are  sets with tips $D_1,D_2\subset \C$ with tips $v_1,v_2$ respectively such that 
\begin{equation}
\label{eq:ddp}
\forall w_1\in D_1,\ w_2\in D_2\quad\text{ there is }\quad w\in D\quad\text{ such that }w_1+w_2\leq w.
\end{equation}
\end{definition}
Intuitively, the definition asks the following. Suppose that $(v_n)$ is an increasing sequence with supremum $v$ and write $v$ as $v=w+z$. Since $v_n\leq w+z$ we can hope to be able to write $v_n=w_n+z_n$ for some $w_n,z_n$ such that $w_n\leq  w$ and $z_n\leq z$. Also, since $(v_n)$ is increasing we might expect to find these $(w_n)$ and $(z_n)$ to be increasing as well and then  to expect that the suprema of these sequences are $w,z$ respectively. Being able to get such $(w_n),(z_n)$ out of $(v_n)$ is, technicalities apart,  the content of the above definition. The similarities between this definition and that of truncation property (Definition \ref{def:trunc}) are not only formal: see Proposition \ref{prop:tdptrunc} below.

If a cone has the directed decomposition property, the following slightly more general property holds:  for any $v\in \C$, any directed set $D$ with $\sup D=v$ and any decomposition $v=\lambda_1 v_1+\lambda_2 v_2$ there are  sets with tips $D_1,D_2\subset \C$ with tips $v_1,v_2$ respectively such that 
\begin{equation}
\label{eq:ddp2}
\forall w_1\in D_1,\ w_2\in D_2\quad\text{ there is }\quad w\in D\quad\text{ such that }\lambda_1 w_1+\lambda_2 w_2\leq w.
\end{equation}
Indeed, if either $\lambda_1$ or $\lambda_2$ is 0 the conclusion is trivial, otherwise just rescale the  sets with tips in \eqref{eq:ddp} by factors $\lambda^{-1}_1$ and $\lambda_2^{-1}$.

The relevance of this notion in relation to maps with the $\mcp$ is due to the next result, to be compared to Proposition \ref{prop:prsubadditive}:
\begin{proposition}[Projection of linear maps - revisited]\label{prop:linpr}
Let $\C$ be a cone with the directed decomposition property, $\J$ a \lattice,  and $L:\C\to \J$ linear.

Then   ${\sf Pr}(L)$ is linear.
\end{proposition}
\begin{proof} By Proposition \ref{prop:prsubadditive} it suffices to prove superadditivity and to this aim we proceed by transfinite recursion and show that if $L$ is superadditive, then so is $P(L)$ (as defined by \eqref{eq:defprt}).

Fix such $L$, let $v\in \C$ and  then $v_1,v_2\in \C$  and $\lambda_1,\lambda_2\in[0,\infty)$ so that $v=\lambda_1 v_1+\lambda_2 v_2$. Then let  $D\subset \C$ be a directed subset having supremum $v$  and use the assumption on $\C$ to find $D_1,D_2\subset \C$ so that \eqref{eq:ddp2} holds. Then \eqref{eq:ddp2}  and the monotonicity of the superadditive map $L$ yield
\[
\begin{split}
\sup_{w\in D}L(w)&\geq \sup_{w_1\in D_1,\, w_2\in D_2}L(\lambda_1 w_1+\lambda_2 w_2)\\
&\geq \sup_{w_1\in D_1,\, w_2\in D_2}\big(\lambda_1 L( w_1)+\lambda_2 L( w_2)\big)\stackrel{\eqref{eq:basemult},\eqref{eq:basesomma}}= \lambda_1 \sup_{w_1\in D_1} L( w_1)+\lambda_2\sup_{w_2\in D_2} L( w_2).
\end{split}
\]
In particular, we have $\sup_{w\in D}L(w)\geq \lambda_1 P(L)(w_1)+\lambda_2 P(L)(w_2)$ and by the arbitrariness of $D$ we conclude. 
\end{proof}
In conjunction with Theorem \ref{thm:exthm}, the above has the following important consequence:
\begin{theorem}\label{thm:ddpmaps}
Let $\C$ be a cone with the directed decomposition property and $\J$ a \lattice. Then:
\begin{itemize}
\item[i)] With the same assumptions and notation of Theorem \ref{thm:exthm}, the map $\tilde M$ in the conclusion can  be taken to be also linear. In fact, the choice $\tilde M:={\sf Pr}(M)$ does the job.
\item[ii)] On  $\Hom(\C,\J)$ the orders $\leq$,  $\leq_\L$ and $\leq_\p$ agree. In particular $\Hom(\C,\J)$ is a cone and the supremum of a directed family is given by the pointwise formula \eqref{eq:pointsup}.
\item[iii)] Assume that also $\C$ is a \lattice. Then $\Hom(\C,\J)$ is also a cone with joins and for any family $(L_i)_{i\in I}\subset\Hom(\C,\J)$ we have
\begin{equation}
\label{eq:suplsup}
\sup_iL_i=\Lsup_i L_i\qquad\text{ and }\qquad\inf_iL_i={\sf Pr}\big(\Linf_i L_i\big).
\end{equation}
In particular, the Riesz-Kantorovich formulas \eqref{eq:RKpgen} and \eqref{eq:RKp} for the sup hold.
\end{itemize}
\end{theorem}
\begin{proof}
For $(i)$ we apply  Theorem \ref{thm:exthm} and notice that in this case  Proposition \ref{prop:linpr} grants  the linearity of $\tilde M:={\sf Pr}(M)$.

For $(ii)$ we notice that thanks to \eqref{eq:3ordini} it suffices to prove that for $L_1,L_2\in\Hom(\C,\J)$  with $L_1\leq_\p L_2$ we have $L_1\leq L_2$. To see this, use first Proposition \ref{prop:lincone} to find $M:\C\to \J$ linear with $L_1+M=L_2$, then apply item $(i)$ above  with $\varphi=\psi=L_2$ to deduce from $L_2\leq_\p L_1+M\leq_\p L_2 $ that ${\sf Pr}(M)\in\Hom(\C,\J)$ also satisfies $L_2\leq_\p L_1+{\sf Pr}(M)\leq_\p L_2 $, i.e.\ $L_1+{\sf Pr}(M)=L_2$, proving that $ L_1 \leq L_2$, as desired.

For $(iii)$ we observe that the definitions and the linearity proved in Proposition \ref{prop:linpr} directly imply that
\[
\sup_iL_i={\sf Pr}\big(\Lsup_i L_i\big)\qquad\text{ and }\qquad\inf_iL_i={\sf Pr}\big(\Linf_i L_i\big).
\] 
To prove \eqref{eq:suplsup} it thus suffices to prove that ${\sf Pr}\big(\Lsup_i L_i\big)= \Lsup_i L_i$. The inequality $\leq_\p$ is obvious from the definition of ${\sf Pr}$. To prove $\geq_\p$ we notice that ${\sf Pr}\big(\Lsup_i L_i\big)\geq_\p {\sf Pr}(L_i)=L_i$   for every $i\in I$. Then  the conclusion follows from  formula \eqref{eq:RKpgen} for $\Lsup_i L_i$ and the linearity of ${\sf Pr}\big(\Lsup_i L_i\big)$.

To conclude the proof that $\Hom(\C,\J)$ is a \lattice\ we need to prove that $M+\inf_iL_i=\inf_i(M+L_i)$ holds for arbitrary $M,L_i\in \Hom(\C,\J)$, $i\in I$, $I$ being an arbitrary set of indexes. Using \eqref{eq:suplsup}, this follows from
\[
\begin{split}
\inf_i(M+L_i)&={\sf Pr}\big(\Linf_i(M+L_i)\big)\\
\text{(as $\Lin(\C,\J)$ is a \lattice, see Prop. \ref{prop:linlattice})}\qquad &={\sf Pr}\big(M+\Linf_i L_i\big)\\
\text{(by the equality case in \eqref{eq:prsuperlin})}\qquad&=M+{\sf Pr}\big(\Linf_i L_i\big)=M+\inf_iL_i,
\end{split}
\]
and the proof is complete.
\end{proof}
It is worth to notice that  a slight strengthening of the directed decomposition property implies that a cone with joins  is the completion of its finite part:
\begin{proposition}\label{prop:tdptrunc}
Let $\C$ be a cone  with the following property: for any $v\in \C$, any  set with tip $D\subset \C$ with $\tip D=v$  and any decomposition $v=v_1+v_2$ there are  sets with tips $D_1,D_2\subset \C$ with tips $v_1,v_2$ respectively such that \eqref{eq:ddp} holds.

Then $\C$ has the truncation property (recall Definition \ref{def:trunc}). In particular, if the finite part $\C_{\rm finite}$ of $\C$ as in Remark \ref{re:finitepart} is dense, then $\C$ is, together with the inclusion, the directed completion of  $\C_{\rm finite}$.
\end{proposition}
\begin{proof}
Let $D\subset \C$ be with tip and $v_1\in \C$ be such that $v_1\leq v:=\tip D$. Then there is $v_2\in \C$ so that $v_1+v_2=v$ and our assumption yields the existence of $D_1,D_2$ as in the statement. Since $v_1=\tip D_1\in \widehat{D_1}$, to check that property \eqref{eq:truncpropr} holds (with $a:=v$ and $B:=D$) it suffices to check that $D_1\subset (\downarrow D)\cap(\downarrow v_1)$. The inclusion $D_1\subset \downarrow v_1$ is obvious and the one $D_1\subset (\downarrow D)$ is a direct consequence of \eqref{eq:ddp}.

Thus the truncation property is proved.  The second claim is a direct consequence of what just proved and of Proposition \ref{prop:trunccompl}. 
\end{proof}

\begin{remark}{\rm
The cone with joins $\{0,\infty\}$ satisfies the assumption in Proposition \ref{prop:tdptrunc} but its finite part $\{0\}$ is not dense.
}\fr\end{remark}

We conclude the section with a comment on duality. As usual, in presence of a dual operation we have canonical mapping in the bidual. More precisely, given a wedge $\W$ and $v\in \W$ we define
\begin{equation}
\label{eq:defPhi}
\begin{array}{rlll}
\Phi(v):&\W^*&\to& [0,\infty],\\
&L&\mapsto& \Phi(v)(L):=L(v).
\end{array}
\end{equation}
Notice that in general $\Phi$ is not injective, see also Example $(e)$ in Section \Ref{se:finitedim}. We  have the following:
\begin{proposition}\label{prop:Wstarcone}
Let $\W$ be a wedge and assume that on $\W^*:=\Hom(\W,[0,\infty])$ the order relations $\leq $ and $\leq_\p $ agree.

Then $\W^*$ is a cone and the map $\Psi$ defined above in \eqref{eq:defPhi} takes values in $\W^{**}$. 
\end{proposition}
\begin{proof}
The fact that $\W^*$ is a cone follows from the fact that $\Hom(\W,\C)$ is a wedge, the completeness property \eqref{eq:linhom} and the assumption. The linearity of $\Phi(v)(\cdot)$ is obvious. To check the $\mcp$ we need to prove that if $(L_i)\subset \W^*$ is $\leq$-directed with supremum $L$, then $\sup_i L_i(v)=L(v)$. This, however, is another consequence of the fact that  $\leq$ and $\leq_\p$ are assumed to agree.
\end{proof}

\end{subsection}

\begin{subsection}{Hyperbolic norms and hyperbolic Banach spaces}\label{se:hbs}

\begin{definition}[Hyperbolic norm]
Let $\W$ be a prewedge. A hyperbolic norm $\hn$ on $\W$ is a $[0,+\infty]$-valued superadditive function on it, i.e.\ a map $\hn:\W\to[0,+\infty]$ that is 0 in 0 and  such that
\begin{equation}
\label{eq:defnorm}
\hn(\lambda_1v_1+\lambda_2v_2)\geq \lambda_1\hn(v_1)+\lambda_2\hn(v_2)\qquad\forall v_1,v_2\in\W,\ \lambda_1,\lambda_2\in[0,+\infty).
\end{equation}

A hyperbolic Banach space is given by a cone and a hyperbolic norm on it.

\end{definition}

\begin{definition}[Operator norm]
Let $\W_1,\W_2$ be wedges with hyperbolic norms $\hn_1,\hn_2$ respectively and $L\in\Lin(\W_1,\W_2)$. Then the operator norm of $L$ is defined as 0 if $L\equiv 0$ and otherwise as
\begin{equation}
\label{eq:opnorm}
\hn_{\rm op}(L):=\inf_{v\in \W_1:\hn_1(v)\geq 1} \hn_2(L(v))=\sup\{C\in[0,+\infty]:\hn_2(L(v))\geq C\hn_1(v)\ \forall v\in \W_1\}.
\end{equation}
If $\W_2=[0,+\infty]$ (with the identity as norm), we write $\hn_*$ in place of $\hn_{\rm op}$.
\end{definition}
The verification of the second equality in \eqref{eq:opnorm} follows from routine manipulations.  It is also immediate to check that $\hn_{\rm op}$ is a hyperbolic norm on $\Lin(\W_1,\W_2)$, as 
\[
\begin{split}
\hn_{\rm op}(\lambda_1 L_1+\lambda_2 L_2)&=\inf_{\hn_1(v)\geq 1} \hn_2(\lambda_1 L_1(v)+\lambda_2 L_2(v)))\\
&\geq \inf_{\hn_1(v)\geq 1} \big(\lambda_1\, \hn_2(L_1(v))+\lambda_2\, \hn_2(L_2(v)))\big)\\
&\geq \inf_{\hn_1(v)\geq 1} \lambda_1 \,\hn_2(L_1(v))+\inf_{\hn_1(v)\geq 1} \lambda_2 \,\hn_2(L_2(v)))=\lambda_1\,\hn_{\rm op}( L_1)+\lambda_2\,\hn_{\rm op}( L_2).
\end{split}
\]
\begin{remark}{\rm
Let $M=\R^{n+1}$ be the standard $(n+1)$-dimensional Minkowski spacetime, let $g$ be its metric tensor, with signature $(+,-\cdots,-)$, and $F\subset M$ the future cone. In particular, for $v\in F$ we have $g(v,v)\geq 0$ so that the quantity 
\[
|v|:=\sqrt{g(v,v)}
\]
is well defined for any $v\in F$. The reverse Cauchy-Schwarz inequality $g(v,v')\geq |v||v'|$, valid for any $v,v'\in F$, implies the reverse  triangle inequality, that in turn shows that $|\cdot|$ is a hyperbolic norm.

Now  let $M^*$ be the dual of $M$, $g^*$ the bilinear form induced by $g$ on it  (via the formula $g^*(\omega_1,\omega_2)=g(v_1,v_2)$ for $\omega_i=g(v_i,\cdot)$, $i=1,2$) and $F^*\subset M^*$ be the dual future cone, defined as: $\omega\in F^*$ provided $\omega(v)\geq 0$ for every $v\in F$. The reverse  Cauchy-Schwarz inequality gives that $g^*(\omega,\omega)\geq 0$ for any $\omega\in F^*$, hence the quantity
\[
|\omega|_*:=\sqrt{g^*(\omega,\omega)}
\]
is well defined for any $\omega\in F^*$ and then as above we can check that this is an hyperbolic norm.

The  reverse  Cauchy-Schwarz inequality also offers a variational perspective for the concept of dual norm, as it is immediate to verify that
\[
|\omega|_*=\inf_{v\in F\atop |v|\geq 1} \omega(v).
\]
The right hand side of this formula does not require any bilinear form, but just a hyperbolic norm: this is the observation behind our definition of dual, and more generally operator, norm given above.
}\fr\end{remark}
Proposition \ref{prop:Wstarcone} immediately yields the following:
\begin{proposition}\label{prop:bidual}
Let $\W$ be a wedge so that on $\W^*$ the order relations $\leq$ and $\leq_\p$ agree (this happens for instance if $\W$ is a cone with the directed decomposition property -- recall Theorem \ref{thm:ddpmaps}). Then for any hyperbolic norm $\hn$ on $\W$ we have
\begin{equation}
\label{eq:bidual}
\hn_{**}(\Psi(v))\geq \hn(v)\qquad\forall v\in \W,
\end{equation}
where $\Psi$ is defined in \eqref{eq:defPhi}.
\end{proposition}
\begin{proof}
Just notice that $\hn_{**}(\Psi(v))=\inf_{\hn_*(L)\geq 1}\Psi(v)(L)=\inf_{\hn_*(L) \geq 1}L(v)\geq \hn(v)$.
\end{proof}
At this level of generality we do not know whether the bidual $\W^{**}$ is a cone, nor whether equality holds in \eqref{eq:bidual}. In the standard Banach setting equality in (the analog of) \eqref{eq:bidual} follows from the Hahn-Banach theorem. It therefore comes as no surprise that when we can apply our version of the extension result we can also prove that equality holds:
\begin{proposition}\label{prop:normabiduale}
Let $(\C,\hn )$ be a hyperbolic Banach space with $\C$ having  the directed decomposition property and $\hn $ having the $\mcp$.

Then  equality holds in \eqref{eq:bidual}.
\end{proposition}
\begin{proof}
Notice that by Theorem \ref{thm:ddpmaps} we can apply Proposition \ref{prop:bidual} above, so that the inequality \eqref{eq:bidual} holds and we are left to prove the converse inequality.

Let $\underline v\in \C$. If $\hn(\underline v)=+\infty$ there is nothing to prove, thus we shall assume that $\hn(\underline v)<+\infty$.

Let $\W':=\{\lambda \underline v:\lambda\in[0,+\infty)\}\subset \C$ and define $M:\W'\to[0,+\infty)$ as $M(\lambda \underline v):=\lambda \hn(\underline v)$ (notice that if $\hn(\underline v)=0$, and only in this case if $\hn(v)<+\infty$, it can be that $\lambda \underline v=\eta \underline v$ for $\lambda\neq \eta$, i.e.\ that $\underline v=\eps \underline v$; this does not affect the argument). We want to extend $M$ to a linear map on the whole $\C$ via Theorem \ref{thm:extnuovo}:  pick $\varphi:=\hn$ and $\psi$  to be identically $+\infty$ on $\C\setminus\{0\}$, and 0 in 0. We need to verify that \eqref{eq:assextnuovo} holds, which given our choices means checking that
\[
\lambda \underline v+b\leq\eta \underline v\qquad\Rightarrow\qquad \lambda\hn(\underline v)+\hn(b) \leq\eta\hn(\underline v).
\]
This, however, is trivial, by  the monotonicity of the norm and the reverse triangle inequality.

Thus Theorem \ref{thm:extnuovo} can be applied and  we can extend $M$ to a linear functional, that we shall still denote by $M$, on the whole $\C$ such that
\[
\hn(v)\leq M(v)\qquad\forall v\in \C. 
\]
We now apply item $(i)$ of the existence Theorem \ref{thm:ddpmaps}  with the same choices of $\varphi,\psi$ as above: the assumptions  of Theorem \ref{thm:exthm} are trivially satisfied ($(iii)$ holds with the  choice $\bar v=0$  as   we assumed $\hn$ to have the $\mcp$). 

It follows that ${\sf Pr}(M)$ belongs to $\C^*$ and satisfies $\hn(v)\leq{\sf Pr}(M)(v)$ for every $v\in \C$. By definition of dual norm we deduce $\hn_*({\sf Pr}(M))\geq 1$. The construction also grants that ${\sf Pr}(M)(\underline v)\leq M(\underline v)= \hn(\underline v)$, thus giving
\[
\hn_{**}(\Psi(\underline v))\leq \Psi( \underline v)({\sf Pr}(M))={\sf Pr}(M)(\underline v)\leq\hn(\underline v),
\]
as desired.
\end{proof}
We conclude this short section pointing out that there are interesting examples of hyperbolic Banach spaces for which the norm does \emph{not} possess the $\mcp$. This is the case, for instance, of the $L^p$-norm for $p<0$ that we are going to study in Section \ref{se:lp} below, see Remark \ref{rem:lpmcp}. Still, in this case for any function $f$ with $\|f\|_p>0$ the dominated convergence theorem shows that the map $g\mapsto\|g+f\|_p$ has the $\mcp$: we have stated assumption $(iii)$ in Theorem \ref{thm:exthm} precisely to give a little bit of wiggle room to potentially cover situations like this one.
\end{subsection}

\end{section}

\begin{section}{Further existence results}
\begin{subsection}{An order-theoretic Baire category theorem and the chronological topology}\label{se:Baire}
In order to highlight the similarities between our version of Baire category theorem and the classical one valid on complete metric spaces, let us begin recalling a (very minor) generalization of this well known statement. The generalization aims at highlighting  that the limiting procedure underlying   assumption $(ii)$ below might have nothing to do with the given topology. 

Below we put $\bar B_r(x):=\{y\in\X:\sfd(x,y)\leq r\}$ and call this a $\sfd$-closed ball of radius $r$.
\begin{theorem}[Baire category theorem -- metric version]\label{thm:Baireclassic}
Let $(\X,\sfd)$ be a metric space and $\tau$ a topology on $\X$. Assume that:
\begin{itemize}
\item[i)] For every non-empty $\tau$-open set $U$ there is $x\in U$  such that for all  $r>0$ sufficiently small the $\sfd$-closed ball  $\bar B_r(x)$ is contained in $U$ and has    non-empty $\tau$-interior,
\item[ii)] $(\X,\sfd)$ is complete as metric space. 
\end{itemize}
Then for every sequence $(U_n)$ of $\tau$-dense $\tau$-open sets the set $\cap_nU_n$ is $\tau$-dense. 
\end{theorem}
\begin{proof}
Let $(U_n)$ be as in the statement and $V\subset\X$ be non-empty and $\tau$-open. We shall recursively define a sequence $(B_n)$ of nested $\sfd$-closed balls so that the radius of $B_n$ is $\leq \tfrac1n$ and   $B_n\subset V\cap U_1\cap\cdots\cap U_n$.

By $(i)$ the  non-empty  $\tau$-open set $V\cap U_1$ contains a $\sfd$-closed ball $B_1$ of radius $\leq 1$ with non-empty $\tau$-interior $V_1$. Having defined the $\sfd$-closed ball $B_n$ with non-empty $\tau$-interior $V_n$ satisfying the recursion assumption, we use $(i)$ and the $\tau$-density of $U_{n+1}$ to find a $\sfd$-closed ball $B_{n+1}\subset V_n\cap U_{n+1}$ with radius $\leq\tfrac1{n+1}$ and non-empty $\tau$-interior.

This produces the desired sequence of balls. The centres of these balls form a Cauchy sequence, that by $(ii)$ admits a $\sfd$-limit. Since the balls are $\sfd$-closed, such limit belongs to each of them, and thus to the intersection $\cap_nB_n$. As $\cap_nB_n\subset V\cap(\cap_nU_n)$, by the arbitrariness of $V$ we conclude. 
\end{proof}
We turn to the order theoretic version. Below we put $J(a,b):=\{c\in\X:a\leq c\leq b\}$ and call this a diamond (as opposed to interval: here we are borrowing terminology from the mathematic of General Relativity).
\begin{theorem}[Baire category theorem - order theoretic version]\label{thm:Baire}
Let $(\X,\leq)$ be a partial order and $\tau$ a topology on $\X$. Assume that:
\begin{itemize}
\item[i)] Any non-empty $\tau$-open set $U$ contains a diamond with non-empty $\tau$-interior,
\item[ii)] For any two sequences $(a_n),(b_n)\subset\X$ such that $a_n\leq a_{n+1}\leq b_{n+1}\leq b_n$ for every $n\in\N$ there is $c\in\X$ with $a_n\leq c\leq b_n$ for every $n\in\N$.
\end{itemize}
Then for every sequence $(U_n)$ of $\tau$-dense $\tau$-open sets the set $\cap_nU_n$ is $\tau$-dense. 
\end{theorem}
\begin{proof}
Let $(U_n)$ be as in the statement and   $V\subset\X$ be non-empty and $\tau$-open. We shall recursively define a sequence $(J_n)$ of nested diamonds so that $J_n\subset V\cap U_1\cap\cdots\cap U_n$.  

By $(i)$ the non-empty $\tau$-open set $V\cap U_1$  contains a diamond $J_1:=J(a_1,b_1)$ with non-empty $\tau$-interior $V_1$.  Having defined the diamond  $J_n=J(a_n,b_n)$ with non-empty $\tau$-interior $V_n$ satisfying the recursion assumption, we use $(i)$ and the $\tau$-density of $U_{n+1}$ to find a diamond  $J_{n+1}=J({a_{n+1}},{b_{n+1}})\subset V_n\cap U_{n+1}$  with non-empty $\tau$-interior.

This produces the desired sequence of diamonds. By construction, for every $n\in\N$ we have   $J_{n+1}\subset J_n$, i.e.\ the inequality  $a_n\leq a_{n+1}\leq b_{n+1}\leq b_n$ holds. It follows by $(ii)$ that there is $c\in\X$ with $a_n\leq c\leq b_n$ for every $n\in\N$, i.e.\ such that $c\in \cap_nJ_n$.   As $\cap_nJ_n\subset V\cap(\cap_nU_n)$, by the arbitrariness of $V$ we conclude. 
\end{proof}

\begin{remark}\label{re:Baire}{\rm
Theorem \ref{thm:Baire} implies   Theorem \ref{thm:Baireclassic} and thus also the classical Baire category theorem for complete metric spaces. To see why, let $(\X,\sfd)$ and $\tau$ be as in Theorem \ref{thm:Baireclassic}. On $\R\times\X$ consider the partial order
\begin{equation}
\label{eq:podad}
(t,x)\leq(s,y)\qquad\text{whenever}\qquad \sfd(x,y)\leq s-t
\end{equation}
and the product topology $\tau_\times$ of $\tau$ and the Euclidean topology on $\R$. If $(U_n)$ is a sequence of $\tau$-open dense sets, then $(\R\times U_n)$ is a sequence of $\tau_\times$-open dense sets, so if we show that $\leq$ and $\tau_\times$ satisfy the assumption of Theorem \ref{thm:Baire}, then our claim follows.

Any non-empty $\tau_\times$-open set contains a set of the form $(a,b)\times U$ with $a<b$ and $U\subset\X$ non-empty and $\tau$-open. By assumption $(i)$ in Theorem \ref{thm:Baireclassic} there are  $x\in U$ and $\bar r> 0$ so that $\bar B_r(x)$ is contained in $U$ and has non-empty $\tau$-interior for every $r\in(0,\bar r)$. Then let $[\bar a,\bar b]\subset(a,b)$ be with $\bar b-\bar a\in(0,2\bar r)$. We claim that the diamond $J:=J((\bar a,x),(\bar b,x))$, that is obviously contained in $(a,b)\times U$, has non-empty $\tau_\times$-interior. Indeed, for every $r,s>0$ with $r+s<\tfrac{\bar b-\bar a}2$ we have that $(\tfrac{\bar b+\bar a}{2}-s,\tfrac{\bar b+\bar a}{2}+s)\times \bar B_r(x)$ is contained in $(a,b)\times U$ and since $\bar B_r(x)$ has non-empty $\tau$-interior, the claim is proved. This shows that assumption $(i)$ in Theorem \ref{thm:Baire} holds. 

Assumption $(ii)$ follows from the fact that Cauchy completeness of $(\X,\sfd)$ implies  local completeness of the partial order in \eqref{eq:podad}, see Proposition \ref{prop:2completi} for the simple proof, so the choice $c:=\sup_na_n$ (or $c:=\inf b_n$, that also exists in this case) does the job.

We also notice that the statements above show that the particular metric/order theoretic structure plays little role in the proof: what truly matters is having a class of objects contained/containing open sets and with non-empty countable intersections. It would thus  be possible  to further stretch the statements above to obtain a single theorem implying both the above and that also   covers the case of Baire's theorem for locally compact Hausdorff spaces. Due to the lack of significant examples, we will not do so.
%
}\fr\end{remark}
%

Let us discuss how Theorem \ref{thm:Baire} can be applied to our context. In short, the idea is that the partial order should be that of a cone with joins and the topology   should be the chronological topology induced by the hyperbolic norm.

Notice indeed that  in a hyperbolic Banach space $(\C,\hn)$ we can use the norm to induce a `way below' relation $\ll$ by declaring 
\begin{equation}
\label{eq:waybelow}
v\ll w\qquad\text{ provided  there is $z\in \C$ with $\hn(z)>0$ such that $v+z=w$.}
\end{equation}
It is clear that $v\ll w$ implies $v\leq w$ and that the viceversa does not hold in general (just take  $v=w$ for some $v$ with $\hn(v)<+\infty$). Also, from the reverse triangle inequality for the norm we easily get that  the `push up' property
\begin{equation}
\label{eq:pushup}
\begin{split}
v\leq w\ll z\qquad\Rightarrow\qquad v\ll z,\\
v\ll w\leq z\qquad\Rightarrow\qquad v\ll z.
\end{split}
\end{equation}
In particular, $\ll$ is a transitive. In what follows we shall also use the trivial implications
\begin{equation}
\label{eq:addll}
\begin{split}
v\ll w\qquad&\Rightarrow\qquad \lambda v\ll \lambda w\quad\forall \lambda\in(0,+\infty),\\
v_1\ll w_1,\quad v_2\leq w_2\qquad&\Rightarrow \qquad v_1+v_2\ll w_1+w_2,
\end{split}
\end{equation}
that are direct consequences of the definitions. We also notice that  if $\C$ is a cone with joins then the cancellation property holds also for $\ll$, i.e.\ we have
\begin{equation}
\label{eq:canc2}
a+v\ll b+v \qquad\Rightarrow\qquad a+\eps v\ll b+\eps v.
\end{equation}
Indeed, the assumption tells that $a+v+z=b+v$ for some $z\in \C$ with positive norm, then the standard cancellation \eqref{eq:cancellationtrue} implies $a+\eps v+z=b+\eps v$ and the claim follows.

\medskip

We can then use the chronological relation to induce a topology as follows:
\begin{definition}[Chronological topology]
Let $(\C,\hn)$ be a hyperbolic Banach space. The chronological topology is the order topology generated by $\ll$. 

In other words, it is the one generated by sets of the form $I^+(v):=\{w\in \C:v\ll w\}$ and $I^-(v):=\{w\in \C:w\ll v\}$ as $v$ varies in $\C$.
\end{definition}
From Theorem \ref{thm:Baire} we get  the following:
\begin{theorem}\label{thm:BaireJ}
Let $(\J,\hn)$ be a hyperbolic Banach space with $\J$  being a cone with joins and let $\tau$ be the chronological topology induced by $\hn$. 

Then the intersection of a countable collection of $\tau$-open and $\tau$-dense sets is $\tau$-dense.
\end{theorem}
\begin{proof}
We shall apply Theorem \ref{thm:Baire} with $\X:=\J$, the partial order the one coming with the wedge structure of $\J$ and with the chronological topology $\tau$. Since $(\J,\leq)$ is a directed complete partial order, in fact a a complete lattice, assumption $(ii)$ is trivially satisfied (just take $c:=\sup_na_n$).

Thus we only need to prove that that any non-empty open set contains a diamond with non-empty interior. By definition, a base of the topology is given by sets of the form
\begin{equation}
\label{eq:base}
U\big((v_i)_{i};(w_j)_{j}\big):=\Big(\bigcap_{i\in I}I^+(v_i)\Big)\cap \Big(\bigcap_{j\in J}I^-(w_j)\Big)
\end{equation}
for some finite, possibly empty, collections $(v_i),(w_j)\subset \C$. To conclude it is therefore enough to show that any non-empty set of this form contains a diamond with non-empty interior. Thus let $U=U\big((v_i)_{i};(w_j)_{j}\big)$ be   as in \eqref{eq:base} and non-empty and let 
\[
\begin{split}
v:=\sup_iv_i\qquad\text{ and }\qquad w:=\inf_jw_j.
\end{split}
\] 
We claim that
\begin{equation}
\label{eq:vwi}
v_i\ll w\qquad\forall i\in I\qquad\text{and}\qquad v\ll w_j\quad\forall j\in J.
\end{equation}
Indeed, by assumption there is $\bar z\in U$ and for such $\bar z$ we have   $v_i\ll \bar z\ll  w_j$ for every $i,j$. Taking either the inf in $j$ or the sup in  $i$ yields \eqref{eq:vwi}.

By \eqref{eq:vwi} it follows that  $v\leq w$, hence there is $z\in \J$ such that $v+z=w$. Put  
\begin{equation}
\label{eq:defvwbar}
\bar v:=\tfrac23 v+\tfrac13w=v+\tfrac13z\qquad\text{ and }\qquad \bar w=\tfrac13v+\tfrac23w=v+\tfrac23z.
\end{equation}
By \eqref{eq:vwi} and \eqref{eq:addll} we see that $v_i\ll \bar v$ and $\bar w\ll w_j$ for every $i\in I$ and $j\in J$. Hence $J(\bar v,\bar w)\subset U$.

To conclude we need to prove that $J(\bar v,\bar w)$ has non-empty interior. To this aim,   put 
\[
\tilde v_i:=v_i+\tfrac13z\qquad\text{ and }\qquad \tilde w_j:=w_j-\tfrac13z
\]
(recall \eqref{eq:defdiff} and notice that the $\tilde w_j$'s are well defined  because $\frac13 z\leq z\leq w\leq w_j$ for any $j$). If we prove that   $J({\bar v},{\bar w})\supset \tilde U:=U((\tilde v_i),(\tilde w_j))$ and that $\tilde U$ is not empty we are done.

To see that  $J({\bar v},{\bar w})\supset \tilde U$ it suffices to prove that $\bar v\leq \sup_i\tilde v_i$ and $\bar w\geq\inf_j\tilde w_j$. For the first we   have  $\sup_i\tilde v_i=\sup_iv_i+\frac13z =v+\frac13z =\bar v$. For the second  we start from $ \inf_j\tilde w_j+\tfrac13z =\inf w_j=w=\bar w+\tfrac13z$ to deduce, via the cancellation property \eqref{eq:cancellationtrue}, that $ \inf_j\tilde w_j\leq \bar w+ \eps z\leq\bar w+\eps w$. Since by \eqref{eq:defvwbar} we get $\eps \bar w=\eps w$, we conclude that   $ \inf_j\tilde w_j\leq \bar w$, as desired.

To prove that $\tilde U$ is not empty we shall show that  $\tfrac12v+\tfrac12w\in \tilde U$. For every $i\in I$ we have $v\geq v_i$ and, by \eqref{eq:vwi}, $w\gg v_i$, hence from \eqref{eq:addll} we get
\[
\tfrac12v+\tfrac12w=\tfrac12v+\tfrac13w+\tfrac16w\gg\tfrac12v+\tfrac13w+\tfrac16v_i=\tfrac56v+\tfrac13z+\tfrac16v_i\geq v_i+\tfrac13z=\tilde v_i.
\]
Similarly,    for every $j$ we have $w\leq w_j$ and $v\ll w_j$, thus
\[
\begin{split}
\big(\tfrac12v+\tfrac12w\big)+\tfrac13z=\tfrac16 v+\tfrac56w\ll w_j=\tilde w_j+\tfrac13z
\end{split}
\]
so that from the cancellation property \eqref{eq:canc2} we get $\tfrac12v+\tfrac12w\ll \tilde w_j+\eps z$. To conclude, notice that $\eps z\leq \eps w_j$ and recall \eqref{eq:epsdiff} to get $\eps w_j=\eps\tilde w_j$.
\end{proof}
We collect a couple of comments about the chronological topology:
\begin{itemize}
\item[a)] It tends to be ill-behaved at infinity. Indeed, for $v\in \J$ with $\eps v=v$ and  $\hn(v)=+\infty$ the identity $v+v=v$ implies $v\ll v$ making the singleton $\{v\}$ chronologically open.
\item[b)] In some circumstances it can be the discrete topology even in the finite part of finite dimensional structures. Consider for instance $\J:=[0,+\infty]^2$ equipped with the natural operations and the norm $\hn(a,b):=\|(a,b)\|_p:=(a^p+b^p)^{\frac1p}$ for some $p\in(0,1)$. The discussion in Section \ref{se:lp} shows that this is a hyperbolic norm and it is clear that $\|(a,b)\|_p>0$ for any $(a,b)\neq(0,0)$. 

Thus $I^+((a,b))=\{(a',b')\gneq (a,b)\}$ and $I^-((a,b))=\{(a'',b'')\lneq (a,b)\}$ for any $(a,b)\in[0,+\infty)^2$. Fix  $a,b\in(0,+\infty)$, let $a'\in(a,+\infty)$, $b'\in(b,+\infty)$, $a''\in(0,a')$ and $b''\in(0,b)$  and notice that the set
\[
I^-((a',b))\,\cap \, I^-((a,b'))\,\cap\, I^+((a'',b))\,\cap \,I^+((a,b''))
\]
is the singleton $\{(a,b)\}$, which therefore is open. Analogous arguments are valid if either one, or both, of $a,b$ are 0.

In particular, recalling also point (a) above we see that $J((a,b),(a,b))=\{(a,b)\}$ is open for every $(a,b)\in[0,+\infty]^2$. 

Notice that this tells that  we can have diamonds $J(v,w)$ with non-empty interior even if $v \not\ll w$. For the same reason, in \eqref{eq:vwi} it might be not true that $v\ll w$.
\end{itemize}

\end{subsection}
\begin{subsection}{Comments about `weak$^*$' topologies}\label{se:wstar}

Tichonoff's theorem can be used to endow functional spaces with the compact topology of `pointwise convergence', a fact that is particularly relevant in presence of algebraic constraints, as these force some rigidity and prevent oscillations, making the topology useful. This is the essence of  the classical Banach-Alaoglu theorem.

In this short section we collect some comments about the effect of such principle in our context: the content here is by no means satisfactory or complete, but should rather be interpreted as source of stimulus for potential further investigations.

The   crucial additional difficulty in our setting compared to   the classical one is the following: in the classic framework  continuity can be quantified and the pointwise limit of uniformly continuous (=uniformly bounded) linear functional is still continuous. In here, instead,   the $\mcp$ is not quantifiable and certainly is not stable by pointwise convergence (see Example \ref{ex:mcpnostable} below). We stress, in particular, that the $\mcp$ has nothing to do with bounds on the (hyperbolic) norm of a functional.

Thus here we limit ourselves to noticing the following: the abstract procedure behind Banach-Alaoglu's result can be put in place to endow the collection $\Lin(\C,[0,+\infty])$ of linear maps from the cone/wedge $\C$ to $[0,+\infty]$ with a compact, possibly Hausdorff, topology. Then if  the cone $\C$ has the directed decomposition property we have a natural projection map  from $\Lin(\C,[0,+\infty])$ to $\C^*=\Hom(\C,[0,+\infty])$ that is surjective (recall Proposition \ref{prop:linpr}), so we can endow $\C^*$ with the quotient topology, that will therefore be compact (but the Hausdorff property is most likely lost).

\medskip

We give some details. Let $\tau_{\sf Eu}$  be the standard Euclidean topology on $[0,+\infty]$ and $\tau_{\sf lsc}$ the topology  generated by sets of the form $(a,+\infty]$ for $a\in[0,+\infty]$ (notice that $[0,+\infty]$-valued maps that are $\tau_{\sf lsc}$ continuous are precisely the lower semicontinuous functions). These are both compact, with $\tau_{\sf Eu}$ being also Hausdorff. Below we shall let $\tau$ be any of these two.

Let now $\C$ be a wedge and let us endow the set $[0,+\infty]^\C$ of functions from $\C$ to $[0,+\infty]$ with the product topology: this only requires the topology $\tau=\tau_{\sf Eu},\tau_{\sf lsc}$ on the target space and by Tichonoff's theorem it is compact as $\tau$ is so (and trivially also Hausdorff if $\tau=\tau_{\sf Eu}$).

Regardless of the choice of $\tau=\tau_{\sf Eu},\tau_{\sf lsc}$  we have:
\begin{equation}
\label{eq:linchiuso}
\Lin(\C,[0,+\infty])\quad\text{ is a closed subset of }\quad [0,+\infty]^\C.
\end{equation}
Indeed, we have
\[
\Lin(\C,[0,+\infty])=\bigcap_{\lambda,\eta\in[0,+\infty),\atop v,w\in \C}\big\{f\in [0,+\infty]^\C\ :\ f(\lambda v+\eta w)=\lambda f(v)+\eta f(w)\big\}
\]
so that, as in the classical case, the claimed closure follows noticing that for every $\lambda,\eta,v,w$ as above, the map $f\mapsto  f(\lambda v+\eta w)$ is continuous in the product topology and, using that both the sum and the multiplication by a non-negative scalar are continuous on $([0,+\infty],\tau)$, also $f\mapsto  \lambda f( v)+\eta f(w)\in [0,+\infty]$ is continuous. The claim \eqref{eq:linchiuso} follows. We thus obtain  that $\Lin(\C,[0,+\infty])$ is compact (and Hausdorff we chose $\tau=\tau_{\sf Eu}$).

If $\C$ is a cone with the directed decomposition property, then by Proposition \ref{prop:linpr} the projection map  ${\sf Pr}$ (recall Definition \ref{def:prmcp})  sends $\Lin(\C,[0,+\infty])$ into $\C^*$ and  is clearly surjective. We can therefore endow $\C^*$ with the quotient topology induced by such projection and the topology we had on  $\Lin(\C,[0,+\infty])$ (i.e.\ the strongest one making the projection continuous): we shall call this resulting topology $w^*_{\sf Eu}$ if we started from $\tau=\tau_{\sf Eu}$   and  $w^*_{\sf lsc}$ is we instead started from $\tau_{\sf lsc}$. Since the quotient of a compact topology is itself compact, both $w^*_{\sf Eu}$ and $w^*_{\sf lsc}$ are compact topologies on $\C^*$. It is unclear to us under which conditions $w^*_{\sf Eu}$ is also Hausdorff but in any case thinking at the case of $L^p-L^q$ duality and thus to the cone $\LL(\X)$ discussed in Section \ref{se:lp}, Fatou's lemma seems to indicate that the appropriate weak$^*$ topology in this setting should be $w^*_{\sf lsc}$, which is not  Hausdorff in any reasonable case (e.g.\ if $\C=[0,+\infty]$ then quite clearly $\C^*\sim [0,+\infty]$ and via such isomorphism  $w^*_{\sf lsc}$ is $\tau_{\sf lsc}$).

\begin{example}\label{ex:mcpnostable}{\rm
Consider the cone $\LL(\X)$ for $\X:=[0,1]$ with the Lebesgue measure (described in Section \ref{se:lp}, where also the $L^p$ norms mentioned below are discussed). For every $n\in\N$ consider the functions $f_n:=1+\infty\nchi_{[0,\frac1n]}$ and let $L_n:\LL(X)\to\LL(\X)$, $n\in\N$, be the functionals induced by them, i.e.\ $L_n(g):=\int f_ng\,\d\mathcal L^1$.

Then clearly $L_n\in \LL(\X)^*$ and for every $g\in\LL(\X)$ we have $L_n(g)\to L(g)$, where  $L\in\Lin(\LL(\X),[0,+\infty])$ is defined as
\[
L(g):=\left\{
\begin{array}{ll}
\displaystyle{\int g\,\d\mathcal L^1},&\quad\text{ if there is $\eps>0$ so that $g=0$ a.e.\ on $[0,\eps]$},\\
+\infty,&\quad\text{ otherwise}.
\end{array}
\right.
\]
It is clear that $L$ is linear and that it does not have the $\mcp$: the increasing sequence $g_n:=\nchi_{[\frac1n,1]}$ converges to $g_\infty:=\nchi_{[0,1]}$ a.e., we have $L(g_n)=1-\tfrac1n$ but $L(g_\infty)=+\infty$.

Notice also that for $q<0$ we have $\|f_n\|_{q}=(\int_{\frac1n}^11\,\d\mathcal L^1)^{\frac1q}=(1-\frac1n)^{\frac1q}$ so that $\|f_n\|_q\in[1,2^{\frac1q}]$ for every $n\geq 2$, showing that the problem with this example is not that the operator norms of the $L_n$'s are too close to 0 or infinity.
}\fr\end{example}

\end{subsection}

\end{section}

\begin{section}{Examples}
\begin{subsection}{$L^p$ spaces for $p\in[-\infty,1]$}\label{se:lp}

Let $(\X,\mathcal A,\mm)$ be a $\sigma$-finite measured space and $\LL(\X)$ be defined as
\[
\LL(\X):=\big\{f:\X\to[0,+\infty]\ :\ \mathcal A-\text{measurable}\big\}/\sim
\]
 where $f\sim g$ whenever $f=g$ $\mm$-a.e.. With a, rather standard, abuse of notation, we shall think at elements of $\LL(\X)$ as functions defined $\mm$-a.e., rather than as equivalence classes of functions.

It is clear that $\LL(\X)$ is a wedge w.r.t.\ the natural pointwise a.e.\ defined operation of sum and product by a positive scalar. The order relation $\leq$ related to such wedge structure obviously satisfies
 \[
 g\leq f\qquad\text{ if and only if }\qquad g(x)\leq f(x)\quad \mm-a.e.\ x.
 \]
 It is a well-known and easy prove fact that $(\LL(\X),\leq)$ is a complete lattice, the sup and inf w.r.t.\ this order being called essential supremum and essential infimum of the given family, see for instance \cite[Theorem 4.7.1]{Bogachev07}. As we are going to need such completeness and some further properties of this order, we give the short proof:
 \begin{lemma}[Essential sup/inf]\label{le:esssup}
 Let $A\subset \LL(\X)$ be a subset. Then there is a unique $f\in\LL(\X)$, called $\mm$-essential supremum of $A$ (or simply essential supremum), with the following properties:
 \begin{itemize}
 \item[i)] for any $g\in A$ we have $g\leq f$ $\mm$-a.e.,
 \item[ii)] if $h\in \LL(\X)$ is such that $g\leq h$ $\mm$-a.e.\ for every $g\in A$, then $f\leq h$ $\mm$-a.e..
 \end{itemize}
Moreover, there is an at most countable subset $(g_n)\subset A$ such that $f(x)=\sup_ng_n(x)$ for $\mm$-a.e.\ $x\in\X$. Analogously for the infimum.

In particular, $\LL(\X)$ is first countable as partial order in the sense of Definition \ref{def:sep}.
 \end{lemma}
 \begin{proof} The uniqueness is obvious, so we turn to existence. Since $\mm$ is $\sigma$-finite, there is a probability measure $\mm'$ on $\mathcal A$ with the same negligible sets: replacing $\mm$ with $\mm'$ we do not affect the statement and thus can assume that the total mass is 1. Also, post-composing all the functions in $A$ with the same order isomorphism $\varphi$ of $[0,+\infty]$ and $[0,1]$ we can assume that the functions in $A$ take values in $[0,1]$ (the original supremum will be found post-composing the supremum of the newly built family with $\varphi^{-1}$). For $I\subset A$ finite, define $g_I(x):=\max_{g\in I}g(x)$, let $B\subset\LL(\X)$ be the collection of functions  obtained this way and then put $s:=\sup_{g\in B}\int g\,\d\mm\in[0,1]$.

Obviously, there is a a countable collection $(\tilde g_k)\subset B$ such that $s=\sup_k\int \tilde g_k\,\d\mm$ and thus there is a countable subset $(g_n)$ of $A$ such that $s=\sup_n\int g_1\vee\cdots \vee g_n\,\d\mm$.  Define $f\in\LL(\X)$ as $f(x):=\sup_ng_n(x)$ for $\mm$-a.e.\ $x\in\X$. If we prove that such $f$ has properties $(i),(ii)$ we are done. $(ii)$ is clear, because any $h$ as in the statement is $\geq g_n$ $\mm$-a.e.\ for each $n\in\N$, and thus $\geq f$ $\mm$-a.e.. For $(i)$ we argue by contradiction: suppose that there is $g\in A$ such that $\mm(\{g>f\})>0$. Then by Beppo Levi's monotone convergence theorem we get
\[
\lim_n\int g\vee g_1\vee\cdots \vee g_n\,\d\mm=\int g\vee f\,\d\mm\stackrel*>\int f\,\d\mm=\lim_n\int g_1\vee\cdots \vee g_n\,\d\mm=s,
\]
where in the starred inequality we used that $\mm(\X)<\infty$ and $f,g\leq 1$ $\mm$-a.e.. Thus for some $n\in\N$ we have $\int g\vee g_1\vee\cdots \vee g_n\,\d\mm>s$ and since $ g\vee g_1\vee\cdots \vee g_n\in B$ we found the desired contradiction.

For the last claim let $D\subset\LL(\X)$ be directed and $(g_n)\subset D$ be the countable collection with the same supremum. Define recursively a non-decreasing sequence $n\mapsto f_n\in D$ by putting $f_0:=g_0$ and then recursively finding $f_n\in D$ that is $\mm$-a.e.\ $\geq f_{n-1}\vee g_n$. The fact that $D$ is directed ensures that $f_n$ exists and the construction grants that $(f_n)$ has the same supremum of $D$. 
 \end{proof}
It is not hard to check that $\LL(\X)$ has the directed  decomposition property (in fact in its more general formulation as in the statement of Proposition \ref{prop:tdptrunc}). Indeed, let $\mathcal F\subset\LL(\X)$ be a  set with tip with $\tip \mathcal F=f=g+h$. To each $f'\in\mathcal F$ we associate the functions $G(f'),H(f')\in\LL(\X)$ defined as follows:
\begin{equation}
\label{eq:GH}
\begin{split}
\big(G(f'),H(f')\big):=\left\{
\begin{array}{lll}
\big(g\wedge f',&f'-f'\wedge g\big),&\qquad \text{on }\{g<+\infty\}\cap\{h<+\infty\},\\
\big(g\wedge f',&f'-f'\wedge g\big),&\qquad \text{on }\{g<+\infty\}\cap\{h=+\infty\},\\
\big(f'-f'\wedge h,&h\wedge f'\big),&\qquad \text{on }\{g=+\infty\}\cap\{h<+\infty\},\\
\big( f'/2,&f'/2\big),&\qquad \text{on }\{g=+\infty\}\cap\{h=+\infty\},
\end{array}
\right.
\end{split}
\end{equation}
being intended that the term $f'-f'\wedge g$ is equal to $+\infty$ if $f'=+\infty$, and similarly for $f'-f'\wedge h$. It is then immediate to check, e.g.\ by transfinite recursion, that   $G(f')+H(f')=f'$ and that the families $\mathcal G:=\{G(f'):f'\in\mathcal F\}$ and $\mathcal H:=\{H(f'):f'\in\mathcal F\}$ have tips, with tips $g,h$ respectively. Thus the claim is proved.

Directly from the definition we also see that for  $f\in\LL(\X)$ its  part at infinity  $\eps f\in\LL(\X)$ is given by
\begin{equation}
\label{eq:epslx}
(\eps f)(x)=\eps(f(x))=\left\{
\begin{array}{ll}
0,&\qquad\text{ if }f(x)<+\infty,\\
+\infty,&\qquad\text{ if }f(x)=+\infty,
\end{array}\right.\qquad\mm-a.e.\ x\in\X.
\end{equation}

A function $f\in\LL(\X)$ naturally induces a map $L_f:\LL(\X)\to[0,\infty]$ via the formula
\begin{equation}
\label{eq:dualL}
L_f(g):=\int fg\,\d \mm,
\end{equation}
 where the positivity of $f,g$ ensures that the integral is well-defined, possibly equal to $+\infty$. We have:
\begin{proposition}[$\LL(\X)$ as dual of $\LL(\X)$]\label{prop:lxlx}
For any $f\in \LL(\X)$ the map $L_f$ defined in \eqref{eq:dualL} belongs to $\LL(\X)^*$. Moreover, the map $\LL(\X)\ni f\mapsto L_f\in \LL(\X)^*$ is a linear bijection and an order isomorphism.
\end{proposition}
\begin{proof}
The linearity of $L_f$ is obvious from basic integration theory. To see that  $L_f$ has the $\mcp$ we apply Proposition \ref{prop:sepmcp}, that can be done thanks to the last claim in Lemma \ref{le:esssup} above: then the validity of \eqref{eq:mcpsep} is, obviously, nothing but  Beppo  Levi's monotone convergence theorem.  Viceversa, let $L\in\LL(\X)^*$ and consider the map $\mathcal A\ni A\mapsto \mu(A):=L(\nchi_A)\in[0,+\infty]$. This is clearly additive on disjoint sets and, by the $\mcp$ of $L$, also $\sigma$-additive. Hence it is a measure and by definition it is also absolutely continuous w.r.t.\ $\mm$. Let $\mathcal F\subset\mathcal A$ be the collection of sets $F$ such that $\mu(F)<\infty$ and let $\bar F$ be the $\mm$-essential union of these sets (in other words, consider the essential supremum of the functions $\nchi_F$ as $F$ varies in $\mathcal F$, notice -- e.g.\ by the last claim in Lemma \ref{le:esssup} -- that this is the characteristic function of a set and let $\bar F$ be that set, well defined up to $\mm$-negligible sets). By definition and the last claim in Lemma \ref{le:esssup}, the restriction of $\mu$ to $\bar F$ is $\sigma$-finite, hence by the Radon-Nikodym theorem there is an $\mathcal A$-measurable function $f:\bar F\to[0,\infty)$ such that $\mu(A)=\int f\,\d\mm$ for every $\mathcal A$-measurable $A\subset\bar F$. We extend $f$ to the whole $\X$ by declaring it to be $+\infty$ on $\X\setminus\bar F$: it is clear from the definition and linearity of $L$ that $L(g)=\int fg\,\d\mm$ holds for any $g$ attaining only a finite number of values, then the conclusion for arbitrary $g\in\LL(\X)$ follows from the $\mcp$ of $L$, again Beppo Levi's theorem and the fact that functions with finite range are dense in $\LL(\X)$, as it is easy to check. From the construction it is also clear that such $f$ is $\mm$-a.e.\ unique and the claim that $f\mapsto L_f$ is a linear  order isomorphism is obvious.
\end{proof}
In what follows we shall define $0^p$ to be 0 if $p\in(0,1)$ and $+\infty$ if $p<0$. Analogously, we shall put $\infty^p=\infty$ for $p\in(0,1)$ and $\infty^p=0$ if $p<0$.
\begin{definition}[$L^p$ norms for $p\in {[-\infty,1]}\setminus \{0\}$]
For $p\in[-\infty,1]\setminus \{0\}$ and $f\in\LL(\X)$ we define
\[
\|f\|_p:=\Big(\int f^p\,\d\mm\Big)^{\frac1p}.
\]
We also put
\[
\|f\|_{-\infty}:=\essinf f.
\]
\end{definition}
The conventions just established grant the well-posedness of the definition. We have:
\begin{proposition}[$L^p$-$L^q$ duality for $p,q \in {[-\infty,1]}\setminus \{0\}$]\label{prop:lplq} For $p\in [-\infty,1]\setminus \{0\}$ the $\|\cdot\|_p$-norm is a hyperbolic norm on $\LL(\X)$.  Moreover, let  $q\in [-\infty,1]\setminus \{0\}$ be so that  $\tfrac1p+\tfrac1q=1$.  Then
\begin{equation}
\label{eq:revholder}
\int fg\,\d\mm\geq \|f\|_p\|g\|_q\qquad\forall f,g\in \LL(\X)
\end{equation}
and more precisely
\begin{equation}
\label{eq:lpdual}
\|f\|_p=\inf_{g\in\LL(\X),\ \|g\|_q\geq1} \int fg\,\d\mm\qquad \forall f\in\LL(\X).
\end{equation}
In other words,  the map $(\LL(\X),\|\cdot\|_q)\ni f\mapsto L_f\in (\LL(\X)^*,\|\cdot\|_p^*)$ defined in \eqref{eq:dualL} is norm-preserving.
\end{proposition}
\begin{proof}
The case $p,q\in (-\infty,1)\setminus \{0\}$ is essentially the content of \cite[Lemma 2.13]{Octet}, phrased with the terminology introduced here. We thus discuss the remaining cases.

The fact that $\|\cdot\|_1$ and $\|\cdot\|_{-\infty}$ are hyperbolic norms is obvious, and so is  the inequality \eqref{eq:revholder} in this case. This shows inequality $\leq$ in \eqref{eq:lpdual}.  The choice $g\equiv 1$ proves equality for $p=1$. For the case $p=-\infty$ we fix  $\eps>0$ and use the definition of $\essinf$ to find $A\in\mathcal A$ with $\mm(A)\in(0,\infty)$ so that $g\leq \|g\|_{-\infty}+\eps$ $\mm$-a.e.\ on $A$. Then  $f:=\mm(A)^{-1}\nchi_A$ satisfies $\|f\|_1=1$ and we have
\[
  \int fg\,\d\mm=\mm(A)^{-1}\int_Ag\,\d\mm\leq  \|g\|_{-\infty}+\eps.
\]
The conclusion follows from the arbitrariness of $\eps>0$.
\end{proof}
\begin{remark}[Comments on the axiomatization chosen]\label{re:perchenocanc}{\rm The study of $L^p-L^q$-duality for $p,q<1$ is an important source of justifcation for the axiomatization chosen in this manuscript. In particular as to why we:
\begin{itemize}
\item[-] Opted for wedges/cones rather than vector spaces;
\item[-] Did not impose the cancellation property;
\item[-] Did not insist on hyperbolic norms to be finite;
\item[-] Considered directed/forward completeness rather than filtered/backward one.
\end{itemize}
To see why, let us consider, as possibile alternative, the vector space $\tilde L^p(\X):=\{f:\X\to\R:\int|f|^p\,\d\mm<\infty\}$ for, say, $\X:=[0,1]$ with the Lebesgue measure, where here as usual we identify functions agreeing a.e.. Then the function $g\equiv 1$ belongs to $\tilde L^p([0,1])$ for every $p<1$ and we might consider the functional $L_g$ as in \eqref{eq:dualL} on $\tilde L^p([0,1])$.

A first problem in doing so is that such functional is \emph{not} well defined on the whole $\tilde L^p$. Indeed, the finiteness assumption $\int_0^1|f|^p<+\infty$ for $p<1$ surely does not imply $f\in L^1([0,1])$ and thus $L_g(f)=\int_0^1f$ might be not well defined. Here the problem is that we only have at disposal the reverse H\"older inequality \eqref{eq:revholder} and thus cannot bound from above $\int|f||g|\,\d\mm$ in terms of $\||f|\|_p$ and $\||g|\|_q$ only. In fact, as just seen, we can have $\int fg\,\d\mm=+\infty$ for $f,g\geq0$ with $\|f\|_p,\|g\|_q<+\infty$ and $\tfrac1p+\tfrac1q=1$.

\medskip

This already pushes toward the `one-sided theory' of wedges as opposed to vector spaces and shows that the reasons for this is in the fact that the image of a morphism can include the value $+\infty$. In particular, we should not expect the cancellation property to hold on the image of morphisms, making unnatural to impose it on sources of morphisms as well. Notice also that in these examples having non-zero part at infinity (which is the same as not being cancellable and equivalent, by \eqref{eq:epslx}, to attaining the value $+\infty$ on a set of positive measure) does not imply having infinite norm: the function $f\in\LL([0,1])$ that is equal to $1$ on $[0,1/2]$ and to $+\infty$ on $[1/2,1]$ has finite $L^p$-norm for every $p<0$. This makes even more unnatural to exclude functions attaining the value $+\infty$ from the analysis.

\medskip

Once the choices of working with wedges/cones without imposing the cancellation  property are made, it seems there is no good reason to impose hyperbolic norm to be finite: no clear advantage would emerge from such assumption, while disadvantages are clear. Consider for instance the self-duality of $\LL(\X)$ discussed in Proposition \ref{prop:lxlx}  and the function $f=\chi_{[0,1/2]}+\infty\nchi_{[1/2,1]}$ considered above. Then $\|f\|_p=+\infty$ for $p\in(0,1)$ and $\|f\|_p<+\infty$ for $p<0$. Restricting to functions with finite norm would exclude such $f$ from $L^p$-cones for $p\in(0,1)$ and include it in those with $p<0$, despite the fact that the functional $L_f$ it induces on $\LL([0,1])$ is independent on $p$. It is clear from  Proposition \ref{prop:lplq} that the difference can be also be read via the fact that $L_f$ has infinite norm when seen as element of $(\LL(\X),\|\cdot\|_p)^*$ for $p<0$ and finite norm when seen as element of  $(\LL(\X),\|\cdot\|_p)^*$ for $p\in(0,1)$, but the advantages of restricting to operators of finite norm are not clear, either. On the other hand, removing them from the analysis would force to give up directed completeness in favour of the weaker local completeness (=directed subset with an upper bound have supremum): we question the advantage of such trade off. 

\medskip

It remains to comment on why we opted for forward, as opposed to backward, completeness: since $\LL(\X)$ is, from the order theoretical perspective, a complete lattice, in principle we have no hints in either direction. Yet, again, things get clearer if we look at morphisms. We mentioned in the introduction that out of $[0,+\infty]$ and $[-\infty,0]$ only the first one is a reasonable choice of `field of scalar', as it is the only one of the two closed by products. Then the study of self-duality of $\LL(\X)$ together with the fact that when $\X$ is a singleton we aim at getting back the `field of scalars' led to define $\LL(\X)$ as the cone of $[0,+\infty]$-valued, as opposed to $[-\infty,0]$-valued, functions.

$\LL(\X)$ and $[0,+\infty]$ are both complete lattices. What breaks the symmetry between forward and backward completeness is, ultimately, Beppo Levi's Monotone Convergence Theorem: functions $f\in\LL(\X)$ induce functionals $L_f$ on $\LL(\X)$ that respect suprema of increasing sequences, not infima of decreasing ones. We are therefore led to the notion of functional with the $\MCP$, which in turn is naturally linked, as made clear by Definition \ref{def:completamento}, to the notion of directed complete partial order.

One might question why we insisted on the full $\MCP$ as opposed to the sequential version discussed in Section \ref{se:seqcompl}, and thus why insist on full directed completeness rather than just the sequential counterpart. On first countable spaces as in Definition \ref{def:sep} the two versions of $\MCP$ are equivalent, thus perhaps this distinction is not so relevant as we expect most interesting examples to be first countable. Still, from the perspective of a general theory the continuity associated to the `full $\MCP$' is surely more natural. 
}\fr\end{remark}

\begin{remark}[Lack of $\mcp$ for the $L^p$-norm, $p<0$]\label{rem:lpmcp}{\rm

For $p>0$ a direct application of Beppo Levi's theorem and of the fact that $\LL(\X)$ is first countable show that $\|\cdot\|$ has the $\mcp$.

This fails for $p<0$. Consider for instance the case $\X:=[0,1]$ with the Lebesgue measure and let $f_n:=\nchi_{[\frac1n,1]}$. Then $(f_n)$ is an increasing sequence in $\LL(\X)$ with supremum  $f_\infty:=\nchi_{[0,1]}$, but a direct computation shows that $\|f_n\|_p=0$ for every $n\in\N$ and $p<0$, while $\|f_\infty\|_p=1$.

In relation with the assumption $(iii)$ in Theorem \ref{thm:exthm}, we notice that the dominated convergence theorem grants that, for $p<0$, if $f\in\LL(\X)$ is so that $\|f\|_p>0$, then $g\mapsto \|g+f\|_p$ has the $\mcp$.
}\fr\end{remark}

From now on we shall assume $\mm(\X)=1$.
A simple formal computation, valid e.g.\ if $f$ is bounded, shows that 
\[
\lim_{p\to0}\|f\|_p=\exp\Big(\int \log(f)\,\d\mm\Big),
\]
indicating that the $L^0$ norm of $f$ should be its geometric mean, as  given by the right hand side above. The actual definition needs a bit of care, because $\int\log(f)\,\d\mm$ is in general not well defined for $f:\X\to[0,+\infty]$.

 To give the correct notion, let us recall the notation  $z^+:= z\vee 0$ and $z^-:=(-z)\vee 0$ valid for $z\in[-\infty,+\infty]$ and that  for $\varphi:\X\to[-\infty,+\infty]$ measurable the integral $\int \varphi\,\d\mm$ is well-defined (possibly with value $\pm\infty$) provided at most one of $\int\varphi^+\,\d\mm$ and $\int\varphi^-\,\d\mm$ is equal to $+\infty$: in this case $\int\varphi\,\d\mm$ is set to be $\int\varphi^+\,\d\mm-\int\varphi^-\,\d\mm$. 

Then we define $\int_+\varphi\,\d\mm$ and $\int_-\varphi\,\d\mm$ as
\[
\begin{split}
\int_+\varphi\,\d\mm&=\left\{
\begin{array}{ll}
\int\varphi\,\d\mm,&\ \text{if well defined,}\\
+\infty,&\ \text{otherwise,}
\end{array}
\right.
\qquad\qquad
\int_-\varphi\,\d\mm=\left\{
\begin{array}{ll}
\int\varphi\,\d\mm,&\ \text{if well defined,}\\
-\infty,&\ \text{otherwise.}
\end{array}
\right.
\end{split}
\]
Notice that 
\begin{equation}
\label{eq:intpmlim}
\int_+\varphi\,\d\mm=\lim_{c\downarrow-\infty}\int \varphi\vee c\,\d\mm\qquad\text{ and }\qquad\int_-\varphi\,\d\mm=\lim_{c\uparrow+\infty}\int \varphi\wedge c\,\d\mm.
\end{equation}
\begin{definition}[$L^0$ norms]
Suppose $\mm(\X)=1$. Then for $f\in\LL(\X)$ define
\begin{equation}
\label{eq:norm0}
\begin{split}
\|f\|_{0^+}:=\exp\Big(\int_+\log(f)\,\d\mm\Big),\qquad\qquad
\|f\|_{0^-}:=\exp\Big(\int_-\log(f)\,\d\mm\Big),
\end{split}
\end{equation}
where here and below $\log(0)=-\infty$, $\log(+\infty)=+\infty$, $\exp(-\infty)=0$ and $\exp(+\infty)=+\infty$.
\end{definition}
\begin{proposition}[$L^{0^+}\!\!-L^{0^-}$ duality]\label{prop:l0}
Let $\mm(\X)=1$. Then $\|\cdot\|_{0^+}$ and $\|\cdot\|_{0^-}$ are hyperbolic norms on $\LL(\X)$. Moreover
\begin{equation}
\label{eq:l01}
\int fg\,\d\mm\geq\|f\|_{0^+}\|g\|_{0^-},\qquad\forall f,g\in\LL(\X)
\end{equation}
and more precisely
\begin{equation}
\label{eq:duall0}
\|f\|_{0^\pm}=\inf_{g\in\LL(\X),\ \|g\|_{0^\mp}\geq1} \int fg\,\d\mm,\qquad\forall f\in\LL(\X).
\end{equation}
Finally, we have
\begin{equation}
\label{eq:0p0mnew}
\|f\|_{0^\pm}=\frac1{\|\tfrac1f\|_{0^\mp}},
\end{equation}
where here and below we set $\tfrac10=+\infty$ and $\tfrac1{+\infty}=0$  (notice that then ${1/\tfrac1z}=z$ and $z\cdot\tfrac1z\leq 1$ for all $z\in[0,+\infty]$). 
\end{proposition}
\begin{proof}
We start proving \eqref{eq:l01}. To see this, notice that from the first in \eqref{eq:intpmlim} and the classical Jensen inequality we get
\[
\log\Big(\int fg\,\d\mm \Big)\geq \int_+\log(fg)\,\d\mm.
\]
Now let us agree from here to the end of the proof that $(+\infty)+(-\infty)=(-\infty)+(+\infty)=-\infty$. Notice that together with $0\cdot(+\infty)=(+\infty)\cdot0=0$, this makes the identities $\log(ab)=\log(a)+\log (b)$ and $e^{z+w}=e^ze^w$ valid for any $a,b\in[0,+\infty]$ and $z,w\in[-\infty,+\infty]$, so that from the above we see that to prove \eqref{eq:l01} it suffices to prove that
\begin{equation}
\label{eq:pmint}
 \int_+\log(f)+\log(g)\,\d\mm\geq  \int_+\log(f)\,\d\mm+ \int_-\log(g)\,\d\mm.
\end{equation}
If either $  \int_+\log(f)\,\d\mm$ or $  \int_-\log(g)\,\d\mm$ is equal to $-\infty$ the claim holds trivially by the convention just introduced, so we can assume that both are $>-\infty$. In particular $\int \log(g)^-\,\d\mm<+\infty$, so if we also have $\int \log(f)^-\,\d\mm<+\infty$ then all the integrals in \eqref{eq:pmint} are well defined and equality holds. It remains to discuss the case  $\int \log(g)^-\,\d\mm<+\infty$ and  $\int \log(f)^\pm\,\d\mm=+\infty$: in this case we have $\log(g)>-\infty$ $\mm$-a.e.\ and thus, as it is easy to check, $\mm$-a.e.\ we have $\log(f)^+-\log(g)^-\leq (\log(f)+\log(g))^+$ which implies $\int (\log(f)+\log(g))^+\,\d\mm=+\infty$ and the conclusion \eqref{eq:pmint}, and thus also \eqref{eq:l01}, follows. 

We now prove \eqref{eq:0p0mnew} and put for brevity   $g:=\tfrac1f$. Then  $\log(g)=-\log (f)$, thus  if $\|f\|_{0^+}\in(0,+\infty)$, then \eqref{eq:0p0mnew}  is trivial. Suppose now that $\|f\|_{0^+}=0$. Then we have $\int_+\log(f)\,\d\mm=-\infty$ hence $\int(\log(f))^-\,\d\mm=+\infty$ and  $\int(\log(f))^+\,\d\mm<+\infty$ which implies $\int(\log(g))^+\,\d\mm=+\infty$ and  $\int(\log(g))^-\,\d\mm<+\infty$ that in turn gives $\int_-\log(g)\,\d\mm=+\infty$, i.e.\ $\|g\|_{0^-}=+\infty$, so that \eqref{eq:0p0mnew} holds also in this case. The case $\|f\|_{0^+}=+\infty$ is analog: in this case we have $\int_+\log(f)\,\d\mm=+\infty$, thus $\int(\log(f))^+\,\d\mm=+\infty$ hence $\int(\log(g))^-\,\d\mm=+\infty$ that implies $\int_-\log(g)\,\d\mm=-\infty$, i.e.\ $\|g\|_{0^-}=0$. Thus  \eqref{eq:0p0mnew} is proved.

We now prove  \eqref{eq:duall0} and notice that \eqref{eq:l01} gives the inequality $\leq$. We turn to equality; we shall deal with the left hand side being $\|f\|_{0^+}$, the case $\|f\|_{0^-}$ being handled analogously. If $\|f\|_{0^+}=+\infty$ there is nothing to prove. If $\|f\|_{0^+}\in(0,+\infty)$ then the choice $g:=\tfrac{\|f\|_{0^+}}{f}$, the identity \eqref{eq:0p0mnew} and the (trivial) positive homogeneity of the norms give the conclusion (notice that $fg\leq \|f\|_{0^+}$). Finally, if $\|f\|_{0^+}=0$, then for every $c>0$ the function $g:=\tfrac c{f}$ satisfies, by \eqref{eq:l01}, $\|g\|_{0^-}=+\infty$ and, by definition, $fg\leq c$ $\mm$-a.e.. Integrating and choosing $c>0$ arbitrarily small we see that also in this case \eqref{eq:duall0} holds.
 
It remains to prove that $\|\cdot\|_{0^+}$ and  $\|\cdot\|_{0^-}$ are hyperbolic norms. Positive homogeneity is trivial. Superadditivity follows from the duality formula \eqref{eq:duall0}, as we have
\[
\begin{split}
\|f_0+f_1\|_{0^+}&=\inf_{g:\|g\|_{0^-}\geq1}\int (f_1+f_2)g\,\d\mm\\
&=\inf_{g:\|g\|_{0^-}\geq1}\Big(\int f_1 g\,\d\mm+\int f_2g\,\d\mm\Big)\\
&\geq\inf_{g:\|g\|_{0^-}\geq1}\int f_1 g\,\d\mm+\inf_{g:\|g\|_{0^-}\geq1}\int f_2g\,\d\mm=\|f_0\|_{0^+}+\|f_1\|_{0^+}.
\end{split}
\]  
Similarly for $\|\cdot\|_{0^-}$.
\end{proof}

\begin{remark}[Self-duality and rigidity]{\rm It is worth to comment the above duality results from the viewpoint of Fenchel transform: this highlights similarities and differences with the classical `elliptic' case.

For  $u:\R\to\R\cup\{\pm\infty\}$ convex define the Legendre-Fenchel transform as
\begin{equation}
\label{eq:dualconv}
u^*(w):=\sup_{z\in\R}wz-u(z).
\end{equation}
The standard $L^p-L^q$ duality for $p,q>1$ is related to the fact that for 
\begin{equation}
\label{eq:up}
u_p(z):=\tfrac1p z^p
\end{equation}
we have $u_p^*=u_q$, where $\tfrac1p+\tfrac1q=1$. It is also well-known  that  the only convex function $u$ such that $u=u^*$ is $u(z)=\frac{z^2}2$. Indeed, $\tfrac{z^2}{2}$ is clearly self-dual in this sense, and if $u$ has such property plugging $z=w$ in the trivial bound $u(z)+u^*(w)\geq zw$ we get $u(z)\geq \tfrac{z^2}{2}$ for all $z\in\R$. Then taking the Legendre-Fenchel transform on both sides and noticing that $u\geq v$ implies $u^*\leq v^*$ we see that we also have $u(z)\leq \tfrac{z^2}{2}$ for all $z\in\R$, hence the claim. Another way to look at this rigidity is to observe that from \eqref{eq:dualconv} it follows that $(u^*)'\circ u' ={\rm Id}$ (say everything is smooth) and that there is only one non-decreasing involution (the identity).

On the other hand, there are many non-increasing involutions on $\R$, which leads to a lack of rigidity in the concave case. To be explicit,   for   $u:\R\to\R\cup\{\pm\infty\}$ concave we define the Legendre-Fenchel transform as
\begin{equation}
\label{eq:dualconc}
u^*(w):=\inf_{z\in\R}wz-u(z).
\end{equation}
As before, the dual of the (now concave) function $u_p$ for $0\neq p<1$ is $u_q$, where again $\tfrac1p+\tfrac1q=1$.  The case $p=q=0$ is related to the fact that the dual of $\log(z) $ is $\log(w)+1$, or analogously that  $z\mapsto\tfrac12+\log(z)$ is self dual.

The self duality of  $z\mapsto\tfrac12+\log(z)$ has to do with the fact that its derivative is the non-increasing involution $(0,\infty)\ni z\mapsto\tfrac1z\in(0,\infty)$. However,   as said, there  are many  non-increasing involutions $\varphi:\R\to\R$ and to  each such $\varphi$ it corresponds a self-dual concave function, as follows. Let $a\in\R$ be the only number such that $\varphi(a)=a$ (it is easy to check that $\varphi$ is continuous and strictly decreasing, whence existence and uniqueness of $a$ follows). Then define $u:\R\to\R$ as the only function such that $u(a):=\tfrac{a^2}2$ and $u'(z)=\varphi(z)$ for every $z$.  From \eqref{eq:dualconc} we see that, as before, we have  $(u^*)'\circ u' ={\rm Id}$, which in our case gives $(u^*)'=\varphi$ that together with $u^*(a)=\inf_zaz-u(z)=a^2-u(a)=\tfrac{a^2}2=u(a)$ gives $u^*=u$ on $\R$.
}\fr\end{remark}
\begin{remark}{\rm
It is interesting to observe that for, say, bounded functions $f,g$ it holds
\begin{equation}
\label{eq:prodnorm}
\|fg\|_0=\|f\|_0\|g\|_0,
\end{equation}
as it is trivial to check (here assuming $f,g$ bounded dispenses us from dealing with $0^\pm$, as in this case the two $L^0$-norms coincide).

It is worth to compare the above with the classical identity characterizing the self-dual exponent in the standard theory, namely:
\begin{equation}
\label{eq:parall}
\|f+g\|_2^2+\|f-g\|_2^2=2\big(\|f\|_2^2+\|g\|_2^2\big).
\end{equation}
One can ask whether this latter identity holds in an arbitrary Banach space, and obviously its validity identifies Hilbert structures among Banach ones. On the other hand, the analogous question cannot be posed in the hyperbolic setting, because \eqref{eq:prodnorm} relies on the operation of product of two functions, and in general products are not defined in hyperbolic Banach spaces. One could certainly consider hyperbolic Banach algebras, but we are not going to pursue this direction here.

A curious fact, whose relevance is admittedly unclear to us, is that the identity \eqref{eq:parall} is reminiscent of the equation $x^2+y^2=something$, that describes the circle in the plane, whereas \eqref{eq:prodnorm} reminds $xy=something$, that describes the hyperbole in the same plane.
}\fr\end{remark}
\begin{remark}[Bochner integration]{\rm
The standard Bochner integration theory for maps valued in a  Banach space goes as follows. One first defines the integral of simple functions $\sum_{i=1}^nv_i\nchi_{E_i}$ as $\sum_{i=1}^n\mu(E_i)v_i$, then checks that this linear map is continuous w.r.t.\ the natural $L^1$ distance and finally extends by continuity the map to the $L^1$-completion of the space of simple functions, that can be concretely represented via maps into the Banach space considered.

It is tempting to try to adapt the construction in the current setting for maps with values in a cone. As before, the integral of the simple function $\sum_{i=1}^nv_i\nchi_{E_i}$ should be defined as $\sum_{i=1}^n\mu(E_i)v_i$. In order to extend this to more general functions one should verify whether integration has the $\mcp$ on the wedge of simple functions. Should this be true, the integral can be extended to the directed completion of the wedge of simple functions, that hopefully can be realized as a suitable space of maps with values in the cone considered. 

Understanding the feasibility of this plan is outside the scope of this manuscript.
}\fr\end{remark}

\end{subsection}

\begin{subsection}{Banach spacetimes}\label{se:Banst}
Here we study, from the perspective adopted in the manuscript, the procedure of `Lorentzification' of a `Riemannian' structure.

Given a metric space $(\X,\sfd)$ we equip $\R\times\X$ with the partial order
\begin{equation}
\label{eq:leqdadist}
(t,x)\leq(s,y)\qquad\text{ whenever }\qquad \sfd(x,y)\leq s-t.
\end{equation}
Trivial manipulations show that $\leq$ is indeed a partial order.
%
%
%
%
Even though it is technically simple, it is interesting to observe that Cauchy-convergence in $(\X,\sfd)$ is strictly related to order-convergence in $\R\times \X$:
\begin{lemma}\label{le:duelimiti}
Let $(\X,\sfd)$ be a metric  space and $\leq $ be the order relation on $\R\times \X$ defined by \eqref{eq:leqdadist}. Then:
\begin{itemize}
\item[i)] Let $(x_n)\in \X$ be with  $T:=\sum_n\sfd(x_n,x_{n+1})<+\infty $ and $\sfd$-converging to $x\in \X$. Let $t_n:=\sum_{i\leq n}\sfd(x_i,x_{i-1})$. Then $n\mapsto (t_n,x_n)$ is $\leq$-non-decreasing (resp.\  $n\mapsto (-t_n,x_n)$ is $\leq$-non-increasing) and with supremum equal to $(T,x)$ (resp.\ infimum equal to $(-T,x)$).
\item[ii)] Let $(t_i,x_i)_{i\in I}\subset \R\times \X$ be a $\leq$-directed family having $(T,x)$ as supremum. Then $t_i\to T$ and $x_i\to x$ as $i$ varies in the directed set $I$. Analogously for filtered families.
\end{itemize}
\end{lemma}
\begin{proof} \ \\
\noindent{$(i)$} Let $n< m$ and notice that  $\sfd(x_m,x_n)\leq \sum_{i=n+1}^m\sfd(x_i,x_{i-1})=t_m-t_n$, i.e.\ $(t_n,x_n)\leq (t_m,x_m)$ proving that the given sequence is non-decreasing. Analogously, we have $\sfd(x,x_n)\leq \sum_{i\geq n+1}\sfd(x_i,x_{i-1})=T-t_n$, showing that $(t_n,x_n)\leq (T,x)$ i.e.\ that $(T,x)$ is an upper bound for the given sequence. Let $(S,y)\in\R\times \X$ be another upper bound. Then passing to the limit in $\sfd(y,x_n)\leq S-t_n$ we get $\sfd(y,x)\leq S-T$, i.e.\  $(T,x)\leq (S,y)$, showing that $(T,x)$ is the supremum of the sequence, as claimed. The study of  $n\mapsto (-t_n,x_n)$  is analog.

\noindent{$(ii)$} We clearly have $t_i\leq T$ for every $i\in I$, thus $\bar t:=\lim_it_i=\sup_it_i\leq T$. For $i\leq j$ we have $\sfd(x_j,x_i)\leq t_j-t_i$ and thus $\lims_{i,j}\sfd(x_j,x_i)\leq\bar t-\bar t=0$, showing that $(x_i)_{i\in I}$ is Cauchy.

Fix $\eps>0$ and find $\bar \i\in I$ so that $\sfd(x_j,x_i)\leq\eps$ for every $i,j\geq\bar\i$. We claim that $(\bar t+\eps,x_{\bar\i})$ is an upper bound for $(t_i,x_i)_{i\in I}$. For $i\leq\bar\i $ we have $(t_i,x_i)\leq (t_{\bar\i},x_{\bar\i})\leq (\bar t+\eps,x_{\bar\i})$, if instead $i\geq \bar\i$ we have $\sfd(x_i,x_{\bar\i})\leq\eps\leq \bar t+\eps-t_i$, showing that again it holds $(t_i,x_i)\leq (\bar t+\eps,x_{\bar\i})$. Thus our claim is proved, and by definition of supremum we must have $(T,x)\leq (\bar t+\eps,x_{\bar\i})$, i.e.\ $\sfd(x,x_{\bar \i})\leq \bar t-T+\eps\leq\eps$. Taking the $\lims$ in $\bar \i\in I$, we conclude by the arbitrariness of $\eps>0$. Similarly in the filtered case.
\end{proof}
A simple consequence of the above is the next result, that  corroborates the proposal made in \cite{Octet}  of replacing global hyperbolicity with some version of order-completeness in order to accomodate genuine infinite-dimensional theories. Recall that by `diamond' in a partially ordered set we mean a set of the form $J(a,b):=\{c:a\leq c\leq b\}$ for some $a,b$.
\begin{proposition}\label{prop:2completi}
With the above notation we have:
\begin{itemize}
\item[1)] $\R\times \X$ is globally hyperbolic (=diamonds are compact in the product topology) if and only if $(\X,\sfd)$ is proper (=closed balls are compact). 
\item[2)] $(\R\times \X,\leq)$ is locally directed complete if and only if $(\X,\sfd)$ is Cauchy-complete.
\end{itemize}
Moreover, $\R\times \X$ is first countable in the sense of Definition \ref{def:sep} and it is second countable if and only if $(\X,\sfd)$ is second countable in the classical sense.
\end{proposition}
\begin{proof}  (1) is obvious and  (2) is a direct consequence of Lemma \ref{le:duelimiti} above.

For first countability let $D\subset\R\times \X$ be directed with sup $(T,x)$ and let $t_n\uparrow T$ be strictly increasing. Then $(ii)$ of Lemma \ref{le:duelimiti} above shows that for $(t_i,x_i)\in D$ sufficiently big we have $t_i-t_n\geq \sfd(x_i,x)$ and thus $(t_i,x_i)\geq (t_n,x)$. Since clearly the supremum of $(t_n,v)$ is $(T,v)$, the proof is complete. The claims about second countability follow along similar lines.
\end{proof}
From now on we replace the generic metric space $(\X,\sfd)$ with a normed vector space $(V,\|\cdot\|)$ and discuss the relation between real valued linear maps on $V$ and $\R\times V$. Basic linear algebra considerations tell that for $L:\R\times V\to\R$ linear there are unique $s\in\R$ and $M:V\to\R$ linear so that
\begin{equation}
\label{eq:LsM}
L((t,v))=st-M(v)\qquad \forall (t,v)\in \R\times V.
\end{equation}
Such correspondence is clearly byjective and we shall write $L=(s,M)$ to mean that \eqref{eq:LsM} holds.
\begin{proposition}\label{prop:pensopositivo}
Let  $(V,\|\cdot\|)$ be a Banach space and $L:\R\times V\to \R$  linear, $L=(s,M)$. Then the following are equivalent:
\begin{itemize}
\item[a)] $L$ is positive, i.e.\ $L((t,v))\geq 0$ for any $(t,v)\geq (0,0)$.
\item[b)] $M$ belongs to the topological Banach dual $V^*$ of $V$ with $\|M\|_*\leq s$.
\item[c)] $L$ has the $\mcp$ (and, analogously, respects filtered infima).
\end{itemize} 
 \end{proposition}
\begin{proof}\ 

$(a)\Rightarrow(b)$. If $\|v\|\leq 1$ then $(1,v)\geq (0,0)$ and the assumption yields $M(v)\leq s$. 

$(b)\Rightarrow(c)$. Direct consequence of item $(ii)$ in Lemma \ref{le:duelimiti}.

$(c)\Rightarrow(a)$. Consequence of the general fact that maps with the $\mcp$ are monotone.
\end{proof}
Notice that $V^*$ is itself a normed vector space, so that $\R\times V^*$ carries a partial order $\leq$ as in \eqref{eq:leqdadist}. Then this last proposition tells that $L=(s,M)$ is positive, and has the $\mcp$, if and only if $(s,M)\geq (0,0)$  in $\R\times V^*$.

We further study this duality relation in connection with hyperbolic norms. Below we shall denote by $T\subset \R_+^2$ the `triangle' defined as  
\[
T:=\{(t,x)\, :\, 0\leq x\leq t<+\infty\}.
\]
This is clearly a locally complete wedge w.r.t.\ componentwise operations.  We shall be interested in:
\begin{itemize}
\item[-] Hyperbolic norms $|\cdot|$ on $T$ that are  {\bf$x$-decreasing}, meaning that are upper semicontinuous on $T\setminus\{(0,0)\}$ (w.r.t.\ the standard topology) and for every $t\geq 0$ the map $[0,t]\ni x\mapsto|(t,x)|\in[0,+\infty]$ is  non-increasing. 
\item[-] Seeing $T$ as subset of its dual $T^*$ via the map sending $(s,y)$ to the linear map $(t,x)\mapsto ts-xy$. Notice that the image of this injective embedding consists precisely of those finite-valued linear maps that are non-increasing in $x$. In particular, given a hyperbolic norm $|\cdot|$ on $T$ we shall define the dual norm $|\cdot|_*$ on $T$ as
\begin{equation}
\label{eq:dualT}
|(s,y)|_*:=\inf_{|(t,x)|\geq 1}ts-xy.
\end{equation}
Iterating the procedure we can define the bidual norm $|\cdot|_{**}$ on $T$ and it is clear that
\begin{equation}
\label{eq:bidualT}
|(t,x)|_{**}\geq |(t,x)|\qquad\forall (t,x)\in T.
\end{equation}
\end{itemize}
We have:
\begin{lemma}
Let $|\cdot|$ be an  hyperbolic norm on $T$. Then the following are equivalent:
\begin{itemize}
\item[i)] $|\cdot|$ is $x$-decreasing.
\item[ii)] We have $|(t,x)|_{**}= |(t,x)|$ for every  $(t,x)\in T$.
\item[iii)] The norm is upper semicontinuous on $T\setminus\{(0,0)\}$ and for any  normed space  $(V,\|\cdot\|)$ (equivalently: some normed space of dimension at least 1) the map $(t,v)\mapsto\hn(t,v)$  defined as 
\begin{equation}
\label{eq:normadanorma}
\hn(t,v):=|(t,\|v\|)|\qquad \forall(t,v)\in F :=\{(t,v)\in\R\times V\ :\ t\geq \|v\|\}
\end{equation}
is a hyperbolic norm on the `future cone' (a wedge, in our terminology) $F$.
\end{itemize}
\end{lemma}
\begin{proof}\ 

$(i)\Rightarrow(ii)$ If $|\cdot|$ ever attains the value $+\infty$, then by concavity and upper semicontinuity it is identically $+\infty$ on $T\setminus\{(0,0)\}$ and the claim easily follows. Also, in this case the dual norm is identically 0. Dually,  if $|\cdot|$  attains the value 0 on the interior $\{(t,x):0<x<t\}$ of $T$, then by concavity it is identically 0, hence the dual norm is identically $+\infty$ on $T\setminus\{(0,0)\}$  and again the claim follows. We shall thus assume $|\cdot|$ to be real valued and positive in the interior of $T$, in which case upper semicontinuity improves - by concavity - to continuity. Fix $(\bar t,\bar x)$ with $0<\bar x<\bar t$. We shall prove that 
\begin{equation}
\label{eq:bidualTaltra}
|(\bar t,\bar x)|_{**}\leq |(\bar t,\bar x)|,
\end{equation}
that by the arbitrariness of $(\bar t,\bar x)$ as above and the continuity of $|\cdot|_{**}$ (that follows from its finiteness that in turn follows from \eqref{eq:bidualTaltra}) together with \eqref{eq:bidualT} gives the conclusion. Since $|(\bar t,\bar x)|\in(0,+\infty)$, up to scaling we can, and will, assume that $|(\bar t,\bar x)|=1$. The map $(t,x)\mapsto|(t,x)|$ is concave and, by assumption, finite on the point $(\bar t.\bar x)$ that lies in the interior of the domain of definition: it follows by standard -- and easy  to prove, see e.g.\ \cite[Theorem 23.4]{Rock70} -- arguments that its superdifferential at $(\bar t,\bar x)$ is not empty, and thus   that there are $(s,y)\in\R^2$ such that
\begin{equation}
\label{eq:superdiffL}
|(\bar t,\bar x)|+s(t-\bar t)-y(x-\bar x)\geq |(t,x)|\qquad\forall (t,x)\in T.
\end{equation}
We claim that  $(s,y)\in T$. To see this notice that \eqref{eq:superdiffL} and the fact that $|\cdot|$ is $x$-decreasing give
\[
|(\bar t,\bar x)|+y\bar x\geq|(\bar t,0)|\geq |(\bar t,\bar x)|,
\]
that  implies $y\geq 0$.  Also, \eqref{eq:superdiffL} and superadditivity of $|\cdot|$ give
\[
|(\bar t,\bar x)|+s-y\geq|(\bar t+1,\bar x+1)|\geq  |(\bar t,\bar x)|+|(1,1)|\geq  |(\bar t,\bar x)|,
\]
implying that $s\geq y$. Thus the claim $(s,y)\in T$ is proved. We can thus compute its dual norm and obtain
\[
\begin{split}
|(s,y)|_*=\inf_{(t,x)\in T:|(t,x)|= 1}ts-xy\stackrel{\eqref{eq:superdiffL}}\geq  1-|(\bar t,\bar x)|+\bar ts-\bar xy=\bar ts-\bar xy,
\end{split}
\]
having used that $|(\bar t,\bar x)|=1$. Obviously we have  $\bar ts-\bar xy\geq 0$; we shall assume that the strict inequality holds (otherwise a perturbation of the argument below gives the conclusion anyway) and define $\bar s:=\tfrac s{\bar ts-\bar xy}$ and  $\bar y:=\tfrac y{\bar ts-\bar xy}$. Thus we have $|(\bar s,\bar y)|_*\geq 1$ and therefore
\[
|(\bar t,\bar x)|_{**}=\inf_{(s',y')\in T:|(s',y')|_*\geq 1}\bar ts'-\bar xy'\leq \bar t\bar s-\bar x\bar y=1=|(\bar t,\bar x)|,
\]
proving \eqref{eq:bidualTaltra}.

$(ii)\Rightarrow(i)$ The definition \eqref{eq:dualT} ensures that dual norms are always upper semicontinuous and non-increasing in $y$ (as $x\geq0$).  The claim follows.

$(i)\Rightarrow(iii)$ Upper semicontinuity of the norm is part of the assumption and the fact that $\hn$ is 1-homogeneous is obvious. For the reverse triangle inequality we notice that
\[
\begin{split}
|(t+s,\|v+w\|)|\stackrel!\geq|(t+s,\|v\|+\|w\|)|=|(t,\|v\|)+(s,\|w\|)|\geq |(t,\|v\|)|+|(s,\|w\|)|,
\end{split}
\]
where in the marked inequality we used that the norm is non-increasing in $x$.

$(iii)\Rightarrow(i)$ Let $(V,\|\cdot\|)$ be a normed space of dimension at least 1. Pick  $\lambda\in[0,1]$ and then find $v\in V$ with $\|v\|=\lambda$. Then $(1,v),(1,-v)\in F$ and assumptions give
\[
2|(1,0)|=|(2,0)|=\hn((2,0))\geq \hn((1,v))+\hn((1,-v))=2|(1,\lambda)|,
\]
and thus $|(1,0)|\geq |(1,\lambda)|$ for every $\lambda\in[0,1]$. Now let  $\eta\leq\lambda\leq1$, notice that   $(1,\eta)=\tfrac\eta\lambda(1,\lambda)+(1-\tfrac\eta\lambda)(1,0)$ and use   that $|\cdot|$ is assumed to satisfy the reverse triangle inequality to get
\[
|(1,\eta)|\geq \tfrac\eta\lambda|(1,\lambda)|+(1-\tfrac\eta\lambda)|(1,0)|\stackrel{!}\geq  \tfrac\eta\lambda|(1,\lambda)|+(1-\tfrac\eta\lambda)|(1,\lambda)|=|(1,\lambda)|
\]
having used what already proved in the marked inequality. This shows  that $[0,1]\ni\lambda\mapsto|(1,\lambda)|$ is non-decreasing. Since we assumed   upper semicontinuity the proof is complete.
\end{proof}

\begin{example}[$L^p$ Lorentzian norms, $p\geq 1$]\label{ex:lpT}{\rm

For $p\in[1,+\infty)$ we define the $x$-decreasing norm $|\cdot|_p$ on $T$ as
\[
|(t,x)|_p:=\sqrt[p]{t^p-x^p}\qquad\forall (t,x)\in T,
\]
and for $p=+\infty$ we put
\[
|(t,x)|_{+\infty}:=t.
\]
It is then not hard to show that  for $p,q\in[1,+\infty]$ with $\tfrac1p+\tfrac1q=1$ the dual, in the sense of \eqref{eq:dualT}, of $|\cdot|_p$ is $|\cdot|_q$ and viceversa. 

Briefly: the cases $1,+\infty$ can be dealt via direct computation. For the others we claim that
\begin{equation}
\label{eq:dualpqT}
|(t,x)|_p|(s,y)|_q\leq ts-xy\qquad\forall(t,x),(s,y)\in T.
\end{equation}
If $x=0$ this is obvious, otherwise $t\geq x>0$ and one can check that at $(s,y):=(t^{p/q},x^{p/q})\in T$  the differential of the convex function $T\ni (s,y)\mapsto  ts-xy-|(t,x)|_p|(s,y)|_q$ vanishes and the function  also vanishes.

To conclude, notice that for  $(t,x)\in T$ with $x>0$ the choice $(s,y):=(t^{p/q},x^{p/q})\in T$ gives equality in \eqref{eq:dualpqT}. If $x=0$ we argue by approximation using the continuity of the norms.
 }\fr\end{example}

\begin{definition}[Lorentzification of a Banach space]
Let $(B,\|\cdot\|)$ be a Banach space and $|\cdot|$ an $x$-decreasing norm on $T$. 

The $|\cdot|$-Lorentzification of $B$ is the wedge $F:=\{(t,v)\in\R\times B:t\geq \|v\|\}$ equipped with the hyperbolic norm $\hn$ defined in \eqref{eq:normadanorma}.
\end{definition}

Proposition \ref{prop:pensopositivo} establishes a link between (hyperbolic) duals of  Banach spacetime and (classical) duals of the underlying Banach space. Such link naturally extends to norms:
\begin{proposition}[Dual of Lorentzification is Lorentzification of the dual] Let $(B,\|\cdot\|)$ be a Banach space, $|\cdot|$ an $x$-decreasing norm on $T$ and $(F,\hn)$ the $|\cdot|$-Lorentzification of $B$. 

Let $(B^*,\|\cdot\|_*)$ be the Banach dual of $(B,\|\cdot\|)$, $|\cdot|_*$ the dual norm of $|\cdot|$ as in \eqref{eq:dualT} and let $(F^*,\hn_*)$ be  the $|\cdot|_*$-Lorentzification of $B^*$. Then:
\begin{itemize}
\item[i)] For each $(s,M)\in F^*$ the map $F\ni (t,v)\mapsto L(t,v):=ts-M(v)\in \R$ is linear with the $\mcp$ (thus an element of the dual of $F$ as in Definition \ref{def:morphism}) and it holds
\begin{equation}
\label{eq:dualLorT}
\hn_*(L)=\inf_{\hn(t,v)\geq1}ts-M(v).
\end{equation}
\item[ii)] For any $L:F\to[0,+\infty)$ linear (and thus with the $\mcp$ by Proposition \ref{prop:pensopositivo}) there is a unique $(s,M)\in F^*$ such that $L(t,v)=ts-M(v)$ for every $(t,v)\in F$.
\end{itemize}
\end{proposition}
\begin{proof} \

$(i)$ the fact that $L$ is linear with the $\mcp$ follows from Proposition \ref{prop:pensopositivo}. For \eqref{eq:dualLorT} we notice that
\[
\hn_*(L)=|(s,\|M\|_*)|_*=\inf\big\{ts-\|M\|_*\|v\|\ :\ |(t,\|v\|)|\geq 1\big\}.
\]
In the last expression, $v$ enters only via its norm, thus in taking the infimum and recalling the definition of $\|M\|_*$ we see that such expression is equal to $\inf\big\{ts-M(v)\ :\ |(t,\|v\|)|\geq 1\big\}$, proving the claim.

$(ii)$ follows directly from Proposition \ref{prop:pensopositivo}.
\end{proof}
Notice that Lorentzifications of Banach spaces are always locally complete wedges (by Proposition \ref{prop:2completi}) and the norm $\hn$ always has the $\mcp$. To see the latter notice that we have already established that an $x$-decreasing norm is either identically $+\infty$ in $T\setminus\{(0,0)\}$, in which case the induced norm $\hn$ obviously has the $\mcp$, or it is continuous. In this latter case $\hn$ is also continuous (w.r.t.\ the product topology) and the fact that it has the $\mcp$ follows from Lemma \ref{le:duelimiti}. Therefore  $\hn$ uniquely extends to a map with the $\mcp$ on the directed completion $\bar F$ of $F$. It is easy to see that the extension of $\hn$, still denoted by $\hn$, is still a hyperbolic norm. Indeed, if $D_1,D_2\subset F$ are  subsets with tips $v_1,v_2\in \bar F$, respectively, then $D:=\{v_1'+v_2':v_1'\in D_1,v_2'\in D_2\}$ has tip, it being $v_1+v_2$ (by transfinite recursion and the $\mcp$ of the addition) and the $\mcp$ of $\hn$ gives
\[
\hn(v_1+v_2)=\sup_{z\in D}\hn(z)=\sup_{v'_1\in D_1\atop v'_2\in D_2}\hn(v'_1+v'_2)\geq\sup_{v'_1\in D_1\atop v'_2\in D_2}\big(\hn(v'_1)+\hn(v'_2)\big)=\hn(v_1)+\hn(v_2).
\]
1-homogeneity is proved analogously. The following definition is now natural:
\begin{definition}[Banach spacetimes]\label{def:banachspacetime}
A Banach spacetime is a hyperbolic Banach space $(\C,\hn)$ isomorphic to the completion of the Lorentzification, w.r.t.\ some $x$-decreasing norm $|\cdot|$ on $T$, of some Banach space.

We say that the  Banach spacetime space splits off a line if $|\cdot|$ is the Lorentzian 2-norm $|\cdot|_2$ as in Example \ref{ex:lpT} and that it is Minkowskian  if moreover the underlying Banach space is Hilbert. \end{definition}
We collect some comments about this:
\begin{itemize}
\item[i)] The relevance of the above definition is in the fact that we expect the (co)tangent bundles (or perhaps modules, once such definition will be given) of an infinitesimally Minkowskian space to be (interpretable as) bundles whose fibres are Minkowskian spaces in the sense of the above definition. This is in analogy with the elliptic case, where (co)tangent modules of an infinitesimally Hilbertian space can be seen as bundles having Hilbert spaces as fibres. 

 In the elliptic case the space of $L^2$-sections of (co)tangent bundles of a smooth Riemannian manifold  is a Hilbert space itself. This has no analogue in the current framework: for no $p\leq 1$ the space of future vector fields in a, say, smooth spacetime  has the structure of a Minkowskian space in the sense of Definition \ref{def:banachspacetime}. This is related to the fact that the definition of infinitesimal Minkowskianity \cite{Octet} had been given via an a.e.\ parallelogram-like identity, whereas that of infinitesimal Hilbertianity \cite{Gigli12}  can be given by asking for its `integrated version', i.e.\ that $W^{1,2}$ is Hilbert.

One should not be worried by this difference, though, that anyway disappears when considering $L^0$-normed modules in the sense of \cite{Gigli14} (not to be confused with the $L^0$-norms we  discussed  here).
\item[ii)] It is not clear to us whether the directed completion of a Lorentzification of a Banach space is a cone (recall Remark \ref{re:abstrcompl}) and thus whether it is a Banach spacetime in the above sense. On the other hand, for the $|\cdot|_2$-Lorentzification of a Hilbert space we know that this is the case: the proof is a direct consequence of the explicit construction of the directed completion done in Example \ref{ex:mink} in Section \ref{se:exbelli}.
\item[iii)] In general, a Banach spacetime is not coming from a unique Lorentzification, and even if so the embedding of the future cone into the Banach spacetime might be not unique. To be more precise, let $F_1,F_2$ be two Lorentzifications of two Banach spaces $B_1,B_2$ so that their directed completions $\C_1,\C_2$, coming with the inclusions $\iota_1,\iota_2$, are isomorphic as Banach spacetimes via some $\Phi:\C_1\to \C_2$. Then $\iota_2^{-1}\circ\Phi\circ\iota_1$ is a linear map from a subset of $F_1$ to $F_2$ that extends to a linear map, call it $L$, from $\R\times B_1$ to $\R\times B_2$. Quite clearly, $L(\{0\}\times B_1)$ is not $\{0\}\times B_2$; in fact it is not even clear to us if $B_1$ must be isometric to $B_2$.

The situation improves if the spacetime is Minkowskian as in Definition \ref{def:banachspacetime}, as if one of the two Banach spaces, say $B_2$, is Hilbert, the parallelogram identity holds in $F_2$, thus it holds in the whole $\R\times B_2$ and then by restriction on the image of $\{0\}\times B_1$ in $\R\times B_2$ via the isomorphism $L$ discussed above. It follows that a Minkowskian spacetime in the sense of Definition \ref{def:banachspacetime} above  can only be the completion of the $|\cdot|_2$-Lorentzification of a Hilbert space.
\end{itemize}

We conclude pointing out that one might consider, more generally, the product of a normed vector space $(V,\|\cdot\|)$ and of a wedge $\W$ equipped with a hyperbolic norm $\hn$ on it. The natural order relation $\leq$ on $\W\times V$ is then
\[
(w,v)\leq(w',v')\qquad\text{ whenever }\qquad \|v-v'\|\leq\hn(w'-w).
\]
 Denoting by $F:=\{(w,v)\in \W\times V:(0,0)\leq(w,v)\}$ the `future cone', it is easy to see that for an $x$-decreasing norm on $T$, the map $(w,v)\mapsto |(\hn(w),\|v\|)|$ is a hyperbolic norm on the wedge $F$ (observe that $t\mapsto |(t,x)|$ must be non-decreasing for any $x\geq 0$).
 
\end{subsection}

\begin{subsection}{Duality between measures and functions}\label{se:dualfcts}

Let $(\X,\tau)$ be a Polish space and $C_b(\X)$ the space of bounded continuous functions on it. Then for every finite Borel measure $\mu$ on $\X$, its Total Variation norm $\|\mu\|_{\TV}$ satisfies the duality formula
\begin{equation}
\label{eq:measdual1}
\|\mu\|_{\TV}=\sup_{\|f\|_{C_b}\leq 1}\Big|\int f\,\d\mu\Big|\qquad\text{where}\qquad\|f\|_{C_b}:=\sup_{x\in \X}|f(x)|.
\end{equation}
This simple statement tells that $\|\mu\|_\TV$ coincides with the operator norm -- in the sense of standard functional analysis -- of the linear functional $f\mapsto \int f\,\d\mu$ induced by $\mu$ on $C_b(\X)$.  A central result in measure theory is Riesz's theorem, that ensures that if $\X$ is compact, then every bounded linear functional on $C_b(\X)$ can be (uniquely) represented by a measure. If $\X$ is a  general Polish space, then functionals represented by measures as above are precisely those such that 
\begin{equation}
\label{eq:stone}
L(f_n)\to 0\qquad\text{whenever $(f_n)\subset C_b(\X)$ is so that $f_n(x)\downarrow0$ for every $x\in\X$.}
\end{equation}
See for instance \cite[Theorem 7.10.1]{Bogachev07}.

We are interested in studying the same duality from the perspective of hyperbolic functional analysis. Let us denote, for $\X$ Polish, by $\LSC(\X)$ the wedge   of lower semicontinuous functions from $\X$ to $[0,+\infty]$. For every non-negative Borel measure $\mu$ on $\X$ we have
\begin{equation}
\label{eq:measdual2}
\|\mu\|_{\TV}=\inf_{\|f\|_{\LSC}\geq 1}\int f\,\d\mu\qquad\text{where}\qquad\|f\|_{\LSC}:=\inf_{x\in \X}f(x),
\end{equation}
as it is trivial to check. Identity \eqref{eq:measdual2} is in clear analogy with \eqref{eq:measdual1} and allows to reinterpret the classical Riesz' result for non-negative measures in the hyperbolic framework developed here. Before coming to the actual statement, let us collect the basic properties of the space $\LSC(\X)$ in the next statement. Below, we shall denote by $\leq_\p$, as usual, the pointwise order and by $\preceq$ the order relation induced by the wedge structure of $\LSC(\X)$. Also $C_+(\X)\subset\LSC(\X)$ is the set of continuous functions from $\X$ to $[0,+\infty)$.
\begin{proposition}[The wedge $\LSC(\X)$]\label{prop:lsc}
Let $(\X,\tau)$ be a Polish space. Then $\LSC(\X)$ is a wedge and $\|\cdot\|_{\LSC}$ a hyperbolic norm on it. Moreover:
\begin{itemize}
\item[i)] For $f,g\in\LSC(\X)$ we have:
\[
g\preceq f\qquad\Leftrightarrow\qquad g\leq_\p f\ \text{ and }\ \{g<+\infty\}\ni x\mapsto f(x)-g(x)\text{ is lower semicontinuous},
\]
\item[ii)] If the $\preceq$-supremum of a $\preceq$-directed family exsts, then it coincides with the pointwise supremum,
\item[iii)] $C_+(\X)$ is $\preceq$-dense in $\LSC(\X)$,
\item[iv)] $(\LSC(\X),\preceq)$ is second countable (and thus also first countable -- recall Definition \ref{def:sep}).
\end{itemize}
\end{proposition}
\begin{proof}
The fact that $\LSC(\X)$ is a prewedge is obvious. It being a wedge directly follows from $(i)$, so we prove this. The `only if' is clear. For the `if' we define the function
\[
h(x):=\limi_{y\to x\atop y\in \{g<+\infty\}}f(x)-g(x)\qquad\forall x\in\X,
\]
being intended that $h(x)=+\infty$ if $x\notin\overline{\{g<+\infty\}}$. Then $h\in\LSC(\X)$ by construction and the assumption grants that $g+h=f$, concluding the proof.

The fact that $\|\cdot\|_{\LSC}$ is a hyperbolic norm is obvious, so we pass to $(ii)$. Let $(f_i)\subset\LSC(\X)$ be $\preceq$-directed and admitting $\preceq$-supremum $f$. Define $g(x):=\sup_if_i(x)$ for every $x\in\X$ and notice that $g\in\LSC(\X)$ and by $(i)$ we have  $g\leq_\p f$. To conclude it is therefore sufficient to prove that $f_i\preceq g$ for every $i\in I$. To see this notice that for every $x\in\{f_i<+\infty\}$ we have $(g-f_i)(x)=\sup_{j>i}(f_j-f_i)(x)$ and since the family is directed we know by $(i)$ that $f_j-f_i$ is lower semicontinuous on $\{f_i<+\infty\}$ for every $j>i$. It follows that $g-f_i$ is also lower semicontinuous on the same set so again by $(i)$ we conclude.

It is well known, and easy to prove, that any function in $\LSC(\X)$ is the pointwise supremum of a pointwise non-decreasing sequence in $C_+(\X)$. By $(i)$ this suffices to prove the density of $C_+(\X)$ in $\LSC(\X)$ claimed in $(iii)$.

We pass to second countability, that is ultimately related to the Lindel\"of property of $\X$.  Let $(U_n)$ be a countable base for $\tau$ and, for every $n\in\N$, let $(f_{n,k})\subset C_+(\X)$ be pointwise non-decreasing with $\psup_kf_{n,k}=\nchi_{U_n}$. Put 
\[
\tilde N:=\{rf_{n,k}\ :\ n,k\in\N,\ r\in\Q,\ r\geq 0\}
\]
and
\[
N:=\{g_1\vee\cdots\vee g_n\ :\ n\in\N,\ n>0,\ g_i\in\tilde N\ \forall i=1,\ldots,n\}.
\]
Notice that by $(i)$ the meaning of $\vee$ in this last formula is unambiguously defined either as pointwise maximum or as $\preceq$-maximum. We claim that $N$, that is clearly countable, has the property required by Definition \ref{def:sep}. To see this, let $(f_i)\subset\LSC(\X)$ be $\preceq$-directed and $f\in\LSC(\X)$ its $\preceq$-supremum. Let $M:=\{g\in \tilde N:g\preceq f_i\text{ for some }i\in I\}$, let $(g_j)_{j\in\N}$ be an enumeration of the elements of $M$  and then put $h_n:=g_1\vee\cdots\vee g_n\in N$ for every $n\in\N$. It is clear that $(h_n)$ is $\preceq$-non-decreasing with pointwise supremum (and thus $\preceq$-supremum as these are continuous) equal to $f$.

The construction also ensures that for every $j\in\N$ there is $i_j\in I$ so that $g_j\preceq f_{i_j}$ and since $(f_i)$ is $\preceq$-directed, we see that for every $n\in\N$ there is $i\in I$ so that $h_n\preceq f_i$, concluding the proof of $(iv)$.
\end{proof}

In the discussion below we shall denote by   $\mathcal M_+(\X)$ the collection of non-negative and finite Borel measures on $\X$. This is clearly a wedge whose partial order coincides with the set-theoretic one `$\mu(E)\leq\nu(E)$ for every $E\subset\X$ Borel' and $\|\cdot\|_{\TV}$ is a hyperbolic norm on it (satisfying the equality in the reverse triangle inequality).   

Notice that  $\mathcal M_+(\X)$ is locally   complete, as for any directed family $(\mu_i)$ bounded from above, the map $E\mapsto\mu(E):=\sup_i\mu_i(E)$ defined on Borel subsets of $\X$ is easily seen to be a finite Borel measure and the supremum of the given family.

With this said we have the following theorem, that shows that well known results and concepts fit naturally in the current context:
\begin{theorem}[Riesz theorem - hyperbolic version]\label{thm:riesz}
Let $(\X,\tau)$ be a Polish space. Then:
\begin{itemize}
\item[i)] Let $\mu\in\mathcal M_+(\X)$ and define $L_\mu:\LSC(\X)\to [0,+\infty]$ as
\begin{equation}
\label{eq:Lmu}
L_\mu(f):=\int f\,\d\mu\qquad\forall f\in \LSC(\X).
\end{equation}
Then $L_\mu\in \LSC(\X)^*$ with
\begin{equation}
\label{eq:normamis}
\|\mu\|_{\TV}=\|L_\mu\|_{\LSC(\X)^*}.
\end{equation}
\item[ii)] Let $L\in \LSC(\X)^*$ be with $\|L\|_{\LSC(\X)^*}<+\infty$. Then there is a unique $\mu\in\mathcal M_+$ such that $L=L_\mu$, the latter being defined as in \eqref{eq:Lmu}.
\end{itemize}
\end{theorem}
\begin{proof} \

$(i)$ The linearity of $L_\mu$ is clear from the definition. For the $\mcp$ property we use Proposition \ref{prop:sepmcp}, that can be applied thanks to $(iv)$ of Proposition \ref{prop:lsc} above: then \eqref{eq:mcpsep} is, again, nothing but Beppo Levi's monotone convergence theorem. Finally, identity \eqref{eq:normamis} is a restatement of \eqref{eq:measdual2}.

$(ii)$ This is a restatement of classical variants of Riesz's theorem in non-compact spaces, as given for instance in \cite[Theorem 7.10.1]{Bogachev07}:  the already recalled assumption \eqref{eq:stone}  follows from the assumed MCP of $L$ applied to $1-f_n$ and $(i)$ of Proposition \ref{prop:lsc}. Notice that once  we established that  $L_\mu(f)=L(f)$ for $f$ continuous,   the conclusion for lower semicontinuous ones follows  by the fact that both sides have the  $\mcp$ and the density of continuous functions in $\LSC(\X)$.
%
%
\end{proof}
\begin{remark}{\rm The above result would not hold if we replace $\LSC(\X)$ with the dense subset $C_+(\X)$ of continuous functions from $\X$ to $[0,+\infty)$, nor with its completion  (as discussed in  Example \ref{ex:c} in Section \ref{se:exbelli}). The problem with these wedges is that the supremum of an increasing sequence might be different from the pointwise supremum, making the evaluation at points (i.e.\ Dirac masses) not possessing the MCP. 

For instance, consider as in the Example \ref{ex:c} just mentioned the case of $\X:=[0,1]$ and  $f_n(x):=1\wedge(nx)$ for every $n\in\N$. Then in the wedge $C_+(\X)$ the supremum of the increasing sequence  $(f_n)$ is the function identically 1, as it is immediate from the fact that the partial order induced by the wedge structure coincides with the pointwise one (as difference of continuous functions in continuous),
making the map $f\mapsto f(0)=\int f\,\d\delta_0$ not possessing the MCP.
}\fr\end{remark}
\begin{remark}{\rm
A direct consequence of Theorem \ref{thm:riesz} is that any linear functional as in the statement has the $\mcp$ also w.r.t.\ the pointwise order $\leq_\p$, a property that according to $(ii)$ in Proposition \ref{prop:lsc} is stronger than the $\preceq$-$\mcp$. If one knew in advance that the linear $L:\LSC(\X)\to[0,+\infty]$ had the $\leq_\p$-$\mcp$, then the conclusion would follow more directly along the following lines.

Under this assumption,  we can apply Carath\'eodory's construction of outer measures via coverings plus  extraction of measurable sets: this   ensures that a map $\mu$ from the open subsets of $\X$ to $[0,+\infty]$ is the restriction of a (not necessarily unique if $\mu(\X)=+\infty$) Borel measure provided
\begin{equation}
\label{eq:carathe}
\begin{split}
\mu&\text{ is monotone},\\
\mu(U\cup V)&\leq\mu(U)+\mu(V)\quad\text{ with equality if }U\cap V=\emptyset,\\
\mu(\cup_nU_n)&=\sup_n\mu(U_n)\quad\text{ whenever $(U_n)$ is increasing.}
\end{split}
\end{equation}
Then we can observe that for $U\subset\X$ open the characteristic function $\nchi_U$ is lower semicontinuous and define $\mu(U):=L(\nchi_U)$. The above properties are trivially verified, the continuity from below being a direct consequence of the $\leq_\p$-$\MCP$ of $L$. Hence such $\mu$ extends to a Borel measure and we are left to prove that $\int f\,\d\mu=L(f)$ for any $f\in\LSC(\X)$. By definition this is clear if $f\in\LSC(\X)$ is a characteristic function, then by linearity this holds also for functions with finite range, and being these functions (easily seen to be) $\leq_\p$-dense in $\LSC(\X)$, the conclusion follows from the $\mcp$ of $L$ and Beppo Levi's theorem.
}\fr\end{remark}
\begin{remark}\label{re:lscnocone}{\rm
$\LSC(\X)$ is in general not a cone in the sense of our Definition \ref{def:cone}. Consider for instance $\X:=[-1,1]$ and the $\preceq$-increasing sequence $(f_n)\subset \LSC(\X)$ defined as $f_n(x):=0\vee (nx)\wedge1$. The pointwise supremum of this sequence is clearly $f_\infty:=\nchi_{(0,1]}$, thus if the $\preceq$-supremum exists it must, by $(ii)$ of Proposition \ref{prop:lsc} above, coincide with $f_\infty$.

However, $f_\infty$ is not the $\preceq$-supremum, because the function $g\equiv 1$ is clearly $\succeq f_n$ for every $n\in\N$, but $f_\infty\not\preceq g$, as $g-f_\infty$ is not lower semicontinuous (recall $(i)$ of Proposition \ref{prop:lsc}).
}\fr\end{remark}
\begin{remark}[Wedge completion of $C_+(\X)$]\label{re:wedgecomplcpiu}{\rm
It is worth commenting on the considerations given in the previous remark from the perspective of wedge completion (recall Remark \ref{re:abstrcompl}): the fact that $\LSC(\X)$ is not a cone hints at the possibility -- without an actual proof -- that the directed completion of $C_+(\X)$ is not a cone, either.

Let us discuss this and recall that Example \ref{ex:c} in Section \ref{se:exbelli} shows that the directed completion $\overline{C_+(\X)}$ of $C_+(\X)$ can be realized as
\[
\overline{C_+(\X)}:=\{f\in\LSC(\X)\ :\ f=f^{\us \ls}\}
\]
together with the inclusion. Here the operators $f\mapsto f^\us$ and  $f\mapsto f^\ls$ are those sending an arbitrary function to its upper and lower semicontinuous envelopes, respectively. The partial order on $\overline{C_+(\X)}$ is the restriction of that in $\LSC(\X)$, i.e.\ the pointwise one, but the sum is not, as $f,g\in \overline{C_+(\X)}$ does not imply $f+g\in\overline{C_+(\X)}$ (take $\X:=[-1,1]$, $f:=\nchi_{(0,1]}$ and $g:=\nchi_{[-1,0)}$).

We claim that the sum operation $\oplus$ coming from extending the sum in $C_+(\X)$ via the $\mcp$ to its completion is given by
\begin{equation}
\label{eq:aa}
f\oplus g=(f+g)^{\us\ls}\qquad\forall f,g\in \overline{C_+(\X)}.
\end{equation}
Let us prove this. Fix $f,g\in \overline{C_+(\X)}$ and let $(f_i),(g_j)\subset C_+(\X)$ be directed families having $f,g$ as supremum, respectively (by Propositions \ref{prop:complunostep} and \ref{prop:compljoin} these exist). By \eqref{eq:eqnum} we thus have that $f=(\psup_i f_i)^{\us\ls}$ and $g=(\psup_jg_j)^{\us\ls}$. Also, quite clearly, the family $(f_i+g_j)$ is directed and by definition of $\oplus$ its supremum, namely $(\psup_{ij}(f_i+g_j))^{\us\ls}$, is equal to $f\oplus g$. Since $\psup_i f_i,\psup_jg_j\in\LSC(\X)$, we thus see that to prove \eqref{eq:aa} it is sufficient to prove
\begin{equation}
\label{eq:bb}
(f+g)^{\us\ls}=(f^{\us\ls}+g^{\us\ls})^{\us\ls},\qquad\forall f,g\in\LSC(\X).
\end{equation}
The trivial bound $f\leq_\p f^{\us\ls}$ valid for any $f\in\LSC(\X)$ shows that $\leq_\p$ holds in \eqref{eq:bb}. For the opposite inequality notice that the trivial bound $\lims(a_n+b_n)\geq \lims a_n+\limi b_n$ yields $(f+g)^{\us}\geq_\p f^\us+g^\ls=f^\us+g$. Similarly, from $\limi(a_n+b_n)\geq\limi a_n+\limi b_n$ we get $(f^\us+g)^\ls\geq_\p f^{\us\ls}+g^\ls=f^{\us\ls}+g$. By the $\leq_\p$ monotonicity of $f\mapsto f^\us,f^\ls$ it follows that $(f+g)^{\us\ls}\geq_\p f^{\us\ls}+g$, and swapping the roles of $f,g$ we conclude the proof of \eqref{eq:bb}.

Now consider $\X=[-1,1]$ and the same functions $f_\infty=\nchi_{(0,1]}$ and $g\equiv 1$ of Remark \ref{re:lscnocone} above. Notice that these belong to $\overline{C_+(\X)}$ and recall that $g$ is not an upper bound of $f_\infty$ in $\LSC(\X)$ (as $g-f_\infty$ is not lower semicontinuous). However, $g$ is an upper bound of $f_\infty$ in $\overline{C_+(\X)}$, because $\nchi_{[-1,0)}\in \overline{C_+(\X)}$ and
\[
f_\infty\oplus\nchi_{[-1,0)}\stackrel{\eqref{eq:aa}}=(\nchi_{(0,1]}+\nchi_{[-1,0)})^{\us\ls}=\nchi_{[-1,1]}=g.
\]
It follows that the same example of Remark \ref{re:lscnocone} does not disprove $\overline{C_+(\X)}$ to be a cone. We shall not investigate further the topic.
}\fr\end{remark}

\begin{remark}[Functionals with infinite norm]\label{rmk:rieszinfinito}{\rm Theorem \ref{thm:riesz} admits a natural extension to the case of non-finite measures/functionals with infinite norm that we briefly outline here.

We shall consider the wedge
\begin{equation}
\label{eq:defbarm}
\widetilde{\mathcal M}_+(\X):={\mathcal M}_\infty(\X)/\sim\qquad\text{where}\qquad {\mathcal M}_\infty(\X):=\big\{\text{non-negative Borel measures on $\X$}\big\},
\end{equation}
where $\mu\sim\nu$ iff $\mu(U)=\nu(U)$ for every open set $U\subset\X$. It is clear that if $\mu$ is finite, then it is the only element in its equivalence class. The natural operations of addition and multiplication by positive scalar pass to the quotient and give $\widetilde{\mathcal M}_+(\X)$ the structure of prewedge. We claim that the induced partial order is the one obtained by evaluation on open sets, proving also that  $\widetilde{\mathcal M}_+(\X)$ is a wedge, i.e.\ we claim that 
\begin{equation}
\label{eq:potildem}
\text{there is $[\xi]\in\widetilde{\mathcal M}_+(\X)$ such that }[\nu]+[\xi]=[\mu]\qquad\Leftrightarrow\qquad \nu(U)\leq\mu(U)\quad\forall U\subset\X\ \text{open}
\end{equation}
(see also \eqref{eq:claimmpp} below for the analog in ${\mathcal M}_\infty(\X)$). The implication $\Rightarrow$ is obvious. For the converse one let $U_{[\mu]}$ be the domain of local finiteness of $[\mu]$, defined as the union of all the open subsets $V$ of $\X$ with  $\mu(V)<+\infty$. By the Lindel\"of property of open subsets of the Polish space $\X$ we see that $U_{[\mu]}$ is the countable union of such $V$'s, hence $\mu\restr {U_{[\mu]}}$ is $\sigma$-finite and thus $\nu\restr {U_{[\mu]}}$ is $\sigma$-finite as well. The existence of a $\sigma$-finite Borel measure $\xi$ on $U_{[\mu]}$ such that $\nu\restr {U_{[\mu]}}+\xi=\mu\restr {U_{[\mu]}}$ is now obvious. We then extend $\xi$ to all $\X$ by setting $\xi(E)=+\infty$ if $E\not\subset U_{[\mu]}$, observe that $\xi$ is still a measure and that $[\nu]+[\xi]=[\mu]$ holds trivially.

It is now easy to see that $\widetilde{\mathcal M}_+(\X)$ is in fact a cone. Indeed, if $(\mu_i)$ is a directed family of (equivalence classes of) measures in  $\widetilde{\mathcal M}_+(\X)$, then we can define $\mu:\{\text{open subsets of $\X$}\}\to[0,+\infty]$ by putting
\begin{equation}
\label{eq:supsuaperti}
\mu(U):=\sup_i\mu_i(U)\qquad\forall U\subset\X\ {\rm open}.
\end{equation}
It is easy to check that this satisfies the three properties in \eqref{eq:carathe} (for continuity from below notice that $\mu(\cup_nU_n)=\sup_i\mu_i(\cup_nU_n)=\sup_i\sup_n\mu_i(U_n)=\sup_n\sup_i\mu_i(U_n)=\sup_n\mu(U_n)
$), hence by Carath\'eodory's criterion we see that $\mu$ is the restriction to open sets of a (not necessarily unique if $\mu(\X)=+\infty$) Borel measure and the equivalence class of this measure is, by \eqref{eq:potildem}, the supremum of the given family. Hence $\widetilde{\mathcal M}_+(\X)$ is a cone.

A non-negative Borel measure $\mu$ on $\X$ induces a (trivially linear) functional $L_\mu$ on $\LSC(\X)$ via the formula \eqref{eq:Lmu}. It is easy to see that $L_\mu=L_\nu$ iff $\mu\sim\nu$ (one implication follows from Cavalieri's formula $\int f\,\d\mu=\int_0^{+\infty}\mu(\{f>t\})\,\d t$ noticing that $\{f>t\}$ is open, while the other from evaluating the functionals at $\nchi_U$, which is in $\LSC(\X)$ for $U\subset\X$ open). Thus $\mu\mapsto L_\mu$ passes to the quotient and induces a map, still denoted $\mu\mapsto L_\mu$, from $\widetilde{\mathcal M}_+(\X)$ to functionals on $\LSC(\X)$.

Then the statement is:
\[
\mu\quad\mapsto \quad L_\mu\quad\text{ is an isomorphism from $\widetilde{\mathcal M}_+(\X)$ to $(\LSC(\X))^*$ satisfying \eqref{eq:normamis}}.
\]
Here the only non-trivial thing is to show that any $L\in\LSC(\X)^*$ comes from (an equivalence class of) a measure $\mu\in\mathcal M_\infty(\X)$. To see this, for given $L$ defined the domain of local finiteness $U_L$ of $L$ as the union of all open subsets $V\subset\X$ such that $L(\nchi_V)<+\infty$. Again the Lindel\"of property ensures that $U_L$ is the union of a countable number of these $V$'s. Applying Theorem \ref{thm:riesz} to each of these $V$'s and via a patching argument we obtain a $\sigma$-finite Borel measure $\mu$ on $U_L$ such that  \eqref{eq:Lmu} holds for all $f\in\LSC(\X)$ that are 0 outside $U_L$. Extending $\mu$ by setting it to be $+\infty$ on Borel sets $E\subset\X$ not contained in $U_L$ we get the desired surjectivity.
}\fr\end{remark}
\begin{remark}[The completion of $\mathcal M_+(\X)$]\label{re:complmx}{\rm
It is worth to point out that in general $\widetilde{\mathcal M}_+(\X)$ is \emph{not} the directed completion of ${\mathcal M}_+(\X)$ (being intended that the map is  the natural inclusion plus passage to the quotient). Intuitively, this might be guessed from the fact that  $\mathcal M_+(\X)$ only depends on the Borel $\sigma$-algebra of $\X$, whereas the equivalence relation $\sim$ in \eqref{eq:defbarm} refers to its topology. More concretely, say $\X=\R$, let $(r_n)\subset\R$ be an enumeration of the rationals and $\tilde x\notin\Q $. Let $(\mu_n)\subset\mathcal M_+(\X)$ be the increasing sequence given by $\mu_n:=\sum_{i\leq n}\delta_{r_i}$ and notice that its supremum in $\widetilde{\mathcal M}_+(\X)$ is (the equivalence class of) the measure giving infinite mass to every non empty subset, call it $[\infty]$, which is the maximal element in $\widetilde{\mathcal M}_+(\X)$. In particular we have $\delta_{\tilde x}\leq[\infty]$. However it is clear -- see also the discussion below about the truncation property -- that the set $\widehat{\{\mu_n:n\in\N\}}\subset\mathcal M_+(\R)$ coincides with $\downarrow\!\{\mu_n:n\in\N\}$ and thus only consists of measures concentrated on $\Q$. In particular, $\delta_{\tilde x}\notin \widehat{\{\mu_n:n\in\N\}}$, showing that the implication $\Rightarrow$ in \eqref{eq:criterio} does not hold, thus proving our claim that $\widetilde{\mathcal M}_+(\X)$ is not the completion of ${\mathcal M}_+(\X)$.

To describe such completion, let us study the collection $ {\mathcal M}_\infty(\X)$ of all  non-negative Borel measures on $\X$. We claim that the partial order induced by the natural sum operation  coincides with that given by evaluation on Borel sets, proving also that $ {\mathcal M}_\infty(\X)$ is a wedge, i.e.\ we claim that for any $\mu,\nu\in {\mathcal M}_\infty(\X)$ we have
\begin{equation}
\label{eq:claimmpp}
\text{there is $\xi\in {\mathcal M}_\infty(\X)$ such that }\nu+\xi=\mu\ \qquad\Leftrightarrow\qquad\ \nu(E)\leq\mu(E)\quad\forall E\subset\X\text{ Borel}.
\end{equation}
 The implication $\Rightarrow$ is obvious. For the converse one let $\mathcal F_\mu$ be the collection of Borel sets $E\subset\X$ such that $\mu(E)<+\infty$ and ${\mathcal A}_\mu$ the $\sigma$-algebra it generates. Define $\xi:\mathcal F_\mu\to[0,+\infty)$ by putting $\xi(E):=\mu(E)-\nu(E)$. It is clear that $\xi$ is finitely additive and that if $(E_n)\subset\mathcal F_\mu$ is decreasing so that $\cap_nE_n=\emptyset$ then $\xi(E_n)=\mu(E_n)-\nu(E_n)\to 0$. Building the outer measure associated to such $\xi$ and then applying Caratheodory's criterion using the fact that $\mathcal F_\mu$ is stable by finite unions, intersections and relative difference one can check that  $\xi$ extends to a measure defined on the whole ${\mathcal A}_\mu$. Such $\sigma$-algebra ${\mathcal A}_\mu$ has the following property
 \begin{equation}
\label{eq:tildemathcala}
A\in {\mathcal A}_\mu,\ B\subset A,\ \text{Borel}\qquad\Rightarrow\qquad B\in {\mathcal A}_\mu,
\end{equation}
as can be quickly checked by a monotone class argument on the set of $A$'s for which this holds, noticing that this is trivially true for $A\in\mathcal F_\mu$.

We now extend $\xi$ to the whole Borel $\sigma$-algebra by putting $\xi(E):=+\infty$ for $E\subset \X$ Borel not in $ {\mathcal A}_\mu$. The fact that $\nu(E)+\xi(E)=\mu(E)$ for any $E$ Borel is trivial, so to conclude it suffices to prove that $\xi$ is $\sigma$-additive. Let thus $(E_n)$ be a disjoint sequence of Borel sets and notice that if $(E_n)\subset  {\mathcal A}_\mu$, then $\cup_nE_n\in {\mathcal A}_\mu$ and $\sigma$-additivity is already known. If instead we have $E_n\notin {\mathcal A}_\mu$ for some $n$, then by \eqref{eq:tildemathcala} we have that $\cup_nE_n\notin {\mathcal A}_\mu$ and thus we have $\sum_i\xi(E_i)=+\infty=\xi(\cup_nE_n)$, concluding the proof of \eqref{eq:claimmpp}.

From \eqref{eq:claimmpp} we see that the inclusion of  ${\mathcal M}_+(\X)$ in $ {\mathcal M}_\infty(\X)$ is an order isomorphism: we claim that 
\begin{equation}
\label{eq:dsclm+}
\text{the directed-sup-closure $\overline{\mathcal M}_+(\X)\subset {\mathcal M}_\infty(\X) $ of $\mathcal M_+(\X)$ is the completion of $\mathcal M_+(\X)$ }
\end{equation}
(together with the inclusion, obviously). To see this start recalling that ${\mathcal M}_\infty(\X)$ is a complete lattice: for $(\mu_i)_{i\in I}\subset {\mathcal M}_\infty(\X)$ the formulas
\begin{equation}
\label{eq:joinmeetmeas}
\begin{split}
(\sup_i\mu_i)(E)&:=\sup\big\{\sum\mu_{i_n}(E_n)\ :\ (i_n)\subset I\ (E_n)\text{ Borel partition of }E\big\},\\
(\inf_i\mu_i)(E)&:=\inf\big\{\sum\mu_{i_n}(E_n)\ :\ (i_n)\subset I\ (E_n)\text{ Borel partition of }E\big\},
\end{split}
\end{equation}
define the $\sup$ and $\inf$ of the family (recall \eqref{eq:claimmpp}). Since clearly $\mathcal M_+(\X)$ is a lower subset of  ${\mathcal M}_\infty(\X)$, by Proposition \ref{prop:trunccompl} our claim will be proved if we show that $\mathcal M_+(\X)$ truncates   ${\mathcal M}_\infty(\X)$.

To see this  let $(\mu_i)\subset  {\mathcal M}_\infty(\X)$ be a directed family, $\mu$ its supremum (that satisfies $\mu(E)=\sup_i\mu_i(E)$ for any $E\subset\X$ Borel) and let $\nu\in\mathcal M_+(\X)$ be with $\nu\leq\mu$. We need to prove that the supremum of the directed family $(\nu\wedge \mu_i)$  is $\geq\nu$. To see this,  notice that since $\nu$ is a finite measure and $\nu\wedge\mu_i\leq\nu$ for every $i\in I$, the Radon-Nikodym derivative $\rho_i:=\tfrac{\d(\nu\wedge\mu_i)}{\d \nu}$ is well defined for every $i\in I$ and so is its $\nu$-essential supremum $\rho$ (as in Lemma \ref{le:esssup}). The very definition of $\nu$-essential supremum ensures that $\tilde \nu:=\rho\nu$ is the supremum of $(\nu\wedge\mu_i)$. Proving that $\tilde\nu\geq\nu$ is equivalent to proving that $\rho\geq 1$ $\nu$-a.e.. Say not. Then for some $\eps>0$ there is $E\subset\X$ Borel with $\nu(E)>0$ and $\rho\leq 1-\eps$  $\nu$-a.e.\ on $E$. It is then obvious   that the restriction of $\nu\wedge\mu_i$ to $E$ coincides with the restriction of $\mu_i$ to the same set, implying that $\mu(E)=\sup_i\mu_i(E)\leq(1-\eps)\nu(E)$, which contradicts the assumption $\nu\leq\mu$.

\medskip

It is worth to point out that  $\overline{\mathcal M}_+(\X)$ is \emph{not} the whole ${\mathcal M}_\infty(\X)$. An example of measure in ${\mathcal M}_\infty(\X)\setminus \overline{\mathcal M}_+(\X)$ is the comeager measure $\mu_{\sf cm}$ defined as
\[
\mu_{\sf cm}(E):=\left\{
\begin{array}{ll}
0,&\quad\text{ if $E$ is meager},\\
+\infty,&\quad\text{ if not}
\end{array}
\right.
\]
(recall that a set is meager if it is a countable union of closed sets with empty interior). The problem with this measure is that any finite positive Borel measure that is 0 on all meager sets is the zero measure.  Proving this is the same as proving that for any  atomless Borel probability measure $\nu$ there is a closed set with empty interior and positive measure. This, however, is standard, as by inner regularity there is a compact set $K$ with $\nu(K)>0$ and the atomless condition ensures that $\nu(B_r(x))\downarrow0$ as $r\downarrow0$ for any $x\in\X$. Hence if $(x_n)\subset\X$ is countable and dense and $(r_n)\subset(0,1)$ are small enough we have $\sum_n\nu(B_{r_n}(x_n))<\nu(K)$ and thus $K\setminus \cup_nB_{r_n}(x_n)$ is compact, with empty interior and positive $\nu$-mass, as desired.

It is therefore interesting to look for ways to characterize measures in $\overline{\mathcal M}_+(\X)$. A sufficient condition for $\mu\in \mathcal M_\infty(\X)$ to be in  $\overline{\mathcal M}_+(\X)$  is that each $E\subset\X$ Borel with $\mu(E)=+\infty$  either contains a point $x$ with $\mu(\{x\})=+\infty$ or a Borel subset $F\subset E$ with $\mu(F)\in(0,+\infty)$. Indeed, since the collection of $\nu$'s in $\mathcal M_+(\X)$ with $\nu\leq\mu$ is directed (because the join of any two such measures is still $\leq\mu$ by the first formula in \eqref{eq:joinmeetmeas}) to prove the claim it suffices to show that for each $E\subset\X$ Borel and $M\in[0, \mu(E))$ there is $\nu\in\mathcal M_+(\X)$ with $\nu(E)\geq M$ and $\nu\leq\mu$. If $\mu(E)<+\infty$ this is obvious (just take $\nu:=\mu\restr E$) otherwise we use the assumption on $\mu$: if $E$ contains a point $x$ of infinite mass we pick $\nu:=M\delta_x$, if instead contains a subset $F\subset E$ with $\mu(F)\in(0,+\infty)$ we pick $\nu:=\tfrac{M}{\mu(F)}\mu\restr F$.

All this clarifies that  the property $\mu\in\overline{\mathcal M}_+(\X)$ is strictly related to properties of the $\sigma$-ideal  $\mathcal I_\mu$ generated by sets of finite $\mu$-measure. A notable example of measure in  $\overline{\mathcal M}_+(\X)$ not covered by what said in the previous paragraph is the cocountable measure $\mu_{\sf cc}$ defined as
\begin{equation}
\label{eq:mucc}
\mu_{\sf cc}(E):=\left\{
\begin{array}{ll}
0,&\quad\text{ if $E$ is countable},\\
+\infty,&\quad\text{ if not}.
\end{array}
\right.
\end{equation}
To see that $\mu_{\sf cc}\in \overline{\mathcal M}_+(\X)$, up to scaling we need to prove that if $E\subset\X$ is Borel and uncountable, then there is a probability measure $\nu$ concentrated on $E$ without atoms. 
To see this, recall (we refer  to  \cite[Theorem 13.6]{Kec95} for these results) that   since $E$ is Borel, then there is a Polish topology $\tau'\supset\tau$ inducing the same Borel structure  for which $E$ is both open and closed. Since $E$ is uncountable, it contains a $\tau'$-perfect set and thus a $\tau'$-homeomorphic copy of the Cantor set (hence also a $\tau$-homeomorphic copy, as $\tau'\supset\tau$). Then  we can take $\nu$ to be   the push-forward via such homeomorphism of the standard atomless probability measure concentrated on the Cantor set.

It might be worth to add that the completion of ${\mathcal M}_+(X)$ contains the finite part of the cone ${\mathcal M}_\infty(X)$, proving also that   ${\mathcal M}_\infty(X)$ is not the completion of its finite part. To prove the claim above it suffices to prove that if a non-negative measures $\mu$ is so that $\eps \mu=0$, then $\mu$ must be $\sigma$-finite, as clearly $\sigma$-finite measures are in the completion of finite ones. Thus let  $\mu\in{\mathcal M}_\infty(X)$ be arbitrary, consider as above the family  $\mathcal F_\mu$ of Borel sets to which $\mu$ assigns finite measure and then let $\mathcal I_\mu$ be the $\sigma$-ideal generated by it, i.e.\ the collection of Borel sets that can be covered by a countable collection of sets in  $\mathcal F_\mu$. Define the Borel measure $\nu$ by putting $\nu(E):=0$ if $E\in\mathcal I_\mu$ and $\nu(E):=+\infty$ otherwise. It is clear that this is a measure and that $\eps\mu=\nu$, showing that $\eps\mu=0$  if and only if $\X\in \mathcal I_\mu$, i.e.\ if and only if $\mu$ is $\sigma$-finite, as claimed.
}\fr\end{remark}

\begin{remark}[First countability and truncation do not pass to completion]\label{re:noalcomplet}{\rm
We  notice that\linebreak $\mathcal M_+(\X)$ is first countable (recall Definition \ref{def:sep}) and truncates its completion (recall Definition \ref{def:trunc}), while its completion $\overline{\mathcal M}_+(\X)$ is not first countable and does not truncate itself.

We briefly discuss these claims. The fact that  $\mathcal M_+(\X)$ is first countable follows from Lemma \ref{le:esssup}  noticing that  a directed family  $(\nu_i)$ with $\nu_i\leq\nu\in\mathcal M_+(\X)$ has $\nu$ as supremum if and only if the $\nu$-essential supremum of the Radon-Nikodym densities $\tfrac{\d\nu_i}{\d\nu}$ is $\nu$-a.e.\ equal to 1. The fact that $\mathcal M_+(\X)$ truncates $\overline{\mathcal M}_+(\X)$ has been proved in Remark \ref{re:complmx} above.

For the claims about $\overline{\mathcal M}_+(\X)$ let   $\X:=[0,1]$ be with the standard topology and $D$ be the collection of `finite counting measures', i.e.\ measures of the form $\sum_{j=1}^n\delta_{x_j}$ for some $n\in\N$ and some disjoint collection of points $x_1,\ldots,x_n$. This is clearly directed with supremum the counting measure. Clearly, any measure in $\downarrow\!D$ is concentrated on a finite  set and thus the sup of a non decreasing sequence $(\mu_n)\subset\,\downarrow\!D$ is concentrated on a countable set, and thus must differ from the sup of $D$.

To see that  $\overline{\mathcal M}_+(\X)$  does not truncate itself we pick the same $D$ as above and consider the measure $\mu_{\sf cc}\in \overline{\mathcal M}(\X)$ defined in \eqref{eq:mucc}. Then clearly $\mu_{\sf cc}\leq\sup D$, but for every $\mu\in D$ we have $\mu_{\sf cc}\wedge\mu=0$, as  if $F$ is the finite set  where $\mu$ is concentrated, for each $E\subset[0,1]$ we have $(\mu_{\sf cc}\wedge\mu)(E)\leq\mu_{\sf cc}(F)+\mu(E\setminus F)=0$.
}\fr\end{remark}

\begin{remark}{\rm We mentioned in Section \ref{se:seqcompl} the possibility of imposing only the existence of the supremum of increasing sequences and the related sequential $\mcp$. One can then interpret the definition of cone and morphism with this nuance. With this choice a natural example that comes out is: the cone of measurable $[0,+\infty]$-valued functions in a given measurable space $(\X,\mathcal A)$. Its dual (i.e.\ collection of linear maps that respect the supremum of increasing sequences) is then the collection of all non-negative measures on $\mathcal A$: to the functional $L$ we associate the measure $\mu$ defined as $\mu(E):=\nchi_E$ for every $E\in\mathcal A$, its $\sigma$-additivity being a consequence of the sequential $\mcp$ of $L$.
}\fr\end{remark}

\begin{remark}{\rm  So far we have only discussed examples of cones of functions with 0-th order  regularity imposed, but in fact there are very natural examples of wedges/cones of functions defined via first or second order bounds.

For first order, consider the wedge of real valued causal functions on Minkowski spacetime (or any other causal spacetime) modulo constants. Its completion should likely be related to  causal $[-\infty,+\infty]$-valued functions, a topic investigated, precisely for closure/stability considerations, in \cite{Octet} (but the link between these studies and the directed completion is unclear).

For second order, consider the cone of $[0,+\infty]$-valued convex and lower semicontinuous functions on $\R^d$ (or any other geodesic metric space).

It is easy to see that the duals of these wedges/cones contains functionals representable via first/second order distributions, as expected. The precise description of such duals, though, will not be pursued here.
}\fr\end{remark}

\end{subsection}

\begin{subsection}{Matrix considerations}\label{se:dario}
This example is due to D.\ Trevisan.

\bigskip

For  $d\in\N$, $d\geq1$ we shall denote by $\O(\mathbb C^d)$ the space of $d\times d$ Hermitian matrices, by $\O_+(\mathbb C^d)\subset \O(\mathbb C^d)$ the cone of non-negative ones (w.r.t.\ the standard order $A\leq B$ iff $B-A$ has positive spectrum) and by $\O_{++}(\mathbb C^d)\subset \O_+(\mathbb C^d)$ the subset of invertible ones. Similar  examples emerge by considering real non-negative symmetric matrices or other more general structures where spectral calculus is available.

It is clear that $\O_+(\mathbb C^d)$ is a wedge and that it is locally complete. The latter can be seen for instance noticing that to each $A\in \O(\mathbb C^d)$ we can associate the real-valued quadratic form $Q_A:\mathbb C^d\to\R$ defined as
\[
\mathbb C^d\ni v\quad\mapsto\quad Q_a(v):=\la v,Av\ra=\sum_{ij}\bar v_ia_{ij}v_j
\]
and then noticing that $A\leq B$ if and only if $Q_A(v)\leq Q_B(v)$ for any $v\in\mathbb C^d$.
\begin{remark}{\rm
One can use the order preserving  correspondence $A \leftrightarrow Q_A$ to identify the completion of $\O_+(\mathbb C^d)$ as the cone of quadratic forms from $\mathbb C^d$ to $[0,+\infty]$, i.e.\ as the collection  of those maps $Q:\mathbb C^d\to[0,+\infty]$ such that
\[
\begin{split}
Q(\eta v)&=|\eta|^2 Q(v),\\
Q(v+w)+Q(v-w)&=2(Q(v)+Q(w)),
\end{split}
\]
for every $v,w\in\mathbb C^d$ and $\eta\in\mathbb C$.

Alternatively, to the same end we can use spectral analysis and characterize the completion as the collection of formal sums $\sum_i\lambda_ie_i\otimes e_i$ for $(e_i)$ orthonormal base of $\mathbb C^d$ and $\lambda_i\geq 0$. Here two such sums are declared equivalent if they coincide `after evaluation' in the natural sense. Writings of this form can be used to give a meaning to expressions, that we might encounter below, such as $A^p$ for $A\in \O_+(\mathbb C^d)$ not invertible and $p<0$. We will not follow this route.
}\fr
\end{remark}
For  $A\in \O(\mathbb C^d)$ we write $A=\sum_i\lambda_ie_i\otimes e_i$ for its spectral decomposition, i.e.\ with this we intend that $(e_i)$ is an orthonormal base of $\mathbb C^d$ and $(\lambda_i)\subset\R$. If  $A\in \O_{++}(\mathbb C^d)$, hence $\lambda_i>0$ for every $i$, spectral calculus can  be used to give a meaning to $A^p$ for any $p\in(-\infty,1)\setminus\{0\}$: the spectral decomposition gives
\[
A^p=\sum_i\lambda_i^pe_i\otimes e_i\qquad\text{ whenever }\qquad A=\sum_i\lambda_ie_i\otimes e_i. 
\]
We have:
\begin{lemma}\label{le:youngmatr}
Let $p,q\in(-\infty,1)\setminus\{0\}$ be with $\tfrac1p+\tfrac1q=1$. Then for every $A,B\in\O_{++}(\mathbb C^d)$ we have
\begin{equation}
\label{eq:youngmatr}
\Tr ( AB)\geq\tfrac1p\Tr(A^p)+\tfrac1q\Tr(B^q)
\end{equation}
with equality if and only if $A^p=B^q$.
\end{lemma}
\begin{proof} Fix $B\in \O_{++}(\mathbb C^d)$ and notice that by the general result  of Davis \cite{Davis57} the map $\O(\mathbb C^d)\ni A\mapsto f(A):=\Tr(AB)-\tfrac1p\Tr(A^p)$, set to $+\infty$ if $A\notin \O_{++}(\mathbb C^d)$, is convex. The restriction of $f$ to $\O_{++}(\mathbb C^d)$ is smooth: we shall show that in there the differential is 0 if and only if $A^p=B^q$. Since for such $A$ we have equality in \eqref{eq:youngmatr}, the conclusion then follows.

The computation of the  differential of $f$ is standard.  To start, notice that for arbitrary $C\in\O(\mathbb C^d)$ the spectral decomposition $C=\sum_i\lambda_ie_i\otimes e_i$ gives  $\Id+\eps C=\sum_i(1+\eps\lambda_i)e_i\otimes e_i$ and thus  
\begin{equation}
\label{eq:aallap}
(\Id+\eps C)^p=\sum_i(1+p\eps\lambda_i+o(\eps))e_i\otimes e_i=\Id+p\eps C+o(\eps).
\end{equation}
Hence for $A\in \O_{++}(\mathbb C^d)$ and $C\in \O(\mathbb C^d)$, recalling that $\Tr(AC)=\Tr(CA)$ we have
\[
\Tr((A+\eps C)^p)=\Tr(A^p(\Id+\eps A^{-\frac12}CA^{-\frac12})^p)\stackrel{\eqref{eq:aallap}}=
\Tr(A^p)+\eps p\Tr(A^{p-1}C)+o(\eps).
\]
It follows that
\[
f(A+\eps C)=f(A)+\eps(\Tr(CB)-\Tr(A^{p-1}C))+o(\eps)=f(A)+\eps\Tr\big(C(B-A^{p-1})\big)+o(\eps)
\]
and the claim follows.
\end{proof}
For $A\in \O_+(\mathbb C^d)$, $A=\sum_i\lambda_ie_i\otimes e_i$, and $p\in[-\infty,1]$ we define $\|A\|_p\in[0,+\infty)$ as
\begin{equation}
\label{eq:defap}
\|A\|_p:=\text{$L^p$-norm of $\lambda_\cdot:\{1,\ldots,d\}\to[0,+\infty)$ w.r.t.\ the measure $\tfrac1d\sum_{i=1}^d\delta_i$}
\end{equation}
(notice that the $L^{0_+}$ and $L^{0_-}$ norms coincide here). If either $A$ is invertible or $p\in[0,1]$ we have
\[
\begin{array}{ll}
\|A\|_p\!\!\!&=\big(\tfrac1d\,\Tr(A^p)\big)^{\frac1p}\qquad\text{ for }p\in(-\infty,1]\setminus\{0\},\\
\|A\|_0\!\!\! &=(\det(A))^{\tfrac1d}\\
\|A\|_{-\infty}\!\!\!&=\text{minimal eigenvalue}.
\end{array}
\]
We then have:
\begin{proposition}[Duality formula]
Let $p,q\in[-\infty, 1]$ be with $\tfrac1p+\tfrac1q=1$ (the choice $p=q=0$ is admissible in here). Then
\begin{equation}
\label{eq:ineqmatr}
\tfrac1d\Tr(AB)\geq \|A\|_p\|B\|_q\qquad\forall A,B\in\O_+(\mathbb C^d)
\end{equation}
and more precisely
\begin{equation}
\label{eq:dualmatr}
\|A\|_p=\inf_{\|B\|_q\geq1}\tfrac1d\Tr(AB)\qquad\forall A\in\O_+(\mathbb C^d).
\end{equation}
\end{proposition}
\begin{proof} We start with \eqref{eq:ineqmatr}. To see this, notice that by continuity we can assume $A,B\in\O_{++}(\mathbb C^d)$ and $p,q\in(-\infty,1)\setminus \{0\}$ and then by homogeneity we can replace $A,B$ with $\tfrac{A}{\|A\|_p},\tfrac{B}{\|B\|_p}$ respectively. Then the claim directly follows from Lemma \ref{le:youngmatr}.

We shall show that  equality holds in \eqref{eq:dualmatr} even if we restrict the collection of $B$'s to those that are diagonalizable together with $A$. Fixing an orthonormal base on which $A$, and thus also $B$, is diagonalizable and calling $(\lambda_i),(\eta_i)$ the eigenvalues of $A,B$ respectively, we want to prove that
\begin{equation}
\label{eq:lplqagain}
\|(\lambda_i)\|_{p}=\inf_{\|(\eta_i)\|_q\geq 1}\tfrac1d\sum_i\lambda_i\eta_i,
\end{equation}
where the $p,q$-norms are taken in the set $\{1,\ldots,d\}$ equipped with the probability measure  $\tfrac1d\sum_{i=1}^d\delta_i$. Then formula  \eqref{eq:lplqagain} is just a special case of  formula \eqref{eq:lpdual}.
\end{proof}

\begin{corollary}
For any $p\in[-\infty,1]$ the $p$-norm $\|\cdot\|_p$ is a hyperbolic norm on $\O_+(\mathbb C^d)$.
\end{corollary}
\begin{proof}
Homogeneity is clear, thus we need only to prove the reverse triangle inequality. This is a direct consequence of the duality formula \eqref{eq:dualmatr}, as for any $A_1,A_2\in\O_+(\mathbb C^d)$ we have
\[
\begin{split}
\|A_1+A_2\|_p&=\inf_{\|B\|_q\geq 1}\tfrac1d\Tr((A_1+A_2)B)=\inf_{\|B\|_q\geq 1}\Big(\tfrac1d\Tr(A_1B)+\tfrac1d\Tr(A_2B)\Big)\\
&\geq\inf_{\|B\|_q\geq 1}\tfrac1d\Tr(A_1B)+\inf_{\|B\|_q\geq 1}\tfrac1d\Tr(A_2B)=\|A_1\|_p+\|A_2\|_p,
\end{split}
\]
being intended that $\tfrac1p+\tfrac1q=1$.
\end{proof}
Notice that in the case $p=0$ the reverse triangle inequality reduces to the well known inequality
\[
\det(A+B)^{\tfrac1d}\geq\det(A)^{\tfrac1d}+\det(B)^{\tfrac1d}\qquad\forall A,B\in \O_+(\mathbb C^d).
\]

\end{subsection}

\begin{subsection}{The Brunn-Minkowski inequality as reverse triangle inequality}\label{se:antonio}

This example is due to A.\ Lerario.

\bigskip

Let  $d\in\N$, $d\geq 1$, and let $\W$ be the   any of the following sets:
\[
\begin{split}
&\{\text{Souslin subsets of }\R^d\},\\
&\{\text{compact subsets of }\R^d\},\\
&\{\text{open subsets of }\R^d\}\cup\big\{\{0\}\big\},\\
&\{\text{convex subsets of }\R^d\},\\
&\{\text{open convex subsets of }\R^d\}\cup\big\{\{0\}\big\}.
\end{split}
\]
We equip $\W$ with the operations
\begin{equation}
\label{eq:minkadd}
\begin{split}
A+B&:=\{a+b\ :\ a\in A,\ b\in B\},\\
\lambda A&:=\{\lambda a\ :\ a\in A\},
\end{split}
\end{equation}
that clearly take values in $\W$.  Also, we  equip $\W$  with the `norm' 
\[
\|A\|:=\big(\mathcal L^d(A)\big)^{\tfrac1d}.
\]
It is  clear that such norm is homogeneous and that the reverse triangle inequality
\[
\|A+B\|\geq \|A\|+\|B\|
\]
is nothing but the classical Brunn-Minkowski inequality, that in particular takes this new interpretation in the current framework.

Notice that with the exception of the collection of convex or of  open convex sets (and $\{0\}$), the structures $\W$ above are not those of prewedges. Indeed, the distributive law 
\begin{equation}
\label{eq:aconvesso}
(\lambda+\eta )A=\lambda A+\eta A.
\end{equation}
fails to hold in general if $A$ is not convex, the problem being the inclusion $\supset$ (if instead $A$ is convex then \eqref{eq:aconvesso} holds as with a scaling argument we can  reduce to the case $\lambda+\eta=1$ and then use the convexity of $A$ to conclude that for $x,y\in A$ we have $\lambda x+\eta y\in A$).

Notice also that  the Minkowski sum does not induce a partial order in the case of Souslin / compact / convex sets. Indeed, for any $v,w\in\R^d$ we obviously have $\{v\}+\{w-v\}=\{w\}$ and $\{w\}+\{v-w\}=\{v\}$, so that $\{v\}\leq \{w\}$ and  $\{w\}\leq \{v\}$ without having $\{v\}=\{w\}$. 
%

\end{subsection}

\begin{subsection}{Some finite dimensional examples}
\label{se:finitedim}

Recall that given a commutative monoid $M$, its Grothendick group $(G,\iota)$  is defined as a commutative group $G$ together with a monoid morphism $\iota:M\to G$ universal in the following sense:  for any commutative group $G'$ and monoid morphism $\iota':M\to G'$ there is a unique  group homeomorphism $\Psi:G\to G'$ making the following diagram commute:
\begin{equation}
\label{eq:diagramminoGrot}
\begin{tikzcd}[node distance=1.5cm, auto]
M\arrow{dr}[swap]{\iota'} \arrow{r}{\iota}&G\arrow{d}{\Psi}\\
& G'
\end{tikzcd}
\end{equation}
Uniqueness up to unique isomorphism comes from the definition and existence from simple constructions (e.g.\ take the free abelian group generated by $M$ and quotient it by the subgroup generated by $[a+b]-[a]-[b]$ as $a,b$ vary in $M$, $[a]$ being the generator corresponding to $a$, see for instance \cite[Section I.7]{lang84} for more details).

A wedge $\W$ is, in particular, a commutative monoid under sum, hence it admits a Grothendick group $V$. It is easy to see that such $V$ carries also a natural structure of vector space, the product by a scalar $\lambda\in\R$ being defined as follows. If $\lambda\geq 0$ (resp.\ $\lambda\leq 0$) we pick $G':=V$ and $\iota'(v):=\iota(\lambda v)$ (resp.\ $\iota'(v):=-\iota((-\lambda )v)$) for $v\in \W$. Then $\iota'$ is a monoid morphism and thus the universal property gives the existence of a unique (additive) group morphism $\Psi_\lambda:V\to V$ making \eqref{eq:diagramminoGrot} commute. Routine manipulation show that the given sum and the product by scalar defined as $\lambda v:=\Psi_\lambda(v)$ give $V$ the structure of a vector space.

It is not hard to see that the map $\iota:M\to G$ coming with the definition of Grothendick group is injective if and only if the monoid is cancellative (see e.g.\ \cite[Section I.7]{lang84} again for one implication, the other is obvious). In our setting, this surely is not necessarily the case: for instance, if $\W$ is a cone then $\infty$ is an absorbing element for the addition (i.e.\ \eqref{eq:abs} holds) and thus its Groethendick group $V$ reduces to the sole $\{0\}$, as it is immediate to check. 

Hence, in light of \eqref{eq:cancellationtrue} if we want to associate a vector space to a wedge it is better to restrict to elements in the finite part, i.e.\ elements $v$ such that $\eps v=0$ (see also Remark \ref{re:finitepart}). We are therefore led to the following definition (see also \cite{KO25}):
\begin{definition}[Vector space associated to a wedge]
Let $\W$ be a wedge and assume that its finite part $\W_{\rm finite}$ is also a wedge (equivalently: it is closed by sum) and has the cancellation property (by  \eqref{eq:cancellationtrue} this holds automatically if $\W$ is a cone with joins).

Then we define the vector space associated to $\W$ as the Groethendick group $V$ of $\W_{\rm finite}$.

The dimension of such $\W$ is, by definition, that of the real vector space $V$.
\end{definition}
Recall that we don't have examples of wedges whose finite part is not a wedge. The definition of dimension as just given is necessary as the concepts of linear independence and generating set work well only when one is allowed to freely take differences of vectors, something that in wedges is not allowed. For instance: there is no finite set $S$ in the future cone in the standard 4-dimensional Minkowski spacetime  such  that any vector of the cone is a linear combination with positive coefficients of the elements of $S$. This should not distract us into thinking that such cone in infinite dimensional, rather than  4-dimensional, as it is also according to the definition just given.

From standard linear algebra we know that up to isomorphism there is only  one vector space of given finite dimension $n$. This is far from true for wedges, even if we restrict the attention to those coinciding with their finite part. For instance, already in dimension 3 we can find several convex cones in $\R^3$ that are not isomorphic as wedges (notice that  any wedge isomorphism extends to an  invertible linear map on $\R^3$).

There is less freedom on two dimensional wedges, but still there are a few non-isomorphic ones. We collect the examples below in the spirit of guiding possible future investigation and axiomatizations, also pointing out at potential pathologies. In what follows, operations will always be defined componentwise and to $(\lambda,\eta)\in[0,+\infty]^2$ we shall associate the linear functional $L_{\lambda,\eta}$ sending $(a,b)$ to $\lambda a+\eta b$.
\begin{enumerate}[label=\emph{\alph*)}]
\item  Let $\C:=[0,+\infty]^2$. Then this is a cone and $\C^*\sim [0,+\infty]^2$ (via the map $(\lambda,\eta)\mapsto L_{\lambda,\eta}$). This is trivial to check.
\item Let $\C:=\{0\}\cup(0,+\infty]^2$. Then this is a cone and $\C^*\sim [0,+\infty]^2$ (via the map $(\lambda,\eta)\mapsto L_{\lambda,\eta}$). This is trivial to check and, together with the previous example, shows that two non isomorphic cones can have isomorphic duals. For the same reason,  finite dimensional cones that are the completion of their finite part can be non-reflexive. 

\item Let $\C:=[0,+\infty)^2\cup\{+\infty\}$. Then this is a cone and $\C^*\sim\{0\}\cup (0,+\infty)^2\cup\{+\infty\}$ (via the map $(\lambda,\eta)\mapsto L_{\lambda,\eta}$; notice in particular that $\C^*$ is a cone). Indeed, clearly linear functionals on $\C$ must be of the form $L_{\lambda,\eta}$ for some $\lambda,\eta\in[0,+\infty]$ (properly defined at $+\infty\in \C$). We check when they have the $\mcp$: for non-zero functionals this reduces to checking whether $L_{\lambda,\eta}(a_n,b_n)\to+\infty$ whenever $\sup_i(a_i,b_i)=+\infty$, i.e.\ whenever either $\sup_ia_i=+\infty$ or $\sup_ib_i=+\infty$ (or both). If $\lambda,\eta>0$ this is clear, so to conclude by symmetry and scaling it suffices to check  that $L_{0,1}$ has not the $\mcp$. To see this pick $a_n:=n$ and $b_n:=0$ for every $n$: then $n\mapsto(a_n,b_n)$ is increasing in $\C$ with $\sup_n(a_n,b_n)=+\infty$ but $L_{0,1}(a_n,b_n)=0$ for every $n$.
\item Let $\C:=\{0\}\cup (0,+\infty)^2\cup\{+\infty\}$. Then $\C^*\sim \{0\}\cup (0,+\infty)^2\cup\{+\infty\}$ (via the map $(\lambda,\eta)\mapsto L_{\lambda,\eta}$). To prove the claim we need to check, as before, that $L_{0,1}$ does not have the $\mcp$. To see this let $a_n:=n$ and  $b_n:=1-\tfrac1n$. Then $n\mapsto(a_n,b_n)$ is increasing w.r.t.\ the partial order of $\C$ and has supremum $+\infty$ but $L_{0,1}(a_n,b_n)= 1-\tfrac1n \to1 <+\infty$.
\item Let $\W:=\{0\}\cup (\R\times(0,+\infty)\} $. Then $\W$ is a two dimensional locally complete wedge whose dual is one dimensional. Indeed, a linear operator on $\W$ clearly extends to a linear operator on $\R^2$ and thus is of the form $L(a,b)=\lambda a+\eta b$ for every $(a,b)\in \W$ and some $\lambda,\eta\in\R$. Then positivity on $\W$ forces $\lambda=0$.

It is then clear that the bidual is also 1-dimensional, and thus that the natural map $\Phi:\W\to\W^{**}$ defined in \eqref{eq:defPhi} might be not injective. In our case, $\Phi(a,b)=\Phi(a',b')$ if and only if $b=b'$.

Notice that the topological closure of $\W$ in $\R^2$ is not a wedge, as it contains the vectors $(-1,0)$ and $(1,0)$ whose sum is the zero element of $\W$.
\item\label{ex:Roman} This example is due to R. Oleinik. Let $\W:=\big(\{0\}\times[0,+\infty)\big) \cup\big((0,+\infty)\times\R\big)\subset\R^2$ be equipped with the natural operations. It is not hard to check that  the induced partial order is the lexicographic one (and in particular it is actually a partial order), i.e.\ that
\begin{equation}
\label{eq:roman}
(a_1,b_1)\leq(a_2,b_2)\text{ in }\W\qquad\Leftrightarrow\qquad \text{either $a_1<a_2$ \ or \ $(a_1=a_2\text{ and }b_1\leq b_2)$.}
\end{equation}
The interest of this example is in that $(1,0)$ is \emph{not} the supremum of $\{\lambda(1,0):\lambda\in(0,1)\}$. Indeed, while clearly $(1,0)$ is an upper bound for the given set, the same holds for  $(1,b)$ for any $b\in\R$. In particular, this holds for $b<0$, and since in this case we have $(1,b)\leq(1,0)$, our claim is proved.

This is an example of structure satisfying $(i),(ii)$ and $(iii)$ in Definition \ref{def:wedge} but not $(iv)$.
\end{enumerate}
We conclude pointing out that for a sane development of the theory it seems natural to restrict the attention to cones with dense finite part. Among other things, if this fails we can encounter situations where we have multiple `layers of infinity', as we discuss now. Let $\W:=\N$ be equipped with the operations:
\[
\begin{array}{rll}
n+m&\!\!\!:=\max\{n,m\},\qquad&\forall n,m \in\N\\
\lambda n&\!\!\!:=n\qquad&\forall \lambda\in(0,+\infty),\ n\in\N.
\end{array}
\]
Then quite clearly this is a wedge, its finite part being $\{0\}$. More generally, let $(\W_n)$ be a sequence of wedges and equip $\W:=\cup_{n\in\N}\{n\}\times \W_n$ with the operations:
\[
\begin{array}{rll}
(n,v)+(m,w)&\!\!\!:=\left\{\begin{array}{ll}
(n,v+w),&\qquad\text{ if }n=m,\\
(n,v),&\qquad\text{ if }m<n,
\end{array}
\right.\\
\lambda (n,v)&\!\!\!:=(n,\lambda v)
\end{array}
\]
for every $(n,v),(m,w)\in \W$ and $\lambda\in(0,+\infty)$. Then once again $\W$ is a wedge whose finite part is $\{0\}\times`\text{finite part of $\W_0$}$'. In particular, if $\W_0$ is finite dimensional then so is $\W$ according to the definition above, even though $\W_n$ might be infinite dimensional for every $n>0$. Notice also that if all the $\W_n$'s are cones, then $\W\cup\{\infty\}$ is also a cone. Observe that here $\N$ can be replaced by any other total order admitting minimum.
\end{subsection}

\end{section}

%
%

\begin{thebibliography}{10}

\bibitem{AbstrConcrCat}
{\sc J.~r. Ad\'amek, H.~Herrlich, and G.~E. Strecker}, {\em Abstract and
  concrete categories: the joy of cats}, Repr. Theory Appl. Categ.,  (2006),
  pp.~1--507.
\newblock Reprint of the 1990 original [Wiley, New York; MR1051419].

\bibitem{AB85}
{\sc C.~D. Aliprantis and O.~Burkinshaw}, {\em Positive operators}, vol.~119 of
  Pure and Applied Mathematics, Academic Press, Inc., Orlando, FL, 1985.

\bibitem{Octet}
{\sc T.~Beran, M.~Braun, M.~Calisti, N.~Gigli, A.~Ohanyan, R.~J. McCann,
  F.~Rott, and C.~S\"{a}mann}, {\em A nonlinear d'{A}lembert comparison theorem
  and causal differential calculus on metric measure spacetimes}.
\newblock Preprint, arXiv: 2408.15968.

\bibitem{Bogachev07}
{\sc V.~Bogachev}, {\em Measure theory. {V}ol. {I}, {II}}, Springer-Verlag,
  Berlin, 2007.

\bibitem{BraunPers}
{\sc M.~Braun}, {\em New perspectives on the d'{A}lembertian from general
  relativity. an invitation}, Indagationes Mathematicae,  (2025).

\bibitem{CavMonLorSurvey}
{\sc F.~Cavalletti and A.~Mondino}, {\em A review of {L}orentzian synthetic
  theory of timelike {R}icci curvature bounds}, Gen. Relativity Gravitation, 54
  (2022), pp.~Paper No. 137, 39.

\bibitem{CavMonTCD}
\leavevmode\vrule height 2pt depth -1.6pt width 23pt, {\em Optimal transport in
  {L}orentzian synthetic spaces, synthetic timelike {R}icci curvature lower
  bounds and applications}, Camb. J. Math., 12 (2024), pp.~417--534.

\bibitem{Cheeger00}
{\sc J.~Cheeger}, {\em Differentiability of {L}ipschitz functions on metric
  measure spaces}, Geom. Funct. Anal., 9 (1999), pp.~428--517.

\bibitem{Davis57}
{\sc C.~Davis}, {\em All convex invariant functions of hermitian matrices},
  Arch. Math., 8 (1957), pp.~276--278.

\bibitem{EBS24}
{\sc S.~Eriksson-Bique and E.~Soultanis}, {\em Curvewise characterizations of
  minimal upper gradients and the construction of a {S}obolev differential},
  Anal. PDE, 17 (2024), pp.~455--498.

\bibitem{FHM11}
{\sc J.~L. Flores, J.~Herrera, and M.~S\'anchez}, {\em On the final definition
  of the causal boundary and its relation with the conformal boundary}, Adv.
  Theor. Math. Phys., 15 (2011), pp.~991--1057.

\bibitem{CompendiumLattices}
{\sc G.~Gierz, K.~H. Hofmann, K.~Keimel, J.~D. Lawson, M.~W. Mislove, and D.~S.
  Scott}, {\em A compendium of continuous lattices}, Springer-Verlag,
  Berlin-New York, 1980.

\bibitem{Gigli12}
{\sc N.~Gigli}, {\em On the differential structure of metric measure spaces and
  applications}, Mem. Amer. Math. Soc., 236 (2015), pp.~vi+91.

\bibitem{Gigli14}
\leavevmode\vrule height 2pt depth -1.6pt width 23pt, {\em Nonsmooth
  differential geometry---an approach tailored for spaces with {R}icci
  curvature bounded from below}, Mem. Amer. Math. Soc., 251 (2018), pp.~v+161.

\bibitem{Har98}
{\sc S.~G. Harris}, {\em Universality of the future chronological boundary}, J.
  Math. Phys., 39 (1998), pp.~5427--5445.

\bibitem{Jech2003}
{\sc T.~Jech}, {\em Set Theory: The Third Millennium Edition}, Springer, 2003.

\bibitem{JohnstoneScott}
{\sc P.~T. Johnstone}, {\em Scott is not always sober}, in Continuous Lattices,
  B.~Banaschewski and R.-E. Hoffmann, eds., Berlin, Heidelberg, 1981, Springer
  Berlin Heidelberg, pp.~282--283.

\bibitem{Kec95}
{\sc A.~S. Kechris}, {\em Classical descriptive set theory}, vol.~156 of
  Graduate Texts in Mathematics, Springer-Verlag, New York, 1995.

\bibitem{KO25}
{\sc E.~Kharitonov and A.~Ohanyan}, {\em Topological cones and positively
  polarizable hyperbolic norms}.
\newblock Preprint, arXiv: 2508.17052.

\bibitem{KS18}
{\sc M.~Kunzinger and C.~S\"{a}mann}, {\em Lorentzian length spaces}, Ann.
  Global Anal. Geom., 54 (2018), pp.~399--447.

\bibitem{lang84}
{\sc S.~Lang}, {\em Algebra}, Addison-Wesley Publishing Company, Advanced Book
  Program, Reading, MA, second~ed., 1984.

\bibitem{Lott-Villani09}
{\sc J.~Lott and C.~Villani}, {\em Ricci curvature for metric-measure spaces
  via optimal transport}, Ann. of Math. (2), 169 (2009), pp.~903--991.
\newblock arXiv:math/0412127.

\bibitem{MacLane98}
{\sc S.~Mac~Lane}, {\em Categories for the working mathematician}, vol.~5 of
  Graduate Texts in Mathematics, Springer-Verlag, New York, second~ed., 1998.

\bibitem{Mark76}
{\sc G.~Markowsky}, {\em Chain-complete posets and directed sets with
  applications}, Algebra Universalis, 6 (1976), pp.~53--68.

\bibitem{MR03}
{\sc D.~Marolf and S.~F. Ross}, {\em A new recipe for causal completions},
  Classical Quantum Gravity, 20 (2003), pp.~4085--4117.

\bibitem{McCannSurvey}
{\sc R.~McCann}, {\em Trading linearity for ellipticity: a nonsmooth approach
  to {E}instein's theory of gravity and the {L}orentzian splitting theorems}.
\newblock Preprint, arXiv: 2501.00702.

\bibitem{McCann24}
{\sc R.~J. McCann}, {\em A synthetic null energy condition}, Comm. Math. Phys.,
  405 (2024).

\bibitem{MNP91}
{\sc P.~Meyer-Nieberg}, {\em Banach lattices}, Universitext, Springer-Verlag,
  Berlin, 1991.

\bibitem{Ming10}
{\sc E.~Minguzzi}, {\em Time functions as utilities}, Comm. Math. Phys., 298
  (2010), pp.~855--868.

\bibitem{MinSuh24}
{\sc E.~Minguzzi and S.~Suhr}, {\em Lorentzian metric spaces and their
  {G}romov-{H}ausdorff convergence}, Lett. Math. Phys., 114 (2024), pp.~Paper
  No. 73, 63.

\bibitem{MS25}
{\sc A.~Mondino and C.~S\"{a}mann}, {\em Lorentzian {G}romov-{H}ausdorff
  convergence and pre-compactness}.
\newblock Preprint, arXiv:2504.10380.

\bibitem{Peressini67}
{\sc A.~L. Peressini}, {\em Ordered topological vector spaces}, Harper \& Row
  Publishers, New York, 1967.

\bibitem{Rock70}
{\sc R.~T. Rockafellar}, {\em Convex analysis}, vol.~No. 28 of Princeton
  Mathematical Series, Princeton University Press, Princeton, NJ, 1970.

\bibitem{Sturm06I}
{\sc K.-T. Sturm}, {\em On the geometry of metric measure spaces. {I}}, Acta
  Math., 196 (2006), pp.~65--131.

\bibitem{Sturm06II}
\leavevmode\vrule height 2pt depth -1.6pt width 23pt, {\em On the geometry of
  metric measure spaces. {II}}, Acta Math., 196 (2006), pp.~133--177.

\bibitem{Weaver01}
{\sc N.~Weaver}, {\em Lipschitz algebras and derivations. {II}. {E}xterior
  differentiation}, J. Funct. Anal., 178 (2000), pp.~64--112.

\bibitem{DCPO10}
{\sc D.~Zhao and T.~Fan}, {\em Dcpo-completion of posets}, Theoret. Comput.
  Sci., 411 (2010), pp.~2167--2173.

\end{thebibliography}
\def\cprime{$'$} \def\cprime{$'$}

\end{document}